\documentclass[
a4paper,				
12pt,							
headsepline,					
abstracton,					  
]
{scrartcl}

\usepackage[ngerman, english]{babel} 
\usepackage[T1]{fontenc}
\usepackage[ansinew]{inputenc}

\usepackage{lmodern} 

\usepackage[hyphens,obeyspaces,spaces]{url}
\usepackage{hyperref}

\usepackage{color}

\usepackage{amssymb}
\usepackage{mathrsfs} 

\usepackage{amsmath}
\usepackage{amsthm}
\usepackage{amsfonts}

\usepackage{paralist}
\usepackage{stmaryrd}
\usepackage{graphicx}
\usepackage{caption}

\usepackage[arrow, matrix, curve]{xy}

\usepackage{tikz}
\usetikzlibrary{matrix,arrows,decorations.pathmorphing}

\newtheorem{satz}{Theorem}[section]
\newtheorem{lemma}[satz]{Lemma}
\newtheorem{prop}[satz]{Proposition}
\newtheorem{korollar}[satz]{Corollary}

\newtheorem{bem}[satz]{Remark}
\newtheorem{definition}[satz]{Definition}

\newtheorem*{einleitung1}{Merzlyakov's Theorem}
\newtheorem*{gen}{Generalized Merzlyakov's Theorem}
\newtheorem*{tarski}{Tarski-problems}

\newtheorem*{einleitungbeispiel}{Example}
\theoremstyle{definition}
\newtheorem{beispiel}[satz]{Example}

\newcommand{\A}{\mathbb{A}}
\newcommand{\N}{\mathbb{N}}
\newcommand{\B}{\mathbb{B}}
\newcommand{\R}{\mathbb{R}}
\newcommand{\C}{\mathbb{C}}
\newcommand{\D}{\mathbb{D}}

\newcommand{\h}{\mathbb{H}}
\newcommand{\U}{\mathbb{U}}
\newcommand{\Z}{\mathbb{Z}}
\newcommand{\mO}{\mathcal{O}}

\newcommand{\vp}{\varphi}

\DeclareMathOperator{\tr}{tr}
\DeclareMathOperator{\CH}{CH}
\DeclareMathOperator{\Ax}{Ax}
\DeclareMathOperator{\id}{id}
\DeclareMathOperator{\rk}{rk}
\DeclareMathOperator{\Cl}{Cl}
\DeclareMathOperator{\CF}{CF}
\DeclareMathOperator{\Th}{Th}
\DeclareMathOperator{\im}{Im}
\DeclareMathOperator{\stab}{Stab}
\DeclareMathOperator{\isom}{Isom}
\DeclareMathOperator{\Hom}{Hom}
\DeclareMathOperator{\Cay}{Cay}

\DeclareMathOperator{\GCl}{GCl}

\DeclareMathOperator{\Aut}{Aut}
\DeclareMathOperator{\Res}{Res}
\DeclareMathOperator{\Mod}{Mod}

\DeclareMathOperator{\Comp}{Comp}
\DeclareMathOperator{\dist}{dist}

\begin{document}
\pagestyle{plain}
\pagenumbering{roman}

\title{Test sequences and formal solutions over hyperbolic groups}
\author{Simon Heil
}
\maketitle

\begin{abstract}
\noindent In 2006 Z. Sela and independently O. Kharlampovich and A. Myasnikov gave a solution to the Tarski problems by showing that two non-abelian free groups have the same elementary theory. Subsequently Z. Sela generalized the techniques used in his proof of the Tarski conjecture to classify all finitely generated groups elementary equivalent to a given torsion-free hyperbolic group.
One important step in his analysis of the elementary theory of free and torsion-free hyperbolic groups is the Generalized Merzlyakov's Theorem.\\
In our work we prove a version of the Generalized Merzlyakov's Theorem over hyperbolic groups (possibly with torsion). To do this, we show that given a hyperbolic group $\Gamma$ and $\Gamma$-limit group $L$, there exists a larger group $\Comp(L)$, namely its completion,  into which $L$ embeds, and a sequence of points $(\lambda_n)$ in the variety $\Hom(\Comp(L),\Gamma)$ from which one can recover the structure of the group $\Comp(L)$.
Using such a test sequence $(\lambda_n)$ we are finally able to prove a version of the Generalized Merzlyakov's Theorem over all hyperbolic groups.\end{abstract}
\pagestyle{plain}
\tableofcontents
\newpage
\section{Introduction}
\pagestyle{headings} 
\pagenumbering{arabic}
Model theory is a branch of mathematical logic where mathematical structures are studied by considering the first-order sentences true in those structures and the sets definable by first-order formulas. Intuitively, a structure is a set equipped with distinguished functions, relations and elements. We then use a language where we can talk about this structure. For example the first-order language $L_0$ of groups uses the following symbols:
\begin{itemize}
\item The binary function group multiplication $"\cdot"$, the unary function inverse $" ^{-1}"$, the constant $"1"$ and the equality relation $"="$.
\item Variables $"x,y,z"$ or tuples of variables $"\underline{x}=(x_1,\ldots,x_n)"$ (which will have to be interpreted as individual group elements).
\item Logical connectives $"\wedge"$ ("and"), $"\vee"$ ("or"), $"\neg"$ ("not"), the quantifiers $"\forall"$ ("for all") and $"\exists"$ ("there exists") and parentheses $"("$ and $")"$.
\end{itemize}

\noindent A \textit{formula} in the language of groups is a logical expression using the above symbols.
A variable in a formula is called \textit{bound} if it is bound to a quantifier and \textit{free} otherwise. A \textit{sentence} is a formula in $L_0$ where all variables are bound. \\
Given a group $G$ and a sentence $\sigma$ we will say that \textit{$\sigma$ is a true sentence in $G$} (denoted by $G\models \sigma$) if $\sigma$ is true if interpreted in $G$ (i.e. variables interpreted as group elements in $G$). The \textit{first-order} or \textit{elementary theory} of a group $G$, denoted by $\Th(G)$, is the set of all sentences of $L_0$ true in $G$.\\
The groups whose first-order theory we are mainly interested in, are free and hyperbolic groups.
The interest of model theorist in the first-order theory of free groups dates back to the 1940s when Alfred Tarski posed his famous questions, which are now widely known as the Tarski-problems.

\begin{tarski}
\begin{enumerate}[(1)]
\item $\Th(F_n)=\Th(F_m)$ for all $m,n\geq 2$?
\item Is the first-order theory of a non-abelian free group $F_n$ decidable?
\end{enumerate}
\end{tarski}

There was some early progress on the subject by Vaught who proved that two non-abelian free groups of infinite rank have the same elementary theory. Meanwhile Merzlyakov was able to show that all non-abelian free groups have the same positive theory, that is all sentences not containing negations.\\
In the 1980s Makanin then described an algorithm that could decide if a given system of equations (with constants) has a solution in a free group. Later Razborov refined this process to give a complete description of the  set of solutions for a system of equations in a free group. This is now called the Makanin-Razborov algorithm.\\
Rips then recognized that one could use this algorithm to study group actions on real trees. This culminated in a structure theorem for the actions of finitely presented groups on real trees, which is usually referred to as the Rips machine. The original version of the Rips machine has later been generalized by Sela to actions of finitely generated groups and further refined by Guirardel \cite{guirardel}.\\

Still the original questions of Tarski remained unsolved and became known for their notorious difficulty. It took well over 60 years until finally in 2006 Zlil Sela and independently  O. Kharlampovich and A. Myasnikov published their astonishing solutions to Tarski's problems (\cite{sela1}-\cite{sela6} and \cite{kharlampovich/myasnikov}). Sela only considered the elementary theory of free groups, while Kharlampovich/Myasnikov also answered Tarski's second question to the positive. Despite this now being more than ten years ago, up to the present day the mathematical community has not yet fully absorbed the sophisticated ideas underlying Sela's voluminous work.\\
His proof splits up into seven papers. In \cite{sela1} he uses a geometric interpretation of ideas going back to Makanin and Razborov's algorithm for solving systems of equations over free groups to give a parametrization of $\Hom(G,F_k)$ for a fixed non-abelian free group $F_k$ and an arbitrary finitely generated group $G$.
In \cite{sela} he proves an implicit function-type theorem over free groups, which he calls Generalized Merzlyakov's Theorem and shows that given an $\forall\exists$-sentence true in $F_k$, there exist so-called formal solutions witnessing the truthfulness of this sentence.\\
In his third paper (\cite{sela3}) he shows that there exists a uniform bound on exceptional solutions to systems of equations with parameters. 
In \cite{sela4} he uses the results from the first three papers to show that all non-abelian free groups have the same $\forall\exists$-theory. This is the induction start for the general proof of the elementary equivalence of non-abelian free groups. Namely he proves quantifier elimination over free groups in \cite{sela5.1} and \cite{sela5.2}, that is he shows that an arbitrary first-order sentence over a given non-abelian free group is logically equivalent to a boolean combination of $\forall\exists$-sentences, hence using the results from \cite{sela4} mentioned before he finally was able to prove that all non-abelian free groups have the same elementary theory in \cite{sela6}.\\

Since then many of these results have been generalized and led to other important discoveries. The following list covers some of the developments following Sela's solution to the Tarski problems but is by no means complete.
\begin{itemize}
\item Solutions to systems of equations over
\begin{itemize}
\item torsion-free hyperbolic groups by Sela (\cite{selahyp}),
\item all hyperbolic groups by Reinfeldt-Weidmann (\cite{rewe}),
\item torsion-free toral relatively hyperbolic groups by Groves (\cite{grovessolo}),
\item acylindrical hyperbolic groups by Groves-Hull (\cite{groves})
\item free products of groups by Jaligot-Sela (\cite{jaligot})
\item free semigroups by Sela. (\cite{sela11})
\end{itemize}
\item The elementary theory of torsion-free hyperbolic groups \cite{selahyp}.
\item The elementary theory of free products of groups \cite{sela10}.
\end{itemize}

On the one hand the first paper of Sela on the elementary theory of free groups is now fairly well-understood and has been generalized in several ways, on the other hand for none of the remaining papers there is a generalization from someone else besides Sela himself. The only work we are aware of, dealing with one of the later papers of Sela's solution is the recent PhD thesis of J. Gonz\'ales, who proves a Generalized Merzlyakov's theorem over, what he calls, $\pi$-free groups \cite{javier} and the work of S. Andr\'e, who shows that hyperbolicity is preserved under elementary equivalence \cite{simon}. Some parts of the original results from Sela's second paper are also reproved in \cite{sklinos}.\\ 

We are mainly interested in the first step in Sela's proof of the Tarski problems, that is to show that $\Th_{\forall\exists}(F_m)=\Th_{\forall\exists}(F_n)$ for given $m,n\geq 2$, where $\Th_{\forall\exists}(F_m)$ consists of all sentences with two quantifiers, true in the free group $F_m$. These sentences are called $\forall\exists$- or \textit{AE-sentences}. 
An $\forall\exists$-sentence $\sigma$ is of the form
$$\forall y\ \exists x\ \Sigma(x,y)=1\wedge \Psi(x,y)\neq 1,$$
where $x$ and $y$ are tuples of variables, $\Sigma(x,y)=1$ is a system of equations in the variables $x,y$ and $\Psi(x,y)\neq 1$ is a system of inequalities.\\
Now fix a non-abelian free group $F_k$ and let $\sigma$ be an $\forall\exists$-sentence, true in $F_k$, that does not contain inequalities. Then $F_k\models\sigma$ if and only if for every homomorphism $\vp:F(y)\to F_k$ there exists an extension $\bar{\vp}:G_{\Sigma}\to F_k$ to the group $G_{\Sigma}:=\langle x,y\ |\ \Sigma(x,y)=1\rangle$ making the following diagram commute.
\begin{center}
\begin{tikzpicture}[descr/.style={fill=white}]
\matrix(m)[matrix of math nodes,
row sep=4em, column sep=4em,
text height=1.5ex, text depth=0.25ex]
{G_{\Sigma}& \\
F(y)& F_k\\};
\path[->,font=\scriptsize]
(m-2-1) edge node[descr] {$\nu$}(m-1-1)
(m-2-1) edge node[descr] {$\vp$} (m-2-2)
(m-1-1) edge node[descr] {$\bar{\vp}$} (m-2-2);
\end{tikzpicture}
\end{center}
Here $\nu: F(y)\to G_{\Sigma}$ is the canonical embedding. Clearly if there exists a retraction $\pi:G_{\Sigma}\to F(y)$ such an extension always exists by just setting $\bar{\vp}=\vp\circ\pi$.\\
It turns out that the other direction is also true, as was shown by Yu. I. Merzlyakov.

\begin{einleitung1}[\cite{merzlyakov}]
Let $F_k$ be a non-abelian free group and suppose that
$$F_k\models \forall y\ \exists x\ \Sigma(x,y)=1.$$
Then there exists a retraction $\pi: G_{\Sigma}\to F(y)$.
\end{einleitung1}

The retraction $\pi$, respectively $\pi(x)$, is called a \textit{formal solution}. The existence of formal solutions does not depend on the rank of the free group $F_k$ and therefore it follows immediately that 
$\Th^+_{\forall\exists}(F_m)=\Th^+_{\forall\exists}(F_n),$ i.e. $F_m$ and $F_n$ have the same positive theory.\\
Unfortunately the picture becomes a lot more complicated as soon as inequalities are involved.
A general $\forall\exists$-sentence
$$\forall y\ \exists x\ \Sigma(x,y)=1\wedge \Psi(x,y)\neq 1$$
is true in $F_k$ if and only if
for every homomorphism $\vp:F(y)\to F_k$ there exists an extension $\bar{\vp}:G_{\Sigma}\to F_k$ such that $\vp=\bar{\vp}\circ\nu$ and
with the additional property that $\bar{\vp}(v(x,y))\neq 1$ for all $v(x,y)\in \Psi(x,y)$. By Merzlyakov's Theorem there still exists a retraction $\pi:G_{\Sigma}\to F(y)$ but the formal solution $\pi(x)$ fails to verify the sentence for all homomorphisms $\vp: F(y)\to F_k$ with the property that $\vp\circ\pi(v(x,y))=1$ for some $v(x,y)\in\Psi(x,y)$.
Let  $$v_1'(y):=v_1(\pi(x),y)\neq 1, \ldots, v_r'(y):=v_r(\pi(x),y)\neq 1$$ be the inequalities of $\Psi(x,y)$ after replacing the variables $x$ by $\pi(x)\in F(y)$. Fix $i\in\{1,\ldots, r\}$. We assume for simplicity that $R_i(y):=\langle y\ |\ v_i'(y)=1\rangle$ is a limit group (in general one first has to pass to finitely many limit quotients of $R_i(y)$). Clearly $R_i(y)$ is a proper quotient of $F(y)$ and we define the basic definable set corresponding to the limit group $R_i(y)$ as
$$B_i(y)=\{y_0\in F_k^l\ |\ v_i'(y_0)=1\}.$$
If $$y_0\notin V_B(y):=B_1(y)\cup\ldots\cup B_r(y)$$ then $v_j'(y_0)\neq 1$ for every $j\in\{1,\ldots,r\}$.
The formal solution $\pi(x)$ given by Merzlyakov's Theorem proves the validity of our sentence on the set $F^l_k\setminus V_B(y)$ and hence it only remains to find a witness for the truth of the sentence
$$\forall y\in V_B(y)\ \exists x\ \Sigma(x,y)=1\wedge \Psi(x,y)\neq 1.$$
So the question becomes: Does Merzlyakov's Theorem still hold when the universal variables $y$ are restricted to some variety?\\
Unfortunately, a naive generalization of Merzlyakov's Theorem cannot be true as the following example demonstrates.

\begin{einleitungbeispiel}
Let $L=\langle \underline{y}\ |\ \Theta(\underline{y})\rangle$ be a limit group with JSJ decomposition $$L=R_1\ast_{C} R_2,$$ where $C=\langle c\rangle\cong \Z$ and $R_1=\langle \underline{y_1}\rangle ,R_2=\langle \underline{y_2}\rangle $ are rigid. Let $$S=\langle \underline{s}\ | \prod_{j=1}^g [s_{2j-1},s_{2j}]=1\rangle$$ be the fundamental group of a closed orientable surface of genus $g\geq 2$ and suppose that $S$ is the unique maximal shortening quotient of $L$ with corresponding quotient map $\eta:L\to S$. This means that for every homomorphism $\vp\in\Hom(L,F_k)$ there exists a Dehn-twist $\alpha\in \Aut(L)$ along the single edge of the JSJ decomposition of $L$ and $\bar{\vp}\in \Hom(S,F_k)$ such that $\vp=\bar{\vp}\circ\eta\circ\alpha$. Let $V_{\Theta}=\{ \underline{y}\in F_k^l\ |\ \Theta(\underline{y})=1\}$ and $s_i(\underline{y_1}, t\underline{y_2}t^{-1})$ an element of $(\eta\circ\alpha)^{-1}(s_i)$. Then
$$F_k\models \forall \underline{y}\in V_{\Theta}\ \exists t,\underline{s}\ \underbrace{([t,c]=1)\wedge (\underline{s}=\underline{s}(\underline{y_1}, t\underline{y_2}t^{-1})\wedge (\prod_{j=1}^g[s_{2j-1},s_{2j}]=1)}_{\Sigma(t,\underline{s},\underline{y})=1},$$
i. e. the sentence is true in $F_k$ (here we implicitly view $c$ as a word in $\underline{y}$). Suppose that there exists a retraction $$\pi: G_{\Sigma}=\langle t, \underline{s}, \underline{y}\ |\ V_{\Theta}, \Sigma(t,\underline{s},\underline{y})\rangle \to L=\langle \underline{y}\ |\ V_{\Theta}\rangle.$$
Then in particular $$\pi(\prod_{j=1}^g[s_{2j-1},s_{2j}])=1$$ and hence $\pi$ induces a homomorphism $\iota: S\to L$ such that $\eta\circ\iota=\id$ and therefore $\iota$ is injective.
\end{einleitungbeispiel} 

This example shows that if one wants to prove a version of Merzlyakov's Theorem in the case that the universal variables are restricted to some variety, then one has to guarantee that the maximal shortening quotients of the corresponding limit group $L$ are already contained as subgroups in $L$.\\
Obviously this is not true in general but nonetheless one can show the following. Let $L:=\langle y\ |\ B(y)=1\rangle$ be a limit group and $G_{\Sigma}:=\langle  x,y \ |\ B(y), \Sigma(x,y)\rangle$ then one can embed $L$ into a new group $\Comp(L):=\langle y,z\rangle$, which is called its completion, such that every homomorphism $\vp:L\to F_k$ can be extended to a homomorphism $\hat{\vp}: \Comp(L)\to F_k$ (in general there exist finitely many completions, but for now we suppose that it is only a single one). Moreover, for the group $\Comp(L)$ a generalized version of Merzlyakov's Theorem holds.

\begin{gen}
Let $F_k$ be a non-abelian free group, 
$$L=\langle y\ |\ B(y)=1\rangle$$ a limit group, $V_{B}=\{ y\in F_k^l\ |\ B(y)=1\}$ and suppose that
$$F_k\models \forall y\in V_{B}\ \exists x\ \Sigma(x,y)=1\wedge\Psi(x,y)\neq 1.$$
Then $L$ can be embedded into its completion $\Comp(L)$ via a homomorphism $\iota$ and there exists a homomorphism $\pi: G_{\Sigma}\to \Comp(L)$ such that $\pi\circ\nu=\iota$, where $\nu:L\to G_{\Sigma}$ is the canonical embedding.
\end{gen}

\begin{center}
\begin{tikzpicture}[descr/.style={fill=white}]
\matrix(m)[matrix of math nodes,
row sep=4em, column sep=4em,
text height=1.5ex, text depth=0.25ex]
{ &G_{\Sigma}& \\
L& \Comp(L) &F_k\\};
\path[->,font=\scriptsize]
(m-2-1) edge node[descr] {$\nu$}(m-1-2)
(m-2-1) edge node[descr] {$\iota$} (m-2-2)
(m-2-2) edge node[descr] {$\hat{\vp}$} (m-2-3)
(m-1-2) edge node[descr] {$\pi$} (m-2-2)
(m-2-1) edge[bend right=40] node[above] {$\vp$} (m-2-3);
\end{tikzpicture}
\end{center}

It should be noted though that this theorem is not entirely true as it stands, due to the difficulties involved when $L$ contains non-cyclic abelian subgroups. For a precise formulation of the Generalized Merzlyakov's Theorem see \cite{sela} (or Theorem \ref{maintheorem} for the corresponding theorem over hyperbolic groups).\\
Note that in contrast to the ordinary Merzlyakov's Theorem, the map $\pi$ is no longer a retraction.
Still this gives us a formal solution $\pi(x)=u(y,z)\in \Comp(L)=\langle y,z\rangle$.
But the price for finding this formal solution was that we had to introduce new variables $z$.\\
Now for all $i\in\{1,\ldots,r\}$ let $M_i$ be the maximal limit quotient of $\Comp(L)$ corresponding to all homomorphisms from $\Comp(L)$ to $F_k$ which map $v_i(u(y,z),y)$ to the identity (again we ignore the fact that in general there are finitely many quotients). This is a proper quotient of $\Comp(L)$. Let $$D_i(y)=\{ y\in B(y)\ |\ \exists z\in F_k^l:\ v_i(u(y,z),y)=1\}$$
be the (Diophantine) definable set corresponding to $M_i$. Set 
$$D(y):=D_1(y)\cup\ldots\cup D_r(y).$$
Then $D(y)\subsetneq B(y)$ and the formal solutions constructed by Merzlyakov's Theorem and the Generalized Merzlyakov's Theorem prove the validity of our sentence on the co-Diophantine set
$$F_k^l\setminus D(y).$$
Hence it remains to find a witness for the truth of the sentence
$$\forall y\in D(y)\ \exists x\ \Sigma(x,y)=1\wedge \Psi(x,a)\neq 1.$$

Repeatedly applying the Generalized Merzlyakov's Theorem yields smaller and smaller sets on which the sentence has to be analyzed. But a priori there is no reason why this process should eventually stop. Still Sela in \cite{sela4} is able to show that the completion was constructed carefully enough such that this iterative procedure indeed terminates. Since all of this does not depend on the rank of the free group $F_k$ it follows that 
$\Th_{\forall\exists}(F_m)=\Th_{\forall\exists}(F_n),$ i.e. $F_m$ and $F_n$ have the same $\forall\exists$-theory.\\

In our work we build on the results of Reinfeldt-Weidmann to generalize Sela's results to hyperbolic groups with torsion. More precisely we prove a  version of the Generalized Merzlyakov's Theorem for all hyperbolic groups. In Sela's paper on the elementary theory of torsion-free hyperbolic groups \cite{selahyp} a version of this theorem for torsion-free hyperbolic groups is announced, with the remark that the proof is fairly similar to the one in the free group case. Over all hyperbolic groups we have to address several problems caused by the presence of torsion, which makes some easy arguments from the torsion-free case a lot more technically involved. We therefore intend to give a short explanation here and there of the corresponding proof in the torsion-free case, supplementing Sela's result in \cite{selahyp}.\\

We start in chapter \ref{1} by recalling the fundamental ideas of graphs of groups and Rips theory, namely how one can deduce a graph of groups decomposition of a group $G$ from its (appropriate) action on a real tree. In chapter \ref{2} we collect basic properties of $\Gamma$-limit groups, describe their Dunwoody- and JSJ decomposition and their actions on real trees. Most of these results are borrowed from \cite{rewe}.\\
In chapter \ref{3} we introduce some model theory and state the Tarski problems. Then in chapter \ref{4} we review the main results from \cite{rewe}, that is the shortening argument and the existence of Makanin-Razborov diagrams for finitely generated groups over a hyperbolic group. We end this chapter with a description of special resolutions appearing in such a diagram, so-called strict and well-structured resolutions.\\
In chapter \ref{7} we construct the completion of a well-structured resolution $\Res(L)$, that is a $\Gamma$-limit group $\Comp(L)$ into which $L$ can be embedded. The main significance of the completion is that it admits so-called test sequences, i.e. sequences of generic points in the variety $\Hom(\Comp(L),\Gamma)$ which allow to recover the group $\Comp(L)$ from it. Constructing these test sequences is the core of our proof of the Generalized Merzlyakov's Theorem  and will take up all of chapter \ref{8}.\\
In chapter \ref{9} we are then finally able to prove a version of the Generalized Merzlyakov's Theorem over all hyperbolic groups. It should be noted though that we were not able to prove a full generalization of the Generalized Merzlyakov's Theorem as stated in \cite{selahyp} to the case of hyperbolic groups with torsion. Namely in order to prove that the collection of closures which appears in the theorem forms a covering closure, we have to restrict ourselves to the torsion-free case (see Theorem \ref{maintheorem}). We expect that this restriction can be removed, so that a full generalization of the Generalized Merzlyakov's Theorem over all hyperbolic groups in fact holds.

\section{Groups acting on trees and Rips theory}\label{1}
The structure of groups acting on simplicial trees is well-understood, thanks to Bass-Serre theory. The picture becomes a lot more complicated when actions of groups on real trees are considered. Still under some mild stability assumptions, Rips theory provides a description of these actions. This theory was developed by Rips (unpublished) who applied ideas from the Makanin-Razborov rewriting process for system of equations over free groups  to describe free actions of finitely presented groups on real trees (see \cite{gaboriau}). His ideas have then been generalized by Bestvina and Feighn \cite{feighn} to stable actions of finitely presented groups and by Sela \cite{sela9} to super-stable actions of finitely-generated torsion-free groups. Guirardel \cite{guirardel} then generalized Sela's version by weakening the stability assumptions and allowing the presence of torsion.\\

In this chapter we fix notations for graphs of groups and the corresponding Bass-Serre tree and we recall the definition of a graph of actions.
We then review some material and state the main theorem of \cite{guirardel} and explain certain simple actions on real trees which appear as the building blocks for general actions of groups on real trees.\\
We expect the reader to be familiar with the ideas of Bass-Serre theory and only fix the notation here, which will be used later on. For more details on the subject see \cite{serre} or \cite{kmw}.

\subsection{Graphs of groups}
A \textit{graph} $A$ is a tuple $A=(VA,EA, \alpha,\omega, ^{-1})$ consisting of a vertex set $VA$, a set of oriented edges $EA$, a fixed point free involution $^{-1}$ and maps $\alpha,\omega: EA\to VA$, which assign to every edge $e\in EA$ its initial vertex $\alpha(e)$ and terminal vertex $\omega(e)$, such that $\alpha(e^{-1})=\omega(e)$.\\
A \textit{simplicial tree} is a connected graph which contains no non-trivial closed paths.

\begin{definition}
A \textnormal{graph of groups} $\A$ is a tuple 
$$(A, (A_v)_{v\in VA}, (A_e)_{e\in EA}, (\alpha_e)_{e\in EA}, (\omega_e)_{e\in EA}),$$
where for all $v\in VA$, $e\in EA$
\begin{itemize}
\item $A$ is a connected graph,
\item $A_v$ and $A_e$ are groups, called vertex and edge groups and
\item $\alpha_e: A_e\to A_{\alpha(e)}$,\quad $\omega_e=\alpha_{e^{-1}}: A_e\to A_{\omega(e)}=A_{\alpha(e^{-1})}$ are monomorphisms.
\end{itemize} 
We call the maps $\alpha_e$ and $\omega_e$ boundary monomorphisms of the edge $e$.
\end{definition}

To every graph of groups $\A$ we can associate a group, namely its fundamental group $\pi_1(\A)$. Moreover we can associate a simplicial tree to $\A$ on which its fundamental group acts, its \textit{Bass-Serre tree} $T_A$.
 
\begin{definition}
Let $\A$ be a graph of groups and suppose that $A_v=\langle X_v\ |\ R_v\rangle$ for all $v\in VA$ and $A_e=\langle X_e\rangle$ for all $e\in EA$. Let $T$ be a maximal subtree of $A$. We then define the \textnormal{fundamental group} $\pi_1(\A,T)$ to be given by
$$\pi_1(\A,T)=\langle \bigcup_{v\in VA} X_v, \{s_e|e\in EA\}\ | \bigcup_{v\in VA} R_v, N\rangle,$$
where
$$N=\{s_e\ |\ e\in ET\}\cup\{s_es_{e^{-1}}\ |\ e\in EA\}\cup\{s_e\omega_e(x)s_e^{-1}=\alpha_e(x)\ |\ e\in EA, x\in X_e\}$$
\end{definition} 

\begin{bem}
Up to isomorphism $\pi_1(\A,T)$ does not depend on the choice of the maximal subtree $T$ and hence we denote the fundamental group of $\A$ by $\pi_1(\A)$. 
\end{bem}

Let $\A$ be a graph of groups. An \textit{$\A$-path} from $v\in VA$ to $w\in VA$ is a sequence $$(a_0,e_1,a_1,\ldots,e_k,a_k),$$ where
$e_1,\ldots,e_k$ is an edge-path from $v$ to $w$ in the underlying graph $A$, $a_0\in A_v$ and $a_i\in A_{\omega(e_i)}$ for $i\in\{1,\ldots,k\}$. Let $\sim$ be the equivalence relation on the set of $\A$-paths generated by the elementary equivalences
$$(a_0,e,a_1)\sim (a_0\alpha_e(c),e,\omega_e(c^{-1})a_1) \text{ and }$$
$$(a_1,e,1,e^{-1},a_2)\sim (a_1a_2).$$
We denote the equivalence class of an $\A$-path $p$ by $[p]$. For every base vertex $v_0\in VA$, the fundamental group of $\A$ with respect to $v_0$ is defined as 
$$\pi_1(\A,v_0)=\{[p]\ |\ p \text{ is an }\A\text{-path from } v_0 \text{ to }v_0\}.$$
It is easy to see that this definition does not depend (up to isomorphism) on the choice of the  base vertex and $\pi_1(\A,v_0)\cong \pi_1(\A,T)$.\\
For all $e\in EA$ let $C_e$ be a set of left coset representatives of $\alpha_e(A_e)$ in $A_{\alpha(e)}$. Then each $\A$-path $p$ is equivalent to a reduced $\A$-path $p'=(a_0,e_1,\ldots,e_k,a_k)$ such that $a_{i-1}\in C_{e_i}$ for all $i\in\{1,\ldots,k\}$. We say that $p'$ is in \textit{normal form} (with respect to the set $\{C_e\ |\ e\in EA\}$).

\begin{definition}\label{defdehntwist}
Let $\A$ be a graph of groups, $e\in EA$ and $g\in Z(A_e)$ an element of the center of $A_e$. The \textnormal{Dehn-twist} along $e$ by $g$ is the automorphism 
$$\vp: \pi_1(\A)\to \pi_1(\A), [a_0,e_1,\ldots,e_k,a_k]\to [\bar{a}_0,e_1,\ldots, e_k,\bar{a}_k],$$
where
$$\bar{a}_i=
\begin{cases}
    \omega_e(g)a_i,& \text{if } e_i=e\\
    \alpha_e(g^{-1})a_i,& \text{if } e_i=e^{-1}\\   
    a_i,              & \text{otherwise}
\end{cases}$$
\end{definition}

In the case of a one-edge splitting $\A$, this just recovers the usual definition of a Dehn-twist.

\subsection{Rips theory}
We now turn our attention from simplicial trees to $\R$-trees.
An \textit{$\R$-tree} (or \textit{real tree}) is a $0$-hyperbolic metric space. Suppose a group $G$ acts on an $\R$-tree $T$ by isometries. We call the action \textit{minimal}, if $T$ has no proper $G$-invariant subtree. We call a subtree of $T$ non-degenerate if it contains an arc, where an arc is a set homeomorphic to the interval $[0,1]$.
We now collect the relevant definitions and results from \cite{guirardel} which we are using in this work. For more details on $\R$-trees see \cite{chiswell} or \cite{guirardel}.

\begin{definition}
Let $G$ be a group acting on a real tree $T$ by isometries.
\begin{enumerate}[(a)]
\item $T$ satisfies the \textnormal{ascending chain condition} if for any sequence of arcs $I_1 \supset I_2 \supset \ldots$ in $T$ whose length converges to $0$, the sequence of stabilizers of the sequence is eventually constant.
\item A non-degenerate subtree $S\subset T$ is called \textnormal{stable} if for every arc $I\subset S$, $\stab_G(I)=\stab_G(S)$, where $\stab_G$ denotes the stabilizer in $G$. Otherwise $S$ is called \textnormal{unstable}.
\item $T$ is \textnormal{super-stable} if any arc with non-trivial stabilizer is stable
\item A non-degenerate subtree $S\subset T$ is called \textnormal{indecomposable} if for every pair of arcs $I,J\subset S$, there is a finite sequence $g_1I,\ldots,g_nI$ which covers $J$ such that $g_i\in G$ for all $i\in\{1,\ldots,n\}$ and $g_iI\cap g_{i+1}I$ is non-degenerate.
\end{enumerate}
\end{definition}

\begin{bem}
\begin{enumerate}[(a)]
\item By stabilizer we always mean pointwise stabilizer.
\item Clearly super-stability implies the ascending chain condition.
\end{enumerate}
\end{bem}

A graph of actions is a way of decomposing the action of a group on a real tree into pieces. For more information see \cite{levitt}.

\begin{definition}
A \textnormal{graph of actions} is a tuple
$$\mathcal{G}=\mathcal{G}(\A)=(\A, (T_v)_{v\in VA}, (p_e^{\alpha})_{e\in EA},l),$$
where
\begin{itemize}
	\item $\A$ is a graph of groups with underlying graph $A$,
	\item for each $v\in VA$, $T_v=(T_v,d_v)$ is a real $A_v$-tree,
	\item for each $e\in EA$, $p_e^{\alpha}\in T_{\alpha(e)}$ is a point fixed by $\alpha_e(A_e)$,
	\item $l: EA\to \R_{\geq 0}$ is a function satisfying $l(e)=l(e^{-1})$ for all $e\in EA$.
\end{itemize}
We call $l(e)$ the length of $e$. If $l=0$ then we omit $l$, i.e. we write $$\mathcal{G}=\mathcal{G}(\A)=(\A, (T_v)_{v\in VA}, (p_e^{\alpha})_{e\in EA}).$$
\end{definition}

Given a graph of actions one can canonically construct an associated real tree $T$ by replacing the vertices of the Bass-Serre tree $T_A$ of $\A$ by copies of the trees $T_v$ and any edge $e$ by a segment of length $l(e)$. Clearly the action of $\pi_1(\A)$ on $T_A$ extends naturally to an action of $\pi_1(\A)$ on $T$.\\

We are now ready to state the main theorem of \cite{guirardel}. This result and its relatives are usually referred to as the Rips machine.

\begin{satz}[\cite{guirardel}]\label{rips} Consider a non-trivial action of a finitely generated group $G$ on a real tree $T$ by isometries. Assume that
\begin{enumerate}
	\item $T$ satisfies the ascending chain condition,
	\item for any unstable arc $J\subset T$,
	\begin{enumerate}
		\item $\stab(J)$ is finitely generated
		\item $\stab(J)$ is not a proper subgroup of any conjugate of itself, i.e. for all $g\in G$ $g\stab(J)g^{-1}\subset \stab(J) \Longrightarrow g\stab(J)g^{-1}=\stab(J)$.
	\end{enumerate}
\end{enumerate}
Then either $G$ splits over the stabilizer of an unstable arc or over the stabilizer of an infinite tripod, or $T$ splits as a graph of actions
$$\mathcal{G}=(\A, (T_v)_{v\in VA}, (p_e^{\alpha})_{e\in EA},l),$$
where each vertex action of $A_v$ on the vertex tree $T_v$ is either
\begin{enumerate}
	\item simplicial: a simplicial action on a simplicial tree,
	\item of orbifold type: the action of $A_v$ has kernel $N_v$ and the faithful action of $A_v/N_v$ is dual to an arational measured foliation on a compact $2$-orbifold with boundary, or
	\item axial: $T_v$ is a line and the image of $A_v$ in $\isom(T_v)$ is a finitely generated group acting with dense orbits on $T_v$
\end{enumerate}
\end{satz}

\begin{bem}
\begin{enumerate}[(a)]
\item For a detailed definition of the different types of actions on the vertex trees see \cite{guirardel} or \cite{rewe}.
\item The actions of axial and orbifold type have the indecomposable property.
\end{enumerate}
\end{bem}

Later we will also need a relative version of the Rips machine for pairs of groups:

\begin{definition}
A \textnormal{pair of groups} $(G,\mathcal{H})$ consists of a group $G$ together with a finite set of subgroups $\mathcal{H}=\{H_1,\ldots,H_k\}$ of $G$. A pair $(G,\mathcal{H})$ is \textnormal{finitely generated} if there exists a finite subset $X\subset G$ such that $G$ is generated by $X\cup H_1\cup\ldots\cup H_k$.\\
 An \textnormal{action of a pair} $(G,\mathcal{H})$ on a (real) tree $T$ is an action of $G$ on $T$ such that each subgroup $H_i$ acts elliptically. A \textnormal{splitting} of $(G,\mathcal{H})$ is a graph of groups $\A$ with $\pi_1(\A)=G$ such that every $H_i$ is contained in a conjugate of a vertex group of $\A$.
\end{definition}

The following Theorem is Theorem 5.1 of \cite{guirardel}.

\begin{satz}\label{relativerips}(Relative Version of the Rips machine)
Consider a non-trivial action of a finitely generated pair $(G, \mathcal{H})$ on a real tree $T$ by isometries. Assume that
\begin{enumerate}
	\item $T$ satisfies the ascending chain condition,
	\item there exists a finite family of arcs $I_1,\ldots,I_p$ such that $I_1\cup\ldots\cup I_p$ spans $T$ and such that for any unstable arc $J$ contained in some $I_i$,
	\begin{enumerate}
		\item $\stab(J)$ is finitely generated
		\item $\stab(J)$ is not a proper subgroup of any conjugate of itself, i.e. for all $g\in G$ $g\stab(J)g^{-1}\subset \stab(J) \Longrightarrow g\stab(J)g^{-1}=\stab(J)$
	\end{enumerate}
\end{enumerate}
Then either $(G,\mathcal{H})$ splits over the stabilizer of an unstable arc or over the stabilizer of an infinite tripod (whose normalizer contains a non-abelian free group generated by two elements having disjoint axes), or $T$ splits as a graph of actions
$$\mathcal{G}=(\A, (T_v)_{v\in VA}, (p_e^{\alpha})_{e\in EA},l)$$
where each vertex action of $A_v$ on the vertex tree $T_v$ is either
\begin{enumerate}
	\item simplicial: a simplicial action on a simplicial tree,
	\item of orbifold type: the action of $A_v$ has kernel $N_v$ and the faithful action of $A_v/ N_v$ is dual to an arational measured foliation on a compact $2$-orbifold with boundary, or
	\item axial: $T_v$ is a line and the image of $A_v$ in $\isom(T_v)$ is a finitely generated group acting with dense orbits on $T_v$
\end{enumerate}
\end{satz}

\subsection{Strong convergence of group actions}
In this section we describe a way how a minimal action of a finitely generated pair on a real tree can be approximated by a sequence of better-understandable actions of pairs on real trees. This is used by Guirardel in his proof of the Rips machine. In the end we state a relative version of Guirardel's Extended Scott's Lemma (Theorem 2.4 in \cite{guirardel}).\\
The non-relative versions of the following definitions and results can be found in \cite{levittpaulin} while the relative versions are due to Guirardel (\cite{guirardel}). 
Given two actions $G_1\curvearrowright T_1$, $G_2\curvearrowright T_2$ of groups $G_1,G_2$ on $\R$-trees $T_1,T_2$ we call $$(\vp,f):G_1 \curvearrowright T_1\to G_2\curvearrowright T_2$$ a morphism of group actions if $\vp: G_1\to G_2$ is a group homomorphism and $f:T_1\to T_2$ is a $\vp$-equivariant morphism of $\R$-trees. We say that $(\vp,f)$ is surjective if both $\vp$ and $f$ are.\\

\begin{definition}
A \textnormal{direct system of actions} on $\R$-trees  is a sequence of actions of finitely generated groups  $G_k\curvearrowright T_k$ and an action $G\curvearrowright T$ with surjective morphisms of group actions $(\vp_k,f_k):G_k \curvearrowright T_k\to G_{k+1}\curvearrowright T_{k+1}$ and $(\Phi,F_k):G_k \curvearrowright T_k\to G\curvearrowright T$ such that the following diagram commutes\\

\centering
\begin{tikzpicture}[descr/.style={fill=white}]
\matrix(m)[matrix of math nodes,
row sep=1em, column sep=2.8em,
text height=1.5ex, text depth=0.25ex]
{G_k&G_{k+1}&\cdots&G\\
T_k&T_{k+1}&\cdots&T\\};
\path[->>,font=\scriptsize]
(m-1-1) edge node[above] {$\vp_k$}(m-1-2)
(m-1-1) edge [bend left=60] node[descr] {$\Phi_{k}$} (m-1-4)

(m-1-2) edge [bend left=40] node[descr] {$\Phi_{k+1}$} (m-1-4)

(m-2-1) edge node[above] {$f_k$} (m-2-2)
(m-2-1) edge [bend right=60] node[descr] {$F_{k}$} (m-2-4)
(m-2-2) edge [bend right=40] node[descr] {$F_{k+1}$} (m-2-4);
\path[->,font=\scriptsize]
(m-1-1.south) edge [bend left=60](m-2-1.north)
(m-1-2.south) edge [bend left=60](m-2-2.north)
(m-1-4.south) edge [bend left=60](m-2-4.north);
\end{tikzpicture}
\end{definition}

Following Guirardel, we will use the notation $f_{kk'}=f_{k'-1}\circ\ldots\circ f_k:T_k\to T_{k'}$ and $\vp_{kk'}=\vp_{k'-1}\circ\ldots\circ \vp_k:G_k\to G_{k'}$ for $k'\geq k$.

\begin{definition}
A direct system of minimal actions of finitely generated groups on $\R$-trees $G_k\curvearrowright T_k$ \textnormal{converges strongly} to $G\curvearrowright T$ if 
\begin{itemize}
\item G is the direct limit of the groups $G_k$,
\item for every finite tree $K\subset T_k$, there exists $k'\geq k$ such that $F_{k'}$ restricts to an isometry on $f_{kk'}(K)$.
\end{itemize}
The action $G\curvearrowright T$ is called the \textnormal{strong limit} of this direct system.
\end{definition}
 
Strong convergence of direct systems allows us to make precise what it means to approximate a minimal action of a finitely generated group on a real tree by so-called geometric actions of groups on real trees. Since the definition of a geometric action is rather long, we refer the reader to the discussion before Definition 1.24 in \cite{guirardel} or \cite{levittpaulin}.

\begin{satz}[Theorem 3.7 in \cite{levittpaulin}]\label{geometric} Consider a minimal action of a finitely generated group $G$ on a real tree $T$. Then $G\curvearrowright T$ is a strong limit of a direct system of geometric actions $\{(\Phi_k,F_k):G_k\curvearrowright T_k\to G\curvearrowright T\}$ such that 
\begin{itemize}
\item $\Phi_k$ is one-to-one in restriction to each arc stabilizer of $T_k$,
\item $T_k$ is dual to a 2-complex $X$ whose fundamental group is generated by free homotopy classes of curves contained in leaves.
\end{itemize}
\end{satz}

We will later need a relative version of the above theorem. Note that the last claim in Theorem \ref{relativegeometric} is equivalent to the last claim in Theorem \ref{geometric}.

\begin{satz}[Proposition 1.31 in \cite{guirardel}]\label{relativegeometric}
Consider a minimal action $(G,\mathcal{H})\curvearrowright T$ of a finitely generated pair on an $\R$-tree. Then there exists a direct system of minimal actions $(G_k,\mathcal{H}_k)\curvearrowright T_k$ converging strongly to $G\curvearrowright T$ and such that 
\begin{itemize}
\item $\vp_k$ and $\Phi_k$ are one-to-one in restriction to arc stabilizers of $T_k$.
\item $\vp_k$ (resp. $\Phi_k$) restricts to an isomorphism between $\mathcal{H}_k$ and $\mathcal{H}_{k+1}$ (resp. $\mathcal{H}$).
\item $T_k$ splits as a graph of actions where each non-degenerate vertex action is either axial, thin or of orbifold type and therefore indecomposable, or is an arc containing no branch point except at its endpoints.
\end{itemize}
\end{satz}

Finally we will need the following result, which is Theorem 2.4 in \cite{guirardel}.

\begin{satz}[Relative version of the Extended Scott's Lemma]\label{scott}
Let $$(G_k,\{H_1^k,\ldots,H_p^k\})\curvearrowright S_k$$ be a sequence of non-trivial actions of finitely generated pairs on simplicial trees, and $$(\vp_k,f_k):G_k\curvearrowright S_k\to G_{k+1}\curvearrowright S_{k+1}$$ be epimorphisms mapping $H_i^k$ onto $H_i^{k+1}$. Consider $G=\underrightarrow{\lim}\ G_k$ the inductive limit and $\Phi_k: G_k\to G$ the natural map.
Assume that $(\vp_k,f_k)$ does not increase edge stabilizers in the following sense:
\begin{itemize}
\item $\forall e\in E(S_k),\forall e'\in E(S_{k+1}),\ e'\subset f_k(e)\Rightarrow G_{k+1}(e')=\vp_k(G_k(e)).$
\item $\forall e\in E(S_k), \forall i\in\{1,\ldots,p\},\ \Phi_k(H_i^k)\nsubseteq \Phi_k(e).$
\end{itemize}
Here $G_{k+1}(e')$ denotes the pointwise edge stabilizer of the edge $e'$ in $G_{k+1}$. Then the pair $(G,\mathcal{H})$ has  a non-trivial splitting over the image of an edge stabilizer of some $S_k$. Moreover, any subgroup $H\subset G$ fixing a point in some $S_k$ fixes a point in the obtained splitting of $G$.
\end{satz}

\section{$\Gamma$-limit groups}\label{2}
\subsection{$\Gamma$-limit groups and their actions on real trees}
From now on let $\Gamma$ be a finitely generated non-elementary hyperbolic group (possibly with torsion).
In this section we will define the notion of a $\Gamma$-limit group and investigate which properties these groups have. In particular we will show how one can construct faithful actions of $\Gamma$-limit groups on real trees. Most of the results are borrowed from \cite{rewe} and we will only sketch some of the ideas found there. For a complete proof of all the results we refer the reader to \cite{rewe}. 

\begin{definition}
Let $G$ be a group and $(\varphi_i)_{i\in\N}\subset \Hom(G,\Gamma)$. 
\begin{enumerate}[(1)]
\item The sequence $(\varphi_i)$ is \textnormal{stable} if for any $g\in G$ either $\varphi_i(g)=1$ for almost all $i$ or $\varphi_i(g)\neq 1$ for almost all $i$.
If $(\varphi_i)$ is stable then the \textnormal{stable kernel} of the sequence, denoted by $\underrightarrow{\ker}(\varphi_i)$, is defined as 
$$\underrightarrow{\ker}(\varphi_i):=\{ g\in G\ |\ \varphi_i(g)=1 \text{ for almost all }i\}.$$
We call the sequence $(\varphi_i)$ \textnormal{stably injective} if $(\varphi_i)$ is stable and $\underrightarrow{\ker}(\varphi_i)=1$.
\item Assume that $(\vp_i)$ is stable. We then call the quotient $G/\underrightarrow{\ker}(\varphi_i)$ the \textnormal{$\Gamma$-limit group} associated to $(\varphi_i)$ and the projection $\pi: G\to G/\underrightarrow{\ker}(\varphi_i)$ the 
\textnormal{$\Gamma$-limit map} associated to $(\varphi_i)$.
\item We call a group $L$ a $\Gamma$-limit group if there exists some group $H$ and a stable sequence $(\vp_i)\subset \Hom(H,\Gamma)$ such that $L=H/\underrightarrow{\ker}(\varphi_i)$.
\end{enumerate}
\end{definition} 

\begin{bem}
\begin{enumerate}[(1)]
\item Let $L$ be a $\Gamma$-limit group. Then there exists a stably injective sequence $(\vp_i)\subset \Hom(L,\Gamma)$ such that $L=L/\underrightarrow{\ker}(\varphi_i)$.
\item The definition of a $\Gamma$-limit group also works for arbitrary groups $\Gamma$, but here we are only interested in the case that $\Gamma$ is hyperbolic.
\end{enumerate}
\end{bem}

\begin{definition}
A group $G$ is called 
\begin{itemize}
\item \textnormal{residually $\Gamma$} if for every non-trivial element $g\in G$ there exists a homomorphism $\vp:G\to \Gamma$ such that $\vp(g)\neq 1$.
\item \textnormal{fully residually $\Gamma$} if for every finite subset $S\subset G$ there exists a homomorphism $\vp:G\to \Gamma$ such that $\vp|_S$ is injective.
\end{itemize}
\end{definition}

The following theorem shows that any finitely generated $\Gamma$-limit group is fully residually $\Gamma$ and the converse also holds.

\begin{satz}[\cite{rewe} Lemma 1.3 and Corollary 5.4]
Let $L$ be a finitely generated group. Then the following are equivalent
\begin{enumerate}[(1)]
\item $L$ is a $\Gamma$-limit group.
\item $L$ is fully residually $\Gamma$
\end{enumerate}
\end{satz}

Let $G$ be a group. We now explain how a sequence of suitable group actions of $G$ on $\delta$-hyperbolic spaces can be used to construct an action of $G$ on a real tree. This is usually referred to as the Bestvina-Paulin method (\cite{paulin}).\\

Let $G$ be a finitely generated group. A pseudo-metric $d$ on $G$ is called \textit{$G$-invariant} if $d(h_1,h_2)=d(gh_1,gh_2)$ for all $g,h_1,h_2\in G$. We denote by $\mathcal{ED}(G)$ the space of all $G$-invariant pseudo-metrics on $G$, with the compact-open topology (where $G$ is given the discrete topology). Thus a sequence $(d_n)$ of $G$-invariant pseudo-metrics on $G$ converges in $\mathcal{ED}(G)$ against a pseudo-metric $d$ if and only if the sequence $(d_n(1,g))$ converges in $\R$ to $d(1,g)$ for all $g\in G$.

\begin{definition}\label{definition}
\begin{enumerate}[(1)]
\item A (based) \textnormal{$G$-space} is a tuple $(X,x_0,\rho)$ of a metric space $X$, a base point $x_0$ and an action $\rho$ of a group $G$ on $X$. If $g\in G$ and $x\in X$, we will denote the element $\rho(g)(x)\in X$ simply by $gx$ if the action $\rho$ is understood. Let $X=(X,x,\rho)$ be a based $G$-space. Then the $G$-action $\rho$ on $X$ induces a pseudo-metric
$$d_{\rho}^x: G\times G\to \R_{\geq 0}$$
on $G$ given by
$$d_{\rho}^x(g,h)=d_X(\rho(g)(x),\rho(h)(x)).$$
\item Let $G$ be a group with finite generating set $S_G$ and $X=(X,x,\rho)$ a based $G$-space. The norm of the action $\rho$ with respect to the base point $x$ (and the generating set $S_G$), denoted by $|\rho|_x$, is defined as
$$|\rho|_x:=\sum_{s\in S_G}d_X(x,\rho(s)x).$$
\end{enumerate}
\end{definition}

\begin{satz}[Theorem 1.11 in \cite{rewe}]\label{1.11}
Let $G$ be a group with finite generating set $S_G$. For $i\in \N$ let $\delta_i\geq 0$ and $X_i=(X_i,x_i,\rho_i)$ a based $\delta_i$-hyperbolic $G$-space. Assume that the sequence 
$(d_{\rho_i}^{x_i})$ converges in $\mathcal{ED}(G)$ to a limit sequence $d_{\infty}$ and 
$$\lim_{i\to\infty}\delta_i=0.$$
Then there exists a based real $G$-tree $(T,x,\rho)$ such that $T$ is spanned by ${\rho}(G)x$ and $d_{\rho}^x=d_{\infty}$.
If moreover for any $i$ and any $y\in X_i$,
$$|\rho_i|_y\geq|\rho_i|_{x_i},$$
then the limit action of $G$ on $T$ is minimal. 
\end{satz}

We explain now how the previous statement implies that $\Gamma$-limit groups act on real trees.\\
Let $S_{\Gamma}$ be a generating set of the hyperbolic group $\Gamma$ and $X:=\Cay(\Gamma, S_{\Gamma})$ the Cayley graph of $\Gamma$. Let $G$ be a group with finite generating set $S_G$. Let $\delta$ be the hyperbolicity constant of $X$.  Note that any $\varphi\in \Hom(G,\Gamma)$ defines an action of $G$ on $X$, namely
$$G\times X\to X, (g,x)\mapsto \vp(g)x.$$
It follows that $(X,1,\varphi)$ is a based $G$-space for any $\varphi\in\Hom(G,\Gamma)$. This action induces the $\delta$-hyperbolic pseudo metric $d_{\varphi}:=d_{\varphi}^1$ on $G$ given by 
$$d_{\varphi}(g,h)=d_X(\varphi(g), \varphi(h))=|\varphi(g^{-1}h)|_{S_{\Gamma}}.$$
Then $|\varphi|:=|\varphi|_1=\sum_{s\in S_G}d_X(1,\varphi(s))$ denotes the norm of $\varphi$.\\
For any $g\in \Gamma$ let $c_g:\Gamma \to\Gamma$ be the inner automorphism given by $c_g(h):=g^{-1}hg$ for all $h\in \Gamma$.

\begin{definition}
We call a homomorphism $\varphi\in\Hom(G,\Gamma)$ \textnormal{conjugacy-short} if $|\varphi|\leq|c_g\circ\varphi|$ for all $g\in\Gamma$.
\end{definition}

Clearly a homomorphism $\varphi\in\Hom(G,\Gamma)$ is conjugacy-short if and only if $|\varphi|\leq|\varphi|_g$ for all $g\in G$, since $d_{c_g\circ\varphi}=d_{\varphi}^g$ as
$$d_{c_g\circ\varphi}(h,k)=d_X(g^{-1}\varphi(h)g,g^{-1}\varphi(k)g)=d_X(\varphi(h)g,\varphi(k)g)=d_{\varphi}^g(h,k)$$
for all $h,k\in G$ and therefore in particular $|c_g\circ\varphi|=|\varphi|_g$.\\ 
We now consider a sequence of homomorphisms $(\varphi_i)\subset\Hom(G,\Gamma)$ such that the following hold:
\begin{enumerate}
	\item $\varphi_i$ is conjugacy-short for all $i\in\N$.
	\item $(\varphi_i)$ does not contain a subsequence $(\varphi_{j_i})$ such that $\ker \varphi_{j_i}=\ker \varphi_{j_{i'}}$ for all $i,i'\in\N$.
\end{enumerate}
The second condition implies hat we may assume that the $\varphi_i$ are pairwise distinct after passing to a subsequence. It follows in particular that $\lim_{i\to\infty}|\varphi_i|=\infty$ as for any $K$ there are only finitely many homomorphisms of norm at most $|K|$.\\
Let now $X_i=(X_i,d_{X_i})$ be the metric space obtained from $X$ by scaling with the factor $\frac{1}{|\varphi_i|}$, thus $X_i=X$ (the underlying sets) and $d_{X_i}=\frac{1}{|\varphi_i|}d_X$. Clearly $G$ also acts on $X_i$ by isometries where the action on the underlying sets $X=X_i$ coincide, hence we obtain a based $G$-space $(X_i,1,\rho_i)$ where the action $\rho_i: G\times X_i\to X_i$ is given by
$$\rho_i(g)x=\varphi_i(g)x \text{ for all } x\in X_i=X.$$
It is immediate that $d_{\rho_i}=\frac{1}{|\varphi_i|}d_{\varphi_i}$. Moreover $d_{\rho_i}$ is $\delta_i$-hyperbolic with 
$\delta_i:= \frac{\delta}{|\varphi_i|}$ and $\lim_{i\to\infty}\delta_i=0$.
Applying Theorem \ref{1.11} to the sequence of based $G$-spaces $(X_i,x_i,\rho_i)$ yields the following Theorem.

\begin{satz}[Theorem 1.12 in \cite{rewe}]\label{paulin}
Let $G$ be a finitely generated group, $\Gamma$ a hyperbolic group and $(\varphi_i)\subset \Hom(G,\Gamma)$ a sequence of conjugacy-short homomorphisms. Then one of the following holds:
\begin{enumerate}[(1)]
	\item $(\varphi_i)$ contains a constant subsequence
	\item A subsequence  of $(\frac{1}{|\varphi_i|}d_{\varphi_i})$ converges to $d^x_{\rho}$ for some non-trivial, minimal based real
				$G$-tree $(T,x,\rho)$.
\end{enumerate} 
\end{satz}

\begin{bem}
\begin{enumerate}[(1)]
\item Suppose we are in case $(2)$ of Theorem \ref{paulin} and moreover that the sequence $(\vp_i)$ is stable. Then the action of $G$ on the limit real tree $T$ induces an action of the limit group $L=G/{\underrightarrow{\ker}(\varphi_i)}$ on $T$ without a global fixed point.
\item Let $(\varphi_i)\subset \Hom(G,\Gamma)$ be a stable sequence of conjugacy-short homomorphisms and suppose $(\varphi_i)$ contains a constant subsequence. Then after passing to this subsequence (still denoted $(\varphi_i)$) we get that $L=G/{\underrightarrow{\ker}(\varphi_i)}=G/\ker(\vp_i)\cong H$ for some subgroup $H$ of $\Gamma$. 
\end{enumerate}
\noindent Hence if $(\vp_i)\subset \Hom(G,\Gamma)$ is a stable sequence of conjugacy-short homomorphisms either the limit group $G/{\underrightarrow{\ker}(\varphi_i)}$ is isomorphic to a subgroup of $\Gamma$ or it acts on a minimal limit real tree without a global fixed point.
\end{bem}

\subsection{Restricted $\Gamma$-limit groups}
We will often not be interested in arbitrary $\Gamma$-limit groups ($\Gamma$ still a hyperbolic group) but in $\Gamma$-limit groups which contain an isomorphic copy of $\Gamma$ as a distinguished subgroup. This leads to the following definition.

\begin{definition}
Let $G$ be a group and $\iota: \Gamma\hookrightarrow G$ an embedding. A homomorphism $\vp: G\to\Gamma$ is called \textnormal{restricted} if $\vp\circ\iota=\id$. Let $$\Hom_{\iota}(G,\Gamma)=\{\vp\in\Hom(G,\Gamma)\ |\ \vp\circ\iota=\id\}.$$
A group $L$ is called a \textnormal{restricted $\Gamma$-limit} group if there exists a group $G$, an embedding $\iota:\Gamma\to G$ and a stable sequence $(\vp_n)\subset \Hom_{\iota}(G,\Gamma)$ such that $L=G/\underrightarrow{\ker}(\vp_n)$.
\end{definition}

From now on we ignore the embedding $\iota:\Gamma\to G$ and consider $\Gamma$ as a subgroup of $G$. Clearly every restricted homomorphism is injective on $\Gamma$ and hence $\Gamma$ is contained in every $\Gamma$-limit group which arises as a quotient of $G$ by a stable kernel of some sequence of restricted  homomorphisms.\\
Now let $G$ be a group and $(\vp_n)\subset \Hom(G,\Gamma)$ a stable sequence of pairwise distinct restricted homomorphisms. Denote by $L$ the associated restricted $\Gamma$-limit group. Theorem \ref{1.11} again provides us with a limit based real tree $(T,x,\rho)$ onto which $G$ (and also $L$) acts. But this tree $T$ may a priori not be minimal or the action of $G$ on $T$ may contain a global fixed point since we are not allowed to postcompose the homomorphisms $\vp_n$ by arbitrary inner automorphisms of $\Gamma$.
The following Theorem shows that we can still guarantee that the limit tree is minimal and the action of $G$ on it is without a global fixed point.

\begin{satz}[Lemma 8.2 in \cite{reinfeld}]
Let $G$ be a finitely generated group, $\Gamma$ a hyperbolic group and $(\varphi_i)\subset \Hom(G,\Gamma)$ a sequence of pairwise distinct  restricted homomorphisms. Then a subsequence of $(\frac{1}{|\varphi_i|}d_{\varphi_i})$ converges to $d^x_{\rho}$ for some non-trivial, minimal based real $G$-tree $(T,x,\rho)$.
\end{satz}

\begin{bem}
Suppose that the sequence $(\vp_i)$ is stable and let $L$ be the corresponding restricted $\Gamma$-limit group. Then as before the action of $G$ on the limit tree $T$ induces an action of $L$ on $T$ and moreover $\Gamma$ fixes the base point of $T$.
\end{bem}

\subsection{Properties of $\Gamma$-limit groups}
In this section we collect some results mainly from \cite{rewe}, which will be used repeatedly in the sequel. It is well-known that in a hyperbolic group $\Gamma$ there exists a uniform bound on the cardinality of torsion subgroups (in particular torsion subgroups are finite). It turns out that the same holds for $\Gamma$-limit groups.

\begin{prop}[Proposition 1.17 and Lemma 1.18 in \cite{rewe}]\label{eig} Let $\Gamma$ be a hyperbolic group and $L$ a $\Gamma$-limit group. Then there exists a constant $N(\Gamma)$ such that the following hold:
\begin{enumerate}[(1)]
\item Every torsion subgroup of $\Gamma$ has at most $N(\Gamma)$ elements.
\item Every torsion subgroup of $L$ has at most $N(\Gamma)$ elements.
\end{enumerate}
\end{prop}

\begin{definition}
\begin{enumerate}[(1)]
\item A group is called \textnormal{$K$-virtually abelian} if it contains an abelian subgroup of index $K\in\N$ and \textnormal{virtually abelian} if it is $K$-virtually abelian for some $K\in\N$.
\item A group is called \textnormal{finite-by-abelian} if it contains a finite normal subgroup $N$ such that $G/N$ is abelian.
\end{enumerate}
\end{definition}

The following Theorem allows us to use the Rips machine (Theorem \ref{rips}) to analyze the action of a $\Gamma$-limit group on its associated limit real tree as constructed in Theorem \ref{paulin}. 

\begin{satz}[Theorem 1.16 in \cite{rewe}]\label{eigenschaften}
Let $G$ be a finitely generated group, $\Gamma$ a hyperbolic group, $(\vp_i)\subset \Hom(G,\Gamma)$ a convergent sequence of pairwise distinct conjugacy-short homomorphisms with corresponding $\Gamma$-limit group $L$ and $T$ be the associated limit real tree. Then the following hold for the action of $L$ on $T$:
\begin{enumerate}[(1)]
\item The stabilizer of any non-degenerate tripod is finite.
\item The stabilizer of any unstable arc is finite.
\item The stabilizer of any non-degenerate arc is finite-by-abelian.
\item Every subgroup of $L$ which leaves a line in $T$ invariant and fixes its ends is finite-by-abelian.
\end{enumerate}
\end{satz}

Note that beside having stated the above Theorem for unrestricted $\Gamma$-limit groups, the same holds for restricted $\Gamma$-limit groups. And since the subgroup $\Gamma$ of a restricted $\Gamma$-limit group $L$ fixes the base point of the associated limit real tree $T$, this allows us to analyze the action of $L$ on $T$ by applying the relative version of the Rips machine (Theorem \ref{relativerips}).\\

Now we turn our attention to virtually abelian subgroups of $\Gamma$-limit groups.

\begin{lemma}[Lemma 1.21 in \cite{rewe}]\label{abelian}
Let $A$ be an infinite $\Gamma$-limit group. Then the following hold:
\begin{enumerate}[(1)]
\item If $A$ is finite-by-abelian with center $Z(A)$ then $|A:Z(A)|<\infty$, i.e. $A$ is virtually abelian.
\item $A$ is virtually abelian if and only if $A$ is either finite-by-abelian or contains a unique finite-by-abelian subgroup of index $2$.
\item $A$ is virtually abelian if and only if all finitely generated subgroups are virtually abelian.
\end{enumerate}
\end{lemma}

Often we will have to deal with maximal virtually abelian subgroups of a $\Gamma$-limit group $L$ and it turns out that every element $a\in L$ of infinite order is contained in a unique maximal virtually abelian subgroup $M(a)$ and $M(a)$ is almost malnormal in $L$.

\begin{prop}[Corollary 1.34 and 1.24 in \cite{rewe}]\label{abe2}
Let $L$ be a $\Gamma$-limit group. Then the following hold:
\begin{enumerate}[(1)]
\item If $a\in L$ is an element of infinite order, then $$M(a):=\langle\{a'\in L\ |\ \langle a,a'\rangle \text{ is virtually abelian }\}\rangle$$ is the unique maximal virtually abelian subgroup of $L$ containing $a$.
\item Any infinite virtually abelian subgroup $H$ is contained in a unique maximal virtually abelian subgroup $M(H)$. In addition the group $M(H)$ is almost malnormal, i.e. if $M(H)\cap gM(H)g^{-1}$ is infinite, then $g\in M(H)$.
\end{enumerate}
\end{prop}

\begin{definition}
Let $K>0$ be some integer.
A group $G$ is called \textnormal{$K$-CSA} if:
\begin{enumerate}[(a)]
\item Any finite subgroup has cardinality at most $K$.
\item Any element $g\in G$ of infinite order is contained in a unique maximal virtually abelian subgroup $M(g)$ and $M(g)$ is $K$-virtually abelian.
\item $M(g)$ is its own normalizer.
\end{enumerate}
\end{definition}

Combining Proposition \ref{eig}, Lemma \ref{abelian} and Proposition \ref{abe2} immediately yields the first assertion of the following result. The second one is Corollary 9.10 in \cite{guirardeljsj}.

\begin{korollar}
$\Gamma$-limit groups are $N(\Gamma)$-CSA. Moreover any subgroup of a $\Gamma$-limit group $G$ is either $N(\Gamma)$-virtually abelian or contains a non-abelian free subgroup.
\end{korollar}

\subsection{Dunwoody decompositions of $\Gamma$-limit groups}
Let $\Gamma$ again be a non-elementary hyperbolic group and $L$ a restricted $\Gamma$-limit group. A splitting of $L$ is a graph of groups $\A$ with $\pi_1(\A)\cong L$ such that $\Gamma$ is contained in a vertex group. In this section we are interested in all splittings of $L$ along finite subgroups, while the next section covers the (more difficult) case of all splittings along virtually abelian subgroups. We only state the results for unrestricted $\Gamma$-limit groups, but it is obvious that they also hold in the restricted case.

\begin{definition}
A group $G$ is called \textnormal{accessible} if there exists a reduced graph of groups $\A$ such that 
\begin{enumerate}[(a)]
\item $\pi_1(\A)\cong G$,
\item any edge group of $\A$ is finite and
\item any vertex group is one-ended or finite.
\end{enumerate}
We call any such graph of groups a \textnormal{Dunwoody decomposition} of $G$.
\end{definition}

\begin{bem}
The maximal vertex groups of a Dunwoody decomposition are unique up to conjugacy. Indeed they are precisely the maximal one-ended subgroups of $G$.
\end{bem}

\begin{satz}
Let $\Gamma$ be a hyperbolic group and $L$ a $\Gamma$-limit group. Then $L$ is accessible.
\end{satz}

This follows from the following theorem of P. Linnell and the fact that the order of finite subgroups of $L$ is bounded by a constant $N(\Gamma)>0$ (Proposition \ref{eig}).

\begin{satz}\label{linnell}\cite{linnell}
Let $G$ be a finitely generated group. Suppose that there exists some constant $C>0$ such that any finite subgroup $H$ of $G$ is of order at most $C$. Then $G$ is accessible.
\end{satz}

We end the section with the following observation:

\begin{lemma} Let $\Gamma$ be hyperbolic, $L$ a one-ended $\Gamma$-limit group. Then either $L$ is virtually abelian, or a finite extension of the fundamental group of a closed orbifold, or has a non-trivial virtually abelian JSJ decomposition.
\end{lemma}

\begin{proof} Assume that $L$ is neither virtually abelian nor a finite extension of the fundamental group of a closed orbifold. By Theorem \ref{rips} $L$ splits as a graph of actions. Since the action of $L$ on the corresponding limit tree is faithful, $L$ does not act with a global fixpoint and hence the splitting is non-trivial.
\end{proof}

\subsection{The virtually abelian JSJ decomposition of a $\Gamma$-limit group}
Let $G$ be a finitely generated group and $\mathcal{H}$ a class of groups (finite, abelian, free, etc.). One could ask, if there exists a splitting of $G$ as a graph of groups $\A$ such that one can read off (in an appropriate sense) all splittings of $G$ along groups contained in $\mathcal{H}$ from $\A$. JSJ theory deals with this question. This theory goes back to the work of Jaco-Shalen and Johannsen on $3$-dimensional topology. For groups the first results were due to Rips and Sela for splittings of finitely presented freely indecomposable groups along infinite cyclic subgroups (\cite{rips}). Since then much has been done for splittings of groups along various classes of subgroups $\mathcal{Z}$. A summary of the historical development together with the most recent results in this area can be found in the excellent work of V. Guirardel and G. Levitt on JSJ decompositions of groups (\cite{guirardeljsj}).\\

Let $\mathcal{Z}$ be a class of subgroups of a group $L$. Given two $\mathcal{Z}$-trees $T_1$ and $T_2$, i.e. simplicial trees on which $L$ acts with edge stabilizers in $\mathcal{Z}$, we say that $T_1$ \textit{dominates} $T_2$ if any group which is elliptic in $T_1$ is also elliptic in $T_2$. A $\mathcal{Z}$-tree is \textit{universally elliptic} if its edge stabilizers are elliptic in every $\mathcal{Z}$-tree.

\begin{definition}\label{defjsj}
Let $L$ be a finitely generated group and $T$ a $\mathcal{Z}$-tree such that
\begin{enumerate}[(a)]
\item $T$ is universally elliptic and
\item $T$ dominates any other universally elliptic $\mathcal{Z}$-tree $T'$. 
\end{enumerate}
We call $T$ a \textnormal{JSJ tree} and the quotient graph of groups $\A=T/L$ a \textnormal{JSJ decomposition} of $L$ (with respect to $\mathcal{Z}$).
\end{definition}

For arbitrary finitely generated groups JSJ decompositions do not always exist, but since we only need the case that $L$ is a one-ended $\Gamma$-limit group ($\Gamma$ hyperbolic) and $\mathcal{Z}$ is the class of all virtually abelian subgroups, we will show that under these assumptions a JSJ decomposition of $L$ does exist.\\
Unfortunately JSJ decompositions are not unique in general, but rather form a deformation space, denoted by $\mathcal{D}_{JSJ}$, which consists of all JSJ trees of $L$. Hence two JSJ trees $T_1,T_2$ are in $\mathcal{D}_{JSJ}$ if and only if they have the same elliptic subgroups.\\
Let now $\A$ be a JSJ decomposition of a finitely generated group $L$. A vertex group $A_v$ of $\A$ is called \textit{rigid} if it is elliptic in every splitting of $L$ along a subgroup from $\mathcal{Z}$ and \textit{flexible} otherwise. The description of flexible vertices of JSJ decompositions (along a given class of groups) is one of the major difficulties in JSJ theory.\\

Let from now on $\Gamma$ be a hyperbolic group and $L$ a $\Gamma$-limit group. As seen in the previous section there exists a Dunwoody decomposition for $L$, which is a JSJ decomposition of $L$ along the class of finite subgroups. Assume now that $L$ is one-ended.
For the study of equations over hyperbolic groups and the proof of a generalization of Merzlyakov's theorem we need to understand splittings of $L$ over virtually abelian subgroups, namely the JSJ decomposition along the class of virtually abelian subgroups.
The construction of this JSJ decomposition and its properties can be found in \cite{rewe}. For convenience we recall the important definitions and results from there.

\begin{definition}\label{defunfolding}
\begin{enumerate}[(1)]
\item Let $\A$ be a graph of groups. Let $e\in EA$ be an edge in $\A$ and $\alpha_e(A_e)\leq C\leq A_{\alpha(e)}$. A \textnormal{folding} along the edge $e$ is the replacement of $A_{\omega(e)}$ by\\ $C\ast_{A_e}A_{\omega(e)}$, $A_e$ by $C$ and the corresponding replacement of the boundary monomorphisms $\alpha_e$ and $\omega_e$. The inverse of a folding is called an \textnormal{unfolding}.
\item Let $\A$ be a splitting of a finitely generated group $G$. We say that $\A$ is a \textnormal{compatible splitting} if all one-ended virtually abelian subgroups are elliptic in it.
\item A vertex group $Q$ of a graph of groups $\A$ is called a \textnormal{QH vertex group} (quadratically hanging vertex group) if the following hold:
\begin{enumerate}[(a)]
\item $Q$ is finite-by-orbifold, i.e. there exists a cone-type orbifold $\mO$ and a finite group $E$ such that $Q$ fits into the short exact sequence
$$1\to E\to Q\xrightarrow{\pi} \pi_1(\mO)\to 1.$$
\item For any edge $e\in EA$ such that $Q$ is the vertex group of $\alpha(e)$ there exists a peripheral subgroup $P_e$ of $\pi_1(\mO)$ such that $\alpha(A_e)$ is in $Q$ conjugate to a finite index subgroup of $\pi^{-1}(P_e)$. 
\end{enumerate}
\item Let $Q$ be a QH subgroup in a graph of groups $\A$ with underlying orbifold $\mO$. Any essential simple closed curve on $\mO$ induces a splitting of $\pi_1(\A)$ along a 2-ended subgroup. we call such a splitting
\textnormal{geometric} with respect to $Q$.
\item A one-edge splitting $\A$ of a group $G$, i.e. either a splitting as an amalgamated product or an HNN-extension, is called \textnormal{hyperbolic-hyperbolic} with respect to another splitting $\B$ of $G$ if the edge group of $\A$ is not conjugate into a vertex group of $\B$ and vice versa.
\end{enumerate}
\end{definition}

Before we proceed we introduce some notation concerning virtually abelian vertex groups of splittings of a group $L$.

\begin{definition}\label{p+}
Let $\A$ be a graph of groups decomposition of a one-ended $\Gamma$-limit group $L$. For every virtually abelian subgroup $A$ of $L$ we denote by $A^+$ the unique maximal finite-by-abelian subgroup of index at most $2$.\\
Let now $A_v$ be a virtually abelian vertex group of $\A$. 
Let $\Delta$ be the set of homomorphisms $\eta: A_v^+\to\Z$ such that $\eta(\alpha_e(A_e^+))=0$ for all $e\in EA$ with $\alpha(e)=v$. We then define $$P_v^+=\{g\in A_v^+\ |\ \eta(g)=0 \text{ for all } \eta\in\Delta\}$$ to be the \textnormal{peripheral subgroup} of $A_v^+$. We call every subgroup $U\leq A_v$ with $U^+=P(A^+_v)$ a \textnormal{peripheral subgroup of $A_v$}. Note that the peripheral subgroup of $A^+_v$ is uniquely defined, while the peripheral subgroup of $A_v$ is not. For every edge $e\in EA$ with $\alpha(e)=v$ we denote the peripheral subgroup of $A_v$ containing $A_e$ by $P(A_v,A_e)$.
\end{definition}

\begin{bem}
$A_v^+/P_v^+$ is a finitely generated free abelian group.
\end{bem}

Now we are ready to define the canonical virtually abelian JSJ decomposition of a one-ended $\Gamma$-limit group.

\begin{definition}\label{jsjdef}
Let $L$ be a one-ended $\Gamma$-limit group and $\A$ be a virtually abelian compatible splitting of $L$. Then $\A$ is called a \textnormal{virtually abelian JSJ decomposition} of $L$ if the following hold.
\begin{enumerate}[(1)]
\item Every splitting over a 2-ended subgroup that is hyperbolic-hyperbolic with respect to another splitting over a 2-ended subgroup is geometric with respect to a QH subgroup of $\A$.
\item Any edge of $\A$ that can be unfolded to be finite-by-abelian is finite-by-abelian.
\item For any virtually abelian vertex group $A_v$, the rank of $A_v^+/P_v^+$ cannot be increased by unfoldings.
\item $\A$ is in normal form and of maximal complexity among all virtually abelian compatible splittings of $L$ that satisfy (1)-(3).
\end{enumerate}
\end{definition}

\begin{bem}
\begin{enumerate}[(1)]
\item The virtually abelian JSJ decomposition is a JSJ decomposition with respect to the class of virtually abelian subgroups in the sense of Definition \ref{defjsj}.
\item Vertex groups in a virtually abelian JSJ decomposition which are neither QH subgroups nor virtually abelian, are rigid.
\item The analogue for restricted limit groups of the graph of groups $\A$ in the above definition is called a \textnormal{restricted JSJ decomposition}.
\end{enumerate}
\end{bem}

\begin{satz}\label{jsj2}
Let $L$ be a one-ended (restricted) $\Gamma$-limit group. Then there exists a (restricted) JSJ decomposition $\A$ of $L$.
\end{satz}

\begin{proof}
This follows from Theorem 3.9 in \cite{rewe} after applying finitely many refinements of non-QH subgroups, unfoldings and the normalization process defined in \cite{rewe} to the graph of groups yielded by this theorem.
\end{proof}

The following theorem explains in which sense one can read off any splitting of a one-ended group along a virtually abelian subgroup from its JSJ decomposition. For a definition of an edge slide and the normal form of a splitting we refer the reader to chapter 3 in \cite{rewe}.

\begin{satz}[Theorem 3.12 in \cite{rewe}]
Let $L$ be a one-ended $\Gamma$-limit group and let $\A$ be a virtually abelian JSJ decomposition of $L$. Then the following hold.
\begin{enumerate}[(1)]
\item Let $\B$ be a virtually abelian compatible splitting of $L$ such that all maximal QH subgroups are elliptic. Assume further that $\B$ is either in normal form or a one-edge splitting. Then $\B$ is visible in a graph of groups obtained from $\A$ by unfoldings followed by foldings and edge slides.
\item Any other JSJ decomposition $\B$ of $L$ can be obtained from $\A$ by a sequence of foldings and unfoldings.
\end{enumerate}
\end{satz}

One of the reasons the virtually abelian JSJ decomposition is so useful, is that one can read off special classes of automorphisms of the group from it.

\begin{definition}\label{modone}
Let $G$ be a one-ended group with virtually abelian JSJ decomposition $\A$. Then $\Mod(G)\leq \Aut(G)$, the \textnormal{modular group of $G$}, is the group generated by the following automorphisms:
\begin{enumerate}[(1)]
\item Inner automorphisms of $G$.
\item If $e\in EA$ with $A_e$ finite-by-abelian and $\A^e$ is the one-edge splitting obtained from $\A$ by collapsing all edges but $e$: A Dehn-twist along $e\in \A^e$ by an elliptic (with respect to $\A$) element $g\in Z(M^+)$ where $M$ is the maximal virtually abelian subgroup containing $A_e$.
\item Natural extensions of automorphisms of QH-subgroups.
\item Natural extensions of automorphisms of maximal virtually abelian vertex groups $A_v$ which restrict to the identity on the peripheral subgroup $P_v^+$ and to conjugation on each virtually abelian subgroup $U\leq A_v$ with $U^+=P_v^+$.\\
 
\end{enumerate}
\end{definition}

\begin{bem}
\begin{enumerate}[(a)]
\item By Corollary 3.18 in \cite{rewe} the modular group of $G$ does not depend on the chosen virtually abelian JSJ decomposition, hence is well-defined. 
\item Once again for restricted $\Gamma$-limit groups one can also define the restricted modular group. This group consists of all modular automorphisms of the restricted virtually abelian JSJ decomposition which fix elementwise the vertex group containing $\Gamma$.
\end{enumerate}
\end{bem}

\begin{definition}
Let $\Gamma$ be a hyperbolic group, $L$ a $\Gamma$-limit group and $\D$ a Dunwoody decomposition of $L$. Then we call the subgroup $\Mod(G)\leq \Aut(G)$, consisting of those automorphisms that restrict to modular automorphisms of the vertex groups of $\D$, the \textnormal{modular group of $L$}.
\end{definition}

Note that modular automorphisms of vertex groups of a Dunwoody decomposition $\D$ of $L$, i.e. of one-ended subgroups of $L$, restrict to conjugation on finite subgroups and hence are naturally extendable to $L$.

\section{Some model theory}\label{3}
\subsection{Introduction}
The first-order language $L_0$ of groups uses the following symbols:
\begin{itemize}
\item The binary function group multiplication $"\cdot"$, the unary function inverse $" ^{-1}"$, the constant $"1"$ and the equality relation $"="$.
\item Variables $"x,y,z"$ or tuples (strings) of variables $"\underline{x}=(x_1,\ldots,x_n)"$ (which will have to be interpreted as individual group elements).
\item Logical connectives $"\wedge"$ ("and"), $"\vee"$ ("or"), $"\neg"$ ("not"), the quantifiers $"\forall"$ ("for all") and $"\exists"$ ("there exists") and parentheses $"("$ and $")"$.
\end{itemize}

A \textit{formula} in the language of groups is a logical expression using the above symbols.
A variable in a formula is called \textit{bound} if it is bound to a quantifier and \textit{free} otherwise. A \textit{sentence} is a formula in $L_0$ where all variables are bound. 

\begin{beispiel}
\begin{enumerate}
\item $\forall x_1 ((x_1x_3=1)\vee (\exists x_2\ x_1x_2=x_3)$ is a formula in $L_0$ but not a sentence, since $x_3$ is a free variable.
\item $\forall x,y\ [x,y]=1$ is a sentence in $L_0$ since both $x$ and $y$ are bound. 
\item The statement $\forall x\exists k\in\N\ x^k=1$ is not allowed in the first-order language of groups, because it quantifies over an integer, and not a group element; similarly one cannot quantify over subsets, subgroups or morphisms. In fact the quantifier does not mention to which group the variables belong, it is the interpretation which specifies the group.
\end{enumerate}
\end{beispiel}

Given a group $G$ and a sentence $\sigma$ we will say that \textit{$G$ satisfies $\sigma$} or \textit{$\sigma$ is a true sentence in $G$} if $\sigma$ is true if interpreted in $G$ (i.e. variables interpreted as group elements in $G$). We denote by $G\models \sigma$ the fact that $\sigma$ is a true sentence in $G$.

\begin{beispiel}
$G\models \forall x,y\ [x,y]=1$ if and only if $G$ is abelian.
\end{beispiel}

\begin{definition}
A \textnormal{universal sentence} (of $L_0$) is a sentence of the form 
$$\forall x_1\ldots\forall x_n\ \vp(x_1,\ldots,x_n),$$ where $\vp(x_1,\ldots,x_n)$ is a quantifier free formula which contains at most the variables $x_1,\ldots,x_n$. Similarly an \textnormal{existential sentence} is one of the form $\exists \underline{x}\ \vp(\underline{x})$, where $\underline{x}=(x_1,\ldots,x_n)$ and $\vp(\underline{x})$ is as above.
An \textnormal{universal-existential-} or \textnormal{$\forall\exists$-sentence} is one of the form $$\forall \underline{y}\exists \underline{x}\ \vp(\underline{x},\underline{y}).$$
A \textnormal{positive sentence} is a sentence containing no negations $\neg$.\\
The \textnormal{universal theory} of a group $G$, which we denote by $\Th_{\forall}(G)$, is the set of all universal sentences of $L_0$ true in $G$. Similarly the \textnormal{existential theory} $\Th_{\exists}(G)$ of a group $G$ is the set of all existential sentences true in $G$.\\
The \textnormal{first-order} or \textnormal{elementary theory} of a group $G$, denoted by $\Th(G)$ is the set of all sentences of $L_0$ true in $G$.
\end{definition}

\begin{bem}
\begin{enumerate}[(1)]
\item It is well-known that every sentence of $L_0$ is logically equivalent to one of the form $Q_1x_1,\ldots Q_nx_n\ \vp(\underline{x}),$
where $\underline{x}=(x_1,\ldots,x_n)$ is a tuple of distinct variables, each $Q_i$ for $1,\ldots,n$ is a quantifier, i.e. either $\forall$ or $\exists$, and $\vp(\underline{x})$ is a formula of $L_0$ containing no quantifiers and at most the variables of $\underline{x}$.
\item Since any existential sentence is logically equivalent to the negation of an universal sentence, it follows that if two groups have the same universal theory, then they also have the same existential theory.
\item If $H\leq G$, a quantifier free formula which is true for every tuple of elements of $G$ is clearly also true for every tuple of elements of $H$, hence $\Th_{\forall}(G)\subset\Th_{\forall}(H)$. We immediately conclude that $\Th_{\forall}(F_2)=\Th_{\forall}(F_n)$ for all $n\geq 3$.
\end{enumerate} 
\end{bem}

Two groups $G$ and $H$ are said to be \textit{elementary equivalent} if they have the same elementary theory, i.e. $\Th(G)=\Th(H)$. The aim is to encode properties of a group in first-order sentences and then conclude that all groups elementary equivalent to this group posses the same properties.

\begin{beispiel}
\begin{enumerate}[(1)]
\item The sentence $\forall x\ (x=1 \vee x^k \neq 1)$ says that the group has no $k$-torsion. Thus, the property of being torsion-free can be encoded by infinitely many elementary sentences. Hence any group elementary equivalent to a torsion-free group is torsion-free.
\item $\Z$ and $\Z^2$ don't have the same elementary theory. Indeed one can encode in a sentence the fact that there are at most $2$ elements modulo the doubles, i.e. the sentence
$$\forall x_1,x_2,x_3\exists x_4\ (x_1=x_2+2x_4)\vee(x_1=x_3+2x_4)\vee(x_2=x_3+2x_4)$$
holds in $\Z$ and not in $\Z^2$ since $\Z/2\Z$ has two elements and $\Z^2/2\Z^2$ has four elements.
\end{enumerate}
\end{beispiel}

If $H$ and $G$ are groups and $f:H\to G$ is a monomorphism, then $f$ is an \textit{elementary embedding} provided whenever $\vp(x_1,\ldots,x_n)$ is a formula of $L_0$ containing at most the distinct variables $x_1,\ldots,x_n$, and $(h_1,\ldots,h_n)\in H^n$, then $\vp(h_1,\ldots,h_n)$ is true in $H$ if and only if $\vp(f(h_1),\ldots,f(h_n))$ is true in $G$.\\
Note that the existence of an elementary embedding $f:H\to G$ is a sufficient but not necessary condition for groups $G$ and $H$ to be elementary equivalent.

\subsection{The Tarski problems}
The following theorem, formerly known as the Tarski conjecture was proved independently by Z. Sela (\cite{sela1}-\cite{sela6}) and Kharlampovich/Myasnikov (\cite{kharlampovich/myasnikov}).

\begin{satz}[Tarski Problem] Finitely generated non-abelian free groups have the same elementary theory.
\end{satz}

\begin{definition}
Let $G$ be a group and $Q$ a formula in the first-order language of the group $G$ (i.e. possibly with coefficients). Let $\underline{p}=(p_1,\ldots,p_n)$ be the tuple of free variables appearing in $Q$. Then the \textnormal{definable set} associated to $Q(\underline p)$ is the set of values $\underline{p_0}$ in $G^n$ such that $Q(\underline{p_0})$ is true in $G$.
\end{definition}

\begin{beispiel}
The following formulas give definable sets:
\begin{enumerate}[(a)]
\item $Q(p)\equiv \exists x,y\ (p=[x,y])$ yields the definable set $\{p_0\in G\ |\  \exists x,y\in G: p_0=[x,y]\}$.
\item $Q(p)\equiv (\exists x_1,x_2,x_3,x_4\ (p=[x_1,x_2][x_3,x_4]))\wedge (\forall y_1,y_2\ (p\neq [y_1,y_2]))$.
\end{enumerate}
\end{beispiel}

Arbitrary definable sets (as in the above example) are Diophantine sets, i.e. they may depend on additional variables aside from $\underline p$. The easiest definable sets are definable sets which are given by sentences without additional variables, thus are varieties.

\begin{definition}\label{definableset}
A \textnormal{basic definable set} is a definable set which corresponds to a formula which is a system of equations, i.e. is of the form
$$Q(\underline p)\equiv (\Sigma(\underline p,\underline a)=1),$$
where $\underline a$ is a tuple of constants from $G$. Hence a basic definable set is a variety over $G$.
\end{definition}

\begin{beispiel}
Let $F_k=\langle a\rangle$ be a non-abelian free group and $$L=\langle y, a\ |\ V(y,a)\rangle$$ a restricted $F_k$-limit group. Then the basic definable set (respectively the variety) associated to $L$ is $$\{y_0\in F_k^n\ |\ V(y_0,a)=1\},$$ i.e. the definable set corresponding to the formula 
$$Q(y)\equiv (V(y,a)=1).$$
\end{beispiel}

\section{Makanin-Razborov resolutions and the shortening argument}\label{4}
\subsection{Systems of equations over hyperbolic groups}
One key part in Sela's solution to the Tarski problems and his subsequent study of the elementary theory of torsion-free hyperbolic groups is the  understanding of the structure of definable sets over free and hyperbolic groups. As seen in Definition \ref{definableset} the easiest of such definable sets are varieties over the given free or hyperbolic group.
Let $\Gamma$ be a hyperbolic group. Recall that a variety over $\Gamma$ is a set $$\{\underline{g}=(g_1,\ldots,g_n)\in G^n\ |\ \Sigma(\underline{g},\underline{a})=1\},$$
where $\Sigma(\underline{x},\underline{a})=1$ is a system of equations with variables $\underline{x}=(x_1,\ldots,x_n)$ and coefficients $\underline{a}$ from $\Gamma$. In the following we will just write $x$ instead of $\underline{x}$. In order to be able to understand definable sets over $\Gamma$, one first of all has to understand solution to systems of equations over $\Gamma$. If $\Gamma$ is a free group this was done by the work of Makanin and Razborov, who developed an algorithm to determine if a given system of equations over a free group has a solution (Makanin) and described the set of solutions to this system of equations (Razborov). This algorithm is commonly known as the Makanin-Razborov algorithm. O. Kharlampovich and A. Myasnikov used a refined version of this algorithm in their solution to the Tarski problems, while Sela used a more geometric approach to get a description of the solution set of a system of equations by heavily exploiting the Rips machine and bypassing much of the combinatorics in the process. The advantage of Sela's approach is, that it can fairly easy be generalized to understand system of equations over other classes of groups. The most important generalizations are for systems of equations over
\begin{itemize}
\item torsion-free hyperbolic groups by Sela (\cite{selahyp}),
\item all hyperbolic groups by Reinfeldt-Weidmann (\cite{rewe})
\item acylindrical hyperbolic groups by Groves-Hull (\cite{groves})
\item free products of groups by Jaligot-Sela (\cite{jaligot})
\item free semigroups by Sela. (\cite{sela11})
\end{itemize}
Note that this list is by no means complete.\\

It turns out that understanding all solutions to a system of equations over a given group $\Gamma$ is equivalent to understanding $\Hom(G,\Gamma)$, i.e. all homomorphisms from a finitely generated group $G$ to $\Gamma$. This can be seen by considering a system of equations $\Sigma(x,a)$ with variables $x=(x_1,\ldots,x_n)$ and coefficients $a=(a_1,\ldots,a_k)$ from $\Gamma$, given by a finite set of words $w_1(x,a),\ldots,w_r(x,a)\in F(x)\ast\Gamma=\langle x,a\rangle$. Then clearly every solution of $\Sigma(x,a)=1$ corresponds to a homomorphism 
$$\vp:G_{\Sigma}:=\langle x, \Gamma\ |\ w_1(x,a)=1,\ldots,w_r(x,a)=1\rangle\to \Gamma$$
and vice versa. In the following we will describe Sela's approach to the solution of systems of equation over hyperbolic groups. The aim is to get a parametrization of $\Hom(G,\Gamma)$, where $G$ is a finitely generated group. Since we need the result for all hyperbolic groups, instead of only torsion-free hyperbolic ones, we are going to present the generalization of Sela's techniques which appear in \cite{rewe}.

\subsection{The shortening argument}
In this section we will give a short account of Z. Sela's shortening argument.   For the original shortening argument in the torsion free case see \cite{rips}. Since we have to deal with torsion, the version of the shortening argument presented here is borrowed from \cite{rewe}. For a detailed discussion on the subject and proofs of the results stated in this chapter we therefore refer the reader to chapter 4.2 in the aforementioned paper. Throughout the whole chapter let $\Gamma$ be a hyperbolic group.

\begin{definition}
Let $G$ be a finitely generated $\Gamma$-limit group with a fixed generating set $S$, $H\leq G$ a subgroup, $\Gamma$ a hyperbolic group with corresponding Cayley graph $(X,d_X)$ and $\vp:G\to \Gamma$ a homomorphism. The \textnormal{length} of $\vp$ is defined as 
$$l(\vp)=\sum_{s\in S} d_X(1,\vp(s)).$$
\begin{enumerate}[(1)]
\item We call $\vp$ \textnormal{short} if $l(\vp)\leq l(c_g\circ\vp\circ\alpha)$ for any $g\in \Gamma$ and any $\alpha\in \Mod(G)$, where $c_g$ denotes conjugation by $g$.
\item $\vp$ is called \textnormal{short with respect to $H$} if $l(\vp)\leq l(c_g\circ\vp\circ\alpha)$ for any $g\in C_{\Gamma}(\vp(H))$ and any $\alpha\in\Mod(G,H)$, i.e. modular automorphisms of $G$ fixing $H$. In particular if   $G$ is a restricted $\Gamma$-limit group, $\vp$ is short with respect to $\Gamma$ if $l(\vp)\leq l(\vp\circ\alpha)$ for any $\alpha\in\Mod(G,\Gamma)$.
\end{enumerate}
\end{definition}

\begin{satz}[Proposition 4.6 in \cite{rewe}]\label{shortening argument}
Let $L$ be a one-ended (restricted) $\Gamma$-limit group and $(\vp_n)\subset \Hom(L,\Gamma)$ a stably injective sequence of pairwise distinct non-injective homomorphisms. Then the $\vp_n$ are eventually not short in the unrestricted case and not short relative $\Gamma$ in the restricted case.
\end{satz}

\begin{proof}
For the sake of contradiction assume that the homomorphisms $\vp_n$ are short.
By Theorem \ref{paulin} the sequence $(\vp_n)$ converges into a minimal action of  $L$ on some limit real tree $(T,t_0)$. By Theorem \ref{eigenschaften} the action satisfies the stability assertions of Theorem \ref{rips}. Since $L$ is one-ended and stabilizers of tripods and unstable arcs are finite, the action of $L$ on $T$ splits as a non-trivial graph of actions $\A$, such that every vertex action is either simplicial, of axial or orbifold type. Now we can use this decomposition as a graph of actions to construct modular automorphisms $\alpha_n\in\Mod(L)$, such that $l(\vp_n\circ\alpha)<l(\vp_n)$ for large $n$, i.e. show that for large $n$ the $\vp_n$ are not short. Denote by $(X, d_X)$ the Cayley graph of $\Gamma$ with respect to some fixed generating set.

\begin{lemma}[Theorem 4.8 in \cite{rewe}]\label{axial shortening}
Let $v\in VA$ be an axial vertex with corresponding vertex tree $T_v$. For any finite subset $S\subset L$ there exists $\alpha\in\Mod(L)$ such that for any $s\in S$ the following hold.
\begin{enumerate}[(a)]
\item If $[t_0,st_0]\subset T$ has a non-degenerate intersection with a translate of $T_v$ in $T$, then $$d_T(t_0,\alpha(s)t_0)<d_T(t_0,st_0),$$
\item otherwise, $$d_T(t_0,\alpha(s)t_0)=d_T(t_0,st_0).$$
\end{enumerate}
\end{lemma}

\begin{lemma}[Theorem 4.15 in \cite{rewe}]\label{orbifold shortening}
Let $v\in VA$ be an orbifold vertex with corresponding vertex tree $T_v$. For any finite subset $S\subset L$ there exists $\alpha\in\Mod(L)$ such that for any $s\in S$ the following hold.
\begin{enumerate}[(a)]
\item If $[t_0,st_0]\subset T$ has a non-degenerate intersection with a translate of $T_v$ in $T$, then $$d_T(t_0,\alpha(s)t_0)<d_T(t_0,st_0),$$
\item otherwise, $$d_T(t_0,\alpha(s)t_0)=d_T(t_0,st_0).$$
\end{enumerate}
\end{lemma}

Suppose we are either in the case of Lemma \ref{axial shortening} or Lemma \ref{orbifold shortening}. Denote by $d_n$ the scaled metric on the Cayleygraph $X$. In the unrestricted case after postcomposing $\vp$ by an inner automorphism of $\Gamma$ we can assume that $(1)$ is an approximating sequence of $t_0$. In the restricted case this is given by the construction of the limit tree. Therefore 
$$\lim_{n\to\infty}d_n(1,\vp_n\circ\alpha(s))=d_T(t_0,\alpha(s)t_0)<d_T(t_0,st_0)=\lim_{n\to\infty}d_n(1,\vp_n(s))$$
and hence for large $n$, $d_X(1,\vp_n\circ\alpha(s)) <d_X(1,\vp_n(s))$. This implies that for large $n$, $\vp_n$ is not short in the unrestricted case and not short with respect to $\Gamma$ in the restricted case.\\
It remains to consider the case of a simplicial vertex tree. In this case it is not possible to directly shorten the action on the limit tree, but we  are still able to shorten the action of $L$ on $X$ for large $n$.\\
If there exists a simplicial vertex tree in the graph of actions $\A$ we can refine the graph of actions such that every simplicial vertex tree is a point. Then in particular there exists an edge $e\in EA$ with non-zero length. By Theorem \ref{eigenschaften} the stabilizer $A_e$ of this edge is finite-by-abelian and hence has infinite center. Let $\alpha$ be a Dehn-twist along the edge $e$ by an element of infinite order from $Z(A_e)$.

\begin{lemma}[Corollary 4.20 in \cite{rewe}]\label{simplicial shortening}
For sufficiently large $n$, there exists $m_n\in\N$ such that 
$$l(\vp_n\circ\alpha^{m_n})<l(\vp_n).$$
\end{lemma}

This completes the proof of Theorem \ref{shortening argument}.
\end{proof}

A combination of Lemma \ref{axial shortening}, Lemma \ref{orbifold shortening} and Lemma \ref{simplicial shortening} is usually referred to as the \textit{shortening argument}.\\

We draw one corollary which allows us later to shorten relative a finite set of inequalities.

\begin{korollar}\label{shortening}
Let $L$ be a (restricted) one-ended $\Gamma$-limit group, $S\subset L\setminus \{1\}$ a finite set and $(\vp_n)$ a stable sequence of pairwise distinct non-injective homomorphisms, such that the $\vp_n$ are short among all homomorphisms which map all $s\in S$ to a non-trivial image. Then $\underrightarrow{\ker} (\vp_n)\neq 1$. In particular $L/\underrightarrow{\ker} (\vp_n)$ is a proper quotient of $L$.
\end{korollar}

\begin{proof}
Suppose that $\underrightarrow{\ker} (\vp_n)= 1$, i.e. $(\vp_n)$ is stably injective. It then follows from the shortening argument that there exist modular automorphisms $\alpha_n\in\Mod(L)$ (respectively $\alpha_n\in\Mod(L,\Gamma)$ in the restricted case) and $g_n\in \Gamma$ such that $l(c_{g_n}\circ\vp_n\circ\alpha_n)<l(\vp_n)$. Note that $g_n=1$ in the restricted case. It is implicitly contained in the proof of the shortening argument in \cite{rewe} that for large $n$ the $\alpha_n$ can be chosen such that $\vp_n\circ\alpha_n(s)\neq 1$ for all $s\in S$. This yields a contradiction to the assumption that the $\vp_n$ are short among all homomorphisms with this property.
\end{proof}

\subsection{Shortening quotients}
Equipped with the shortening argument we are now ready to draw some consequences from Theorem \ref{shortening argument}. Let $\Gamma$ again be a hyperbolic group.

\begin{definition}
Let $L$ be a (restricted) one-ended $\Gamma$-limit group and $(\vp_n)\subset\Hom(L,\Gamma)$ a stable sequence of short homomorphisms. Then $Q:=L/\underrightarrow{\ker}(\vp_n)$ is called a \textnormal{shortening quotient} of $L$.
\end{definition}



We make the following observation.

\begin{lemma}\label{Lemma quotient}
Let $L$ be a (restricted) one-ended $\Gamma$-limit group. 
Then either there exists a proper shortening quotient $Q$ of $L$, or $L$ is isomorphic to a subgroup of $\Gamma$.
\end{lemma}

\begin{proof}
Let $(\vp_n)\subset\Hom(L,\Gamma)$ be a stable sequence of homomorphisms from $L$ to $\Gamma$.
Clearly either $(\vp_n)$ contains a constant subsequence or we can assume after passing to a subsequence that the $\vp_n$ are pairwise distinct. Suppose the second case holds. For all $n\in\N$ let $\alpha_n\in\Mod(L)$ be a modular automorphism such that $(\vp_n\circ\alpha_n)$ is a stable sequence of short homomorphisms. Hence it follows from Corollary \ref{shortening} that $$Q:=L/\underrightarrow{\ker} (\vp_n\circ\alpha_n)$$ is a proper quotient of $L$.\\
Now assume that every stable sequence of homomorphisms from $L$ to $\Gamma$ has a constant subsequence. Since $L$ is a $\Gamma$-limit group, there exists a stably injective sequence $(\vp_n)\subset \Hom(L,\Gamma)$. After passing to the constant subsequence of $(\vp_n)$ we have $\vp_n=\vp$ for all $n\in\N$ and some fixed injective homomorphism $\vp:L\to\Gamma$. In particular
$$L=L/\underrightarrow{\ker}(\vp_n)=L/\ker \vp\cong \im(\vp)\leq \Gamma.$$
\end{proof}

In the next step we will define shortening quotients not only for one-ended $\Gamma$-limit groups but for arbitrary $\Gamma$-limit groups.

\begin{definition}\label{definitionshortening}
Let $L$ be a (restricted) $\Gamma$-limit group and $\D$ a (restricted) Dunwoody decomposition of $L$, i.e. every vertex group is either finite or one-ended (relative $\Gamma$ in the restricted case).
Let $(\vp_n)\subset \Hom(L,\Gamma)$ be a sequence of homomorphisms such that $\vp_n|_{D_v}$ is short for every vertex group $D_v$ of $\D$ for all $n\in\N$.
Then $Q:=L/\underrightarrow{\ker}(\vp_n)$ is called a \textnormal{shortening quotient} of $L$.
\end{definition}

We immediately get an equivalent result to Lemma \ref{Lemma quotient}.

\begin{lemma}
Let $L$ be a (restricted) $\Gamma$-limit group with Dunwoody decomposition $\D$. Then either there exists a proper shortening quotient $Q$ of $L$, or all vertex groups of $\D$ are isomorphic to subgroups of $\Gamma$.
\end{lemma}

We now want to show that there exist only finitely many maximal shortening quotients of a given (restricted) $\Gamma$-limit group $L$ with respect to the following partial order defined on the set of all $\Gamma$-limit quotients of $L$.

\begin{definition}
Let $L$ be a (restricted) $\Gamma$-limit group and $Q_1$, $Q_2$ $\Gamma$-limit quotients of $L$ with corresponding quotient maps $\pi_i:L\to Q_i$ for $i=1,2$. We say that $Q_1\geq Q_2$ if there exists an epimorphism $\eta:Q_1\to Q_2$ such that $\pi_2=\eta\circ\pi_1$. We further say that $Q_1>Q_2$ if $Q_1\geq Q_2$ and $Q_2\ngeq Q_1$.
\end{definition}

\begin{satz}\label{maximal elements}
Let $L$ be a (restricted) $\Gamma$-limit group.
\begin{enumerate}[(1)]
\item Let $Q_1<Q_2<Q_3<\ldots$ be an infinite ascending sequence of $\Gamma$-limit quotients of $L$. Then there exists a maximal $\Gamma$-limit quotient $Q$ of $L$ such that $Q_i<Q$ for all $i\geq 1$.
\item There exist only finitely maximal $\Gamma$-limit quotients of $L$ with respect to $"<"$.
\end{enumerate}
\end{satz}

\begin{proof}
This is a combination of Theorem 6.10 and Theorem 6.11 in \cite{rewe}.
\end{proof}

\begin{bem}
In particular there exist only finitely many maximal shortening quotients of a $\Gamma$-limit group $L$.
\end{bem}

The following result is commonly referred to as the descending chain condition for $\Gamma$-limit groups. It states that there cannot be an infinite descending sequence of $\Gamma$-limit quotients of a given $\Gamma$-limit group. This will be crucial in the sequel, to show that certain iterative procedures terminate after finitely many steps.

\begin{prop}[Descending chain condition]\label{descendingchain} Let $L$ be a (restricted) $\Gamma$-limit group. Then there exists no infinite descending sequence $Q_1>Q_2>Q_3>\ldots$ of $\Gamma$-limit quotients of $L$. 
\end{prop}

\begin{proof}
This follows immediately from Lemma 6.8 in \cite{rewe}.
\end{proof}

Later on we will need more information about the algebraic structure of special types of shortening quotients, the so-called strict shortening quotients.

\begin{definition}
Let $\Gamma$ be a hyperbolic group and $L$ be a one-ended  (restricted) $\Gamma$-limit group. Let $\A$ be a virtually abelian JSJ decomposition of $L$. We say that a proper shortening quotient $Q$ of $L$ with corresponding quotient map $\eta: L\to Q$ is a \textnormal{strict shortening quotient} if the following hold:
\begin{enumerate}[(1)]
\item $\eta$ maps each QH subgroup $A_v$ of $L$ to a non-virtually abelian image. Moreover if $1\to E\to A_v\xrightarrow{\pi} O\to 1$ is the short exact sequence corresponding to $A_v$, then for every peripheral subgroup $O_e$ of $O$, $\pi^{-1}(O_e)$ is mapped isomorphically into $Q$ by $\eta$.
\item If $A_v$ is a virtually abelian vertex group, then $\eta$ maps every subgroup $U\leq A_v$ with $U^+=P_v^+$ injectively into $Q$. (For a definition of $P_v^+$ see Definition \ref{p+}).
\item Every virtually abelian edge group $A_e$ of $\A$ gets embedded into $Q$ by $\eta$.
\item Let $R$ be a rigid vertex group in $\A$. The \textit{envelope} $E(R)$ of $R$ is the fundamental group of a graph of groups $\B$ whose underlying graph is a star with one vertex $v$ with vertex group $R$ connected to all other vertices via an edge of the form $(v,u)$ for some $u\in VB\setminus \{v\}$. Edges in $\B$ are in 1-to-1 correspondence to edges adjacent to $R$ in $\A$ and the edge groups are isomorphic. Let now $e=(v,u)$ be an edge in $\B$ and $e'=(v',u')$ the corresponding edge in $\A$ adjacent to $R$ (where $R$ is the rigid vertex group of $v'$). If $A_{u'}$ is rigid or a QH vertex group, then $B_u$ is the maximal virtually abelian subgroup of $A_{u'}$ containing $A_{e'}$. If $A_{u'}$ is a virtually abelian vertex group, then $B_u$ is the peripheral subgroup of $A_{u'}$ containing $A_{e'}$.\\
Then $\eta$ is injective when restricted to the envelope $E(R)$ of $R$. In particular every rigid vertex group in $\A$ gets embedded into $Q$.
\end{enumerate}
\end{definition}

As before when we generalized shortening quotients from one-ended groups to the case of arbitrary $\Gamma$-limit groups, we can do the same for strict shortening quotients.

\begin{definition}
Let $L$ be a (restricted) $\Gamma$-limit group with (restricted) Dunwoody decomposition $\D_L$. Let $Q$ be a proper shortening quotient of $L$ with (restricted) Dunwoody decomposition $\D_Q$ and denote by $\eta:L\to Q$ the corresponding quotient map. We call $Q$ a \textnormal{strict shortening quotient} of $L$ if the following hold. 
\begin{enumerate}
\item The underlying graph $D_Q$ of $\D_Q$ is a refinement of the underlying graph $D_L$ of $\D_L$.
\item Let $v\in VD_L$. Then either $Q_v:=\eta(D_v)$ is a strict shortening quotient of $D_v$ or $D_v$ is isomorphic to a subgroup of $\Gamma$ and there exists a vertex group $Q_v$ in $\D_Q$ such that $\eta$ maps $D_v$ isomorphically onto $Q_v$. Note that we do not claim that $Q_v$ is one-ended in the first alternative (hence in particular $Q_v$ may not be isomorphic to a vertex group of $\D_Q$).
\end{enumerate}
\end{definition}

The significance of strict shortening quotients will become apparent in the following chapters.

\subsection{Makanin-Razborov diagrams}
In this chapter we will give a short account of the construction of a Makanin-Razborov diagram for a finitely generated group $G$ over a hyperbolic group $\Gamma$. Essentially everything from this chapter is just a reformulation of the corresponding results in \cite{rewe} to better fit our needs. 

\begin{definition}
Let $G$ be a finitely generated accessible group, $\Gamma$ a group. Then we call a homomorphism $\Psi: G\to \Gamma$ locally injective if $\Psi$ is injective when restricted to the vertex groups of a Dunwoody decomposition of $G$.
\end{definition}

The following is the main theorem in \cite{rewe}.

\begin{satz}[Theorem 5.7 in \cite{rewe}]\label{MRdiagram}
Let $\Gamma$ be a hyperbolic group and $G$ be a finitely generated group. Then there exists a finite directed rooted tree $T$ with root $v_0$ satisfying:
\begin{enumerate}[(a)]
\item The vertex $v_0$ is labeled by $G$.
\item Any vertex $v\in VT$, $v\neq v_0$, is labeled by a $\Gamma$-limit group $G_v$.
\item Any edge $e\in ET$ is labeled by an epimorphism $\pi_e:G_{\alpha(e)}\to G_{\omega(e)}$ such that for any homomorphism $\vp:G\to \Gamma$ there exists a directed path $e_1,\ldots,e_k$ from $v_0$ to some vertex $\omega(e_k)$ such that 
$$\vp=\Psi\circ\pi_{e_k}\circ\alpha_{k-1}\circ\pi_{e_{k-1}}\circ\ldots\circ\alpha_1\circ\pi_{e_1}$$
where $\alpha_i\in \Mod G_{\omega(e_i)}$ for $1\leq i\leq k$ and $\Psi$ is locally injective.
\end{enumerate}
We call $T$ a \textnormal{Makanin-Razborov diagram} of $G$, and a directed path as in $(c)$, a \textnormal{Makanin-Razborov resolution} of $G$.
\end{satz}

Despite the growing interest in these diagrams since Sela's solution to the Tarski conjectures, very few actual examples of MR diagrams have been computed. If $\Gamma$ is a non-abelian free group and $G$ is either a surface group, a free abelian group, or a non-abelian free group, Makanin-Razborov diagrams for $G$ are well-known (see for example \cite{champetierguirardel}).
All these examples have in common that the JSJ decomposition of the limit group is trivial, namely consists of a single vertex. In \cite{heil} we construct a Makanin-Razborov diagram of a $F_k$-limit group $L$ which is a double $F_2\ast_{\langle w\rangle}F_2$ along a special word $w\in F_2$. In particular the JSJ decomposition of $L$ consists of two rigid vertices connected by an edge with cyclic edge group, hence is non-trivial. To our knowledge this is the first example of an explicit Makanin-Razborov diagram of a group with non-trivial JSJ decomposition.\\

In our work we will only be interested in Makanin-Razborov diagrams of $\Gamma$-limit groups, that is if the group $G$ which labels the root $v_0$ is a finitely generated $\Gamma$-limit group. Therefore we only sketch the proof of Theorem \ref{MRdiagram} in this case. It should be noted though that this case already covers essentially all the difficulties of the general case. As mentioned before, a complete proof of the above Theorem for an arbitrary finitely generated group $G$ can be found in \cite{rewe}.

\begin{proof}[Proof of Theorem \ref{MRdiagram}.] Since by Proposition \ref{descendingchain} every descending chain of $\Gamma$-limit groups eventually stabilizes, it clearly suffices to show that any $\Gamma$-limit group $L$ admits a finite set $\{q_i:L\to Q_i\ |\ 1\leq i\leq n\}$ of proper quotient maps such that any homomorphism $\vp:L\to \Gamma$ which is not locally injective, factors through some $q_i$ after precomposition with an element of $\Mod(L)$. Such a set of quotient maps is called a \textit{$\Gamma$-factor set of $L$ relative $\Mod(L)$}.\\
With the construction of the set of maximal shortening quotients of $L$ which exists by Theorem \ref{maximal elements} we have already done the main part of the work.\\
So let $\D$ be a Dunwoody decomposition of $L$. By Theorem \ref{maximal elements} there exist finitely many maximal shortening quotients $Q_1,\ldots, Q_m$ of $L$. Denote the corresponding quotient maps by $\pi_i:L\to Q_i$ for all $i\in\{1,\ldots,m\}$.\\
We claim that for every non-locally injective homomorphism $\vp:L\to\Gamma$ there exists a modular automorphism $\alpha\in\Mod(L)$ and (in the unrestricted case) an inner automorphisms $c_g$ of $\Gamma$ such that $c_g\circ\vp\circ\alpha$ factors through one of the $\pi_i$'s.\\
So let $\vp:L\to\Gamma$ be a homomorphism and suppose that $\vp$ is not locally injective. Then there exists a modular automorphism $\alpha\in\Mod(L)$ and (in the unrestricted case) an inner automorphism $c_g$ of $\Gamma$ such that $\hat{\vp}=c_g\circ\vp\circ\alpha$ is short.\\
The constant sequence $(\hat{\vp})$ converges into a $\Gamma$-limit group $M$ with corresponding $\Gamma$-limit map $\hat{\pi}:L\to M$. Since by assumption $\vp$ is not injective when restricted to at least one vertex group $\D_v$, $M$ is a proper quotient of $L$. Since $\hat{\vp}$ was chosen to be short when restricted to every vertex group of $\D$, $M$ is a proper shortening quotient of $L$. Hence by Theorem \ref{maximal elements} either $M$ is itself a maximal shortening quotient of $L$ or there exists a maximal shortening quotient $Q$ of $L$ such that $M<Q$. In both cases $\hat{\vp}$ splits through a maximal shortening quotient of $L$. As mentioned before the descending chain condition for $\Gamma$-limit groups (Proposition \ref{descendingchain}) seals the deal.
\end{proof}

\subsection{Strict and well-structured resolutions}
\begin{definition}
Let $\Gamma$ be a hyperbolic group, $G$ a $\Gamma$-limit group and let 
$$G=G_0\xrightarrow{\eta_0} G_1\xrightarrow{\eta_1}G_2\rightarrow \cdots\xrightarrow{\eta_{m-1}}G_m=\pi_1(\D)$$
be a MR resolution of $\Gamma$, such that $\D$ is a Dunwoody decomposition of $G_m$ and every vertex group of $\D$ can be embedded into $\Gamma$. We say that the given resolution is a \textnormal{strict MR resolution} if $G_i$ is a strict shortening quotient for all $i\in\{1,\ldots,m\}$.
\end{definition}

\begin{prop}\label{prop1.29} Let $L$ be a $\Gamma$-limit group. Then the  Makanin-Razborov diagram of $L$ contains a strict Makanin-Razborov resolution.
\end{prop}

\begin{proof}
Using the results from \cite{rewe} this is identical to the proof of Proposition 1.29 in \cite{selahyp}. Note that in the case that every vertex group of a Dunwoody decomposition of $L$ is isomorphic to a subgroup of $\Gamma$, the empty resolution is considered a strict resolution.
\end{proof}

We are now going to replace the canonical collection of MR resolutions associated with a restricted $\Gamma$-limit group, with a canonical collection of strict MR resolutions.

\begin{prop}\label{prop101}
Let $L$ be a restricted $\Gamma$-limit group. There exists a canonical finite collection of strict MR resolutions $\Res_1(L),\ldots,\Res_s(L)$, such that the $\Gamma$-limit groups associated with the resolutions $\Res_i(L)$ (i.e. the $\Gamma$-limit groups where the resolution starts) are either $L$ itself or a quotient of it, and every homomorphism from $L$ to $\Gamma$, factors through (at least) one of these resolutions.
\end{prop}

\begin{proof}
Let $U$ be the collection of all MR resolutions that appear in the canonical MR diagram of the restricted $\Gamma$-limit group $L$. Each resolution in $U$, which is a strict resolution is taken to be one of the resolutions in our new collection. By Proposition \ref{prop1.29}, U contains at least one strict resolution. Each non-strict resolution in $U$ will be replaced by a finite collection of strict resolutions either of $L$ or of a quotient of it in the following way:
Let 
$$\Res(L):\ L=G_0\xrightarrow{\eta_0} G_1\xrightarrow{\eta_1}G_2\rightarrow \cdots\xrightarrow{\eta_{m-1}}G_m=\pi_1(\D)$$
be a resolution in the MR diagram of $L$, which is not a strict resolution. For each level $j\in\{1,\ldots,m-1\}$ let $R_j$ be the quotient group of $G_j$ obtained from the collection of homomorphisms $G_j\to\Gamma$, that factor through the part $G_j\xrightarrow{\eta_j}G_{j+1}\rightarrow \cdots\xrightarrow{\eta_{m-1}}G_m$ of $\Res(L)$, i.e. the quotient of $G_j$ by the subgroup generated by all elements which lie in the kernel of all such homomorphisms.\\
Since $\Res(L)$ is not a strict resolution at least for some level $j$, $R_j$ is a proper quotient of $G_j$ with corresponding quotient map $\tau_j$. Denote by $R_j^1,\ldots, R_j^m$ the finitely many maximal $\Gamma$-limit quotients of $R_j$ that appear on the level below $R_j$ in the MR diagram of $R_j$. We assume without loss of generality that $R_j$ is already a $\Gamma$-limit group (if this is not the case we proceed in the following with the finitely many $\Gamma$-limit quotients $R_j^1,\ldots, R_j^m$). Let $j_0$ be the highest integer $j$ (i.e. the lowest level in figure \ref{fig7}) for which $R_{j_0}$ is a proper quotient of $G_{j_0}$. We replace the resolution $\Res(L)$ by finitely many resolutions obtained in the following way:
\begin{enumerate}[(1)]
\item The top part of the new resolutions is identical with the top part of the resolution $\Res(L)$ starting at $j_0-1$ and above.
\item The $\Gamma$-limit group at level $j_0$ in all the new resolutions is equal to $R_{j_0}$ and the corresponding (new) epimorphism $\eta'_{j_0-1}: G_{j_0-1}\to R_{j_0}$ is equal to $\tau_{j_0}\circ\eta_{j_0-1}$.
\item Each of the new resolutions continues after the level $j_0$ along one of the resolutions that appear in the MR diagram of the restricted $\Gamma$-limit group $R_{j_0}$.
\end{enumerate}
\begin{figure}[htbp]
\centering
\includegraphics{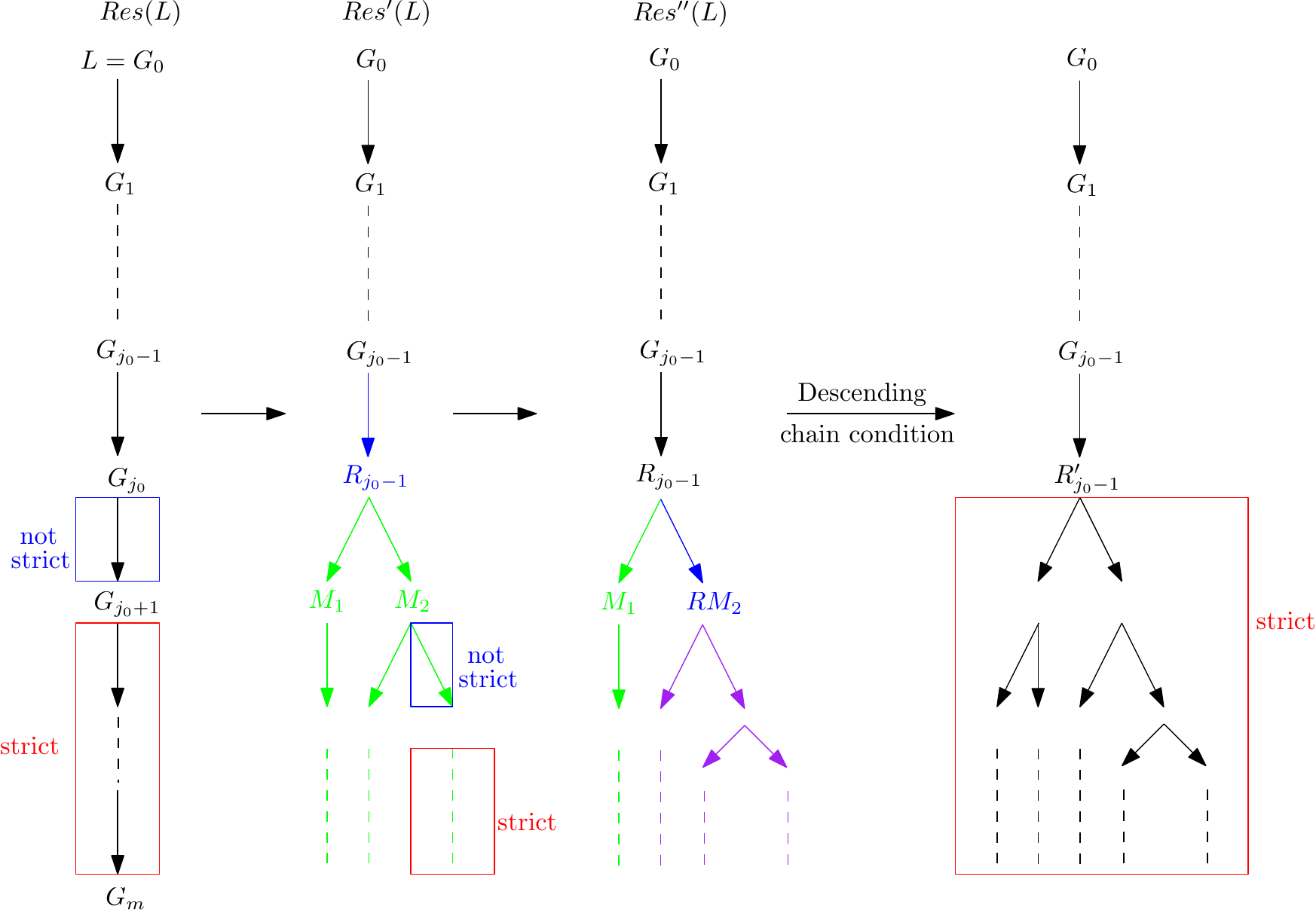}
\caption{Constructing strict resolutions}
\label{fig7}
\end{figure}
Hence every new resolution is of the form
$$L=G_0\xrightarrow{\eta_0} \cdots\xrightarrow{\eta_{j_0-2}}G_{j_0-1}\xrightarrow{\tau_{j_0}\circ\eta_{j_0-1}}R_{j_0}\xrightarrow{\pi_0} M_1\xrightarrow{\pi_1}M_2\rightarrow\cdots\xrightarrow{\pi_{l-1}}M_l=\pi_1(\D')$$
where 
$$R_{j_0}\xrightarrow{\pi_0} M_1\xrightarrow{\pi_1}M_2\rightarrow\cdots\xrightarrow{\pi_{l-1}}M_l$$
is a MR resolution in the MR diagram of $R_{j_0}$.\\
Clearly every homomorphism from $L$ to $\Gamma$ that factors through $\Res(L)$, factors through (at least) one of the new resolutions.
Any strict resolution from the obtained collection of new resolutions is taken to be a resolution in our collection of strict resolutions. If one of the new obtained resolutions is not strict we replace it by finitely many new resolutions obtained by the same procedure as before.\\
Whenever we replace a $\Gamma$-limit group which lies on a lower level then $j_0$ (i.e. on some level $j>j_0$) in the new resolutions with its proper quotient $R_j$, then $R_j$ is clearly also a proper quotient of $G_{j_0}$.
Hence by the descending chain condition (Proposition \ref{descendingchain}), this can happen only finitely many times.
Therefore after finitely many steps all of the new resolutions are strict resolutions below level $j_{0}-1$ (see figure \ref{fig7}) and hence the procedure described above terminates after finitely many steps. By construction after the termination of the procedure every obtained resolution is a strict MR resolution and every homomorphism from $L$ to $\Gamma$ factors through at least one of these strict resolutions.
\end{proof}

\begin{definition}
We call the collection of strict MR resolutions constructed in Proposition \ref{prop101} the \textnormal{strict Makanin-Razborov diagram} of the restricted $\Gamma$-limit group $L$.
\end{definition}

We now introduce some notation:\\
Let $\Gamma=\langle a_1,\ldots,a_k\rangle$ be a hyperbolic group. Let $L$ be a restricted $\Gamma$-limit group and let
$$\Res(L):\ L=L_0\xrightarrow{\eta_0} L_1\xrightarrow{\eta_1}L_2\rightarrow \cdots\xrightarrow{\eta_{l-1}}L_l=\pi_1(\D)$$
be a strict MR resolution of $L$, where $\pi_1(\D)$ is a Dunwoody decomposition of $L_l$, such that every vertex group is isomorphic to a subgroup of $\Gamma$. Each restricted $\Gamma$-limit group $L_i$ that appears along the resolution $\Res(L)$, admits an associated restricted Dunwoody decomposition $\U^{(i)}$ with underlying graph $U^{(i)}$. \\
Denote by $R_1^{0},\ldots,R_{k_0}^{0}$ the vertex groups of $\U^{(0)}$. Let $Q_1^0,\ldots,Q_{s(0)}^0$ be the QH vertex groups in the various virtually abelian JSJ decompositions $\A_j^0$ of the vertex groups $R_j^0$. If $Q_t^0$ is contained in $R_j^0$, let $\B^0_{t}$ be the graph of groups obtained from $\A_j^0$ by collapsing all edges which are not connected to the QH vertex group $Q_t^0$. In particular all edge groups in $\B^0_t$ are virtually cyclic.
We say that the \textit{circumference} of the QH subgroup $Q^0_t$, denoted by $\CF(Q_t^0)$, is the subgroup of $R_j^0$ generated by $Q_t^0$ and the Bass-Serre generators in $\pi_1(B_t^0)$ (i.e. the generators of the fundamental group of the underlying graph).\\
We set $\h_0$ to be the empty subgraph of groups of $\U^{(0)}$ and for all $i\in\{1,\ldots,l\}$ we define the following (possibly trivial, possibly disconnected) subgraphs $\h_i\subset \U^{(i)}$ inductively as follows:\\ 
For all $i\in\{1,\ldots,l\}$ denote by $A_1^{i-1},\ldots A_{k(i-1)}^{i-1}$ the vertex groups of $\U^{(i-1)}$ which are either infinite virtually abelian or finite and by $O_1^{i-1},\ldots,O_{u(i-1)}^{i-1}$ the vertex groups of $\U^{(i-1)}$ which are finite-by-(closed) orbifold groups. We call such vertex groups \textit{isolated virtually abelian} (respectively \textit{isolated QH}) vertex groups.\\
For $i=1$:\\
For every QH vertex group $Q_t^0$ defined above, let $\mathbb{M}_t$ be the  decomposition along finite subgroups of $\eta_0(CF(Q_t^0))$ inherited from the restricted Dunwoody decomposition $\U^{(1)}$ of $L_1$. Then there exist vertex groups of $\mathbb{M}_t$, denoted by $H_{t_1}^0,\ldots,H_{t_b}^0$, which do not contain any conjugate of the image $\eta_0(P)$ of a peripheral subgroup $P$ of $Q_t^0$ (recall that peripheral subgroups correspond to boundary components of the underlying orbifold), while all other vertex groups, say $V_1,\ldots,V_s$, of $\mathbb{M}_t$ do contain a conjugate of the image of a peripheral subgroup of $Q_t^0$. Moreover the image of every peripheral subgroup of $Q_t^0$ can be conjugated into one of the vertex groups $V_1,\ldots,V_s$. To simplify notation we assume that there exists only one vertex group in $\mathbb{M}_t$, say $H_{t}^0$, which does not contain any conjugate of the image $\eta_0(P)$ of a peripheral subgroup $P$ of $Q_t^0$.\\
Let now $\D^{(1)}$ be a Dunwoody decomposition of $L_1$ relative 
$$\{\eta_{0}(A_1^0),\ldots,\eta_{0}(A_{k(0)}^0),\eta_{0}(O_1^0),\ldots,\eta_{0}(O_{u(0)}^0),H_1^0,\ldots H_{s(0)}^0\}$$
and denote the subgraph of $\D^{(1)}$ consisting of the collection of vertices with vertex groups which contain (at least) one of the groups
$$\eta_{0}(A_1^0),\ldots,\eta_{0}(A_{k(0)}^0),\eta_{0}(O_1^0),\ldots,\eta_{0}(O_{u(0)}^0),H_1^0,\ldots H_{s(0)}^0,$$
together with the edges connecting two such vertices, by $\h_1$. Denote the other vertex groups of $\D^{(1)}$ by $R_1^{1},\ldots,R_{k_1}^{1}$.\\
For $i\in\{2,\ldots,l\}$:\\
Let $Q_1^{i-1},\ldots,Q_{s(i-1)}^{i-1}$ be the QH vertex groups in the various virtually abelian JSJ decompositions $\A_j^{i-1}$ of the vertex groups $R_j^{i-1}$. If $Q_t^{i-1}$ is contained in $R_j^{i-1}$, let $\B^{i-1}_{t}$ be the graph of groups obtained from $\A_j^{i-1}$ by collapsing all edges which are not connected to the QH vertex group $Q_t^{i-1}$. In particular all edge groups in $\B^{i-1}_t$ are virtually cyclic.
We say that the \textit{circumference} of the QH subgroup $Q^{i-1}_t$, denoted by $CF(Q_t^{i-1})$, is the subgroup of $R_j^{i-1}$ generated by $Q_t^{i-1}$ and the Bass-Serre generators in $\pi_1(B_t^{i-1})$ (i.e. the generators of the fundamental group of the underlying graph).\\
For every QH vertex group $Q_t^{i-1}$ defined above, let $\mathbb{M}_t$ be the decomposition along finite subgroups of $\eta_{i-1}(CF(Q_t^{i-1}))$ inherited from the restricted Dunwoody decomposition $\U^{(i)}$ of $L_i$. Then there exist vertex groups in $\mathbb{M}_t$, denoted by $H_{t_1}^{i-1},\ldots,H_{t_b}^{i-1}$, which do not contain any conjugate of the image $\eta_{i-1}(P)$ of a peripheral subgroup $P$ of $Q_t^{i-1}$, while all other vertex groups, say $V_1,\ldots,V_s$, of $\mathbb{M}_t$ do contain a conjugate of the image of a peripheral subgroup of $Q_t^{i-1}$. Moreover the image of every peripheral subgroup of $Q_t^{i-1}$ can be conjugated into one of the vertex groups $V_1,\ldots,V_s$. Again to simplify notation we assume that there exists only one vertex group in $\mathbb{M}_t$, say $H_{t}^{i-1}$, which does not contain any conjugate of the image $\eta_{i-1}(P)$ of a peripheral subgroup $P$ of $Q_t^{i-1}$.\\
Denote the vertex groups of $\h_{i-1}$ by $K_1^{i-1},\ldots,K_{q(i-1)}^{i-1}$ and let now $\D^{(i)}$ be a Dunwoody decomposition of $L_i$ relative 
\begin{align*}
 \{&\eta_{i-1}(A_1^{i-1}),\ldots,\eta_{i-1}(A_{k(i-1)}^{i-1}),
 \eta_{i-1}(O_1^{i-1}),\ldots,\eta_{i-1}(O_{u(i-1)}^{i-1}),\\
& H_1^0,\ldots H_{s(i-1)}^0,
 \eta_{i-1}(K_1^{i-1}),\ldots,\eta_{i-1}(K_{q(i-1)}^{i-1})\}
\end{align*}
and denote the subgraph of $\D^{(i)}$ consisting of the collection of vertices with vertex groups which contain (at least) one of the groups
\begin{align*}
 &\eta_{i-1}(A_1^{i-1}),\ldots,\eta_{i-1}(A_{k(i-1)}^{i-1}),
 \eta_{i-1}(O_1^{i-1}),\ldots,\eta_{i-1}(O_{u(i-1)}^{i-1}),\\
& H_1^0,\ldots H_{s(i-1)}^0,
 \eta_{i-1}(K_1^{i-1}),\ldots,\eta_{i-1}(K_{q(i-1)}^{i-1}),
\end{align*}
together with edges connecting two such vertices, by $\h_i$. Denote the other vertex groups of $\D^{(i)}$ by $R_1^{i},\ldots,R_{k_i}^{i}$.\\

So let us briefly recall the setting we are in: $\Gamma=\langle a_1,\ldots,a_k\rangle$ is as always a hyperbolic group. Let $L$ be a restricted $\Gamma$-limit group and let
$$\Res(L):\ L=L_0\xrightarrow{\eta_0} L_1\xrightarrow{\eta_1}L_2\rightarrow \cdots\xrightarrow{\eta_{l-1}}L_l=\pi_1(\D)$$
be a strict MR resolution of $L$, where $\pi_1(\D)$ is a Dunwoody decomposition of $L_l$, such that every vertex group is isomorphic to a subgroup of $\Gamma$.\\
With each restricted $\Gamma$-limit group $L_i$ that appears along the resolution $\Res(L)$, we have constructed an associated (restricted) decomposition $\D^{(i)}$ with finite edge groups, such that every vertex group $V$ of $\D^{(i)}$ either belongs to $\{R_1^{i},\ldots,R_{k_i}^{i}\}$ or to $\h_i$ and these two sets are disjoint as (not necessarily connected)  subgraphs of $\D^{(i)}$. Moreover $L_i=\pi_1(\D^{(i)})$ and the coefficient group $\Gamma$ is contained in one of the vertex groups $R_j^{i}$. With each factor $R_j^{i}$ there is an associated restricted virtually abelian JSJ decomposition $\A_j^i$. 
While reading the following definition we encourage the reader to look at Figure \ref{wellstructured}.

\begin{definition}\label{def well-structured} 
Let $\Res(L)$ be a strict MR resolution of a $\Gamma$-limit group $L$ as above. We call $\Res(L)$ a \textnormal{well-structured resolution} if the following conditions hold for all $i$:
\begin{enumerate}[(1)]
\item $\D^{(i+1)}$ is a refinement of $\D^{(i)}$. In particular the fundamental group $\pi_1(D^{(i)})$ of the underlying graph of $\D^{(i)}$ gets mapped injectively into $\pi_1(D^{(i+1)})$ and the respective edge groups of $\D^{(i)}$ and their images in $\D^{(i+1)}$ are isomorphic.
\item The subgraph $\h_i$ of $\D^{(i)}$ gets mapped isomorphically to a subgraph of $\h_{i+1}$ in $\D^{(i+1)}$.
\item If the factor $R_j^i$ is neither a finite-by-(closed) orbifold group nor a (infinite or finite) virtually abelian group, then $\eta_i(R_j^i)$ has a Dunwoody decomposition $\B$, such that $\B$ is a connected subgraph of groups of $\D^{(i+1)}$, i.e. the vertex and edge groups coincide. This means the vertex groups of $\B$ are isomorphic to (possibly) some of the $R_j^{(i+1)}$ and some of the vertex groups of $\h_{i+1}$. 
\item Let $j\neq j'$. Then $R_j^i$ and $R_{j'}^i$ get mapped onto fundamental groups of distinct (and disjoint) subgraphs of groups of $\D^{(i+1)}$ and these fundamental groups have finite intersection with the image of the vertex groups of $\h_i$ under $\eta_i$. Moreover these subgraphs do not contain any edge of $ED^{(i)}$, where we view $ED^{(i)}$ as a subset of $ED ^{(i+1)}$ (since $\D^{(i+1)}$ is a refinement of $\D^{(i)}$).
\item Since by assumption $\Res(L)$ is a strict MR resolution, $\eta_i$ maps every rigid vertex group and every edge group in $\A_j^i$ monomorphically into $L_{i+1}$ for all $j$. Moreover the image of a QH vertex group under $\eta_i$ is non-virtually abelian. 
\item Let $A_1^i,\ldots,A_{q(i)}^i$ be the (infinite or finite) isolated virtually abelian vertex groups among the $R_j^i$'s in the Dunwoody decomposition of $L_i$. Then $C_u^i:=\eta_i(A_u^i)$ is an (infinite or finite) virtually cyclic vertex group of $\h_{i+1}$ and $C_1^i,\ldots,C_{q(i)}$ have pairwise finite intersection with each other and with the image of every vertex group of $\h_i$.
\item Let $S^i_1,\ldots, S^i_{d(i)}$ be the isolated (i.e. finite-by-closed) orbifold groups among the $R_j^i$'s. Then for all $k\in\{1,\ldots,d(i)\}$, 
$K_k^i:=\eta_i(S_k^i)$ has a decomposition as a subgraph of groups of $\h_{i+1}$ and 
$K_1^i,\ldots,K_{d(i)}^i$ have pairwise finite intersection with each other, with $C_1^i,\ldots,C_{q(i)}^i$ and with the image of every vertex group of $\h_i$.
\item Let $Q_1^i,\ldots,Q_{s(i)}^i$ be the QH vertex groups in the various abelian JSJ-decompositions $\A_j^i$ of the vertex groups $R_j^i$. 
For every QH vertex group $Q_t^i$ let $\mathbb{M}_t$ be the decomposition along finite subgroups of $\eta_i(CF(Q_t^i))$ inherited from $\D^{(i+1)}$ such that several vertex groups of $\mathbb{M}_t$, denoted by $H_{t_1}^i,\ldots,H_{t_b}^i$, do not contain any conjugate of the image $\eta_i(P)$ of a peripheral subgroup $P$ of $Q_t^i$, while all other vertex groups, say $V_1,\ldots,V_{m(t)}$, of $\mathbb{M}_t$ do contain a conjugate of the image of a peripheral subgroup of $Q_t^i$. Moreover the image of every peripheral subgroup of $Q_t^i$ can be conjugated into one of the vertex groups $V_1,\ldots,V_{m(t)}$. Again to simplify notation we assume that there exists only one vertex group in $\mathbb{M}_t$, say $H_{t}^{i}$, which does not contain any conjugate of the image $\eta_{i}(P)$ of a peripheral subgroup $P$ of $Q_t^{i}$.\\
Then
\begin{enumerate}[(a)]
\item $H_1^i,\ldots,H_{s(i)}^i$ are vertex groups of $\h_{i+1}$.
\item $H_1^i,\ldots,H_{s(i)}^i$ have pairwise finite intersection with each other, with $K_1^i,\ldots,K_{d(i)}^i$, $C_1^i,\ldots,C_{q(i)}^i$ and with the image of every vertex group of $\h_i$.
\item $V_1,\ldots, V_{m(t)}\leq \langle R_1^{i+1},\ldots,R_{q(i+1)}^{i+1},\pi_1(D_{i+1})\rangle.$
\end{enumerate} 
\item Denote by $U_1^i,\ldots, U_{f(i)}^i$ the images of vertex groups of $\h_i$ under $\eta_i$. Then a group $X$ is a vertex group of $\h_{i+1}$ if and only if 
$$X\in\{H_1^i,\ldots,H_{s(i)}^i,K_1^i,\ldots,K_{d(i)}^i, C_1^i,\ldots,C_{q(i)}^i, U^i_1,\ldots, U^i_{f(i)}\}.$$
\end{enumerate}
\end{definition}

\begin{figure}[htbp]
\centering
\includegraphics[scale=1.1]{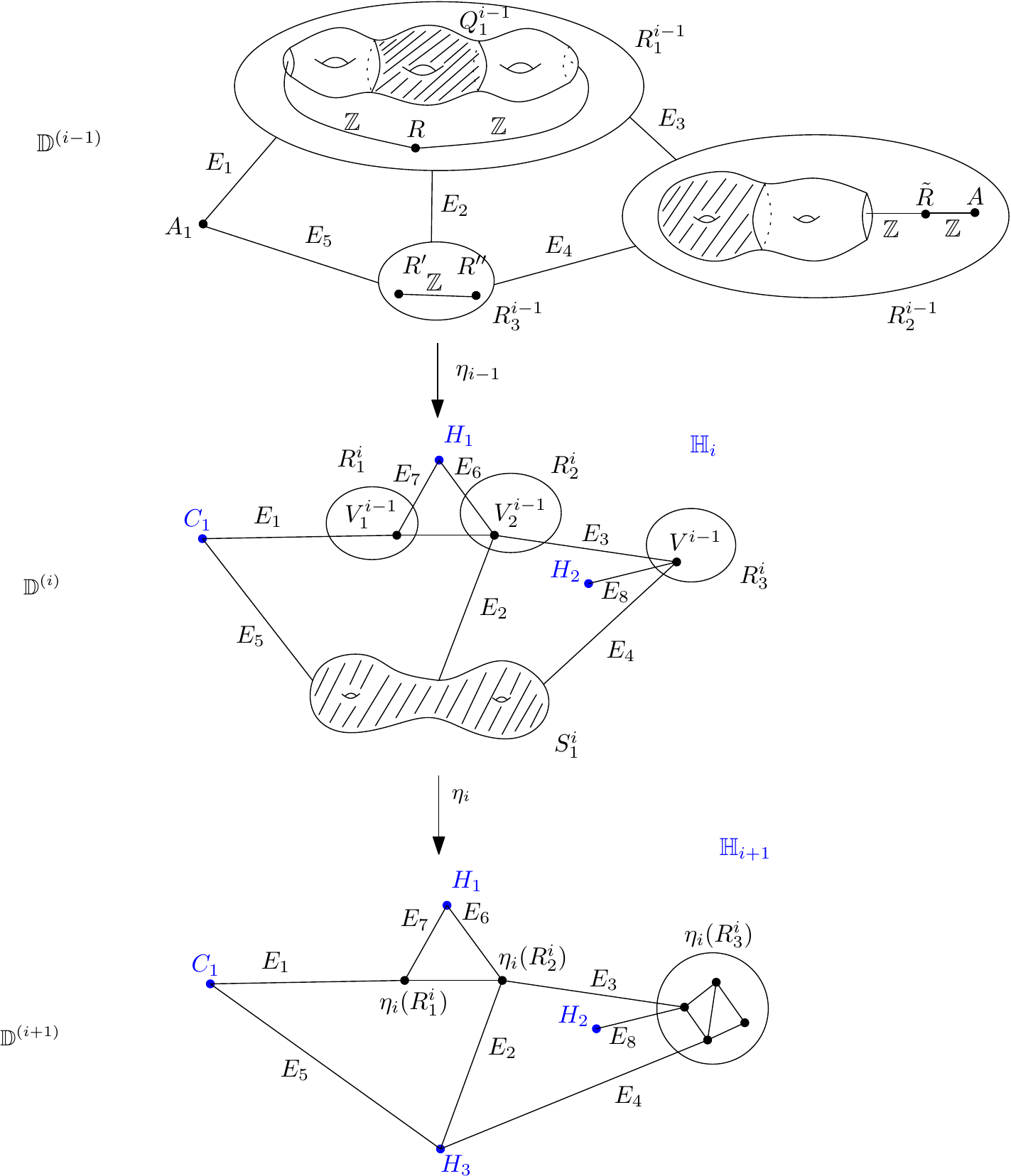}
\caption{A well-structured resolution of a $\Gamma$-limit group}
\label{wellstructured}
\end{figure}

\begin{bem}
It follows from \cite{rewe} that every resolution of a $\Gamma$-limit group which is constructed from a descending chain of maximal shortening quotients is a well-structured resolution. In fact the well-structured resolution is defined in order to capture the properties of a descending chain of shortening quotients.
It follows that every resolution in the strict MR diagram constructed in Proposition \ref{prop101} is well-structured.
\end{bem}

Later we will need an additional property to hold for our well-structured resolution.

\begin{definition}
We call a well-structured MR resolution
$$\Res(L):\ L=L_0\xrightarrow{\eta_0} L_1\xrightarrow{\eta_1}L_2\rightarrow \cdots\xrightarrow{\eta_{l-1}}L_l=\pi_1(\D)$$
of a restricted $\Gamma$-limit group $L$ \textnormal{good} if there exists a homomorphism $\vp: L\to\Gamma$ that factors through $\Res(L)$, such that the centralizer in $\Gamma$ of the image of every edge group of a Dunwoody decomposition $\D$ of $L_l$ contains a quasi-convex non-abelian free group.
\end{definition}
 
\subsection{$\Gamma$-limit tower}
Recall that a subgroup of a group $G$ is called a QH subgroup if it is conjugate to a vertex group $A_v$ of some splitting $\A$ of $G$ such that the following hold:
\begin{enumerate}[(a)]
\item $A_v$ is finite-by-orbifold, i.e. there exists an orbifold $\mathcal{O}$, some finite group $E$ and a short exact sequence 
$$1\to E\to A_v\xrightarrow{\pi} \pi_1(\mO)\to 1.$$
\item For any edge $e\in EA$, such that $\alpha(e)=v$ there exists a peripheral subgroup $O_e$ of $\pi_1(\mO)$ such that $\alpha_e(A_e)$ is in $A_v$ conjugate to a finite index subgroup of $\pi^{-1}(O_e)$.
\end{enumerate}

\begin{definition}
Let $G$ be a group and $H$ be a subgroup of $G$.
\begin{enumerate}[(a)]
\item $G$ has the structure of a \textnormal{virtually abelian flat over $H$}, if $G$ is the amalgamated product $H\ast_{A'}A$, where $A$ is virtually abelian and $A'$ is either an infinite maximal virtually abelian subgroup of $H$ or finite, and if $A'$ is infinite then there exists a retraction $\pi:G\to H$.
\item $G$ has the structure of an \textnormal{orbifold flat over $H$} if $G$ admits a decomposition as a graph of groups $\A$ such that
\begin{itemize}
\item There exists a vertex $v\in VA$ with a QH subgroup $Q=A_v$ as vertex group, which fits into a short exact sequence 
$$1\to E\to Q\xrightarrow{\pi} \pi_1(\mO)\to 1.$$
\item Every edge $e\in EA$ is of the form $e=(v,u)$ for some $u\in V\setminus \{v\}$ and for every edge $e$  with $\alpha(e)=v$, either there exists a peripheral subgroup $O_e$ of $\pi_1(\mathcal{O})$ such that $\alpha_e(A_e)$ is in $A_v$ conjugate to a finite index subgroup of $\pi^{-1}(O_e)$, or $A_e$ is finite. Moreover every boundary component of $\mO$ corresponds to an edge in such a way.
\item $H$ admits a decomposition as a graph of groups along finite subgroups, such that the vertex groups of this decomposition are precisely the vertex groups of the vertices in $VA\setminus\{v\}$.
\end{itemize}
\end{enumerate}
\end{definition}


\begin{definition}
A group $G$ has the structure of a \textnormal{tower} of height $n\in\N$ over a subgroup $H$ if there exists a sequence
$$G=G^n\geq G^{n-1}\geq\ldots\geq G^0=H,$$
such that for each $0\leq i\leq n-1$, one of the following holds:
\begin{itemize}
\item $G^{i+1}$ is an HNN-extension of $G^i$ along some finite group.
\item $G^{i+1}$ is the amalgamated product along a finite group of $G^i$ with a group $K$ which is  isomorphic to a subgroup of $H$.
\item $G^{i+1}$ has the structure of an orbifold flat over $G^i$ and $H$ is contained in one of the vertex groups of the corresponding graph of groups decomposition of $G^i$.
\item $G^{i+1}$ has the structure of a virtually abelian flat $G^{i}\ast_{A_1}A$ over $G^i$ and if $A_1$ is infinite, it cannot be conjugated into a virtually abelian group which corresponds to a virtually abelian flat which is lower in the tower hierarchy.
\end{itemize} 
\end{definition}

\begin{bem}
\begin{enumerate}[(1)]
\item The case that $G^{i+1}$ is a product along a finite group with either a virtually abelian group or with an extension of a closed orbifold group is covered by the virtually abelian/orbifold flat cases.
\item Our definition of a tower is solely designed to describe the completion of a well-structured resolution defined later. We chose this name because in the case of limit groups over free and torsion free hyperbolic groups, the completion of a resolution is always a tower in the classical sense. We refer the reader to \cite{sklinos} for a detailed discussion on towers over free groups.
\item In the torsion free case there always exists a retraction from $G^{i+1}$ onto $G^i$, since for the orbifold flat there exists one by definition and in the other cases there always exists a canonical one. 
In the presence of torsion we do not claim that such a retraction always exist.
\end{enumerate}
\end{bem}

\begin{definition}
Suppose $G$ has the structure of a tower $G=G^n\geq G^{n-1}\geq\ldots\geq G^0=\Gamma$ over $\Gamma$. Then we say that $G$ has the structure of a \textnormal{$\Gamma$-limit tower} if $\Gamma$ is a hyperbolic group and $G_i$ is a $\Gamma$-limit group for all $i\in\{1,\ldots,n\}$.
\end{definition}

\section{Completions of Resolutions}\label{7}
The aim of this section is to show, that every restricted $\Gamma$-limit group $L$ can be embedded into a larger group, called its completion, which is a $\Gamma$-limit tower. The original paper of Sela (\cite{sela}) contains the construction of the completion over free groups but is lacking a proof of the embeddability. In \cite{sklinos} a slightly different construction together with a complete proof of the embeddability appears, but it only covers the case of  minimal-rank well-structured resolutions over free groups (this means that all limit groups which appear along the resolution are one-ended relative to the coefficient free group).\\
We follow the approach of the authors of \cite{sklinos} but have to generalize their ideas in several ways in order to deal with general resolutions over hyperbolic groups and in particular with the various problems caused by the presence of torsion.\\

So let us briefly recall the setting we are in: $\Gamma=\langle a_1,\ldots,a_k\rangle$ is as always a hyperbolic group. Let $L$ be a restricted $\Gamma$-limit group and let
$$\Res(L):\ L=L_0\xrightarrow{\eta_0} L_1\xrightarrow{\eta_1}L_2\rightarrow \cdots\xrightarrow{\eta_{l-1}}L_l=\pi_1(\D)$$
be a well-structured MR resolution of $L$, where $\pi_1(\D)$ is a Dunwoody decomposition of $L_l$, such that every vertex group is isomorphic to a subgroup of $\Gamma$.\\
With each restricted $\Gamma$-limit group $L_i$ that appears along the resolution $\Res(L)$, we have constructed (see Definition \ref{def well-structured}) an associated (restricted) decomposition $\D^{(i)}$ along finite edge groups, such that every vertex group $V$ of $\D^{(i)}$ either belongs to $\{R_1^{i},\ldots,R_{k_i}^{i}\}$ or to $\h_i$ and these two sets are disjoint as (not necessarily connected)  subgraphs of $\D^{(i)}$. Moreover $L_i=\pi_1(\D^{(i)})$ and the coefficient group $\Gamma$ is contained in one of the vertex groups $R_j^{i}$. With each factor $R_j^{i}$ there is an associated restricted (possibly trivial) virtually abelian JSJ decomposition $\A_j^i$.

\subsection{The construction}
We begin with a simple observation how the virtually abelian JSJ decomposition of a one-ended $\Gamma$-limit group can be modified, such that the construction of its completion will be simplified later.\\
We first recall some definitions. Let $\A$ be the virtually abelian JSJ decomposition of a one-ended $\Gamma$-limit group $L$. Let $A_v$ be a virtually abelian vertex group in $\A$ and let $\Delta$ be the set of homomorphisms $\eta: A_v^+\to \Z$ such that $\eta(\alpha_e(A^+_e))=0$ for all $e\in EA$ with $\alpha(e)=v$. We then define the \textit{peripheral subgroup of $A^+_v$} as $$P(A^+_v)=\{g\in A^+_v\ |\ \eta(g)=0 \text{ for all } \eta\in\Delta\}.$$
We call every subgroup $U\leq A_v$ with $U^+=P(A^+_v)$ a \textit{peripheral subgroup of $A_v$}. Note that the peripheral subgroup of $A^+_v$ is uniquely defined, while the peripheral subgroup of $A_v$ is not. For every edge $e\in EA$ with $\alpha(e)=v$ we denote a peripheral subgroup of $A_v$ containing $A_e$ by $P(A_v,A_e)$.\\
Recall that the envelope $E(R)$ of a rigid vertex group $R$ is the fundamental group of a graph of groups $\B$ whose underlying graph is a star with one vertex $v$ with vertex group $R$ connected to all other vertices via an edge of the form $(v,u)$ for some $u\in VB\setminus \{v\}$. Edges in $\B$ are in 1-to-1 correspondence to edges adjacent to $R$ in $\A$ and the edge groups are isomorphic. Let now $e=(v,u)$ be an edge in $\B$ and $e'=(v',u')$ the corresponding edge in $\A$ adjacent to $R$ (where $R$ is the rigid vertex group of $v'$). If $A_{u'}$ is rigid or a QH vertex group, then $B_u$ is the maximal virtually abelian subgroup of $A_{u'}$ containing $A_{e'}$. If $A_{u'}$ is a virtually abelian vertex group, then $B_u$ is the peripheral subgroup of $A_{u'}$ containing $A_{e'}$.\\
Note that in a strict resolution the envelopes of rigid vertex groups and the peripheral subgroups of virtually abelian vertex groups get mapped injectively by the canonical quotient maps.

\begin{lemma}\label{changes}
Let $L$ be a one-ended $\Gamma$-limit group and $\A$ be a virtually abelian JSJ decomposition of $L$. Then we can enlarge every rigid vertex group $R$ to (a subgroup of) its envelope, such that moreover all adjacent edge groups, which connect $R$ to another rigid vertex are maximal finite-by-abelian in it. In addition all edge groups connecting a rigid vertex to a virtually abelian vertex are maximal virtually abelian in the rigid vertex group.
\end{lemma}

\begin{proof}
First note that it is shown in \cite{rewe} that every edge group adjacent to a rigid vertex group is contained in a maximal virtually abelian subgroup and this subgroup is elliptic in $\A$, i.e. it is contained in at least one of the adjacent vertex groups. Let $e$ be an edge with edge group $E$ adjacent to a rigid vertex group $R$ and a virtually abelian vertex group $A$. We then replace $R$ by $R\ast_{E}P(A,E)$ and $E$ by $P(A,E)$, i.e. the subgraph of groups $R\ast_EA$ is replaced by $(R\ast_EP(A,E))\ast_{P(A,E)}A$. Since $R\ast_EP(A,E)\leq E(R)$, the new vertex group is still rigid. We perform these modifications for all edges connecting a rigid and a virtually abelian vertex in parallel and denote the resulting graph of groups again by $\A$.\\
We will now explain how we can change the graph of groups $\A$ such that every edge group connecting two rigid vertices is maximal finite-by-abelian in both adjacent vertex groups.
Assume that $R_1\ast_Z R_2$ is such an edge connecting two rigid vertex groups and suppose that $Z$ is contained in a maximal finite-by-abelian subgroup $C$ which is contained in $R_2$ (the HNN case is similar). Note that 
$$\langle R_1,C\rangle=R_1\ast_ZC\subset E(R_1).$$
We change the splitting $R_1\ast_Z R_2$ to $R_1\ast_ZC\ast_CR_2$ and collapse the edge with edge group $Z$, i.e. we get the splitting $(R_1\ast_ZC)\ast_CR_2$. Note that $(R_1\ast_ZC)$ is still rigid and $C$ is maximal finite-by-abelian in both adjacent vertex groups. After repeating this step for all edges adjacent to two rigid vertex groups we get a graph of groups which satisfies the claim of the Lemma.
\end{proof}

We will use these modifications either explicitly or implicitly in the following constructions whenever a JSJ decomposition of a one-ended $\Gamma$-limit group appears.

\begin{satz}\label{completion}
Let $\Gamma$ be a hyperbolic group and let $\Res(L)$ be a well-structured resolution of a restricted $\Gamma$-limit group $L$. Then $L$ embeds into a (canonical) $\Gamma$-limit tower. We call this tower the \textnormal{completion of $L$} and denote it by $\Comp(L)$.
\end{satz} 

\begin{proof}
We construct the desired $\Gamma$-limit tower iteratively from bottom to top. Assume that the well-structured resolution of $L$ is given by the following decreasing sequence of restricted $\Gamma$-limit groups:
$$\Res(L):\ L=L_0\xrightarrow{\eta_0} L_1\xrightarrow{\eta_1}L_2\rightarrow \cdots\xrightarrow{\eta_{l-1}}L_l=\pi_1(\D^{(l)}),$$
where $\eta_i: L_i\to L_{i+1}$ is the canonical quotient map for all $i\in\{0,\ldots,l-1\}$.\\
For all $i\in\{1,\ldots,l\}$ we construct a group $\Comp(L_i)$ which we call the completion of $L_i$, together with an embedding $\iota_i:L_i\to \Comp(L_i)$ and a decomposition $\B_i$ of $\Comp(L_i)$, such that every vertex group of $\h_i$ is a vertex group in $\B_i$. In fact the image of $\pi_1(\U_i)$ for every connected subgraph of groups $\U_i$ of $\h_i$ under $\iota_i$ admits a decomposition as a subgraph of groups of $\B_i$ which is isomorphic to $\U_i$. In particular $\iota_i|_{\pi_1(\U_i)}$ is an isomorphism.\\
We set $\Comp(L_l)=L_l$ to be the completion of $L_l$ and $\B_l:=\D^{(l)}$. Clearly $\iota_l:=\id_{L_l}$ embeds $L_l$ into its completion and $\h_l$ is a subgraph of $\B_l$.  
Since $L$ is a restricted $\Gamma$-limit group (and so is $L_l$), $\Gamma$ is contained in a vertex group of $\D^{(l)}$. But since every vertex group of $\D^{(l)}$ is isomorphic to a subgroup of $\Gamma$ and one-ended hyperbolic groups are cohopfian, this implies that $\Gamma$ is the whole vertex group (not only a subgroup). Hence $\Gamma$ is a vertex group of $\D^{(l)}$ and therefore it follows immediately that $\Comp(L_l)$ is a $\Gamma$-limit tower. We say that $\B_l=\D^{(l)}$ is the \textit{completed decomposition} of $\Comp(L_l)$.\\

We continue the construction of the completion inductively. Let $i\in \{0,\ldots,l-1\}$ and assume we have already constructed the completed $\Gamma$-limit group $\Comp(L_{i+1})$ on the $(i+1)$-th level together with its completed decomposition $\B_{i+1}$ and an embedding $\iota_{i+1}: L_{i+1}\to\Comp(L_{i+1})$. We now collapse a minimal number of edges in $\B_{i+1}$ necessary to guarantee that all images under $\iota_{i+1}\circ\eta_i$ of non-QH vertex groups in all the virtually abelian JSJ decompositions $A_j^i$ are elliptic in this new graph of groups. By abuse of notation we still call this new graph of groups $\B_{i+1}$.\\
We will construct (in several steps) a new graph of groups $\B_{i}$ and call its fundamental group the completion of $L_i$. Throughout the construction we will heavily use the properties of a well-structured resolution (either explicitly or implicitly).\\
Let $\A$ be a graph of groups which we get by plugging in all the virtually abelian JSJ decompositions for the one-ended vertex groups in $\D^{(i)}$. Note that it follows from Lemma \ref{changes} that after performing finitely many edge slides, we can assume that all edges adjacent to non-isolated virtually abelian vertices have infinite edge groups.\\
If $\A$ does not contain a rigid vertex group, then every vertex group is either a finite-by-closed orbifold, a virtually abelian group or isomorphic to a subgroup of $\Gamma$. Let $V$ be such a vertex group. Then $\iota_{i+1}\circ\eta_{i}(V)$ has a decomposition  as a connected subgraph of groups, say $\mathbb{V}$, of $\B_{i+1}$ (since $\mathbb{V}$ is contained in $\h_{i+1}$) and we replace $\mathbb{V}$ in $\B_{i+1}$ by a vertex with vertex group $V$. We moreover adjust the boundary monomorphisms of adjacent edges with finite edge groups in the obvious way. This is well-defined since $\Res(L)$ is a well-structured resolution. Repeating this for all vertex groups $V$ of $\A$, yields a graph of groups $\B_i$. We set $\Comp(L_i):=\pi_1(\B_i)$ to be the completion of $L_i$ and call $\B_i$ the completed decomposition of $\Comp(L_i)$. Clearly $\Comp(L_i)$ is a $\Gamma$-limit tower.\\
Hence from now on (by introducing redundant rigid vertices if necessary) we can assume that there exists at least one rigid vertex, say $v$, in $\A$.  Let $e_1,\ldots,e_m$ be an enumeration of the edges of $A$ in such a way that $\alpha(e_1)=v$ and for all $t\in\{0,\ldots,m\}$ there exists a graph of groups $\A_t$ with the following properties:
\begin{itemize}
\item $\A_t$ is a connected subgraph of groups of $\A$.
\item $\A_0$ consists of the single vertex $v$.
\item $\A_t\subset \A_{t+1}$ is a proper subgraph of groups and $A_{t+1}$ arises from $\A_t$ by adding the edge $e_{t+1}$ together with possibly one of the vertices $\alpha(e_{t+1})$ or $\omega(e_{t+1})$ (if both adjacent vertices don't already belong to $\A_t$).
\end{itemize}
We start the construction of the completion $\Comp(L_i)$ and its completed decomposition $\B_i$ with the graph of groups $\B_{i+1}$. 
Let $V$ be a finite-by-closed orbifold or a virtually abelian vertex group of $\A$ adjacent only to edges with finite edge group. Recall that we call such vertex groups isolated QH or virtually abelian vertex groups. Then $\iota_{i+1}\circ\eta_{i}(V)$ has a decomposition as a connected subgraph of groups, say $\mathbb{V}$, of $\B_{i+1}$ (since $\mathbb{V}$ is contained in $\h_{i+1}$) and we replace $\mathbb{V}$ in $\B_{i+1}$ by a vertex with vertex group $V$. We also adjust the boundary monomorphisms of edges adjacent to the new vertex with vertex group $V$ in the obvious way. Again this is well-defined since $\Res(L)$ is a well-structured resolution. Repeating this for all vertex groups $V$ of $\A$ of this type, yields a graph of groups $\B'_0$. We set $W_0:=\pi_1(\B'_0)$.\\
We then proceed inductively through the sequence of graphs of groups $\A_0\subset \ldots\subset A_m=\A$ in such a a way, that we construct a graph of groups $\B'_t$ which contains $\B'_0$ (which is constructed out of $\B_{i+1}$) as a subgraph, such that $\pi_1(\A_t)$ embeds into $W_t:=\pi_1(\B'_t)$ via a monomorphism $\nu_t$ for all $t\in\{0,\ldots,m\}$, where $\nu_{t}$ coincides on all rigid vertices of $\A_t$ with $\iota_{i+1}\circ\eta_i$ up to conjugation by an element which is either trivial or contained in $W_t\setminus \Comp(L_{i+1})$. Moreover $\nu_{t}|_{\pi_1(\A_{t-1})}=\nu_{t-1}$ and $W_t=\pi_1(\B'_t)$ is a $\Gamma$-limit tower.\\
For $t=0$ the graph of groups $\A_0$ consists by definition of a single vertex $v$ with rigid vertex group $A_v$. Since $\Res(L)$ is well-structured, $\iota_{i+1}\circ\eta_i$ is injective on $A_v$. Hence $\nu_0:=\iota_{i+1}\circ\eta_i|_{A_v}$ yields the desired embedding of $\pi_1(\A_0)$ into $W_0=\pi_1(\B'_0)$. (This holds since the image of a rigid vertex group under $\iota_{i+1}\circ\eta_i$ has at most finite intersection with the fundamental group of every connected subgraph of groups of $\h_{i+1}$ and all the subgraphs of groups $\mathbb{V}$ which have been replaced in the step before are contained in $\h_{i+1}$). Clearly $W_0$ is a $\Gamma$-limit tower, since $\Comp(L_{i+1})$ is one by induction hypothesis.\\
So lets assume we already have constructed a graph of groups $\B'_{t}$ out of $\B_{i+1}$ together with an embedding $$\nu_{t}:\pi_1(\A_t)\hookrightarrow W_t=\pi_1(\B'_t),$$
which coincides on rigid vertex groups with $\iota_{i+1}\circ\eta_i$ up to conjugation by an element which is either trivial or contained in $W_t\setminus\Comp(L_{i+1})$.
Assume without loss of generality that $\alpha(e_{t+1})\in A_t$. We distinguish several cases, depending on the edge $e_{t+1}$. Denote by $A_e$ the edge group of $e_{t+1}$ and by $\alpha: A_e\hookrightarrow A_{\alpha(e_{t+1})}$, $\omega: A_e\hookrightarrow A_{\omega(e_{t+1})}$ the canonical embeddings into the adjacent vertex groups. 
\begin{enumerate}[(1)]
\item Assume that $A_{e}$ is infinite.
\begin{enumerate}[(a)]
\item Assume that $e_{t+1}$ connects two rigid vertices in $\A$ and hence $A_e$ is finite-by-abelian (see the construction of the virtually abelian JSJ decomposition in \cite{rewe}).\\
Let $f_1,\ldots, f_m$ be the edges in the various JSJ decompositions $\A_j^i$ connecting two rigid vertices. Let $I_1\cup \ldots\cup I_k$ be a partition of $\{1,\ldots,m\}$ with the following property:\\
For any $l\in\{1,\ldots,k\}$ there exists a maximal finite-by-abelian subgroup $M_l^+$ of $\Comp(L_{i+1})$ such that for all $j\in I_l$ the image $\iota_{i+1}\circ\eta_{i}(A_{f_j})$ of the corresponding edge group can be conjugated into $M_l^+$. Moreover $I_l$ is maximal with this property for all $1\leq l\leq k$.\\
Let $l\in\{1,\ldots,k\}$ such that $\iota_{i+1}\circ\eta_{i}(A_{e})$ can be conjugated into $M_l^+$. If the group $M_l^+\oplus\Z^{|I_l|}$ has been amalgamated along $M_l^+$ to $\Comp(L_{i+1})\leq W_{h}$ for some $h<t$ earlier in the progress then we set $\B'_{t+1}:=\B'_t$ and $W_{t+1}:=W_t$. Otherwise we amalgamate the group $M_l^+\oplus\Z^{|I_l|}$ along $M_l^+$ to $\Comp(L_{i+1})\leq W_{t}$. (Note that $M_l^+$ is still contained in $W_t$ since $\Res(L)$ is well-structured, i.e. the maximal finite-by-abelian subgroups containing the images of edge groups connecting two rigid vertices have at most finite intersection with $\pi_1(\h_{i+1})$ and therefore where not replaced in the previous modifications). Hence $\B'_{t+1}$ is the graph of groups $\B'_t$ together with a new vertex with vertex group $M_l^+\oplus\Z^{|I_l|}$ and an edge connecting both with edge group $M_l^+$ (see Figure \ref{fig12}). Therefore $W_{t+1}$ has the structure of a virtually abelian (in fact even finite-by-abelian) flat over $W_t$ and hence is a $\Gamma$-limit tower. 
\begin{figure}[htbp]
\centering
\includegraphics{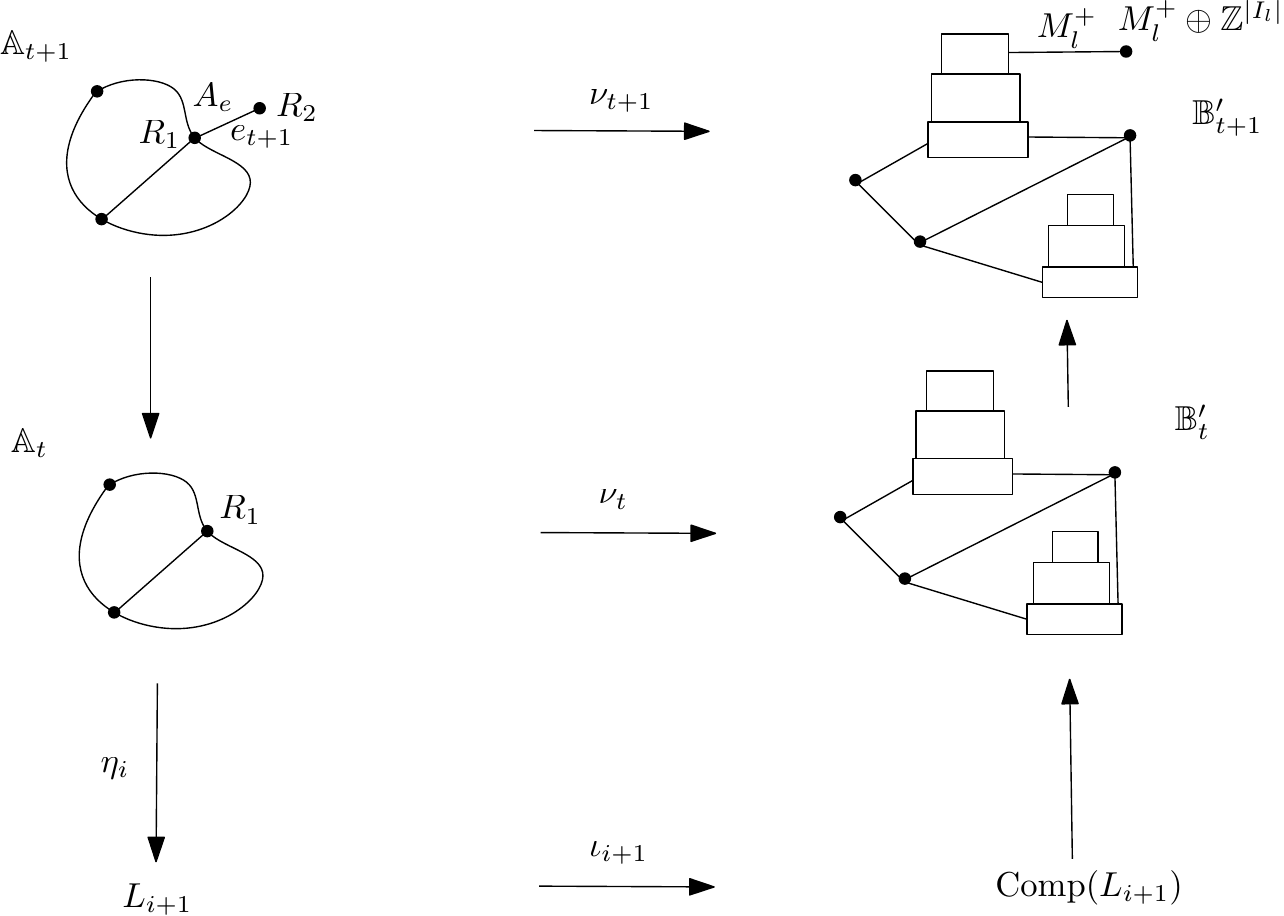}
\caption{Case (1)(a)}
\label{fig12}
\end{figure}
\item Assume that $e_{t+1}$ connects a rigid vertex $r$ with vertex group $R$ and a QH vertex $q$ with vertex group $Q$ in $\A$.\\
If $q=\omega(e_{t+1})$ and $q\notin A_t$, then we add a new vertex $v$ to $\B'_t$ with vertex group $Q$, together with an edge $f$ connecting $v$ to the vertex $w$ with vertex group $B_w$ in $\B'_t$ which contains the image $\iota_{i+1}\circ\eta_i\circ\alpha(A_e)$ of $\alpha(A_e)\leq R$ under $\iota_{i+1}\circ\eta_i$. Note that by the previous modifications applied to $\B_{i+1}$, this image is in fact elliptic in $\B_{i+1}\subset\B'_t$. The edge group of $f$ is $A_e$. The embeddings are the obvious ones, namely 
$$\omega: A_e\to Q\text{  and  } \iota_{i+1}\circ\eta_i\circ\alpha: A_e\to B_w.$$
If $q\in A_t$ then a vertex $v$ with vertex group $Q$ has already been added to $\B'_t$ earlier in our iterative process and we only add a new edge $f$ to $\B'_t$ connecting $v$ to the vertex $w$ with vertex group $B_w$ in $\B'_t$ which contains the image $\iota_{i+1}\circ\eta_i\circ\alpha(A_e)$. The edge group of $f$ is again $A_e$ and the embeddings are also the same as the ones in the previous case.\\
If $q=\alpha(e_{t+1})$ then again a vertex $v$ with vertex group $Q$ has already been added to $\B'_t$ earlier in our iterative process and we only add a new edge $f$ to $\B'_t$ connecting $v$ to the vertex $w$ with vertex group $B_w$ in $\B_t$ which contains the image $\iota_{i+1}\circ\eta_i\circ\omega(A_e)$. The edge group of $f$ is again $A_e$ and the embeddings are also the same as the ones in the previous case.\\

\item Assume that $e_{t+1}$ connects a virtually abelian vertex $a$ with vertex group $A'$ and a rigid vertex $r$ with vertex group $R$ in $\A$. Note that by Lemma \ref{changes} $A_{e}^+=P(A'^+)$, i.e. $A_e$ is a peripheral subgroup of $A'$. For simplicity we set $P(A'):=A_e$. This case is by far the most difficult one, mainly because the simple strategy of the torsion-free case can no longer be applied. Assume that $a=\omega(e_{t+1})\notin A_t$.\\
Let $G:=\pi_1(\A_t)$. Hence if we collapse in $\A_{t+1}$ all edges but $e_{t+1}$ we get the amalgamated free product $G\ast_{P(A')}A'$.
Here we view $P(A')$ as a subgroup of $G$ and $A'$ (to simplify notation), hence both embeddings into the vertex groups are the identity map.
Since $\eta_i:L_i\to L_{i+1}$ is strict, it is in particular injective on $P(A')$. We set $$M:=M_{W_t}(\iota_{i+1}\circ\eta_i(P(A')))$$ to be the maximal virtually abelian subgroup of $W_t$ which contains the (isomorphic) image of $P(A')$. Here we view $\Comp(L_{i+1})$ as a subgroup of $W_t$.\\
Since $W_t$ is a $\Gamma$-limit group there exists a stably injective sequence $\vp_n:W_t\to \Gamma$ such that 
$$W_t=W_t/_{\displaystyle \underrightarrow{\ker} \vp_n}.$$
By the construction of $W_t$ and since $\eta_i$ is a strict quotient map, there exists a sequence of modular automorphisms $\alpha_n\in\Mod(L_i)$, such that $$\mu_n:=(\vp_n\circ \iota_{i+1}\circ\eta_i\circ\alpha_n)|_{A'}$$ is stably injective. Let $$U:=M\ast_{P(A')}A'$$ be the amalgamated product given by the embeddings $$\alpha=\iota_{i+1}\circ\eta_i:P(A')\to M$$ and $$\omega=\id:P(A')\to A'.$$
We define a sequence of homomorphisms $$\Psi_n: W_t\ast_M U\to\Gamma$$ ($W_t\ast_M U$ with the obvious embeddings) via
\begin{itemize}
\item $\Psi_n|_{W_t}=\vp_n$,
\item $\Psi_n|_M=\vp_n$ and
\item $\Psi_n|_{A'}=\mu_n$.
\end{itemize}
We first check that this is a well-defined homomorphisms on $W_t\ast_MU$.
\begin{figure}[htbp]
\centering
\includegraphics{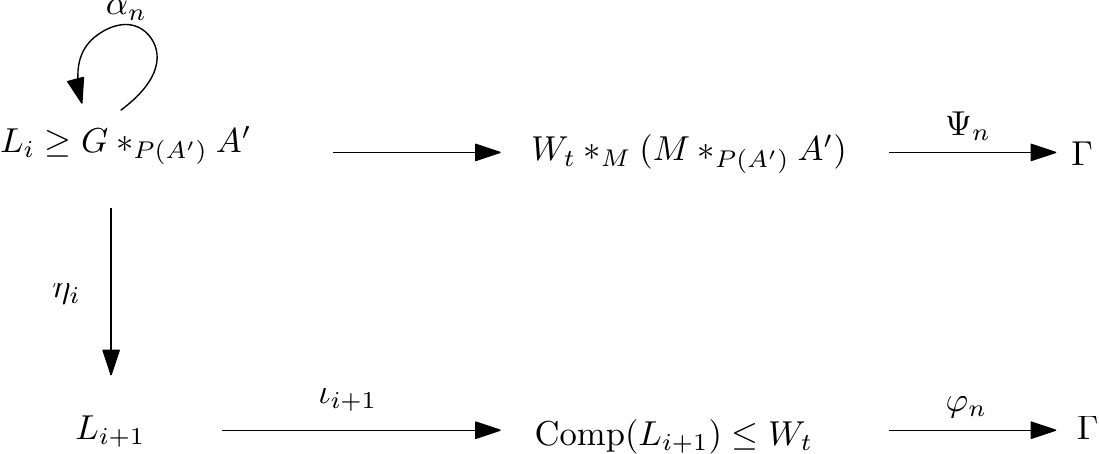}
\caption{Constructing the sequence ($\Psi_n$)}
\end{figure}
First note that since $\alpha_n\in \Mod(L_i)$, it fixes $P(A')$. Now by definition
\begin{align*}
\Psi_n(\omega(P(A')))&=\Psi_n(\underbrace{\id(P(A'))}_{\leq A'})\\
&=\mu_n(P(A'))\\
&=\vp_n\circ \iota_{i+1}\circ\eta_i\circ\alpha_n(P(A'))\\
&=\vp_n\circ \iota_{i+1}\circ\eta_i(P(A'))\\
&=\Psi_n(\underbrace{\iota_{i+1}\circ\eta_i(P(A'))}_{\leq M\subset W_t})\\
&=\Psi_n(\alpha(P(A')))
\end{align*}
and hence $\Psi_n$ is a well-defined homomorphism. Note however that (after passing to a subsequence) the sequence $(\Psi_n)\subset\Hom(W_t\ast_MU,\Gamma)$ is stable but not necessarily stably injective.\\
We set $$A:=U/_{\displaystyle \underrightarrow{\ker}\Psi_n}=(M\ast_{P(A')}A')/_{\displaystyle \underrightarrow{\ker}\Psi_n}.$$
Since $\Psi_n|_M=\vp_n$ is stably injective, we have that $M\leq A$. We now define $W_{t+1}:=W_{t}\ast_MA$. Since $\Res(L)$ is a well-structured resolution, $W_{t+1}$ admits a decomposition as a graph of groups $\B'_{t+1}$ which is the graph of groups $\B'_t$ together with a new edge with edge group $M$ connecting a new vertex $v$ with vertex group $A$ to the vertex in $\B'_t$ containing $M$.\\
\begin{figure}[htbp]
\centering
\includegraphics{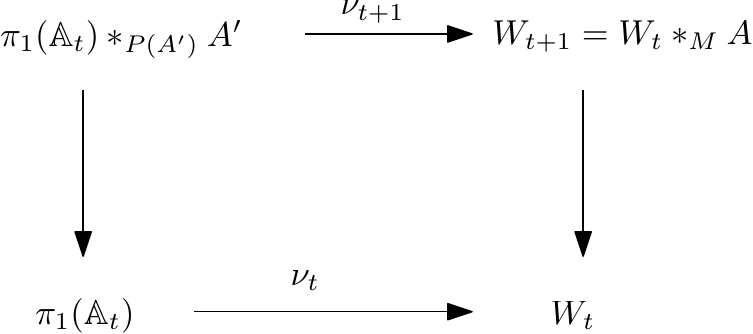}
\caption{Adding a virtually abelian flat to $W_t$}
\end{figure}
Clearly $W_{t+1}$ is a $\Gamma$-limit group.
It remains to show that $W_{t+1}=W_t\ast_MA$ is a virtually abelian flat over $W_t$. Hence we have to show that $A$ is in fact virtually abelian and moreover that there exists a retraction from $W_{t+1}$ to $W_t$.\\
We start by showing that $A$ is virtually abelian. Since $P(A')$ is infinite virtually abelian and $\Psi_n|_{P(A')}$ is by construction stably injective, it holds that for large $n$, $\Psi_n(P(A'))$ is an infinite subgroup of a maximal virtually cyclic subgroup $C_n$ of $\Gamma$. Hence $\Psi_n(A')$ and $\Psi_n(M)$ are also contained in $C_n$. This implies that $\Psi_n(M\ast_{P(A')}A')\leq C_n$.\\
Hence after passing to a subsequence of $(\Psi_n)$ we can assume that $C_n=C$ for all $n$ for some maximal virtually cyclic subgroup $C$ of $\Gamma$. Hence $$A=M\ast_{P(A')}A'/_{\displaystyle \underrightarrow{\ker} \Psi_n}$$
is a $C$-limit group and therefore virtually abelian (since otherwise there exists a commutator in $A$ which cannot be mapped to a non-trivial element in $C$ for any $n$).
It only remains to show that there exists a retraction from $W_{t+1}=W_t\ast_M A$ onto $W_t$.\\
It clearly suffices to show that there exists a retraction from $A$ onto $M$. Recall that $A'^+$ denotes the unique finite-by-abelian subgroup of $A'$ of index $\leq 2$. It follows from the discussion in \cite{rewe} after Definition 3.10 that $A'^+/P(A')^+\cong \Z^k$ is free abelian of rank $k$ for some $k\in\N$. Let $\{x_1,\ldots,x_k\}$ be a $\Z$-basis of $A'^+/P(A')^+$. By abuse of notation we denote a lift of this basis to $A'$ again by $\{x_1,\ldots,x_k\}$. It follows from section 4.2.1 in \cite{rewe} that for large enough $n$ we can choose the sequence of modular automorphisms $(\alpha_n)$ earlier in the construction such that  $\alpha_n$ fixes $P(A')$ and such that the tuple of word lengths
$$(|\vp_n\circ\iota_{i+1}\circ\eta_i\circ\alpha_n(x_i)|_{\Gamma})=(\lambda_1,\ldots,\lambda_k)$$
is linearly independent over $\mathbb{Q}$, where $|\cdot|_{\Gamma}$ denotes the word length with respect to a fixed finite generating set of $\Gamma$.\\
Now the sequence $\Psi_n|_{M^+\ast_{P(A')^+}A'^+}$ converges to the action of $A^+$ on a real tree $T$ isometric to a line. The kernel of this action is $M^+$, while $x_1,\ldots,x_k$ act hyperbolically by translations on $T$. Since $(\lambda_1,\ldots,\lambda_k)$ is linearly independent over $Q$, 
$\{x_1,\ldots,x_k\}$ generate a free abelian subgroup of rank $k$ of $A^+$. Since $M^+$ is the kernel of the action, it is a normal subgroup of $A^+$, i.e. by definition of the sequence, $A^+/M^+\cong \Z^k$. Hence $A^+$ fits into the short exact sequence
$$1\to M^+\to A^+\to\Z^n\to 1.$$
Therefore for all $i,j\in\{1,\ldots,k\}$, $p\in\{1,\ldots,l\}$ there exist elements $e_{ij}$, $\bar{e}_{ip}$ of finite order in $M^+$ such that $A^+$ has the presentation 
\begin{align*}
A^+=\langle &M^+,x_1,\ldots,x_k\ |\ [x_i,x_j]=e_{ij}, [x_i,y_p]=\bar{e}_{ip}\\
&\text{ for all } i,j\in\{1,\ldots,k\}, p\in\{1,\ldots,l\}\rangle,
\end{align*}
where $\{y_1,\ldots,y_l\}$ is a generating set for $M^+$. Note that there exists an element $s$ such that $M$ is generated by $\{s,y_1,\ldots,y_l\}$ and either $s=1$ or the image of $s$ in the quotient of $A'$ by its torsion subgroup has order $2$. Using Lemma 4.9 in \cite{rewe}, we conclude that $A$ has the presentation
\begin{align*}
A=\langle &M,x_1,\ldots,x_k\ |\ [x_i,x_j]=e_{ij}, [x_i,y_p]=\bar{e}_{ip}, sx_is^{-1}x_i=\tilde{e}_{i},\\ 
&\text{ for all } i,j\in\{1,\ldots,k\}, p\in\{1,\ldots,l\}\rangle,
\end{align*}
where $\tilde{e}_i$ is an element of finite order in $M^+$ for all $i\in\{1,\ldots,k\}$.\\ 
For all $i\in\{1,\ldots,k\}$, $n\in\N$ we set $m_n^i:=\iota_{i+1}\circ\eta_i\circ\alpha_n(x_i)$. Let $E$ be the torsion subgroup of $M^+$. Since $A'$ is virtually abelian it holds for all $i,j\in\{1,\ldots,k\}$ that there exist some $e_{ij}^{(n)}, \tilde{e}_{ij}^{(n)}\in E$ such that 
$$[m_n^i,m_n^j]=e_{ij}^{(n)} \text{ and } sm_n^is^{-1}m_n^i=\tilde{e}_{ij}^{(n)}.$$
Moreover for all $n\in\N$, $\iota_{i+1}\circ\eta_i\circ\alpha_n$ defines an action of $A'$ on $M$ and since $M$ is virtually abelian, there exists elements $\bar{e}_{ij}^{(n)}\in E$ for all $i\in\{1,\ldots,k\}$, $j\in\{1,\ldots,l\}$ such that
$$[m_n^i,y_j]=\bar{e}_{ij}^{(n)}.$$
Since $|E|$ is finite, after passing to a subsequence of $(\alpha_n)$ (still denoted $(\alpha_n)$) we can assume that for all $i,j\in\{1,\ldots,k\}$, $p\in\{1,\ldots,l\}$:
\begin{align*}
&[m_n^i,m_n^j]=e_{ij},\\
&sm_n^is^{-1}m_n^i=\tilde{e}_{ij} \text{ and}\\
&[m_n^i,y_p]=\bar{e}_{ij}
\end{align*}
holds. Therefore there exists some fixed $n\in\N$ such that the map 
$$\pi_0:A=(M\ast_{P(A')}A')/_{\displaystyle \underrightarrow{\ker} \Psi_n}\to M$$ defined via
$$\pi_0|_M=\id \text{ and } \pi_0(x_i)=m_n^i$$
extends to a retraction $\pi: A\to M$. Hence $W_{t+1}$ has the structure of a virtually abelian flat over $W_t$.\\
Now assume that the virtually abelian vertex $a$ is already contained in $A_t$. In this case we set $B_{t+1}':=B_t'$ and $W_{t+1}=W_t$.
\end{enumerate}
\item Assume that $A_e$ is finite. By the construction of $\Comp(L_{i+1})$ and $\B_{i+1}$ there exists precisely one distinguished edge, say $f$, with edge group $A_f$ (isomorphic to) $A_e$ in $\B_{i+1}$ (which corresponds to $e_{t+1}$). We set $A_w:=A_{\alpha(e_{t+1})}$, $A_u:=A_{\omega(e_{t+1})}$ and again have to distinguish several cases. Recall that by assumption $\alpha(e_{t+1})\in A_t$.
\begin{enumerate}[(a)]
\item First suppose that $A_u$ and $A_w$ are either isolated QH vertex groups or isolated virtually abelian vertex groups in $\A$. In this case we set $\B'_{t+1}:=\B'_t$ and $W_{t+1}:=W_t$.
\item Suppose that $A_u\in \A_t$.
\begin{itemize} 
\item If $A_w$ and $A_u$ are rigid we replace $f$ by a new edge $f'$ (with edge group $A_e$) connecting the vertices $u'$ and $w'$ of $\B'_t$ containing $\iota_{i+1}\circ\eta_i(A_u)$ and $\iota_{i+1}\circ\eta_i(A_v)$ respectively. The boundary monomorphisms are given by $$\alpha_{f'}=\iota_{i+1}\circ\eta_i\circ\alpha_{e+1}:A_e\to B_{v'}$$ and $$\omega_{f'}=\iota_{i+1}\circ\eta_i\circ\omega_{e+1}:A_e\to B_{u'}.$$
\item If $A_w$ and $A_u$ are QH vertices of one of the JSJ decompositions $\A_i^j$, i.e. in particular not extensions of closed orbifold groups, we replace $f$ by a new edge $f'$ (with edge group $A_e$) connecting the vertices $w'$ and $u'$ of $\B'_t$ whose vertex groups are the isomorphic copies of $A_w$ and $A_u$ respectively, which have been added earlier in the process. Denote by $\tau_1:A_w\to B_{w'}$ and $\tau_2:A_u\to B_{u'}$ the corresponding canonical isomorphisms. Then the boundary monomorphisms of $f'$ are given by $\alpha_{f'}=\tau_1\circ\alpha_{e+1}:A_e\to B_{w'}$ and $\omega_{f'}=\tau_2\circ\omega_{e+1}:A_e\to B_{u'}$.
\item Suppose that either $A_u$ is rigid and $A_w$ is a QH vertex or vice versa. Without loss of generality we assume that $A_w$ is rigid. Then again we replace $f$ by a new edge $f'$ (with edge group $A_e$) connecting the vertex $w'$ of $\B'_t$ whose vertex group contains $\iota_{i+1}\circ\eta_i(A_w)$ to the vertex whose vertex group is the isomorphic copy of $A_u$, which has been added earlier in the process. Denote by $\tau:A_u\to B_{u'}$ the corresponding canonical isomorphism. Then the boundary monomorphisms of $f'$ are given by $\alpha_{f'}=\iota_{i+1}\circ\eta_i\circ\alpha_{e+1}:A_e\to B_{w'}$ and $\omega_{f'}=\tau\circ\omega_{e+1}:A_e\to B_{u'}$.
\item Suppose that either $A_w$ is rigid or a QH vertex of one of the non-trivial JSJ decompositions $A_j^i$ and $A_u$ is an isolated QH or virtually abelian vertex group or vice versa. Without loss of generality we assume that $A_w$ is rigid or a non-isolated QH vertex. Then again we replace $f$ by a new edge $f'$ (with edge group $A_e$) connecting the vertex $w'$ of $\B_t$ whose vertex group contains $\iota_{i+1}\circ\eta_i(A_w)$ (or the isomorphic copy of $A_w$) with the vertex $u'$ with vertex group $B_{u'}=A_u$. If $A_w$ is a QH vertex group, denote by $\tau:A_w\to B_{w'}$ the corresponding canonical isomorphism. Then the boundary monomorphisms of $f'$ are given by 
$$\alpha_{f'}=\iota_{i+1}\circ\eta_i\circ\alpha_{e+1}:A_e\to B_{w'}$$ (or $\alpha_{f'}=\tau\circ\alpha_{e+1}:A_e\to B_{w'}$) and
$$\omega_{f'}=\omega_{e+1}:A_e\to A_u=B_{u'}.$$
\end{itemize}
\item Suppose that $A_u\notin \A_t$.
\begin{itemize}
\item If $A_u$ is a rigid vertex, there exists a vertex $u'$ in $\B'_t$ which contains $\iota_{i+1}\circ\eta_i(A_u)$ and we proceed as in the respective subcases (1),(3) and (4) of $(b)$
\item If $A_u$ is an isolated QH or virtually abelian vertex group, we either already have dealt with this edge in case (a) or we proceed as in the subcase $(4)$ of $(b)$.
\item If $A_u$ corresponds to a QH vertex in one of the non-trivial JSJ decompositions we replace $f$ by a new edge $f'$ together with a new vertex $u'$, where $f'$ connects $u'$ to the vertex $v'$ which either contains $\iota_{i+1}\circ\eta_i(A_w)$ (in case that $A_w$ is rigid), or the isomorphic copy of $A_w$ (in the QH case), or the group $A_w$ (in the isolated QH/Virtually abelian vertex case). The edge group $A_{f'}$ is isomorphic to $A_e$ and the embeddings are constructed in the same way as in the various subcases of $(b)$.\\
It is possible that this new graph of groups is not connected. In this case we add the edge, say $f$, which we have replaced before again, making the graph connected again. We then proceed with the iterative procedure until removing the edge $f$ does no longer disconnect the graph and then remove $f$.  
\end{itemize}
\end{enumerate}
\end{enumerate}

By now we have constructed a new graph of groups $\B'_{t+1}$ whose fundamental group $W_{t+1}=\pi_1(\B'_{t+1})$ has the structure of a tower over $\Comp(L_{i+1})$. It remains to construct the embedding $\nu_{t+1}:\pi_1(\A_{t+1})\to W_{t+1}$. To do this, we again distinguish several cases, depending on the edge $e_{t+1}$.\\ 
Recall that we started with a graph of groups $\B'_{t}$ together with an embedding $$\nu_{t}:\pi_1(\A_t)\hookrightarrow W_t=\pi_1(\B'_t),$$
which coincides on rigid vertex groups with $\iota_{i+1}\circ\eta_i$ up to conjugation by an element, which is either trivial or contained in $W_t\setminus\Comp(L_{i+1})$.
Moreover $w:=\alpha(e_{t+1})\in \A_t$ and $A_e$ denotes the edge group of $e_{t+1}$ with the corresponding embeddings into the adjacent vertex groups $\alpha: A_e\hookrightarrow A_w:=A_{\alpha(e_{t+1})}$ and $\omega: A_e\hookrightarrow A_u:=A_{\omega(e_{t+1})}$.
\begin{enumerate}[(1)]
\item Assume that $A_e$ is infinite.
\begin{enumerate}[(a)]
\item Assume that $e_{t+1}$ connects two rigid vertices in $\A$. Then there exists a maximal finite-by-abelian subgroup $M_l^+$ of $\Comp(L_{i+1})$ into which $\iota_{i+1}\circ\eta_i(A_e)$ can be conjugated by an element $g\in \Comp(L_{i+1})$. Suppose without loss of generality that we are in the case that $W_{t+1}$ is the amalgamated product $$W_{t+1}=W_t\ast_{M_l^+}(M_l^+\oplus \Z^{|I_l|}).$$ Let $\{z_1,\ldots,z_{|I_l|}\}$ be a basis of $\Z^{|I_l|}$. Moreover we assume without loss of generality that the edge $e_{t+1}$ corresponds to the first element of the ordered set $I_l$.
\begin{itemize}
\item We first consider the case that $w=\alpha(e_{t+1})\in\A_t$ and $u=\omega(e_{t+1})\notin \A_t$. 
By induction hypothesis there exists $h\in (W_t\setminus \Comp(L_{i+1}))\cup\{1\}$ such that 
$$\nu_t|_{A_w}=c_h\circ\iota_{i+1}\circ\eta_{i}|_{A_w}.$$
We define $\nu_{t+1}:\pi_1(\A_{t+1})\to W_{t+1}$ by 
$$\nu_{t+1}|_{\pi_1(\A_t)}=\nu_t$$ and $$\nu_{t+1}|_{A_u}=c_h\circ c_{g^{-1}z_1g}\circ\iota_{i+1}\circ\eta_i.$$
We first show that this is in fact a homomorphism. Let $a\in A_e$. Then for $\omega(a)\in A_u$ holds that $g(\iota_{i+1}(\eta_i(\omega(a)))g^{-1}\in M_l^+$ and hence commutes with $z_1$. It follows that
\begin{align*}
c_h\circ c_{g^{-1}z_1g}\circ\iota_{i+1}\circ\eta_i(\omega(a))
&=c_hc_{g^{-1}g}\circ\iota_{i+1}\circ\eta_i(\omega(a))\\
&=c_h\circ\iota_{i+1}\circ\eta_i(\alpha(a))\\
&=\nu_t(\alpha(a))
\end{align*}
and therefore $\nu_{t+1}$ is a well-defined homomorphism. Moreover $hg^{-1}z_1g\notin \Comp(L_{i+1})$. It remains to show the injectivity of $\nu_{t+1}$. Let $$y=a_1b_1\cdots a_kb_k\in \pi_1(\A_{t+1})\setminus\{1\}$$ be reduced with respect to the amalgamated product decomposition 
$$\pi_1(\A_{t+1})=\pi_1(\A_t)\ast_{A_e}A_u.$$
We show that
\begin{align*}
\nu_{t+1}(y)=&\nu_t(a_1)hg^{-1}z_1g(\iota_{i+1}\circ\eta_i(b_1))g^{-1}z_1^{-1}gh^{-1}\nu_t(a_2)\cdots\\
&\cdots \nu_t(a_k)hg^{-1}z_1g(\iota_{i+1}\circ\eta_i(b_k))g^{-1}z_1^{-1}gh^{-1}
\end{align*}
is reduced with respect to the decomposition 
$$W_{t+1}=W_t\ast_{M_l^+}(M_l^+\oplus \Z^{|I_l|}).$$
Suppose not. Then either there exists $j\in\{1,\ldots,k\}$ such that 
$$g(\iota_{i+1}\circ\eta_i(b_j))g^{-1}\in M_l^+$$
or there exists $j\in\{2,\ldots,k\}$ such that 
$$gh^{-1}\nu_t(a_j)hg^{-1}\in M_l^+.$$
Suppose we are in the first case. There exists a non-trivial element $x\in \omega(A_e)\leq A_u$ such that $g(\iota_{i+1}\circ\eta_i(x))g^{-1}\in M^+_l$. Hence $\iota_{i+1}\circ\eta_i(x)$ and  $\iota_{i+1}\circ\eta_i(b_j)$ lie in the same maximal finite-by-abelian subgroup $g^{-1}M_l^+g$ of $\Comp(L_{i+1})$. But since $\iota_{i+1}\circ\eta_i$ is injective on $A_u$ this implies that $b_j$ and $x$ lie in the same maximal finite-by-abelian subgroup of $A_u$. By Lemma \ref{changes} $\omega(A_e)$ is maximal finite-by-abelian in $A_u$ and therefore $b_j\in \omega(A_e)$, a contradiction to the reducedness of $y$.\\
Now suppose we are in the second case. There exists a non-trivial element $x\in \alpha(A_e)\leq A_w$ such that $g(\iota_{i+1}\circ\eta_i(x))g^{-1}\in M_l^+$. Hence $h^{-1}\nu_t(a_j)h$ and $\iota_{i+1}\circ\eta_i(x)$ lie in the same maximal finite-by-abelian subgroup of $W_t$. It follows that $h(\iota_{i+1}\circ\eta_i(x))h^{-1}$ and $\nu_t(a_j)$ lie in the same maximal finite-by-abelian subgroup of $W_t$. But since by the induction hypothesis $$\nu_t(a_j)=h(\iota_{i+1}\circ\eta_i(a_j))h^{-1}$$ and $\iota_{i+1}\circ\eta_i$ is injective on $A_w$ we have that $a_j$ and $x$ lie in the same maximal finite-by-abelian subgroup of $A_w$. Again Lemma \ref{changes} yields a contradiction to the reducedness of $y$.
\item Now consider the case that $\alpha(e_{t+1}), \omega(e_{t+1})\in \A_t$, i.e. $\pi_1(\A_{t+1})$ splits as an HNN-extension $\pi_1(\A_{t+1})=\pi_1(\A_t)\ast_{A_e}$. By the induction hypothesis there exist elements $h_1,h_2\in (W_t\setminus \Comp(L_{i+1}))\cup\{1\}$ such that 
$$\nu_t|_{A_w}=c_{h_1}\circ\iota_{i+1}\circ\eta_i|_{A_w}$$
and 
$$\nu_t|_{A_u}=c_{h_2}\circ\iota_{i+1}\circ\eta_i|_{A_u}.$$
We define $\nu_{t+1}:\pi_1(\A_{t+1})\to W_{t+1}$ by 
$$\nu_{t+1}|_{\pi_1(\A_t)}=\nu_t$$ and 
$$\nu_{t+1}(p)=h_2(\iota_{i+1}\circ\eta_i(p))g^{-1}z_1gh_1^{-1},$$
where $p$ denotes the stable letter of the HNN-extension $\pi_1(\A_{t+1})=\pi_1(\A_t)\ast_{A_e}$. It follows from similar arguments as in the amalgamated product case that $\nu_{t+1}$ is an injective homomorphism.
\end{itemize}
\item Assume that $e_{t+1}$ connects a rigid and a QH vertex in $\A$. First assume that $\omega(e_{t+1})=u\notin\A_t$ is the QH vertex.
In this case 
$$\pi_1(\A_{t+1})=\pi_1(\A_t)\ast_{A_e}A_u$$
and 
$$W_{t+1}=W_t\ast_{A_e}A_u.$$
By the induction hypothesis there exists $h\in (W_t\setminus \Comp(L_{i+1}))\cup\{1\}$ such that 
$$\nu_t|_{A_w}=c_h\circ\iota_{i+1}\circ\eta_{i}|_{A_w}.$$
We define $\nu_{t+1}:\pi_1(\A_{t+1})\to W_{t+1}$ by 
$$\nu_{t+1}|_{\pi_1(\A_t)}=\nu_t$$ and $$\nu_{t+1}|_{A_u}=c_h.$$
It is a straightforward calculation that this defines in fact an injective homomorphism.\\
Now assume that the QH vertex is $\alpha(e_{t+1})=w\in\A_t$. Then either
$$\pi_1(\A_{t+1})=\pi_1(\A_t)\ast_{A_e}A_u$$ for some rigid vertex group $A_u$ or 
$$\pi_1(\A_{t+1})=\pi_1(\A_t)\ast_{A_e}.$$
In both cases $W_{t+1}=W_t\ast_{A_e}$. Assume we are in the first case.\\
We define $\nu_{t+1}:\pi_1(\A_{t+1})\to W_{t+1}$ by 
$$\nu_{t+1}|_{\pi_1(\A_t)}=\nu_t$$ and $$\nu_{t+1}|_{A_u}=c_s\circ \iota_{i+1}\circ\eta_i|_{A_u},$$
where $s$ is the stable letter of the HNN-extension $W_{t+1}=W_t\ast_{A_e}$. Clearly $s\notin\Comp(L_{i+1})$. Once again it is straightforward to show that $\nu_{t+1}$ is an injective homomorphism.\\
Now assume that we are in the second case, i.e. $\pi_1(\A_{t+1})=\pi_1(\A_t)\ast_{A_e}.$ Denote by $p$ the stable latter of this HNN-extension and by $s$ the stable letter of the HNN-extension $W_{t+1}=W_t\ast_{A_e}$. By the induction hypothesis there exists $h\in (W_t\setminus \Comp(L_{i+1}))\cup\{1\}$ such that 
$$\nu_t|_{A_u}=c_h\circ\iota_{i+1}\circ\eta_{i}|_{A_u}.$$
We define $\nu_{t+1}:\pi_1(\A_{t+1})\to W_{t+1}$ by 
$$\nu_{t+1}|_{\pi_1(\A_t)}=\nu_t$$ and $$\nu_{t+1}(p)=hs.$$
As in the cases before one can easily check that this is an injective homomorphism.
\item Assume that $e_{t+1}$ connects a rigid and a virtually abelian vertex in $\A$. First assume that $\omega(e_{t+1})=u\notin\pi_1(\A_t)$. In this case 
$$\pi_1(\A_{t+1})=\pi_1(\A_t)\ast_{A_e}A_u$$
and $A^+_e=P(A^+_u)$. Let $M$ be the maximal virtually abelian subgroup of $W_t$ containing $\nu_t(A_e)$. Then by construction
$$W_{t+1}=W_t\ast_M A,$$
where $A$ is virtually abelian and contains $A_u$. Denote by $\tau: A_u \hookrightarrow A$ the corresponding embedding. Note that by construction 
$$\tau(\omega(A_e))=\iota_{i+1}\circ\eta_i(\omega(A_e))\leq M\leq A.$$
By the induction hypothesis there exists $h\in (W_t\setminus \Comp(L_{i+1}))\cup\{1\}$ such that 
$$\nu_t|_{A_w}=c_h\circ\iota_{i+1}\circ\eta_{i}|_{A_w}.$$
We define $\nu_{t+1}:\pi_1(\A_{t+1})\to W_{t+1}$ by 
$$\nu_{t+1}|_{\pi_1(\A_t)}=\nu_t$$ and $$\nu_{t+1}|_{A_u}=c_h\circ \tau.$$
One can immediately check that this is a homomorphism since
\begin{align*}
\nu_{t+1}(\alpha(A_e))&=c_h\circ\iota_{i+1}\circ\eta_{i}(\alpha(A_e))\\
&=c_h\circ\iota_{i+1}\circ\eta_{i}(\omega(A_e))\\
&=c_h\circ\tau(\omega(A_e))\\
&=\nu_{t+1}(\omega(A_e)).
\end{align*}
For the injectivity let $y=a_1b_1\cdots a_kb_k\in \pi_1(\A_{t+1})\setminus \{1\}$ be reduced with respect to the amalgamated product decomposition 
$$\pi_1(\A_{t+1})=\pi_1(\A_t)\ast_{A_e}A_u.$$
Suppose that 
$$\nu_{t+1}(y)=\nu_t(a_1)h\tau(b_1)h^{-1}\nu_t(a_2)\cdots\nu_t(a_k)h\tau(b_k)h^{-1}$$
is not reduced with respect to the splitting $W_{t+1}=W_t\ast_M A$. Then either 
$h^{-1}\nu_t(a_j)h\in M$ for some $j\in\{2,\ldots,k\}$ or $\tau(b_j)\in M$ for some $j\in\{1,\ldots,k\}$.\\
First suppose that $\tau(b_j)\in M$. By construction $$A=(M\ast_{A_e}A_u)/_{\displaystyle \underrightarrow{\ker} \Psi_n},$$ where $(\Psi_n)\subset \Hom(M\ast_{A_e}A_u,\Gamma)$ is a stable sequence, and hence since $b_j\notin A_e$ we have that $\tau(b_j)b^{-1}_j\in \underrightarrow{\ker} \Psi_n$. Once again by the construction of the stable sequence $(\Psi_n)$ this would imply that $b_j\in A_e$, a contradiction to the reducedness of $y$.\\
So suppose now that $h^{-1}\nu_t(a_j)h\in M$. Then there exists an element $x\in \omega(A_e)\leq A_u$ of infinite order, such that $\tau(x)=\iota_{i+1}\circ\eta_i(x)\in M$. Hence $h^{-1}\nu_t(a_j)h$ and $\iota_{i+1}\circ\eta_i(x)$ lie in the same maximal virtually abelian subgroup of $W_t$. It follows that $\nu_t(a_j)$ and $c_h\circ\iota_{i+1}\circ\eta_i(x)$ lie in the same maximal virtually abelian subgroup of $W_t$. We can consider $x$ also as an element of $A_w$ and since by the induction hypothesis 
$\nu_t|_{A_w}=c_h\circ\iota_{i+1}\circ\eta_i$ and $\nu_t$ is injective on $\pi_1(\A_t)$ we have that $a_j$ and $x$ lie in the same maximal virtually abelian subgroup of $\pi_1(\A_t)$. Hence $a_j\in A_e$, a contradiction to the reducedness of $y$.\\
Now assume that $w=\alpha(e_{t+1})\in \A_t$ is the virtually abelian vertex. In this case $\pi_1(\A_{t+1})=\pi_1(\A_t)\ast_{A_e}A_u$ for some rigid vertex group $A_u$ or $\pi_1(\A_{t+1})=\pi_1(\A_t)\ast_{A_e}$. In both cases $W_{t+1}=W_t$ has a decomposition as $W_{t+1}=H\ast_M A$, where $A$ is a virtually abelian group containing $A_w$, $M$ is the maximal virtually abelian subgroup containing $\iota_{i+1}\circ\eta_i(A_e)$ and $\iota_{i+1}\circ\eta_i(A_u)\leq \Comp(L_{i+1})\leq H$.\\
Suppose we are in the first case. We define $\nu_{t+1}:\pi_1(\A_{t+1})\to W_{t+1}$ by 
$$\nu_{t+1}|_{\pi_1(\A_t)}=\nu_t$$ and $$\nu_{t+1}|_{A_u}=\iota_{i+1}\circ\eta_i|_{A_u}.$$
Now suppose we are in the second case, i.e. $\pi_1(\A_{t+1})=\pi_1(\A_t)\ast_{A_e}$. Let $t$ be the stable letter of this HNN-extension. By induction hypothesis there exists an element $h\in (W_t\setminus \Comp(L_{i+1}))\cup\{1\}$ such that  
$$\nu_t(A_u)=c_h\circ\iota_{i+1}\circ\eta_{i}.$$
We then define $\nu_{t+1}:\pi_1(\A_{t+1})\to W_{t+1}$ by 
$$\nu_{t+1}|_{\pi_1(\A_t)}=\nu_t$$ and $$\nu_{t+1}(t)=\iota_{i+1}\circ\eta_i(t)h^{-1}.$$
In both cases one can check as before that $\nu_{t+1}$ is an injective homomorphism.
\end{enumerate}
\item Assume that $A_e$ is finite. Then by Lemma \ref{changes} the vertex groups of the adjacent vertices of $e_{t+1}$ are either rigid, QH subgroups, isolated QH vertex groups or isolated virtually abelian vertex groups. We again distinguish several cases.
\begin{enumerate}[(a)]
\item First assume that $A_w$ is rigid and $u=\omega(e_{t+1})\notin \A_t$, i.e. $\A_{t+1}=\A_t\ast_{A_e}A_u$. By construction there exists a distinguished edge $f$ in $\B_{t+1}$, whose edge group is isomorphic to $A_e$ (we ignore the isomorphism and consider them equal). Denote by $\alpha_f:A_e\to B_{v'}:=B_{\alpha(f)}$ and $\omega_f: A_e\to B_{u'}:=B_{\omega(f)}$ the corresponding boundary monomorphisms. By the induction hypothesis there exists an element $h\in (W_t\setminus \Comp(L_{i+1}))\cup\{1\}$ such that  
$$\nu_t(A_w)=c_h\circ\iota_{i+1}\circ\eta_{i}.$$
\begin{itemize}
\item Suppose that $A_u$ is rigid. Note that then $\iota_{i+1}\circ\eta_i(A_u)$ is contained in $B_{\omega(f)}$. We define $\nu_{t+1}:\pi_1(\A_{t+1})\to W_{t+1}$ by 
$$\nu_{t+1}|_{\pi_1(\A_t)}=\nu_t$$ and $$\nu_{t+1}|_{A_u}=c_h\circ \iota_{i+1}\circ\eta_i|_{A_u}.$$
\item Suppose that $A_u$ is a QH vertex group, an isolated QH vertex group or an isolated virtually abelian vertex group. Then $B_{u'}$ is an isomorphic copy of $A_u$ witnessed by the isomorphism $\tau: A_u\to B_{u'}$.  
We define $\nu_{t+1}:\pi_1(\A_{t+1})\to W_{t+1}$ by 
$$\nu_{t+1}|_{\pi_1(\A_t)}=\nu_t$$ and $$\nu_{t+1}|_{A_u}=c_h\circ\tau.$$
\end{itemize} 
In both cases it follows straightforward from the construction that $\nu_{t+1}$ defines a monomorphism.
\item Suppose that $A_w$ is either a QH vertex group in one of the non-trivial JSJ decompositions $\A_j^i$, an isolated QH vertex group or an isolated virtually abelian vertex group in $\A$. Assume moreover that $u=\omega(e_{t+1})\notin \A_t$, i.e. $\A_{t+1}=\A_t\ast_{A_e}A_u$. By construction there exists a distinguished edge $f$ in $\B_{t+1}$, whose edge group is isomorphic to $A_e$ (we ignore the isomorphism and consider them equal). Denote by $\alpha_f:A_e\to B_{v'}:=B_{\alpha(f)}$ and $\omega_f: A_e\to B_{u'}:=B_{\omega(f)}$ the corresponding boundary monomorphisms
\begin{itemize}
\item Suppose that $A_u$ is rigid. Note that then $\iota_{i+1}\circ\eta_i(A_u)$ is contained in $B_{\omega(f)}$. We define $\nu_{t+1}:\pi_1(\A_{t+1})\to W_{t+1}$ by 
$$\nu_{t+1}|_{\pi_1(\A_t)}=\nu_t$$ and $$\nu_{t+1}|_{A_u}=\iota_{i+1}\circ\eta_i|_{A_u}.$$
\item Suppose that $A_u$ is a QH vertex group, an isolated QH vertex group or an isolated virtually abelian vertex group. Then $B_{u'}$ is an isomorphic copy of $A_u$ witnessed by the isomorphism $\tau: A_u\to B_{u'}$.  
We define $\nu_{t+1}:\pi_1(\A_{t+1})\to W_{t+1}$ by 
$$\nu_{t+1}|_{\pi_1(\A_t)}=\nu_t$$ and $$\nu_{t+1}|_{A_u}=\tau.$$ 
\end{itemize} 
Again in both cases it follows straightforward from the construction that $\nu_{t+1}$ defines a monomorphism which satisfies the claim.
\item Assume that $A_w$ is rigid and $u=\omega(e_{t+1})\in \A_t$, i.e. $\A_{t+1}=\A_t\ast_{A_e}$. Denote by $t$ the stable letter of this HNN-extension. By construction there exists a distinguished edge $f$ in $\B_{t+1}$, whose edge group is isomorphic to $A_e$ (we ignore the isomorphism and consider them equal). Denote by $\alpha_f:A_e\to B_{v'}:=B_{\alpha(f)}$ and $\omega_f: A_e\to B_{u'}:=B_{\omega(f)}$ the corresponding boundary monomorphisms. Furthermore denote by $s$ the generator of $\pi_1(\B_{t+1})$ corresponding to $f$. By the induction hypothesis there exists an element $h\in (W_t\setminus \Comp(L_{i+1}))\cup\{1\}$ such that  
$$\nu_t|_{A_w}=c_h\circ\iota_{i+1}\circ\eta_{i}|_{A_w}.$$
\begin{itemize}
\item Suppose that $A_u$ is rigid. Note that then $\iota_{i+1}\circ\eta_i(A_u)$ is contained in $B_{\omega(f)}$. By the induction hypothesis there exists an element $\tilde{h}\in (W_t\setminus \Comp(L_{i+1}))\cup\{1\}$ such that  
$$\nu_t|_{A_u}=c_{\tilde{h}}\circ\iota_{i+1}\circ\eta_{i}|_{A_u}.$$
We define $\nu_{t+1}:\pi_1(\A_{t+1})\to W_{t+1}$ by 
$$\nu_{t+1}|_{\pi_1(\A_t)}=\nu_t$$ and $$\nu_{t+1}(t)=\tilde{h}sh^{-1}.$$
We show that this is a well-defined homomorphism. Let $x\in A_e$. Then
\begin{align*}
\nu_{t+1}(t\alpha_{e_{t+1}}(x)t^{-1})&= \nu_{t+1}(t)h(\iota_{i+1}\circ\eta_i(\alpha_{e_{t+1}}(x)))h^{-1}\nu_{t+1}(t)^{-1}\\
&=\tilde{h}sh^{-1}h(\iota_{i+1}\circ\eta_i(\alpha_{e_{t+1}}(x)))h^{-1}hs^{-1}\tilde{h}^{-1}\\
&=\tilde{h}s(\iota_{i+1}\circ\eta_i(\alpha_{e_{t+1}}(x)))s^{-1}\tilde{h}^{-1}\\
&=\tilde{h}(\iota_{i+1}\circ\eta_i(\omega_{e_{t+1}}(x)))\tilde{h}^{-1}\\
&=\nu_{t+1}(\omega_{e_{t+1}}(x))
\end{align*}
and hence we have shown that $\nu_{t+1}$ is a homomorphism. The injectivity is now an easy computation. 
\item Suppose that $A_u$ is a QH vertex group, an isolated QH vertex group or an isolated virtually abelian vertex group. Then $B_{u'}$ is an isomorphic copy of $A_u$ witnessed by the isomorphism $\tau: A_u\to B_{u'}$.  
We define $\nu_{t+1}:\pi_1(\A_{t+1})\to W_{t+1}$ by 
$$\nu_{t+1}|_{\pi_1(\A_t)}=\nu_t$$ and $$\nu_{t+1}(t)=sh^{-1}.$$
As in the case before one can show that $\nu_{t+1}$ is in injective homomorphism. 
\end{itemize} 
\item Suppose that $A_w$ is either a QH vertex group in one of the non-trivial JSJ decompositions $\A_j^i$, an isolated QH vertex group or an isolated virtually abelian vertex group in $\A$. Assume moreover that $u=\omega(e_{t+1})\in \A_t$, i.e. $\A_{t+1}=\A_t\ast_{A_e}$. Denote by $t$ the stable letter of this HNN-extension. By construction there exists a distinguished edge $f$ in $\B_{t+1}$, whose edge group is isomorphic to $A_e$ (we ignore the isomorphism and consider them equal). Furthermore denote by $s$ the generator of $\pi_1(\B_{t+1})$ corresponding to $f$. Denote by $\alpha_f:A_e\to B_{v'}:=B_{\alpha(f)}$ and $\omega_f: A_e\to B_{u'}:=B_{\omega(f)}$ the corresponding boundary monomorphisms
\begin{itemize}
\item Suppose that $A_u$ is rigid. Note that then $\iota_{i+1}\circ\eta_i(A_u)$ is contained in $B_{\omega(f)}$. By the induction hypothesis there exists an element $h\in (W_t\setminus \Comp(L_{i+1}))\cup\{1\}$ such that  
$$\nu_t|_{A_u}=c_h\circ\iota_{i+1}\circ\eta_{i}|_{A_u}.$$
We define $\nu_{t+1}:\pi_1(\A_{t+1})\to W_{t+1}$ by 
$$\nu_{t+1}|_{\pi_1(\A_t)}=\nu_t$$ and $$\nu_{t+1}(t)=hs.$$
\item Suppose that $A_u$ is a QH vertex group, an isolated QH vertex group or an isolated virtually abelian vertex group. Then $B_{u'}$ is an isomorphic copy of $A_u$ witnessed by the isomorphism $\tau: A_u\to B_{u'}$.  
We define $\nu_{t+1}:\pi_1(\A_{t+1})\to W_{t+1}$ by 
$$\nu_{t+1}|_{\pi_1(\A_t)}=\nu_t$$ and $$\nu_{t+1}(t)=s.$$ 
\end{itemize} 
Again in both cases it follows straightforward from the construction that $\nu_{t+1}$ defines a monomorphism.
\end{enumerate}
\end{enumerate}

\noindent Therefore in all cases we have constructed an embedding $\nu_{t+1}:\pi_1(\A_{t+1})\to W_{t+1}$ which coincides on rigid vertex groups with $\iota_{i+1}\circ\eta_i$ up to conjugation by an element which is either trivial or does not belong to $\Comp(L_{i+1})$.\\ 
After finitely many steps of this iterative procedure we have constructed a  graph of groups $\B'$ with $W:=\pi_1(\B')$ such that $L_i=\pi_1(\A)$ embeds into $W$. For the purpose of a possible trial and error procedure for quantifier elimination similar to the torsion-free case in \cite{selahyp}, we have to remove some vertices and edges from $\B'$.\\
Let $Q$ be a non-isolated QH vertex group in $\A$ and assume that $\eta_i(Q)$ has infinite intersection with $\pi_1(\U)$ for some connected subgraph of groups $\U$ of $\h_{i+1}$. Then $\eta_i(Q)$ splits as a subgraph of groups of the Dunwoody decomposition $\D^{(i+1)}$ of $L_{i+1}$ and we denote the vertex groups of this graph of groups which belong to $\h_{i+1}$ by $P_1,\ldots, P_k$. Since $\h_{i+1}$ is isomorphic to a subgraph of groups of $\B_{i+1}$, by the construction of $\B'$, there exists for all $s\in\{1,\ldots,k\}$ a distinguished vertex group $\tilde{P}_i$ in $\B'$ isomorphic to $P_i$. We then remove all the vertices with vertex groups $\tilde{P}_1,\ldots, \tilde{P}_k$ together with adjacent edges from $\B'$ and call this new graph of groups $\B_{i}$. Note that $\B_{i}$ is still connected and there exists an embedding $\iota_i:L_i\to \pi_1(\B_i)$ (namely the same embedding as before).
This completes the proof of Theorem \ref{completion}.
\end{proof}

\begin{definition}
Keeping the notation from the above construction, we say that 
$$\Comp(L):=\Comp(L_0)$$
is the \textnormal{completed limit group} or the \textnormal{completion} of the  well-structured resolution $\Res(L)$ and the sequence of completed limit groups $\Comp(L_i)$, $1\leq i\leq l$, together with the quotient maps $\Comp(\eta_i)$ is the \textnormal{completed resolution} of $\Res(L)$, which we denote by $\Comp(\Res(L))$.
\end{definition}

The following Lemma is now a straightforward computation, making use of the construction before.

\begin{lemma}
$\Comp(L)$ is a $\Gamma$-limit tower and for all $i\in\{0,\ldots,l-1\}$ the quotient map $\eta_i:L_i\to L_{i+1}$ can be naturally extended to a quotient map $\Comp(\eta_i): Comp(L_i)\to\Comp(L_{i+1})$.
\end{lemma}

\subsection{Closures of completions}
The next lemma follows immediately from the construction of the completion.

\begin{lemma}\label{completion property}
Let $\Res(L)$ be a well-structured resolution of a restricted $\Gamma$-limit group $L$ and let $\Comp(\Res(L))$ be its completed resolution. Denote by $\iota:L\hookrightarrow\Comp(L)$ the embedding of $L$ into its completion. Then:
\begin{enumerate}[(1)]
\item The completed resolution $\Comp(Res(L))$ is well-structured.
\item The rank  of the completed resolution $\Comp(\Res(L))$ is at most the rank of the well-structured resolution $\Res(L)$.
\item Any homomorphism $\vp:L\to \Gamma$ that factors through the resolution $\Res(L)$ can be extended to a homomorphism $\bar{\vp}:\Comp(L)\to\Gamma$ that factors through the completed resolution $\Comp(\Res(L))$, such that $\vp=\bar{\vp}\circ\iota$. 
\end{enumerate}
\end{lemma}

Given a $\Gamma$-limit tower $T$, we can associate a slightly larger tower to it, a so-called closure of $T$.
 
\begin{definition}\label{closure}
\begin{enumerate}[(1)]
\item Let $T=T_n\geq\ldots\geq T_1\geq T_0=\Gamma$ be a $\Gamma$-limit tower. A \textnormal{tower closure} of $T$ is a $\Gamma$-limit tower $\Cl(T)$ together with an embedding $\iota:T\to \Cl(T)$ such that the following diagram commutes
\begin{center}
\begin{tikzpicture}[descr/.style={fill=white}]
\matrix(m)[matrix of math nodes,
row sep=3em, column sep=1.8em,
text height=1.5ex, text depth=0.25ex]
{T=T_n&T_{n-1}&\cdots& T_1&T_0=\Gamma\\
\Cl(T)=\bar{T}_n&\bar{T}_{n-1}&\cdots&\bar{T}_1&\bar{T}_{0}=\Gamma\\};
\path[->]
(m-1-1) edge [white] node[black] {$\geq$} (m-1-2)
(m-1-2) edge [white] node[black] {$\geq$} (m-1-3)
(m-1-3) edge [white] node[black] {$\geq$} (m-1-4)
(m-1-4) edge [white] node[black] {$\geq$} (m-1-5)
(m-2-1) edge [white] node[black] {$\geq$} (m-2-2)
(m-2-2) edge [white] node[black] {$\geq$} (m-2-3)
(m-2-3) edge [white] node[black] {$\geq$} (m-2-4)
(m-2-4) edge [white] node[black] {$\geq$} (m-2-5)
(m-1-1) edge node[right] {$\iota$} (m-2-1)
(m-1-2) edge node[right] {$\iota_{n-1}$} (m-2-2)
(m-1-4) edge node[right] {$\iota_1$} (m-2-4)
(m-1-5) edge node[right] {$\iota_0$} (m-2-5);   
\end{tikzpicture}
\end{center}
where for all $i\in\{0,\ldots,n-1\}$ the map $\iota_i=\iota|_{T_i}$ is an embedding and the horizontal lines are the tower decompositions. Moreover the following hold:
\begin{itemize}
\item If $T_i=T_{i-1}\ast_MA$ is a virtually abelian flat over $T_{i-1}$ then $\bar{T}_i=\bar{T}_{i-1}\ast_{\iota(M)} \bar{A}$ is a virtually abelian flat over $\bar{T}_{i-1}$, where $\iota|_A: A\to \bar {A}$ is an embedding which is an isomorphism when restricted to $M$, such that $\iota(A)$ is a subgroup of finite index in $\bar{A}$ and $Z_{\iota(M)}$ is a direct factor of $Z_{\bar{A}}$. Here we denote by $Z_{\bar{A}}$ and $Z_{\iota(M)}$ the maximal free abelian subgroups of $\bar{A}$ and $\iota(M)$ respectively.
\item If $T_i=T_{i-1}\ast_E$ is an HNN-extension along a finite subgroup, then $\bar{T}_i=\bar{T}_{i-1}\ast_{E'}$, where $E'$ is a finite subgroup containing $\iota(E)$.
\item If $T_i\ast_E H$ is an amalgamated product with a subgroup $H$ of $\Gamma$ along a finite group $E$, then $\bar{T}_i=\bar{T}_{i-1}\ast_{E'} V$, where $E'$ is a finite group containing $\iota(E)$, $V$ is a subgroup of $\Gamma$ and $\iota(H)\leq V$.
\item If $T_i$ is an orbifold flat over $T_{i-1}$ witnessed by the graph of groups $\A$ with QH vertex group $Q$ fitting into the short exact sequence
$$1\to E\to Q\to \pi_1(\mO)\to 1,$$
then $\bar{T}_i$ is a an orbifold flat over $\bar{T}_{i-1}$ witnessed by the graph of groups $\B$ with QH vertex group $Q'$ fitting into the short exact sequence
$$1\to E'\to Q'\to \pi_1(\mO)\to 1,$$
where $E'$ is a finite group containing $\iota(E)$ and $\iota(Q)\leq Q'$. Moreover the underlying graphs of $\A$ and $\B$ are isomorphic and edge groups of $\B$ are finite index supergroups of the corresponding edge groups in $\A$.
\end{itemize} 
\item Let $\Res(L)$ be a well-structured resolution of a $\Gamma$-limit group $L$ and $\Comp(\Res(L))$ its completed resolution. A \textnormal{closure of the resolution} $\Res(L)$, denoted by $\Cl(\Res(L))$, is defined as follows: 
We replace the completed limit group $\Comp(L)$ by a tower closure of itself and all other limit groups which appear along the completed resolution by the tower closure which is induced by the tower closure of $\Comp(L)$.
\end{enumerate}
\end{definition}

%
%
\begin{definition}
Let $$\Cl(\Res(L)):\ \Cl(\Comp(L))=L_0\xrightarrow{\pi_1}L_1\xrightarrow{\pi_2}L_2\xrightarrow{\pi_3}\ldots\xrightarrow{\pi_k}L_k$$  be a closure of a well-structured resolution of a $\Gamma$-limit group $L$. Let $G$ be a $\Gamma$-limit group and $$\GCl(\Res(L)):G=G_0\xrightarrow{\eta_1}G_1\xrightarrow{\eta_2}G_2\xrightarrow{\eta_3}\ldots\xrightarrow{\eta_k}G_k$$ a resolution of $G$ with the following properties:
\begin{enumerate}[(a)]
\item $G_0$ admits a decomposition as a graph of groups $\D_0$ with finite edge groups, such that one vertex group is isomorphic to $L_0$ and all other vertex groups $H_1,\ldots,H_l$ are isomorphic to subgroups of $\Gamma$. 
\item For all $i\in\{1,\ldots,k\}$, $G_i$ admits a decomposition as a graph of groups $\D_i$ which arises from the graph of groups $\D_0$ by replacing the vertex group $L_0$ by $L_i$ and adjusting the adjacent boundary morphisms in the obvious way. 
\item $\eta_i|_{L_i}=\pi_i$ for all $i\in\{1,\ldots,k\}$.
\item $\eta_i|_{H_j}=\id$ for all $j\in\{1,\ldots,l\}$, for all $i\in\{1,\ldots,k\}$.
\end{enumerate}
Then we call $\GCl(\Res(L))$ a \textnormal{generalized closure of the resolution $\Res(L)$} and $$\GCl(\Comp(L)):=G$$ a \textnormal{generalized closure of the completion of $L$}.
\end{definition}

Closures of the completion of a well-structured resolution of a $\Gamma$-limit group $L$ will appear in a natural way in the proof of the Generalized Merzlyakov's Theorem.

\begin{definition}
Let $\Res(L)$ be a well-structured resolution of a $\Gamma$-limit group $L$. We call a finite set $\{\Cl(\Res(L))_1,\ldots,\Cl(\Res(L))_q\}$ of closures of $\Res(L)$ a \textnormal{covering closure} if for any homomorphism $\vp: L\to \Gamma$ that factors through $\Res(L)$ there exists an extension $\bar{\vp}: \Cl(\Comp(L))_i\to \Gamma$ of $\vp$ that factors through $\Cl(\Res(L))_i$ for (at least) one index $i\in\{1,\ldots,q\}$.
\end{definition}

In order to understand what it means for a finite set $\{\Cl(\Res)_1,\ldots,\Cl(\Res)_q\}$ to be a covering closure we look at the following toy example.

\begin{beispiel}
Let $F=F(a_1,a_2)$ and $L=\langle a_1,a_2,c\ |\ [a_2,z]\rangle$. Then 
$$L=F\ast_{\Z} \Z^2=\langle a_1,a_2\rangle \ast_{\{a_2=c_1\}}\langle c_1,c_2\ |\ [c_1,c_2]\rangle$$
is a restricted $F$-limit group that admits a well-structured resolution $\Res(L): L\xrightarrow{\eta} F$. Clearly $\Comp(L)=L$ and $\Comp(\Res(L))=\Res(L)$.\\
Let $C:=\langle c_1,c_2\rangle\leq L$, $Z=\langle c_1, z_2\ |\ [c_1,z_2]\rangle\cong \Z^2$ and consider the homomorphism 
$f: C\to Z$ given by $f(c_1)=c_1$ and $f(c_2)=c_1^2z_2^3$. Then since $\langle c_1, c_1^2z_2^3\rangle$ is of index $3$ in $Z$, the map $f$ defines a closure 
$$\Cl(L)=F\ast_{\{a_2=c_1\}}Z$$ of $L$.
To the closure embedding $f$ we can assign the following system of linear equation(s) with two variables and one equation: 
$$\Sigma(x,y):\quad x=2+3y.$$
For any $p\in\Z$ the system of equations $\Sigma(p,y)$ (in one variable $y$) has a solution if and only if $p\in 2+3\Z$.\\
Let now $\vp\in \Hom(L,F)$ be a restricted homomorphism and assume that $\vp$ can be extended to a homomorphism $\tilde{\vp}: \Cl(L)\to F$. Clearly $\vp(c_2)=a_2^p$ for some $p\in\Z$.
Since $\vp$ can be extended to $\Cl(L)$, this implies that 
$$a_2^p=\vp(c_2)=\tilde{\vp}\circ f(c_2)=\tilde{\vp}(c_1^2z_2^3)=a_2^{2+3k}$$ for some $k\in\Z$. In particular the system of equations $\Sigma(p,y)$ has a solution.\\
On the contrary suppose we are given a finite index subgroup $U=m\Z$ of $\Z$. Then for every integer $k$ the coset $k+U$ defines a closure embedding $h: C\to Z$ by $h(c_1)=c_1$ and $h(c_2)=c_1^kz_2^m$. A homomorphism $\vp: L\to F$ can be extended to a homomorphism of $\Cl_h(L)\to F$ if and only if $\vp(c_2)=a_2^p$ for some $p\in k+U$ (here $\Cl_h(L)$ denotes the closure given by $h$).\\
Therefore a set $\{\Cl(\Res(L))_1,\ldots,\Cl(\Res(L))_q\}$ is a covering closure of $L$ if and only if the union of the cosets $$\bigcup_{i=1}^qk_i+U_i=\bigcup_{i=1}^qk_i+m_i\Z$$
associated to the (systems of equations of the) closure embeddings cover $\Z$.
\end{beispiel}

The following observation is now immediate.

\begin{bem}\label{abeliantorsionfree}
\begin{enumerate}[(1)]
\item Let $A=\langle a_1,\ldots,a_m\rangle$ and $\bar{A}=\langle z_1,\ldots,z_m\rangle$ be free abelian groups of rank $m$ and $f: A\to \bar{A}$ be an embedding which  maps $A$ to a finite index subgroup of $\bar{A}$. By the Elementarteilersatz we can assume that for all $i\in\{1,\ldots,m\}$ there exists $k_i\in\Z$ such that $f(a_i)=k_iz_i$. Hence $$f(A)\cong k_1\Z\oplus k_2\Z\oplus\ldots\oplus k_m\Z=:U_f\leq \Z^m.$$
Let $\Gamma$ be a torsion-free hyperbolic group and $\vp:A\to \Gamma$ an arbitrary homomorphism. Clearly $\vp(A)\leq\langle u\rangle\cong \Z$ for some element $u\in \Gamma$ with no root. For all $i\in\{1,\ldots,m\}$ there exists $p_i\in \Z$ such that $\vp(a_i)=u^{p_i}$. Then the following hold:
\begin{align*}
&\vp \text{ can be extended to }\bar{A}\\
\Longleftrightarrow &\text{ there exists a homomorphism } \bar{\vp}:\bar{A}\to F \text{ such that }\vp=\bar{\vp}\circ f\\
\Longleftrightarrow &\ u^{p_i}=\vp(a_i)=\bar{\vp}\circ f(a_i)=\bar{\vp}(k_iz_i)=\bar{\vp}(z_i)^{k_i} \text{ for all } i\in\{1,\ldots,m\}\\
\Longleftrightarrow &\ p_i\in k_i\Z\text{ for all } i\in\{1,\ldots,m\}\\
\Longleftrightarrow &\ (p_1,\ldots,p_m)\in U_f.
\end{align*}
Hence given finitely many embeddings $f_1,\ldots,f_k$ of $A$ into finite index abelian supergroups $\bar{A}_1,\ldots,\bar{A}_k$ the following holds:\\
Any homomorphism $\vp:A\to \Gamma$ can be extended to one of the groups $\bar{A}_1,\ldots,\bar{A}_k$ if and only if the union $\bigcup _{i=1}^k U_{f_i}$ of the associated subgroups of $\Z^m$ covers $\Z^m$.
\item Let now  $A=\langle c,a_1,\ldots,a_m\rangle$ and $\bar{A}=\langle c, z_1,\ldots,z_m\rangle$ be free abelian groups of rank $m+1$ and $f: A\to \bar{A}$ be an embedding which  maps $A$ to a finite index subgroup of $\bar{A}$ and stays the identity on $c$. Let $\vp:A\to\Gamma$ be an arbitrary homomorphism that maps $c$ to an element $u\in\Gamma$ with no root. Again $\vp(A)=\langle u\rangle\cong \Z$ and for all $i\in\{1,\ldots,m\}$ there exists $p_i\in \Z$ such that $\vp(a_i)=u^{p_i}$.
 Similar to case $(1)$ there exist $K_f\in \Z^m$ and a subgroup $U_f\leq \Z^m$ such that $\vp:A\to \Gamma$ can be extended to $\bar{A}$ if and only if $(p_1,\ldots,p_m)\in K_f+U_f$.\\
Hence given finitely many embeddings $f_1,\ldots,f_k$ of $A$ into finite index abelian supergroups $\bar{A}_1,\ldots,\bar{A}_k$ that stay the identity on $c$ the following holds:\\
Any homomorphism $\vp:A\to \Gamma$ can be extended to one of the groups $\bar{A}_1,\ldots,\bar{A}_k$ if and only if the union $\bigcup _{i=1}^k K_{f_i}+U_{f_i}$ of the associated cosets of subgroups of $\Z^m$ covers $\Z^m$.
\end{enumerate}
\end{bem}

We will use the ideas from this Remark later in the proof of the Generalized Merzlyakov's Theorem. But before we can start proving the theorem we have to define so-called test sequences for $\Gamma$-limit groups.

\section{Test sequences}\label{8}
We are going to define test sequences for the completion of a good well-structured resolution of a $\Gamma$-limit group $L$. We are doing this iteratively by defining these test sequences from bottom to top of the $\Gamma$-limit tower associated to the completion. We have to look at all the possible cases how the tower can be build, namely if we add virtually abelian or orbifold flats or we perform an HNN-extension or an amalgamated product along some finite subgroup.
The idea of a test sequence $(\lambda_n)$ for a $\Gamma$-limit tower $$T=T_0\geq T_{1}\geq\ldots\geq T_l=\Gamma$$ is the following:\\
Let $G$ be a group containing $T$ and $(\vp_n)\subset \Hom(G,\Gamma)$ a stable sequence of extensions of $(\lambda_n)$, which converges into the action of a $\Gamma$-limit group $L=G/\underrightarrow{\ker}\ \vp_n$ on some real tree $Y$ such that:
\begin{itemize}
\item $T\leq L$,
\item $L$ is one-ended relative $T$ and
\item $T$ acts without a global fixed point on $Y$.
\end{itemize}
Then the splitting of $T$ coming from the top level tower decomposition is "visible" in a decomposition of $L$ associated to its action on the real tree $Y$. To achieve this we need to control which elements of $T=T_0$ act hyperbolically on $Y$, while we have to guarantee that all elements of $T_1$ act elliptically on $T$. 

\begin{definition}
Let $G, \Gamma$ be groups with fixed finite generating sets and $(\lambda_n)\subset \Hom(G,\Gamma)$ be a sequence of homomorphisms. For $g,g'\in G$ we say that $\lambda_n(g)$ dominates the growth of $\lambda_n(g')$ if 
$$\lim_{n\to\infty}\frac{|\lambda_n(g')|_{\Gamma}}{|\lambda_n(g)|_{\Gamma}}=0.$$
For two subgroups $A,B\leq G$ we say that $\lambda_n|_A$ dominates the growth of $\lambda_n|_B$ if
$$\lim_{n\to\infty}\frac{|\lambda_n(b)|_{\Gamma}}{|\lambda_n(a)|_{\Gamma}}=0$$
for all $a\in A\setminus\{1\}, b\in B\setminus\{1\}$.
\end{definition}

We are now ready to define test sequences.

\begin{definition}\label{testseq} Let $\Gamma$ be a non-elementary hyperbolic group, 
$$L=T_m\geq T_{m-1}\geq\ldots\geq T_0=\Gamma$$
 a $\Gamma$-limit tower and $\vp_n\subset\Hom(L,\Gamma)$ a stably injective sequence. For $i\in\{1,\ldots,m\}$ let $\B_i$ be the graph of groups which is given by the i-th level of the tower decomposition of $L$. We call $(\vp_n)$ a \textnormal{test sequence} for the tower $L$ if the following holds
for any restricted $\Gamma$-limit group $M_i$ which contains $T_i$ and is one-ended relative $T_i$:\\
For any stably injective convergent sequence $(\Psi_n)\subset\Hom(M_i,\Gamma)$ such that
\begin{itemize}
\item for all $n\in\N$, $\Psi_n$ is an extension of $\vp_n|_{T_i}$ which is short relative to $T_i$ with respect to the word metric on $\Gamma$ and
\item $T_i$ does not act with a global fixpoint on the real tree $T$ into which the sequence $(\Psi_n)$ converges, 
\end{itemize}
the following hold:
\begin{enumerate}[(1)]
\item If $T_i$ has the structure of an orbifold flat over $T_{i-1}$ with QH vertex group $Q$ fitting into the short exact sequence 
$$1\to E\to Q\to\pi_1(\mO)\to 1,$$
then $M_i$ admits an orbifold flat decomposition $\A$ over a group $V$, which has one vertex $v$ with  QH vertex group $Q'$ fitting into the short exact sequence 
$$1\to E'\to Q'\to\pi_1(\mO)\to 1$$ 
and several vertices with vertex groups $V_1,\ldots,V_r$ corresponding to the various orbits of point stabilizers in the action of $M_i$ on $T$. Moreover the underlying graphs of $\A$ and $\B_i$ are isomorphic and the edge groups of $\A$ are finite index supergroups of the corresponding edge groups in $\B_i$, and $T_{i-1}\leq V$.
\item If $T_i$ has the structure of a virtually abelian flat $T_i=T_{i-1}\ast_CK$ over $T_{i-1}$ then there exist finitely many $\Gamma$-limit quotients $N_1,\ldots,N_p$ of the $\Gamma$-limit group $M_i$ and each $N_i$ admits a decomposition $\D$ relative $T_i$ along finite groups with one distinguished vertex group $U$ and all other vertex groups of $\D$ are isomorphic to subgroups of $\Gamma$. Moreover $U$ admits a decomposition as an amalgamated product $\A$ of the form $U=V\ast_{A_1}A$, where the subgroup $A$ is virtually abelian and contains the virtually abelian subgroup $K$ as a subgroup of finite index, $A_1$ contains $C$ as a subgroup of finite index and $T_{i-1}\leq V$.
\item If $T_i=T_{i-1}\ast_{E'}K$ is an extension of $T_{i-1}$ along a finite subgroup and $K$ is isomorphic to a one-ended subgroup of $\Gamma$, then there exist finitely many $\Gamma$-limit quotients $N_1,\ldots,N_p$ of the $\Gamma$-limit group $M_i$ and each $N_i$ admits a decomposition $\D$ relative $T_i$ along finite groups with one distinguished vertex group $U$ and all other vertex groups of $\D$ are isomorphic to subgroups of $\Gamma$. Moreover $U$ admits a splitting as a graph of groups $\A$ of the form $U=V\ast_E P$, where $E$ is a finite subgroup containing $E'$, $T_{i-1}\leq V$ and $P$ is isomorphic to a subgroup of $\Gamma$ containing $K$.
\item If $T_i=T_{i-1}\ast_{E'}$ is an extension of $T_{i-1}$ along a finite subgroup then $M_i$ admits a splitting $\A$ as an HNN-extension $M_i=V\ast_E$, where $E$ is a finite subgroup containing $E'$ and $T_{i-1}\leq V$. 
\end{enumerate}
In all cases the induced decomposition of $T_i$ inherited from $\A$ yields a graph of groups which is isomorphic to $\B_i$.
\end{definition}

\subsection{Orbifold flats}
\begin{satz}\label{orbifoldflat}
Assume that $G$ has the structure of an orbifold flat over some subgroup $H$ with QH vertex group $Q$, and there exists a test sequence $\lambda_n^H$ for $H$. Assume moreover that for all $n\in\N$ there exists a homomorphism $\Theta_n:Q\to\Gamma$ with non-virtually abelian image, such that (up to conjugacy) $\Theta_n(C)=\lambda_n^H(C)$ for all edge groups $C$ adjacent to $Q$ in the orbifold flat decomposition of $G$. Then there exists a test sequence $\lambda_n$ for $G$.
\end{satz}

Before we start, we observe that extensions of some orbifold groups cannot appear as QH subgroups of $\Gamma$-limit groups. Recall that a QH subgroup is a finite extension of a compact hyperbolic cone-type 2-orbifold group.

\begin{lemma}\label{smallorbifolds}
Let $L$ be a one-ended $\Gamma$-limit group and $\A$ be a virtually abelian JSJ decomposition of $L$. Let $Q$ be a finite extension of the fundamental group of a cone-type orbifold $\mO$ (possibly with boundary) such that the complement of the cone points $\mO^c$ of $\mO$ is a pair of pants, i.e. $\mO$ is either a sphere with 3 cone points, a disc with 2 cone points, an annulus with 1 cone point or a pair of pants. Then $Q$ is not isomorphic to a QH subgroup of $L$. In particular if $Q$ is a vertex group of $\A$, then $Q$ is rigid.
\end{lemma}

\begin{proof}
This follows from Proposition 5.12 and Proposition 5.20 in \cite{guirardeljsj}.
\end{proof}

The following result for hyperbolic groups is well-known so we omit its proof.

\begin{lemma}\label{potenz}
Let $\Gamma$ be a hyperbolic group. Then there exists $N>0$ such that for all $g,h \in\Gamma$ the following holds:\\
If $g$ has finite order and $\langle h, g\rangle$ is virtually cyclic, then $gh^N=h^Ng$. 
\end{lemma}

Before we start constructing test sequences for orbifold flats, we will prove some preliminary results. Throughout the whole chapter we will not distinguish between a simple closed curve on an orbifold and the corresponding element of the fundamental group.

\begin{lemma}\label{scc}
Let $Q$ be the fundamental group of a compact hyperbolic cone-type orbifold $\mathcal{O}$, such that the complement $\mO^c$ of the cone points has Euler characteristic at most $-2$ or is a once-punctured torus. Let $\mu: Q\to \Gamma$ be a homomorphism such that the following hold:
\begin{itemize}
\item $\mu(Q)$ is an infinite non-virtually abelian (i.e. non-virtually cyclic) subgroup of $\Gamma$,
\item $\mu$ is injective on boundary components and finite subgroups of $Q$ and
\item there exists a $\Gamma$-limit group $H$, a stably injective sequence $(\vp_n)\subset \Hom(H,\Gamma)$ and an epimorphism $\pi:Q\to H$ with non-virtually abelian image, such that $\mu=\vp_n\circ\pi$ for some $n\in\N$.
\item There exists a collection of non-homotopic, non-boundary parallel, essential simple closed curves $C=\{c_1,\ldots,c_m\}$ on $\mO^c$, which get mapped to elements of infinite order by $\mu$, such that every connected component $\Sigma$ of $\mO\setminus C$ is either a once-punctured sphere with 2 cone points of order $2$, or the complement $\Sigma^c$ of its cone points is a pair of pants and $\pi_1(\Sigma)$ is non-virtually abelian.
\end{itemize}
Then there exists a collection $B=\{b_1,\ldots,b_m\}$ of essential, non-homotopic, non-boundary parallel simple closed curves on $\mO^c$ such that the following hold:
\begin{enumerate}[(1)]
\item Every connected component $\Sigma$ obtained by cutting $\mathcal{O}$ along $B$  is either a once-punctured sphere with 2 cone points of order $2$, or its complement $\Sigma^c$ of the cone points has Euler characteristic $-1$ and $\mu(\Sigma)$ is non-virtually abelian in $\Gamma$.
\item If $x$ corresponds to a cone point on a connected component $\Sigma$ of $\mO\setminus B$, whose fundamental group is not isomorphic to $D_{\infty}$, and the curve $b_i$ is a boundary component of $\Sigma$, then for all $k\geq 1$ either $x^k=1$ or $\langle \mu(x^k),\mu(b_i)\rangle$ is non-virtually abelian.
\end{enumerate}
\end{lemma}

\begin{bem}
\begin{enumerate}[(a)]
\item The assumptions of the lemma hold for example if $Q$ is a QH subgroup of a one-ended $\Gamma$-limit group $L$ and $\mu=\vp_n\circ\pi|_Q$ where $\pi:L\to M$ is a quotient map onto a strict shortening quotient $M$ and $(\vp_n)\subset \Hom(M,\Gamma)$ is a test sequence (i.e. in particular stably injective). In particular these assumptions are fulfilled in the setting of Theorem \ref{orbifoldflat}.
\item To guarantee that a collection of curves $C$ as in the assumption of the lemma exists, we might have to precompose $\mu$ by a properly chosen automorphism $\beta\in\Mod(Q)$. We then denote $\mu\circ\beta$ in abuse of notation again by $\mu$.
\end{enumerate}
\end{bem}

\begin{proof}
To improve readability we restrict ourselves to the case that the underlying surface of $\mO$ is orientable. The proof can easily be adapted to the non-orientable case though. Let $C=\{c_1,\ldots,c_m\}$ be the collection of curves which exists by assumption.\\
Since $\Gamma$ is hyperbolic there exist up to isomorphism only finitely many virtually cyclic subgroups of $\Gamma$. For every virtually cyclic subgroup of $\Gamma$, say $A$, fix a maximal cyclic subgroup $Z_A$ of finite index. Then $|A:Z_A|=:k_A<\infty$ and there exist only finitely many of these indices. So let $K\in \N$ be a positive integer such that for every element $g$ in a virtually cyclic subgroup $A$ of $\Gamma$, $g^K$ is contained in $Z_A$.\\
For every connected component $\Sigma$ of $\mO\setminus C$, whose fundamental group is not isomorphic to $D_{\infty}$, choose elements $x,y\in \pi_1(\Sigma)$ such that $w_{\Sigma}:=[x^K,y^K]\neq 1$. Such elements exist since $\pi_1(\Sigma)\ncong D_{\infty}$ and hence contains a non-abelian free group.\\
Since $\pi(Q)$ is non-virtually abelian there exists an automorphism $\alpha\in \Aut(Q)$ such that the following hold:
\begin{itemize}
\item $\mu\circ\alpha(w_{\Sigma})=\vp_n\circ\pi\circ\alpha(w_{\Sigma})\neq 1$ for all connected components $\Sigma$ of $\mO\setminus C$, whose fundamental groups are not isomorphic to $D_{\infty}$.
\item If $l\geq 1$ and $x$ is an element corresponding to a cone point on a connected component $\Sigma$ of $\mO\setminus C$, whose fundamental group is not isomorphic to $D_{\infty}$, such that $x^l\neq 1$ then
$$[\mu\circ\alpha(x^l),\mu\circ\alpha(c_i^p)]\neq 1$$ for all $c_i$ corresponding to boundary components of $\Sigma$ and all $p\leq N$, where $N$ is chosen according to Lemma \ref{potenz}.
\end{itemize} 
Let $\Sigma$ be a connected component of $\mO\setminus C$, such that $\pi_1(\Sigma)\neq D_{\infty}$ and suppose that $\mu\circ\alpha(\pi_1(\Sigma))$ is virtually abelian. It follows that $$\mu\circ\alpha(w_{\Sigma})=[\mu\circ\alpha(x)^K,\mu\circ\alpha(y)^K]\in [\Z,\Z]=1,$$ 
a contradiction.\\
Now suppose that $x$ corresponds to a cone point on $\Sigma$ and let $l\geq 1$ such that $x^l\neq 1$. Since $[\mu\circ\alpha(x^l),\mu\circ\alpha(c_i^N)]\neq 1$ it follows from Lemma \ref{potenz} that $\langle \mu\circ\alpha(x^l),\mu\circ\alpha(c_i)\rangle$ is non-virtually abelian in $\Gamma$. Now $B:=\{\alpha(c_1),\ldots,\alpha(c_m)\}$ is the desired collection of curves.
\end{proof}

\begin{prop}\label{curves}
Let $Q$ be the fundamental group of a compact hyperbolic cone-type orbifold $\mathcal{O}$, such that the complement $\mO^c$ of the cone points has Euler characteristic at most $-2$ or is a once-punctured torus. Let $\mu: Q\to \Gamma$ be a homomorphism such that the following hold:
\begin{itemize}
\item $\mu(Q)$ is an infinite non-virtually abelian subgroup of $\Gamma$,
\item $\mu$ is injective on boundary components and finite subgroups of $Q$ and
\item there exists a $\Gamma$-limit group $H$, a stably injective sequence $(\vp_n)\subset \Hom(H,\Gamma)$ and an epimorphism $\pi:Q\to H$ with non-virtually abelian image, such that $\mu=\vp_n\circ\pi$ for some $n\in\N$.
\end{itemize}
Then there exist two collections of essential, non-homotopic, non-boundary parallel, disjoint simple closed curves on $\mathcal{O}^c$: $A:=\{a_1,\ldots,a_n\}$ and $B:=\{b_1,\ldots,b_m\}$, and an automorphism $\alpha\in \Aut(Q)$ with the following properties:
\begin{enumerate}[(1)]
\item Each complement of the cone points of a connected component obtained by cutting $\mathcal{O}$ along $A$ has Euler characteristic $-1$, and the homomorphism $\mu\circ\alpha: Q\to \Gamma$ maps the fundamental group of each of these connected components isomorphically onto a quasi-convex subgroup of $\Gamma$.
\item Each of the curves $b_i$ intersects at least one of the curves $a_j$.
\item $A\cup B$ fills $\mathcal{O}$, i.e. $\mathcal{O}^c\setminus (A\cup B)$ is a disjoint collection of discs and annuli.
\end{enumerate}
\end{prop}

\begin{proof}
We again restrict ourselves to the case that the underlying surface of $\mO$ is orientable.
Let $B$ be the collection of curves given by Lemma \ref{scc}. Hence each of the complements of the cone points of a connected component $\Sigma$ obtained by cutting $\mathcal{O}$ along $B$ is a pair of pants and has non-virtually abelian image in $\Gamma$, or $\pi_1(\Sigma)\cong D_{\infty}$.
Note that by construction every component of $\mO\setminus B$ has at most $2$ cone points.\\
Denote by $H$ the fundamental group of such a component. Then either 
$H\cong F_2=\langle a,b\rangle$, where $a,b$ correspond to two boundary components, or $H\cong \Z\ast \Z_p=\langle a,c\ |\ c^p\rangle$, where $a$ corresponds to a boundary component, while $c$ corresponds to a cone point of order $p$, or $H\cong\Z_p\ast\Z_q=\langle x,y\ |\ x^p, y^q\rangle$, where $x,y$ correspond to cone points of order $p$ and $q$ respectively and moreover either $p\neq 2$ or $q\neq 2$, or $H\cong D_{\infty}$.\\

Before we proceed with the proof we need the following criterion for elements in hyperbolic groups to be non-trivial. The equivalent result in free groups was proved by G. Baumslag in \cite{baumslag}.

\begin{lemma}\label{trivial}
Let $\Gamma$ be a hyperbolic group and $z\in \Gamma$ be an element of infinite order. Let moreover $g\in\Gamma$ be an element of the form
$$g=a_0z^{i_1}a_1z^{i_2}a_2\ldots a_{n-1}z^{i_n}a_n,$$
where $n\geq 1$ and whenever $0<k<n$ ($a_0, a_n=1$ is allowed), $\langle a_k,z\rangle$ is not virtually cyclic. Then there exists $K\in\N$ (depending on $z$ and $\max_{i}|a_i|_{\Gamma}$) such that $g\neq 1$ if $|i_k|\geq K$ for all $k\in\{1,\ldots,n\}$. 
\end{lemma}

\begin{proof}
This result is well-known. See for example Lemma 2.4 in \cite{olshanski}.
\end{proof}

Now we are ready to continue the proof of Proposition \ref{curves}. 
Let $\B$ be the graph of groups given by the decomposition of $\mathcal{O}$ along the collection of simple closed curves $B=\{b_1,\ldots,b_m\}$. Denote by $\vp_1,\ldots,\vp_m$ the corresponding Dehn-twists along the edges of this decomposition corresponding to $b_1,\ldots,b_m$ respectively.

\begin{lemma}\label{nontrivial}
Let $D$ be an arbitrary finite collection of essential simple closed curves on $\mathcal{O}$, such that each $d\in D$ has non-trivial intersection with at least one of the curves $b_1,\ldots,b_m$. Then for sufficiently large powers $n_1,\ldots,n_m$ of the Dehn-twists $\vp_1,\ldots,\vp_m$, the image of every $d\in D$ under $\mu\circ\vp_m^{n_m}\circ\ldots\circ\vp_1^{n_1}$ is non-trivial in $\Gamma$.
\end{lemma}

\begin{proof}
By assumption the Euler characteristic of $\mO^c$ is at most $-2$ or $\mO^c$ is a once-punctured torus.\\
We first assume that no fundamental group of a component of $\mO\setminus B$ is isomorphic to $D_{\infty}$ and consider the general case afterwards.\\
Let $d\in D$. We first consider the case that $d$ intersects precisely one curve from $B$, say $b_1$. Then $d$ and $b_1$ are contained in a suborbifold $\Sigma$ of $\mO$ such that either $\Sigma^c$ is a once-punctured torus or a $4$th-punctured sphere.
\begin{figure}[htbp]
\centering
\includegraphics{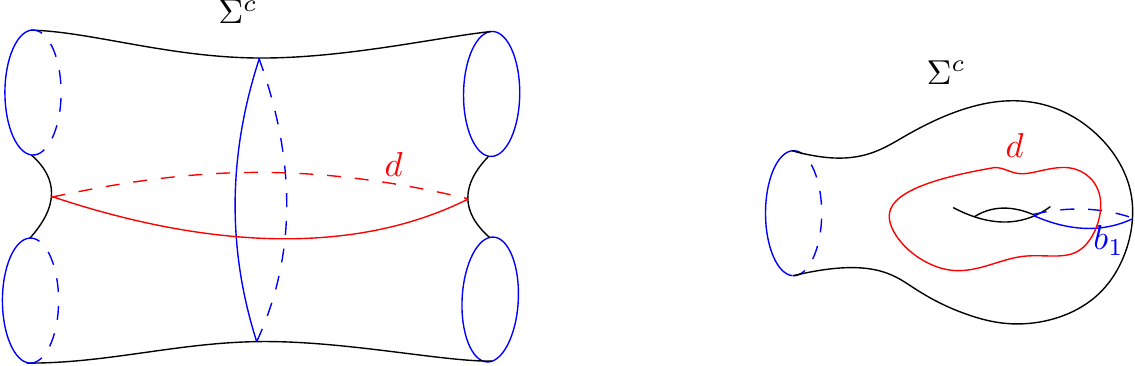}
\caption{$d$ and $b_1$ are contained in a suborbifold $\Sigma$ of $\mO$ such that either $\Sigma^c$ is a once-punctured torus or a $4$th-punctured sphere}
\label{pairs}
\end{figure} 
Suppose that $\Sigma^c$ is a $4$th-punctured sphere. If $d$ either corresponds to a cone point or a boundary component of $\mO$ or a curve in $B$ then $\mu(d)\neq 1$ by assumption. So we can assume that the connected components $P_1^c$ and $P_2^c$ of $\Sigma^c\setminus\{b_1\}$ are pairs of pants (see Figure \ref{pairs}) and there exist $p\geq 1$ and $s_0,\ldots,s_p\in \pi_1(P_1)\cup\pi_1(P_2)\subset \pi_1(\mO)$ such that
$$d=[s_0,e^{\epsilon_1},s_1,e^{\epsilon_2},s_2\ldots,
e^{\epsilon_p},s_p]$$ is a $\B$-path in normal form for $d$, where $\epsilon_j\in \{\pm 1\}$ for all $j\in\{1,\ldots,p\}$. Then $(P_1\setminus \{s_i\})^c$ (respectively $(P_2\setminus\{s_i\})^c$) is an annulus for all $i\in\{0,\ldots,p\}$ 
and therefore $\{b_1,s_i\}$ is a generating set for $\pi_1(P_1)$ (respectively $\pi_1(P_2)$). Since $\mu(\pi_1(P_1))$ and $\mu(\pi_1(P_2))$ are non virtually abelian by construction of the set $B$, it follows that $\langle \mu(b_1), \mu(s_i)\rangle$ is non-virtually abelian for all $i\in\{0,\ldots,p\}$. Hence applying Lemma \ref{trivial} to 
\begin{align*}
\mu\circ\alpha(d)&=\mu\circ\vp_1^{n_1}(d)
=\mu(s_0)\mu(b_1)^{\epsilon_1n_1}\mu(s_1)\mu(b_1)^{\epsilon_2n_1}\mu(s_2)\ldots\mu(b_1)^{\epsilon_pn_1}\mu(s_p)
\end{align*}
yields that $\mu\circ\alpha(d)$ is non-trivial for $n_1$ chosen large enough.\\

So now we assume that $d$ is contained in a suborbifold $\Sigma$ of $\mO$ such that $\Sigma^c$ is a once-punctured torus and cutting $\Sigma$ along  $b_1$ yields an orbifold $P$, such that $P^c$ is a pair of pants. By assumption $d$ has non-trivial intersection with $b_1$. If $d=b_1$ the claim is trivial, so let $p\geq 1$ and  $s_0,\ldots,s_p\in \pi_1(P)\subset \pi_1(\mO)$ such that
$$d=[s_0,e^{\epsilon_1},s_1,e^{\epsilon_2},s_2\ldots,
e^{\epsilon_p},s_p]$$ is a $\B$-path in normal form for $d$, where $\epsilon_j\in \{\pm 1\}$ for all $j\in\{1,\ldots,p\}$. Then $(P\setminus \{s_i\})^c$ is an annulus for all $i\in\{0,\ldots,p\}$ 
\begin{figure}[htbp]
\centering
\includegraphics{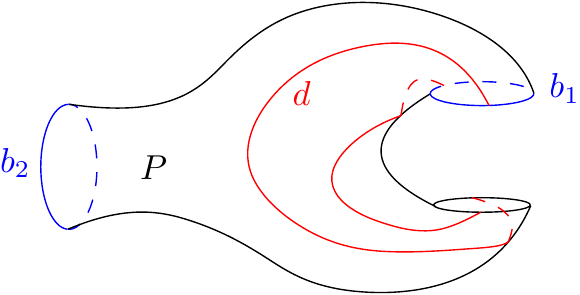}
\caption{$d$ has non trivial intersection with $b_1$}
\end{figure} 
and therefore $\{b_1,s_i\}$ is a generating set for $\pi_1(P)$. Since $\mu(\pi_1(P))$ is non-virtually abelian by construction of the set $B$, it follows that $\langle \mu(b_1), \mu(s_i)\rangle$ is non-virtually abelian for all $i\in\{0,\ldots,p\}$. Hence applying Lemma \ref{trivial} to 
\begin{align*}
\mu\circ\alpha(d)&=\mu\circ\vp_1^{n_1}(d)
=\mu(s_0)\mu(b_1)^{\epsilon_1n_1}\mu(s_1)\mu(b_1)^{\epsilon_2n_1}\mu(s_2)\ldots\mu(b_1)^{\epsilon_pn_1}\mu(s_p)
\end{align*}
yields that $\mu\circ\alpha(d)$ is non-trivial for $n_1$ chosen large enough.\\
So from now on we can assume that $d$ intersects at least two curves, say $b_1,b_2$, from $B$.
Let $$d=[a_0,e_1,a_1,e_2,a_2,\ldots,e_p,a_p]$$ be a $\B$-path in normal form for $d$. Then 
$$\mu\circ\alpha(d)=\mu(a_0)\mu(b_{i_1})^{n_{i_1}}\mu(a_1)\mu(b_{i_2})^{n_{i_2}}\mu(a_2)\ldots\mu(b_{i_p})^{n_{i_p}}\mu(a_p)$$
where $i_1,\ldots,i_p\in\{1,\ldots,m\}$.
Let $j\in\{1,\ldots,p-1\}$. Then $\{b_{i_{j-1}},a_{j}\}$ and $\{a_j,b_{i_{j+1}}\}$ generate the respective fundamental group of a component of $\mO\setminus B$ and therefore $\langle \mu(b_{i_{j-1}}),\mu(a_j)\rangle$ and also $\langle \mu(a_j),\mu(b_{i_{j+1}})\rangle$ is non-virtually abelian.
\begin{figure}[htbp]
\centering
\includegraphics{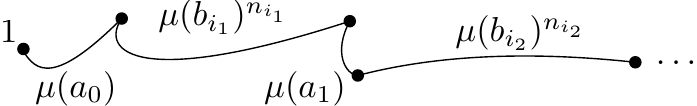}
\caption{The path corresponding to $\mu\circ\alpha(d)$ in $\Cay(\Gamma)$}
\end{figure}
Therefore it follows that for large enough powers of the Dehn-twists along the curves in $B$, the path in $\Cay(\Gamma)$ corresponding to $\mu\circ\alpha(d)$ contains no essential backtracking and hence $\mu\circ\alpha(d)$ is non-trivial.\\

The only thing that remains to be considered is the case that the fundamental group, say $H$, of a connected component $\Sigma$ of $\mO\setminus B$ is isomorphic to $D_{\infty}$. In particular $\Sigma$ is a once-punctured sphere with two cone points $x,y$ of order $2$ and $$H=\langle x,y\rangle\cong \Z_2\ast\Z_2.$$ 
Let as before $$d=[a_0,e_1,a_1,e_2,a_2,\ldots,e_p,a_p]$$ be a $\B$-path in normal form for $d$. Then 
$$\mu\circ\alpha(d)=\mu(a_0)\mu(b_{i_1})^{n_{i_1}}\mu(a_1)\mu(b_{i_2})^{n_{i_2}}\mu(a_2)\ldots\mu(b_{i_p})^{n_{i_p}}\mu(a_p)$$
where $i_1,\ldots,i_p\in\{1,\ldots,m\}$.
Let $j\in\{1,\ldots,p-1\}$ and suppose that $$b_{i_{j-1}}=b_{i_j}\leq H\cong \D_{\infty}.$$ Then $b_{i_j}=xy$ and either $a_{i_{j-1}}=x$ or $a_{i_{j-1}}=y$ (see Figure \ref{2torsion}). It follows that either 
$$\mu(b_{i_{j-1}})^{n_{i_{j-1}}}\mu(a_{i_{j-1}})\mu(b_{i_j})^{n_{i_{j}}}=\mu((xy)^{n_{i_{j-1}}}x(yx)^{n_{i_{j}}})=\mu(xy)^{n_{i_{j-1}}+n_{i_{j}}}\mu(x)$$ or
$$\mu(b_{i_{j-1}})^{n_{i_{j-1}}}\mu(a_{i_{j-1}})\mu(b_{i_j})^{n_{i_{j}}}=\mu((xy)^{n_{i_{j-1}}}y(yx)^{n_{i_{j}}})=\mu(xy)^{n_{i_{j-1}}+n_{i_{j}}-1}\mu(x).$$
Hence there exists no essential backtracking in the path corresponding to $$\mu(b_{i_{j-1}})^{n_{i_{j-1}}}\mu(a_{i_{j-1}})\mu(b_{i_j})^{n_{i_{j}}}$$ in $\Cay(\Gamma)$ and therefore the proof from before generalizes to the
\begin{figure}[htbp]
\centering
\includegraphics{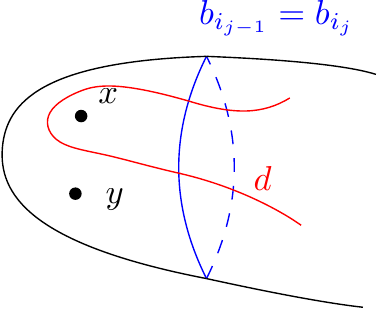}
\caption{The case $H=\langle x,y\rangle\cong D_{\infty}$}
\label{2torsion}
\end{figure}
case when there exists a connected component $\Sigma$ of $\mO\setminus B$ with $\pi_1(\Sigma)\cong D_{\infty}$.
\end{proof}

The aim now is to choose a new collection $A$ of simple closed curves on $\mathcal{O}$ such that the fundamental groups of connected components after cutting along these curves get mapped isomorphically onto quasi-convex subgroups after precomposing $\mu$ by high powers of Dehn-twists along the collection of curves $B$.

\begin{lemma}\label{newcurves}
There exists a collection $A$ of simple closed curves on $\mathcal{O}$ such that each of the components of $\mathcal{O}^c\setminus A$ has Euler characteristic $-1$ and the induced decomposition of the fundamental group of a component of $\mO\setminus A$ with respect to $\B$ is a finite graph of groups with trivial edge groups and finite or infinite cyclic vertex groups (the non-trivial finite ones corresponding to cone points, while the infinite cyclic ones correspond to boundary components of $\mO$). In particular $A\cup B$ fills the surface $\mO^c$.
\end{lemma}

\begin{proof} Denote by $g$ the genus of the surface $S:=\mO^c$ and by $r$ the number of boundary components of $S$. First assume that $S$ is neither the sphere with $4$ punctures nor the once-punctured torus, i.e. neither $g=0$, $r=4$, nor $g=1$, $r=1$. Let $\alpha$ and $\beta$ be two non-homotopic distinct essential simple closed curves on $S$. We denote the geometric intersection number of $\alpha$ and $\beta$ by $i(\alpha,\beta)$.  Note that curves homotopic to the boundary are inessential.\\
Then the complex of curves $C(S)$ of the surface $S$ is defined as follows: The $k$ simplices of $C(S)$ are $(k+1)$-tuples of homotopy classes of distinct, essential simple closed curves, which have pairwise geometric intersection number zero. The 1-skeleton of $C(S)$ is usually referred to as the curve graph $C^1(S)$ of the surface $S$. There is one vertex in $C^1(S)$ for each homotopy class of essential simple closed curves in $S$ and an edge between two vertices $\alpha$ and $\beta$ if and only if $i(\alpha,\beta)=0$.\\
In the case that $S$ is the $4$-punctured sphere or the once-punctured torus, the above defined complex is just an infinite collection of vertices without any edges. Therefore in these cases we alter the definition of $C^1(S)$ in the following way: The vertices are still the homotopy classes of essential simple closed curves on $S$, while there is now an edge between two vertices $\alpha$ and $\beta$ if and only if the geometric intersection number $i(\alpha,\beta)$ is minimal. In this case $C^1(S)$ is the well-known Farey-graph and in all cases Theorem 1.1 in \cite{masur} shows that $C^1(S)$ has infinite diameter.\\
 It is obvious that two (homotopy classes of) essential simple closed curves $\alpha$ and $\beta$ which have distance $\geq 3$ in $C^{1}(S)$ fill the surface, since an essential simple closed curve on $S$ which has trivial intersection with both curves would yield an vertex in $C^1(S)$ connected to both $\alpha$ and $\beta$, contradicting the assumption on the distance.
Moreover the collection of curves $B$ is a subset of diameter $\leq 1$ in $C^1(S)$. Hence the infinite diameter of $C^1(S)$ yields the claim.
\end{proof}

We continue with the proof of Proposition \ref{curves}. Recall that throughout the proof we do not distinguish between a simple closed curve on $\mO$ and the corresponding element of the fundamental group. Let $A$ be the collection of curves chosen according to Lemma \ref{newcurves} and denote by $\A$ the corresponding graph of groups decomposition of $\pi_1(\mathcal{O})$ induced by cutting $\mathcal{O}$ along $A$.\\
So let $v\in VA$ with vertex group $A_v$ and denote by $\mO_v$ the underlying orbifold. Then $\mO_v^c$ is a pair of pants and $\mO_v$ has at least $1$ boundary component  and at most $2$ cone points. Hence $A_v=\langle x_1,x_2\rangle$, where $x_1$ and $x_2$ either correspond to two boundary components or to two cone points.
Let $n_1,\ldots,n_m\in \N$ be large enough such that the conclusion of Lemma \ref{nontrivial} holds for the set of curves $A$. We define $\alpha:=\vp_m^{n_m}\circ\ldots\circ\vp_1^{n_1}$. We assume that every fundamental group of a connected component of $\mO\setminus B$ (i.e. the corresponding vertex group in $\B$) is non-virtually abelian, that is not isomorphic to $D_{\infty}$. If some vertex group of $\B$ is isomorphic to $D_{\infty}$, we can fairly straightforward adjust the arguments of the following proof precisely as we did in the proof of Lemma \ref{nontrivial} to also hold in this case.\\

First suppose that $\mO_v^c$ has two boundary components which are also boundary components of $\mO^c$.
Let $x$ and $y$ be the elements of $A_v$ corresponding to the boundary components of $\mO_v^c$ which are also boundary components of $\mO^c$. That is $x$ (and also $y$) either corresponds to a cone point or to a proper boundary component of $\mO$. In particular $A_v=\langle x\rangle\ast\langle y\rangle$. Since $\mO_v$ is a pair of pants and $A\cup B$ fill the orbifold $\mO$ there exists a curve, say $b_1$, in $B$ such that $x$ and $y$ lie in different components of $\mO_v\setminus \{b_1\}$, as shown in Figure \ref{toyexample}.
\begin{figure}[htbp]
\centering
\includegraphics{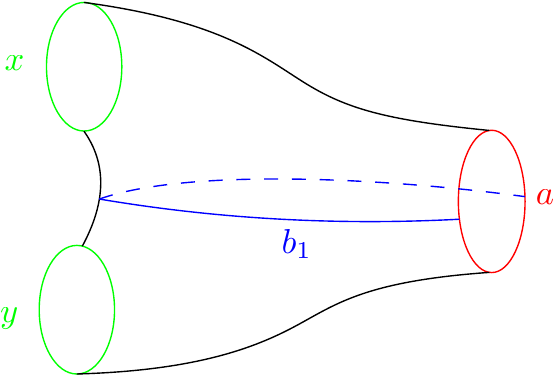}
\caption{Two boundary components $x,y$ of $\mO_v^c$ are boundary components of $\mO^c$}
\label{toyexample}
\end{figure}
For presentation purposes we assume that $b_1$ is the only curve in $B$ which intersects $\mO_v$ non-trivially. The general case (i.e. if there are either multiple curves in $B$, or multiple times the same curve, which separate $x$ and $y$ in $\mO_v$) follows from  the arguments given in the other cases below.\\
Then $y$ and $b_1$ correspond to boundary components of a component of $(\mO\setminus B)^c$, say $\Sigma^c$. Since by the construction of the set of curves $B$, $\mu(\pi_1(\Sigma))$ is non-virtually abelian, it clearly follows that $\langle \mu(y),\mu(b_1)\rangle$ is non-virtually abelian. By the same argument it follows that $\langle \mu(x),\mu(b_1)\rangle$ is non-virtually abelian. Moreover $\mu$ is injective on $\langle x\rangle$ and $\langle y\rangle$. Lemma \ref{scc} yields that $\langle \mu(x^l),\mu(b_1)\rangle$ and $\langle \mu(y^k)\mu(b_1)\rangle$ is non-virtually abelian whenever $x^l\neq 1$ (respectively $y^k\neq 1$) and $x$, $y$ correspond to cone points. If $x$ corresponds to a boundary component of $\mO$ then $\langle \mu(x^l),\mu(b_1)\rangle$ is non-virtually abelian since $\langle \mu(x),\mu(b_1)\rangle$ is non-virtually abelian and elements of infinite order are contained in a unique maximal virtually abelian subgroup of $\Gamma$.\\
Let now $w\in A_v$ be a non-trivial element with normal form 
$$w=x^{i_1}y^{j_1},\ldots x^{i_k}y^{j_k}$$ with respect to the free decomposition $A_v=\langle x\rangle\ast\langle y\rangle$.
We get that
$$\mu\circ\alpha(w)=\mu(x^{i_1})\mu(b_1)^{n_1}\mu(y^{j_1})\mu(b_1)^{-n_1}\mu(x^{i_2})\ldots \mu(b_1)^{n_1}\mu(y^{j_k})\mu(b_1)^{-n_1}.$$
Hence after possibly enlarging $n_1$ there is no essential backtracking in the path corresponding to $\mu\circ\alpha(w)$ in $\Cay(\Gamma)$ and therefore $\mu\circ\alpha$ is injective on $A_v$.\\

Now assume that $\mO_v^c$ has precisely one boundary component which is also a boundary component of $\mO^c$ and denote the corresponding curve on $\mO_v$ by $x$.
\begin{figure}[htbp]
\centering
\includegraphics{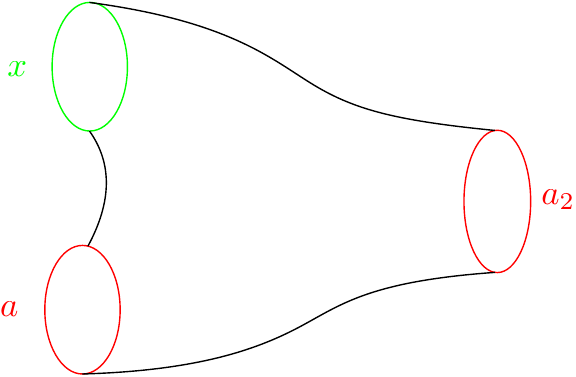}
\caption{Precisely one boundary component of $\mO_v^c$ is a boundary component of $\mO^c$}
\end{figure}
Let $a$ be an element of $A_v$ corresponding to a boundary component of $\mO_v^c$ which is not a boundary component of $\mO^c$, i.e. $a$ corresponds to a curve in $A$. Hence 
$$A_v=\pi_1(\mO_v)=\langle x,a\rangle=\langle x\rangle \ast\langle a\rangle.$$
Let $w=x^{k_1}a^{k_2}x^{k_3}\ldots a^{k_l}$ be a non-trivial element of $A_v$. If either $w=x^{k_1}$ or $w=a^{k_1}$ then $\mu\circ\alpha(w)\neq 1$ since $\mu\circ\alpha$ is injective on $\langle x\rangle$ by assumption and on $\langle a\rangle$ by Lemma \ref{nontrivial}.\\
So we can assume that $l\geq 2$. First consider the case that $|B|=|\{b\}|=1$, i.e. $\mO^c$ is a once punctured torus. Let $\Sigma_B$ be the orbifold obtained from cutting $\mO$ along $b$.
\begin{figure}[htbp]
\centering
\includegraphics{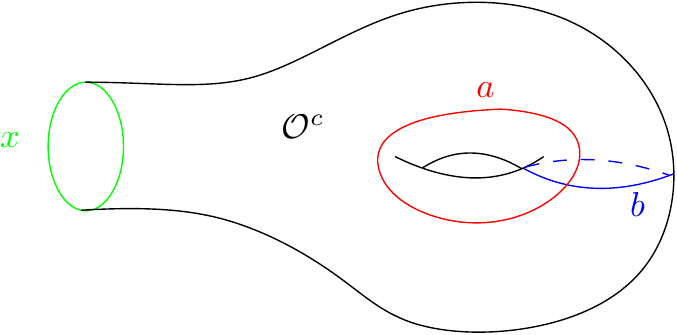}
\caption{The case that $\mO^c$ is a once punctured torus}
\end{figure}
Let $a=[s_1,e^{\epsilon_1},s_2,e^{\epsilon_2},s_3,\ldots,e^{\epsilon_{p-1}},s_p]$ be a $\B$-path in normal form for $a$, where $\epsilon_j\in \{\pm 1\}$ for all $j\in\{1,\ldots,p-1\}$. Then $(\Sigma_B\setminus \{s_i\})^c$ is an annulus for all $i\in\{1,\ldots,p\}$ since $a$ is a simple closed curve, and therefore $\{b,s_i\}$ is a generating set for $\pi_1(\Sigma_B)$. Moreover $s_i=b^{n_i}x$ for some $n_i\in\Z$.\\
Since $\mu(\pi_1(\Sigma_B))$ is non-virtually abelian it follows that $\langle \mu(b), \mu(s_i)\rangle$ is non-virtually abelian, for all $i\in\{1,\ldots,p\}$. Moreover Lemma \ref{scc} yields for all $l\geq 1$, such that $x^l\neq 1$, that $\langle \mu(x)^l,\mu(b)\rangle$ is non-virtually abelian if $\mu(x)$ has finite order (i.e. $x$ corresponds to a cone point on $\mO$). If $\mu(x)$ has infinite order (i.e. $x$ corresponds to a boundary component of $\mO$), then clearly $\langle \mu(x),\mu(b)\rangle$ and therefore also $\langle \mu(x)^l,\mu(b)\rangle$ is non-virtually abelian for all $l\geq 1$, since every element of infinite order is contained in a unique maximal virtually abelian subgroup of $\Gamma$. Hence for $n_1$ chosen large enough, the path in $\Cay(\Gamma)$ corresponding to
\begin{align*}
\mu\circ\alpha(w)&=\mu\circ\vp_1^{n_1}(w)\\
&=\mu(x)^{k_1}\Big( \mu(s_1)\mu(b)^{\epsilon_1n_1}\mu(s_2)\mu(b)^{\epsilon_2n_1}\mu(s_3)\ldots\mu(b)^{\epsilon_{p-1}n_1}\mu(s_p)\Big)^{k_2}\mu(x)^{k_3}\cdots\\
&\Big( \mu(s_1)\mu(b)^{\epsilon_1n_1}\mu(s_2)\mu(b)^{\epsilon_2n_1}\mu(s_3)\ldots\mu(b)^{\epsilon_{p-1}n_1}\mu(s_p)\Big)^{k_l}\\
&=\mu(x)^{k_1}\Big( \mu(b^{n_1})\mu(x)\mu(b)^{\epsilon_1n_1}\mu(s_2)\mu(b)^{\epsilon_2n_1}\mu(s_3)\ldots\mu(b)^{\epsilon_{p-1}n_1}\mu(b^{n_p})\mu(x)\Big)^{k_2}\mu(x)^{k_3}\\
&\cdots\Big( \mu(b^{n_1})\mu(x)\mu(b)^{\epsilon_1n_1}\mu(s_2)\mu(b)^{\epsilon_2n_1}\mu(s_3)\ldots\mu(b)^{\epsilon_{p-1}n_1}\mu(b^{n_p})\mu(x)\Big)^{k_l}
\end{align*}
contains no essential backtracking and therefore $\mu\circ\alpha(w)$ is non-trivial.\\
So now we can assume that $B=\{b_1,\ldots,b_m\}$ for some $m\geq 2$. Let $$a=[s_0,e_1^{\epsilon_1},s_1,e_2^{\epsilon_2},s_2,\ldots,e_{p}^{\epsilon_{p}},s_p]$$ be a $\B$-path in normal form for $a$, where $\epsilon_j\in \{\pm 1\}$ for all $j\in\{1,\ldots,p\}$. To simplify notation we assume without loss of generality that $e_i$ corresponds to the curve $b_i$ in $B$. 
\begin{figure}[htbp]
\centering
\includegraphics{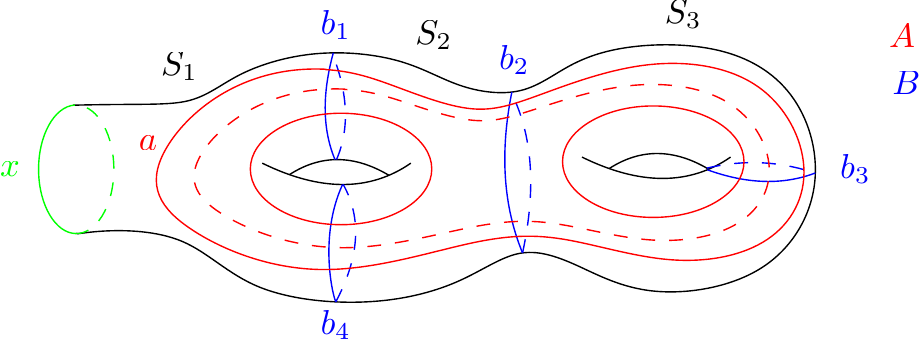}
\caption{The curve $a$ crosses multiple curves from $B$}
\end{figure}
Note that $a$ crosses at least two different simple closed curves $b_i,b_j\in B$.
Hence
\begin{align*}
\mu\circ\alpha(w)&=\mu(x)^{k_1}\Big( \mu(s_0)\mu(b_1)^{\epsilon_1n_1}\mu(s_1)\mu(b_2)^{\epsilon_2n_2}\mu(s_2)\ldots\mu(b_p)^{\epsilon_{p}n_p}\mu(s_p)\Big)^{k_2}\mu(x)^{k_3}\ldots\\
&\Big( \mu(s_0)\mu(b_1)^{\epsilon_1n_1}\mu(s_1)\mu(b_2)^{\epsilon_2n_2}\mu(s_2)\ldots\mu(b_p)^{\epsilon_{p}n_p}\mu(s_p)\Big)^{k_{l}}.
\end{align*}

Let $i\in\{1,\ldots,p-1\}$. Since $a$ is a simple closed curve and components of $(\mO\setminus B)^c$ are pairs of pants, $\{b_{i},s_{i}\}$ and $\{s_i,b_{i+1}\}$ generate the respective fundamental group of a component of $\mO\setminus B$ and therefore $\langle \mu(b_{i}),\mu(s_{i})\rangle$ and also $\langle \mu(s_i),\mu(b_{i+1})\rangle$ are non-virtually abelian. Moreover by the same arguments as before we can assume that $s_0=xb_1^k$, $s_p=b_p^{k'}x$ and $\langle \mu(x^l),\mu(b_1)\rangle$ and $\langle \mu(b_p),\mu(x^l)\rangle$ are non-virtually abelian for all $l\geq 1$ such that $x^l\neq 1$. 
\begin{figure}[htbp]
\centering
\includegraphics{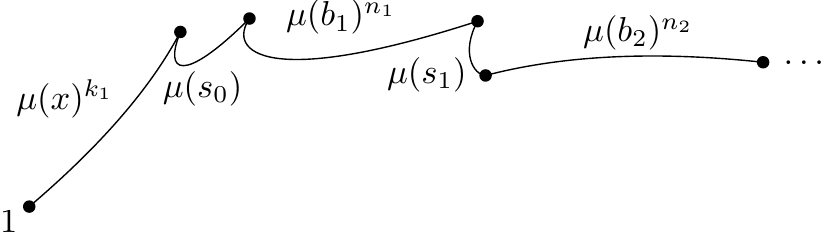}
\caption{The path corresponding to $\mu\circ\alpha(w)$ in $\Cay(\Gamma)$}
\end{figure}
Therefore it follows that there is no essential backtracking in the path corresponding to $\mu\circ\alpha(w)$ in $\Cay(\Gamma)$ and hence $\mu\circ\alpha(w)$ is non-trivial for large enough powers of the Dehn-twists along the set of simple closed curves $B$.\\

So for the last case assume that no boundary component of $\mO^c_v$ is a boundary component of $\mO$. In particular every boundary component of $\mO^c_v$ corresponds to one of the curves in $A$ and $A_v=\pi_1(\mO_v)\cong F_2$. Without loss of generality let $a_1$ and $a_2$ be the curves corresponding to two of the boundary components of $\mO_v$. Then clearly $A_v=\langle a_1,a_2\rangle$.
Let $$a_1=[s_0,e_1,s_1,e_2,s_2,\ldots,e_p,s_p]$$ be a $\B$-path in normal form for $a_1$ and
$$a_2=[r_0,\bar{e}_1,r_1,\bar{e}_2,r_2,\ldots,\bar{e}_q,r_q]$$ be a $\B$-path in normal form for $a_2$. Then 
$$\mu\circ\alpha(a_1)=\mu(s_0)\mu(b_{i_1})^{n_{i_1}}\mu(s_1)\mu(b_{i_2})^{n_{i_2}}\mu(s_2)\ldots\mu(b_{i_p})^{n_{i_p}}\mu(s_p)$$
and 
$$\mu\circ\alpha(a_2)=\mu(r_0)\mu(b_{j_1})^{n_{j_1}}\mu(r_1)\mu(b_{j_2})^{n_{j_2}}\mu(r_2)\ldots\mu(b_{j_q})^{n_{j_q}}\mu(r_q),$$
where $i_1,\ldots,i_p,j_1,\ldots,j_q\in\{1,\ldots,m\}$. 
Let $$w=a_1^{k_1}a_2^{l_1}\ldots a_1^{k_p}a_2^{l_p}$$ be a reduced element in $F(a_1,a_2)\cong A_v$. Then 
\begin{align*}
\mu\circ\alpha(w)&=\Big( \mu(s_0)\mu(b_{i_1})^{n_{i_1}}\mu(s_1)\mu(b_{i_2})^{n_{i_2}}\mu(s_2)\ldots\mu(b_{i_p})^{n_{i_p}}\mu(s_p)\Big)^{k_1}\cdot\\
&\Big(\mu(r_0)\mu(b_{j_1})^{n_{j_1}}\mu(r_1)\mu(b_{j_2})^{n_{j_2}}\mu(r_2)\ldots\mu(b_{j_q})^{n_{j_q}}\mu(r_q)\Big)^{l_1}\cdot\ldots\cdot\\
&\Big(\mu(r_0)\mu(b_{j_1})^{n_{j_1}}\mu(r_1)\mu(b_{j_2})^{n_{j_2}}\mu(r_2)\ldots\mu(b_{j_q})^{n_{j_q}}\mu(r_q)\Big)^{l_p}.
\end{align*}
\begin{figure}[htbp]
\centering
\includegraphics{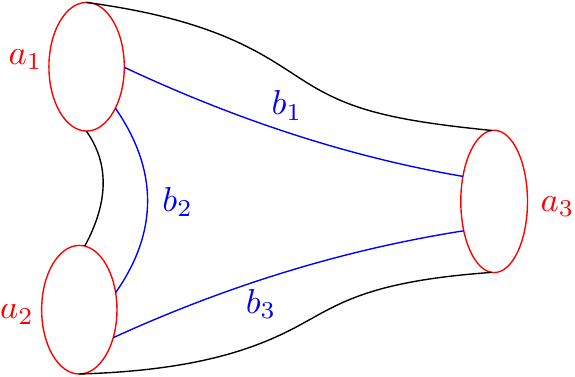}
\caption{All boundary components of $\mO_v^c$ correspond to curves in $A$}
\end{figure}
Since $a_1$ is a simple closed curve, $\{b_{i_j}, s_j\}$ is a generating set of the fundamental group of a component of $\mO\setminus B$ and hence it follows in particular that $\langle \mu(b_{i_j}), \mu(s_j)\rangle$ is non-virtually abelian for all $i\in\{1,\ldots,p\}$. By the same arguments $\langle \mu(s_j), \mu(b_{i_{j+1}})\rangle$, $\langle\mu(b_{j_k}),\mu(r_k)\rangle$ and $\langle \mu(r_k),\mu(b_{j_{k+1}})\rangle$ are also non-virtually abelian.\\
In addition since $a_1\neq a_2$ there exists $k\geq 1$ such that $e_k\neq \bar{e}_k$. The corresponding elements $b_{i_k}$ and $b_{j_k}$ are two different boundary curves of a pair of pants which is a component of $\mO\setminus B$. 
It follows in particular that $\langle \mu(b_{i_k}), \mu(b_{j_k})\rangle$ is non-virtually abelian.\\
Hence for appropriately chosen powers of the Dehn-twists along the set of simple closed curves $B$ the path in $\Cay(\Gamma)$ corresponding to $\mu\circ\alpha(w)$ contains no essential backtracking and therefore $\mu\circ\alpha$ is injective when restricted to $A_v$.\\   

Hence in all cases the fundamental group of every component of $\mathcal{O}\setminus A$ gets mapped isomorphically onto a subgroup of $\Gamma$ by $\mu\circ\alpha$, where $\alpha$ is the product of sufficiently high powers of Dehn-twists along the curves $b_1,\ldots,b_m\in B$. By possibly further increasing the powers of the Dehn-twists we can furthermore guarantee that the images of the fundamental groups of components of $\mathcal{O}\setminus A$ are quasi-convex subgroups of $\Gamma$.\\
We have therefore completed the proof of Proposition \ref{curves}. 
\end{proof}

Now we are ready to proof Theorem \ref{orbifoldflat}.

\begin{proof}[Proof (Theorem \ref{orbifoldflat})] We first assume that $Q$ is just the fundamental group of the cone-type orbifold and deal with the general case later.\\
Applying Proposition \ref{curves} to the homomorphisms $\Theta_n: Q\to \Gamma$ yields automorphisms $\alpha_n\in\Aut(Q)$ and collections of simple closed curves $A=\{a_1,\ldots,a_q\}$ and $B=\{b_1,\ldots,b_t\}$ such that $\mu_n:=\Theta_n\circ \alpha_n$ satisfies the conclusion of Proposition \ref{curves}.
We start by introducing some notation.\\
Let $\vp_1,\ldots,\vp_{q}$ be the automorphisms of $Q$ that correspond to Dehn-twists along the simple closed curves $a_1,\ldots,a_{q}$, and let $\Psi_1,\ldots,\Psi_{t}$ be the automorphisms of $Q$ corresponding to Dehn-twists along $b_1,\ldots,b_t$.
We define the following sequences of automorphisms of $Q$ iteratively:\\
For all $n\in\N$ let $l^{n}\in \N^{t}, m^{n}\in \N^{q}$ be integer tuples. We set 
\begin{itemize}
\item $\tau_1:=\id$ and 
\item $\nu_1:=(\Psi_1)^{l_1^{1}}\circ (\Psi_2)^{l_2^{1}}\circ\ldots\circ (\Psi_{t})^{l_{t}^{1}}$.
\end{itemize}
For $n>1$ we define:
\begin{itemize}
\item $\tau_n=(\vp_1)^{m_1^{n}}\circ (\vp_2)^{m_2^{n}}\circ\ldots\circ (\vp_{q})^{m_{q}^{n}}\circ\nu_{n-1}$ and 
\item  $\nu_n:=(\Psi_1)^{l_1^{n}}\circ (\Psi_2)^{l_2^{n}}\circ\ldots\circ (\Psi_{t})^{l_{t}^{n}}\circ\tau_n$.
\end{itemize}
Let $\{y_1,\ldots,y_s\}$ be a fixed finite generating set of $Q$.
Let $X$ be the Cayley graph of $\Gamma$ with respect to some fixed finite generating set and  $Y$ the Cayley graph of $Q$ with respect to $\{y_1,\ldots,y_s\}$. Let $(T_A,t_A)$ be the Bass-Serre tree corresponding to the decomposition $\A$ of $Q$ along the collection of simple closed curves $A$ with base vertex $t_A$, and let $(T_B,t_B)$ be the Bass-Serre tree corresponding to the decomposition $\B$ of $Q$ along the collection of simple closed curves $B$. We denote by $d_X, d_Y, d_{T_A}$ and $d_{T_B}$ the natural simplicial metric on $X, Y, T_A$ and $T_B$ respectively.\\
For every element $g\in Q$ we set 
$$l_A(g)=d_{T_A}(gt_A,t_A),\quad l_B(g)=d_{T_B}(gt_B,t_B).$$
If $g$ acts hyperbolically on $T_A$ we denote by $\tr_A(g)$ the translation length of the action of $g$ on $T_A$, and similarly if $g$ acts hyperbolically on $T_B$ we denote its translation length by $\tr_B(g)$. For an element $f\in \Gamma$, let $\tr(f)$ be the translation length of $f$ under its action on $X$. Let
$$(a_0^{(y_i)},e_1,a_1^{(y_i)},e_2,\ldots,e_{l(y_i)},a_{l(y_i)}^{(y_i)})$$
be a $\B$-path in normal form for $y_i$, i.e. $[(a_0^{(y_i)},e_1,a_1^{(y_i)},e_2,\ldots,e_{l(y_i)},a_{l(y_i)}^{(y_i)})]=y_i$, that is the corresponding element in $\pi_1(\B)$ is equal to $y_i$. Let 
$$P^{\nu_1}=\left \{[(a_0^{(y_i)},e_1,a_1^{(y_i)},e_2,\ldots,e_j,a_{j}^{(y_i)})]\ |\ i\in\{1,\ldots s\},j\in\{1,\ldots,l(y_i)\}\right \}$$
be the set of all prefixes of the normal forms of the generators $y_i$. We set  $$R^{\tau_1}:=1$$ and $$R^{\nu_1}:=2\cdot \max_{u\in P^{\nu_1}}|u|_Y.$$ 
Moreover we define the constants $R^{\tau_n}$ and $R^{\nu_n}$ for $n\geq 2$ iteratively. For each $n\geq 2$ and $g\in Q$ for which $|g|_Y\leq R^{\nu_{n-1}}$ let 
$$(a_0^{(\nu_{n-1}(g))},e_1,a_1^{(\nu_{n-1}(g))},e_2,\ldots,e_{l(\nu_{n-1}(g))},a_{l(\nu_{n-1}(g))}^{(\nu_{n-1}(g))})$$
be a $\A$-path in normal form for $\nu_{n-1}(g)$. Let 
$$P^{\tau_n}=\left\{[(a_0^{(\nu_{n-1}(g))},e_1,a_1^{(\nu_{n-1}(g))},e_2,\ldots,e_j,a_{j}^{(\nu_{n-1}(g))})]\ |\ |g|\leq R^{\nu_{n-1}},j\in\{1,\ldots,l(\nu_{n-1}(g))\}\right\}$$
be the set of all prefixes of normal forms of elements $g\in Q$ for which $|g|\leq R^{\nu_{n-1}}$. We then define 
$$R^{\tau_n}:=2\cdot \max_{u\in P^{\tau_n}}|\nu_{n-1}^{-1}(u)|_Y.$$
Similarly for each $n\geq 2$ and $g\in Q$ for which $|g|_Y\leq R^{\tau_{n}}$ let 
$$(a_0^{(\tau_n(g))},e_1,a_1^{(\tau_{n}(g))},e_2,\ldots,e_{l(\tau_{n}(g))},a_{l(\tau_n(g))}^{(\tau_n(g))})$$
be a $\B$-path in normal form for $\tau_n(g)$. Let 
$$P^{\nu_n}=\left\{[(a_0^{(\tau_n(g))},e_1,a_1^{(\tau_n(g))},e_2,\ldots,e_j,a_{j}^{(\tau_n(g))})]\ |\ |g|\leq R^{\tau_n},j\in\{1,\ldots,l(\tau_n(g))\}\right\}$$
be the set of all prefixes of normal forms of elements $g\in Q$ for which $|g|\leq R^{\tau_n}$. We then define 
$$R^{\nu_n}:=2\cdot \max_{u\in P^{\nu_n}}|\tau_n^{-1}(u)|_Y.$$
We set for all $n\geq 2$:
\begin{itemize}
\item $HY^{\tau_n}=\{g\in Q\ |\ |g|_{Y}\leq R^{\tau_n} \wedge 0<\tr_A(\nu_{n-1}(g))\}$
\item $NF^{\tau_n}=\{g\in Q\ |\ |g|_{Y}\leq R^{\tau_n} \wedge 0<l_A(\nu_{n-1}(g))\}$
\end{itemize}
and for all $n\geq 1$:
\begin{itemize}
\item $HY^{\nu_n}=\{g\in Q\ |\ |g|_{Y}\leq R^{\nu_n} \wedge 0<\tr_B(\tau_n(g))\}$
\item $NF^{\nu_n}=\{g\in Q\ |\ |g|_{Y}\leq R^{\nu_n}\wedge 0<l_B(\tau_n(g))\}$.
\end{itemize}


\begin{lemma}\label{quadratic}
We keep the notation from above. Let $\lambda_Q^{(n)}:=\mu_n\circ \vp_1^{e_n}\circ\ldots\circ\vp_q^{e_n}\circ\nu_n$. Then for all $n\in\N$ the integer tuples $l^n\in \N^t$, $m^n\in \N^q$ and $e_n\in\N$ can be chosen such that $(\lambda_Q^{(n)},\tau_n,\nu_n)_{n\in\N}$ satisfy the following conditions:\\
\\
To simplify notation, for every index $n$ and every $g\in NF^{\tau_n}$ let
$$\nu_{n-1}(g)=a^{(\nu_{n-1}(g))}_0a^{(\nu_{n-1}(g))}_1\cdots a^{(\nu_{n-1}(g))}_{l(\nu_{n-1}(g))}$$
be the normal form of $\nu_{n-1}(g)$ with respect to $\A$. For every $h\in NF^{\nu_n}$ let 
$$\tau_n(h)=a^{(\tau_n(h))}_0\cdots a^{(\tau_n(h))}_{l(\tau_n(h))}$$
be the normal form of $\tau_n(h)$ with respect to $\B$.\\
\begin{enumerate}[(1)]
	\item For $n>1$ and every $a_i$, $1\leq i\leq q$:
	$$\tr_B((a_i)^{m_i^n})>100\cdot 2^n\cdot \max_{1\leq i\leq q}l_B(a_i)\cdot \sum_{|g|_Y\leq R^{\tau_n}, j\leq l(\nu_{n-1}(g))}l_B(a_j^{(\nu_{n-1}(g))})$$
	\item For $n>1$ and every $b_i$, $1\leq i\leq t$:
	$$\tr_A((b_i)^{l_i^n})>100\cdot 2^n\cdot \max_{1\leq i\leq t}l_A(b_i)\cdot \sum_{|h|_Y\leq R^{\nu_n}, j\leq ln(\tau_n(h))}l_A(a_j^{(\tau_{n}(h))})$$
	\item For every $n>1$ and every $g,g'\in NF^{\tau_n}$:
	$$\left|\frac{l_B(\tau_n(g))l_A(\nu_{n-1}(g'))}{l_B(\tau_n(g'))l_A(\nu_{n-1}(g))}-1\right|<\frac{1}{100\cdot q\cdot 2^n}$$
	\item For every $n>1$ and every $g\in HY^{\tau_n}$:
	$$\left|\frac{\tr_B(\tau_n(g))l_A(\nu_{n-1}(g))}{l_B(\tau_n(g))\tr_A(\nu_{n-1}(g))}-1\right|<\frac{1}{100\cdot q\cdot 2^n}$$
	\item For every $n\geq 1$ and every $h,h'\in NF^{\nu_n}$:
	$$\left|\frac{l_A(\nu_n(h))l_B(\tau_{n}(h'))}{l_A(\nu_n(h'))l_B(\tau_{n}(h))}-1\right|<\frac{1}{100\cdot q\cdot 2^n}$$
	\item For every $n\geq 1$ and every $h\in HY^{\nu_n}$:
	$$\left|\frac{\tr_A(\nu_n(h))l_B(\tau_{n}(h))}{l_A(\nu_n(h))\tr_B(\tau_{n}(h))}-1\right|<\frac{1}{100\cdot q\cdot 2^n}$$
	\item There exist constants $c_1,c_2>0$, so that for every $n\geq 1$ and every $h,h'\in NF^{\nu_n}$:
	$$c_1<\frac{l_A(\nu_n(h))d_X(\lambda_Q^{(n)}(h'),1)}{d_X(\lambda_Q^{(n)}(h),1)l_A(\nu_n(h'))}<c_2$$
	\item There exist constants $c_3,c_4>0$, so that for every $n\geq 1$ and every $h\in HY^{\nu_n}$:
	$$c_3<\frac{\tr_A(\nu_n(h))d_X(\lambda_Q^{(n)}(h),1)}{\tr_X(\lambda_Q^{(n)}(h))l_A(\nu_n(h))}<c_4$$
	\item For every index $n$, the homomorphism $\lambda^{(n)}_Q: Q\to \Gamma$ cannot be extended to an homomorphism $\gamma_n: Q_1\to \Gamma$, such that $\lambda^{(n)}_Q=\gamma_n\circ \iota$ where $\iota: Q\to Q_1$ is an embedding of $Q$ into the fundamental group of a cone-type orbifold $\mathcal{O}_1$, finitely covered by the orbifold $\mathcal{O}$.
\end{enumerate}
Moreover $\lambda_Q^{(n)}$ dominates the growth of $\lambda^H_n$.
\end{lemma}

\begin{proof}

Recall that $\A$ is the graph of groups obtained by decomposing $Q$ along the cyclic groups corresponding to the simple closed curves $a_1,\ldots,a_q$.\\
$(1)-(6)$ For the proof of the inequalities $(1)-(6)$ we refer the reader to the proof of Claim 4.7 of \cite{rips}.\\ 
$(7)$ We claim that given any finite set of elements $M=\{m_1,\ldots, m_k\}\subset Q$, there exist constants $c_1,c_2>0$ and some constant $e\in\N$ so that if $e_n>e$, for every $m_1,m_2\in M$ for which both $\nu_n(m_1)$ and $\nu_n(m_2)$ do not fix the point $t_A\in T_A$, i.e. $l_A(\nu_n(m_i))>0$ for $i=1,2$, the following holds
$$c_1<\frac{l_A(\nu_n(m_1))d_X(\lambda_Q^{(n)}(m_2),1)}{d_X(\lambda_Q^{(n)}(m_1),1)l_A(\nu_n(m_2))}<c_2.$$
To see this choose $e>0$ large enough, such that for $e_n>e$ for all $j\in\{1,\ldots,k\}$:
 $$d_X(\lambda_Q^{(n)}(m_j),1)=2e_n(\sum_{i=1}^{l_A(\nu_n(m_j))}\mu_n(a_{p_i}))+\epsilon_j$$
and $|\epsilon_j|<\frac{1}{1000} e_n$. Then for $e_n>e$ the following hold:
\begin{align*}
&d_X(\lambda_Q^{(n)}(m_j),1)< 2\cdot e_n\cdot l_A(\nu_n(m_j))\cdot \max_{i\in \{1,\ldots,q\}}\mu_n(a_i)+\epsilon_j\\
&d_X(\lambda_Q^{(n)}(m_j),1)> 2\cdot e_n\cdot l_A(\nu_n(m_j))\cdot \min_{i\in \{1,\ldots,q\}}\mu_n(a_i)+\epsilon_j
\end{align*}
Since $\epsilon_j<\frac{1}{1000} e_n$ for all $j\in\{1,\ldots,k\}$
\begin{align*}
&d_X(\lambda_Q^{(n)}(m_j),1)< 3\cdot e_n\cdot l_A(\nu_n(m_j))\cdot \max_{i\in \{1,\ldots,q\}}\mu_n(a_i)\\
&d_X(\lambda_Q^{(n)}(m_j),1)> e_n\cdot l_A(\nu_n(m_j))\cdot \min_{i\in \{1,\ldots,q\}}\mu_n(a_i)
\end{align*}
Applying both inequalities to 
$$\frac{l_A(\nu_n(m_1))d_X(\lambda_Q^{(n)}(m_2),1)}{d_X(\lambda_Q^{(n)}(m_1),1)l_A(\nu_n(m_2))}$$
yields the desired upper and lower bounds $c_1$ and $c_2$ and inequality (7) follows, since $c_1$ and $c_2$ are independent of $n$ and the finite set $M$. Note that for all $n\geq 1$ the sets $NF^{\nu_n}$ and $HY^{\nu_n}$ are finite sets.\\
$(8)$ This inequality follows similar to $(7)$.\\
$(9)$ For the rest of the proof we fix some $n>1$ and the sequence of automorphisms $\{\tau_n,\nu_n\}$. Suppose that there exists an increasing sequence of exponents $(e_j)_{j\in\N}$ for which the homomorphisms 
$$\lambda^j:=\mu_n\circ\vp_1^{e_j}\circ\cdots\circ\vp_q^{e_j}\circ\nu_n$$
can be extended to the fundamental group $Q_1$ of a cone-type orbifold $\mathcal{O}_1$ finitely covered by $\mathcal{O}$. Since up to isomorphism there are only finitely many possibilities for a couple $(Q,Q_1)$, where $Q\leq Q_1$ and $Q_1$ is the fundamental group of a cone-type orbifold finitely covered by $\mathcal{O}$, after passing to a subsequence we may assume that all the homomorphisms $\{\lambda^j\}$  can be extended to the fundamental group $Q_1$ of a fixed orbifold $\mathcal{O}_1$ finitely covered by $\mathcal{O}$. So each homomorphism $\lambda^j$ can be extended to a homomorphism $\gamma_j: Q_1\to \Gamma$, such that $\lambda^j=\gamma_j\circ\iota$ where $\iota:Q\to Q_1$ is the canonical embedding.\\
By Theorem \ref{paulin} the sequence $(\gamma^j)$ converges into an action of $Q_1$ on a limit real tree $(T,t_0)$.
Since each homomorphism $\gamma^j$ is an extension of the homomorphism $\lambda^j: Q\to \Gamma$, and since by the structure of the sequence of homomorphisms $(\lambda^j)$ the finite index subgroup $Q<Q_1$ acts discretely on the limit tree $T$, the group $Q_1$ also acts discretely on $T$, i.e. $T$ is a simplicial tree.\\
Let $c_i:=\nu_n^{-1}(a_i)$ for all $i\in\{1,\ldots,q\}$. Since $\nu_n$ is an automorphism of $Q$, $c_1,\ldots,c_m$ are simple closed curves which induce a splitting of $Q$ along cyclic subgroups, denoted by $\C$. Clearly every edge stabilizer in the action of $Q$ on $T$ is conjugate to a cyclic subgroup generated by one of the $c_i$'s, hence each edge stabilizer in the action of $Q_1$ on $T$ is conjugate to a cyclic subgroup containing the element corresponding to $c_i$ in the orbifold fundamental group $Q_1$.\\
Let $\A_{Q_1}$ be the graph of groups corresponding to the action of $Q_1$ on $T$. By the structure of the homomorphisms $\lambda^j$, every vertex group of $\C$ fixes a point in $T$, hence is contained in a vertex group of $\A_{Q_1}$.\\
We denote the covering map corresponding to the embedding $\iota:Q\to Q_1$ by $\pi:\mathcal{O}\to \mathcal{O}_1$. Since $Q_1$ is the fundamental group of the orbifold $\mathcal{O}_1$ and since $\A_{Q_1}$ is a graph of groups with cyclic edge stabilizers, each vertex group in $\A_{Q_1}$ is isomorphic to the fundamental group of a suborbifold of $\mathcal{O}_1$. Let $\hat{\mathcal{O}}$ be a connected component of 
$\mathcal{O}\setminus\bigcup\{c_1,\ldots,c_m\}$ and denote by $\hat{\mO}^c$ the complement of the cone points of $\hat{\mO}$. By construction $\chi(\hat{\mO}^c)=-1$, and $\hat{Q}:=\pi_1(\hat{\mathcal{O}})$ is conjugate to a vertex stabilizer in $\C$. Moreover $\hat{Q}$ fixes a point in 
$T$ so it can be conjugated into a vertex group of $\A_{Q_1}$. Let $\hat{Q}_1$ be this vertex group and let $\hat{\mathcal{O}}_1$ be the suborbifold of $\mathcal{O}_1$ for which 
$\hat{Q}_1=\pi_1(\hat{\mathcal{O}}_1)$. The covering map $\pi:\mathcal{O}\to \mathcal{O}_1$ maps the suborbifold $\hat{\mathcal{O}}$ into the suborbifold $\hat{\mathcal{O}}_1$, and the boundary of $\hat{\mathcal{O}}^c$ is mapped to the boundary of $\hat{\mathcal{O}}^c_1$. Since in addition $\chi(\hat{\mO}^c)=-1$, then, necessarily, $\pi$ maps $\hat{\mathcal{O}}$ homeomorphically onto $\hat{\mathcal{O}}_1$.\\
Since the restriction of $\pi$ to each suborbifold $\hat{\mathcal{O}}$ is a homeomorphism, and since every edge stabilizer in $\C$ is also contained in some edge stabilizer of $\A_{Q_1}$, the orbifold $\mathcal{O}$ must be a suborbifold of the orbifold $\mathcal{O}_1$, such that $\mO^c$ has the same boundary components as $\mO_1^c$, which clearly implies that $\mathcal{O}=\mathcal{O}_1$. 
Hence $Q=Q_1$, a contradiction.\\
Therefore we can choose an exponent $e_n$ large enough such that the homomorphisms $$\lambda_Q^{(n)}=\mu_n\circ\vp_1^{e_n}\circ\cdots\circ\vp_q^{e_n}\circ\nu_n$$
satisfy properties $(1)-(9)$ and moreover $\lambda_Q^{(n)}$ dominates the growth of $\lambda^H_n$.
\end{proof}

Up to now we have only considered the case that $Q$ is the fundamental group of a cone-type orbifold. So from now on we assume that $Q$ is an extension of a finite group $E$ by the fundamental group $\pi_1(\mO)$ of a cone-type orbifold $\mO$, i.e. there exists a short exact sequence
$$1\to E\to Q\to \pi_1(\mO)\to 1.$$
If $\pi: Q\to \pi_1(\mO)$ is the canonical projection, then we say that an automorphism $\alpha\in \Aut(\pi_1(\mO))$ lifts to $Q$ if there exists  $\tilde{\alpha}\in \Aut(Q)$ such that $\pi\circ\tilde{\alpha}=\alpha\circ\pi$.\\
Let $\mathcal{H}$ be the collection of peripheral subgroups of $\pi_1(\mO)$ and $\tilde{\mathcal{H}}$ the corresponding collection of peripheral subgroups of $Q$,
 i.e. groups of the form $\pi^{-1}(Z)$, where $Z$ corresponds to a boundary component of $\mO$. Let $S\leq \Aut_{\mathcal{H}}(\pi_1(\mO))$ be the subgroup of those automorphisms that lift to $\Aut_{\tilde{\mathcal{H}}}(Q)$. Then, by Lemma 4.17 in \cite{rewe}, $S$ has finite index in $\Aut_{\mathcal{H}}(\pi_1(\mO))$.\\
Hence there exists subsequences of the sequences of automorphisms
$(\tau_n), (\nu_n)$ (still denoted $(\tau_n), (\nu_n)$), such that all $\tau_n$ and $\nu_n$ are in the same left coset $C:=\beta \circ S$ of $S$. Therefore the automorphisms of the sequences $(\beta^{-1}\circ\tau_n)$ and $(\beta^{-1}\circ\nu_n)$ lie in $S$ and hence lift from $\pi_1(\mO)$ to $Q$ and therefore the properties of Lemma \ref{quadratic} also hold for $Q$ in the general case. We define the sequence $(\lambda_n)\subset\Hom(G,\Gamma)$ by $\lambda_n|_H=\lambda_n^H$ and $\lambda_n|_Q=\lambda_Q^{(n)}$ for all $n\in\N$. It remains to show that the properties from Lemma \ref{quadratic} imply that $(\lambda_n)$ is a test sequence for $G$.\\

We briefly recall our setup. $G$ is a $\Gamma$-limit group which has the structure of an orbifold flat over $H$ and $(\lambda_n)\subset\Hom(G,\Gamma)$ is a sequence, such that $(\lambda_n|_Q)$ fulfills the properties of Lemma \ref{quadratic}, where $Q$ is the extension of a finite group by the fundamental group of a cone-type orbifold $\mO$ corresponding to the orbifold flat. Let $\tilde{G}$ be a group which contains $G$ as a subgroup and in abuse of notation let moreover $\lambda_n:\tilde{G}\to \Gamma$ be an extension of $\lambda_n$ for all $n\in\N$. Then (a subsequence of) $(\lambda_n)$ converges into the action of a $\Gamma$-limit group $L$ on some real tree $T$, $G\leq L$ and we assume that $G$ does not act with a global fixpoint on $T$. Assume moreover that $L$ is one-ended relative $G$.\\
Let $\B$ be the graph of groups corresponding to the orbifold flat, i.e. the following hold: 
\begin{itemize}
\item There exists a vertex $v\in VB$ with a QH subgroup $Q=B_v$ as vertex group, which fits into the short exact sequence 
$$1\to E\to Q\to \pi_1(\mO)\to 1.$$
\item Every edge $e\in EB$ is of the form $e=(v,u)$ for some $u\in VB\setminus \{v\}$ and for every edge $e$  with $\alpha(e)=v$, either there exists a peripheral subgroup $O_e$ of $\pi_1(\mathcal{O})$ such that $\alpha_e(B_e)$ is in $B_v$ conjugate to a finite index subgroup of $\pi^{-1}(O_e)$, where $\pi:B_v\to\pi_1(\mathcal{O})$ denotes the canonical projection, or $B_e$ is finite. Moreover every boundary component of $\mO$ corresponds to an edge in such a way.
\item $H$ admits a decomposition as a graph of groups along finite subgroups, such that the vertex groups of this decomposition are precisely the vertex groups of the vertices in $VB\setminus\{v\}$.
\end{itemize}
Since $\lambda_n|_Q$ dominates the growth of $\lambda_n|_H$, every vertex group of $\B$ except for $Q$ acts elliptically on $T$. Since $G$ and therefore also $L$ is not necessarily one-ended relative to its elliptic subgroups we can not simply apply the Rips machine to analyze the action of $L$ on $T$. In order to still be able to get information about the action of $L$ on $T$ we first prove the following:

\begin{lemma}\label{freeaction}
Keeping the notation from before, let $Q$ be a finite extension of the fundamental group of some cone-type orbifold $\mO$, i.e. there exists a finite group $E$ and a short exact sequence
$$1\to E\to Q\to\pi_1(\mO)\to 1.$$
Assume that sequences of automorphisms of $Q$, $(\nu_n),(\tau_n)$ and homomorphisms $\lambda_n: Q\to \Gamma$ satisfy the properties of Lemma \ref{quadratic}. Let $g,g'\in Q$ be elements of infinite order, which do not correspond to boundary components of $\mO$, and suppose that $$\max(d_Y(g,1),d_Y(g',1))\leq R^{\nu_s}$$ for some $s\in\N$. Then there exist constants $C_{g,g'}^1, C_{g,g'}^2,C_{g}^3,C_{g}^4>0$ (depending on $g,g'$) so that for every index $n>s+1$ the following holds:
\begin{align}
	C_{g,g'}^1&<\frac{d_X(\lambda_n(g),1)}{d_X(\lambda_n(g'),1)}<C_{g,g'}^2\\
	C_{g}^3&<\frac{\tr_X(\lambda_n(g))}{d_X(\lambda_n(g),1)}<C_{g}^4\\
	C_{g,g'}^1&<\frac{\tr_X(\lambda_n(g))}{\tr_X(\lambda_n(g'))}<C_{g,g'}^2
\end{align}
\end{lemma}

\begin{proof}
If $\max(d_Y(g,1),d_Y(g',1))\leq R^{\nu_s}$ then $$\max(d_Y(g,1),d_Y(g',1))\leq \min(R^{\nu_n},R^{\tau_n})$$ for every $n>s$. Let $\mO^c$ be the complement of the cone points of the orbifold $\mO$. Since the curves $a_1,\ldots,a_q,b_1,\ldots,b_t$ were chosen to fill the surface $\mO^c$, every  element $u\in \pi_1(\mO)$ of infinite order, which does not correspond to a boundary component, acts hyperbolically on either the Bass-Serre tree $T_A$ or the Bass-Serre tree $T_B$. Since in addition $\max(d_Y(g,1),d_Y(g',1))\leq \min(R^{\tau_n},R^{\nu_n})$ for every $n>s$, each of the elements $g,g'$ belongs to either 
$HY^{\tau_{s+1}}$ or $HY^{\nu_{s+1}}$. Therefore the inequalities $(3)-(8)$ in Lemma \ref{quadratic} hold for the couple $(g,g')$ for every $n>s+1$. Hence inequality (7) yields constants $c_1,c_2>0$ such that 
$$c_1<\frac{l_A(\nu_n(g'))d_X(\lambda_n(g),1)}{d_X(\lambda_n(g'),1)l_A(\nu_n(g))}<c_2.$$
It follows in particular that
\begin{align}\label{stern}
\frac{d_X(\lambda_n(g),1)}{d_X(\lambda_n(g'),1)}<c_2\cdot\frac{l_A(\nu_n(g))}{l_A(\nu_n(g'))}.
\end{align}
From (5) and (3) in Lemma \ref{quadratic} follows:
$$\frac{l_A(\nu_n(g))}{l_A(\nu_n(g'))}<\left(1+\frac{1}{100\cdot q\cdot2^n}\right)\cdot \frac{l_B(\tau_n(g))}{l_B(\tau_n(g'))}$$
and 
$$\frac{l_B(\tau_n(g))}{l_B(\tau_n(g'))}<\left(1+\frac{1}{100\cdot q\cdot2^n}\right)\cdot \frac{l_A(\nu_{n-1}(g))}{l_A(\nu_{n-1}(g'))}.$$
If we plug in the above two inequalities in (\ref{stern}) we get:
$$\frac{d_X(\lambda_n(g),1)}{d_X(\lambda_n(g'),1)}<c_2\cdot\left(1+\frac{1}{100\cdot q\cdot2^n}\right)^2\frac{l_A(\nu_{n-1}(g))}{l_A(\nu_{n-1}(g'))}.$$
So by induction:
$$\frac{d_X(\lambda_n(g),1)}{d_X(\lambda_n(g'),1)}<c_2\cdot\left(1+\frac{1}{100\cdot q\cdot2^n}\right)^n\frac{l_A(g)}{l_A(g')}
<2\cdot c_2\frac{l_A(g)}{l_A(g')}$$
which proves the existence of the desired upper bound. In a complete analogous way we get a lower bound and hence the first part of the claim follows.
Plugging inductively the inequalities (4), (6) of Lemma \ref{quadratic} in inequality (8), we get as before the second part of the claim.
Combining the first two inequalities yields the third one and the proof is complete.
\end{proof}

By using Lemma \ref{freeaction} we will prove the following result, which describes the action of $L$ on $T$. As mentioned before, if $L$ is one-ended relative its elliptic subgroups, we could get the following a lot easier by simply applying the relative version of the Rips machine to the action of $L$ on $T$.

\begin{lemma}\label{orbitest2}
$L$ has a decomposition as a graph of groups $\A$ such that the underlying graph of $\A$ is isomorphic to the underlying graph of $\B$. Moreover every edge group of $\A$ is a finite index supergroup of the corresponding edge group of $\B$. And for the QH vertex group $Q$ of $\B$ the corresponding vertex group in $\A$ is a finite extension of the underlying orbifold of $Q$ and contains $Q$ as a subgroup of finite index.
\end{lemma}

\begin{proof}
We are first going to show that every $g\in Q$ of infinite order, which does not correspond to a boundary component of the underlying orbifold acts hyperbolically on $T$. Let $\{y_1,\ldots,y_m\}$ be a generating set for $Q$ and 
$$\mu_n=\max_{i=1,\ldots,m}d_X(1,\lambda_n(y_i))=\max_{i=1,\ldots,n}|\lambda_n(y_i)|_X$$ for all $n\in\N$, where $d_X$ denotes the metric on the Cayleygraph of $\Gamma$ with respect to some fixed generating set $X$.
Since $G$ does not act with a global fixed point on $T$, $\mu_n\geq 1$. Assume without loss of generality that $\mu_n=|\lambda_n(y_1)|_X$. Suppose for the sake of contradiction that  there exists some $g\in Q\setminus E$, not corresponding to a boundary component or a cone point, that acts elliptically on $T$. Then there exists some point $t$ in the real tree $T$ which is fixed by $g$. Let $(t_n)_{n\in\N}\in \Cay(\Gamma,X)^{\N}$ be an approximating sequence for $t$. Recall that we denote the scaled Cayleygraph by $X_n=(\Cay(\Gamma,X),d_n)$ where $d_n=\frac{d_X}{\mu_n}$. Then 
\begin{align*}
0=d_T(t,gt)&=\lim_{n\to\infty}d_n(t_n,\lambda_n(g)t_n)\\
&=\lim_{n\to\infty}d_n(1,t_n^{-1}\lambda_n(g)t_n)\\
&=\lim_{n\to\infty}\frac{d_X(1,t_n^{-1}\lambda_n(g)t_n)}{\mu_n}\\
&\geq\lim_{n\to\infty}\frac{\tr_X(\lambda_n(g))}{d_X(1,\lambda_n(y_1))}\\
&\overset{\ref{freeaction}(2)}{>}\lim_{n\to\infty}\frac{1}{C_{y_1}^3}\cdot\frac{\tr_X(\lambda_n(g))}{\tr_X(\lambda_n(y_1))}\\
&\overset{\ref{freeaction}(3)}{>}\frac{C_{g,y_1}^1}{C_{y_1}^3}>0,
\end{align*}
a contradiction.\\

Let $H_1,\ldots,H_k$ be the vertex groups of $\B$ which do not correspond to the QH vertex. Since $\lambda_n|_Q$ dominates the growth of $\lambda_n|_{H_i}$ for all $i\in\{1,\ldots,k\}$, $H_i$ fixes a point in the real tree $T$. Let $\mathcal{H}=\{H_1,\ldots,H_k\}$. 
Then by Theorem \ref{relativegeometric} the pair action $(L,\mathcal{H})\curvearrowright T$ is the strong limit of a direct system of geometric pair actions 

\begin{center}
\begin{tikzpicture}[descr/.style={fill=white}]
\matrix(m)[matrix of math nodes,
row sep=1em, column sep=2.8em,
text height=1.5ex, text depth=0.25ex]
{L_k&L_{k+1}&\cdots&L\\
T_k&T_{k+1}&\cdots&T\\};
\path[->>,font=\scriptsize]
(m-1-1) edge node[above] {$\vp_k$}(m-1-2)
(m-1-1) edge [bend left=60] node[descr] {$\Phi_{k}$} (m-1-4)

(m-1-2) edge [bend left=40] node[descr] {$\Phi_{k+1}$} (m-1-4)

(m-2-1) edge node[above] {$f_k$} (m-2-2)
(m-2-1) edge [bend right=60] node[descr] {$F_{k}$} (m-2-4)
(m-2-2) edge [bend right=40] node[descr] {$F_{k+1}$} (m-2-4);
\path[->,font=\scriptsize]
(m-1-1.south) edge [bend left=60](m-2-1.north)
(m-1-2.south) edge [bend left=60](m-2-2.north)
(m-1-4.south) edge [bend left=60](m-2-4.north);
\end{tikzpicture}
\end{center}

such that
\begin{itemize}
\item $\vp_k$ and $\Phi_k$ are one-to-one in restriction to arc stabilizers of $T_k$.
\item $\vp_k$ (resp. $\Phi_k$) restricts to an isomorphism between $\mathcal{H}_k$ and $\mathcal{H}_{k+1}$ (resp. $\mathcal{H}$).
\item The action of $L_k$ on $T_k$ splits as a graph of actions where each non-degenerate vertex action is either axial, thin or of orbifold type and therefore indecomposable, or is an arc containing no branch point except at its endpoints.
\end{itemize}

Since $G$ is finitely presented relative $\mathcal{H}$ there exists $k_0\in\N$ such that for all $k\geq k_0$, $G\leq L_k$. Denote by $T_{\min}$ the minimal $G$-invariant subtree of $T$ and by $T_{\min}^k$ the minimal $G$-invariant subtree of $T_k$ for all $k\geq k_0$. Clearly the action of $G$ on $T_{\min}$ and on $T_{\min}^k$ is minimal. Hence the direct system of pair actions $(G,\mathcal{H})\curvearrowright T_{\min}^k$ converges strongly to $(G,\mathcal{H})\curvearrowright T_{\min}$.\\ 
Denote by $T_{\min}^Q\subset T_{\min}$ the minimal $Q$-invariant subtree. Since every element of $Q$ of infinite order, not corresponding to a boundary component, acts hyperbolically on $T_{\min}^Q$, it follows that the action of $Q$ on $T_{\min}^Q$ is of orbifold type, i.e. there exists a finite normal subgroup $E$ of $Q$ such that the action of $Q/E$ is dual to an arational measured foliation on the underlying orbifold. In particular this action is geometric.\\
By Theorem 2.5 in \cite{levittpaulin} a geometric action of a finitely generated group cannot be a non-trivial strong limit of actions. This implies that there exists $k_1\geq k_0$ such that $F_k$ is an isometry when restricted to the minimal $Q$-invariant subtree of $T_{\min}^k$ for all $k\geq k_1$. In abuse of notation for $k\geq k_1$ we denote this minimal $Q$-invariant subtree of $T_k$ also by $T_{\min}^Q$.\\
Since the action of $L_k$ on $T_k$ splits as a graph of actions $\A_k$ where each non-degenerate vertex action is either axial, thin or of orbifold type and therefore indecomposable, or is an arc containing no branch point except at its endpoints, $T_{\min}^Q$ is contained in an orbifold component $T_{\mO'}$ of this graph of actions.
Denote by $Q':=\stab_{L_k}(T_{\mO'})$ the vertex group of $\A_k$ corresponding to this orbifold component. In particular $Q\leq Q'$.\\
Since $Q'$ is the stabilizer of an orbifold component, there exists a finite subgroup $E'$ of $Q'$ and a short exact sequence
$$1\to E'\to Q'\to \pi_1(\mO')\to 1,$$
where $E'$ is the kernel of the action of $Q'$ on $T_{\mO'}$ and the faithful action of $Q'/E'$ on $T_{\mO'}$ is dual to an arational measured foliation on $\mO'$.
From property $(9)$ of Lemma \ref{quadratic} follows that the underlying orbifolds $\mO$ and $\mO'$ are identical and therefore $Q$ has finite index in $Q'$.\\
It follows that the setwise stabilizer of $T_{\min}^Q\subset T$ is a group (also denoted by) $Q'$ fitting into a short exact sequence $$1\to E'\to Q'\to \pi_1(\mO)\to 1,$$ for some finite group $E'$ containing $E$. In particular $Q$ has finite index in $Q'=\stab_L(T_{\min}^Q)$.\\
We define a family $\mathcal{Z}$ of subtrees of $T$ consisting of:
\begin{itemize}
\item $L$-translates of $T_{\min}^Q$ and 
\item closures of connected components of $T\setminus L\cdot T_{\min}^Q$.
\end{itemize}
Since in any of the approximation trees $T_k$ for $k\geq k_1$, either $gT_{\min}^Q\cap T_{\min}^Q$ is at most one point or $g\in Q'=\stab_{L_k}(T_{\min}^Q)$, the same holds for $L\curvearrowright T$. Hence the family $\mathcal{Z}$ gives rise to an equivariant covering of $T$ such that any two subtrees intersect in at most one point. If we can show that every arc of $T$ is covered by finitely many subtrees of $\mathcal{Z}$ then $\mathcal{Z}$ is precisely what V. Guirardel in \cite{guirardel} calls a transverse covering of the tree $T$. So suppose for the the sake of contradiction that there exists some arc $I\subset T$ which is not covered by finitely many subtrees of $\mathcal{Z}$.
Then there exists $k_2\geq k_1$ and an arc $J\subset F_k^{-1}(I)$ such that $F_k$ maps $J$ isometrically onto $I$ for all $k\geq k_2$. Let now $k\geq k_2$. Then the action of $L_k$ on $T_k$ splits as a graph of actions $\A_k$ and therefore every finite subtree of $T_k$ is covered by finitely many translates of the vertex trees of $\A_k$. In particular there exists a decomposition of $J$ into non-trivial subsegments
$$J=J_1\cup\ldots \cup J_l$$
such that each $J_i$ is contained in a unique translate of a vertex tree of $\A_k$. But this implies that $J$ intersects only finitely many translates of $T_{\min}^Q$ in a segment (i.e. not just in a point). But since $F_k$ is an isometry when restricted to $T_{\min}^Q$, it follows that $I$ also intersects only finitely many translates of $T_{\min}^Q\subset T$ in more than a point. So assume that $I_0\subset I$ is a non-trivial subsegment of $I$ that intersects translates of $T_{\min}^Q$ in at most one point. Then $I_0$ is contained in the closure of a connected component of $T\setminus L\cdot T_{\min}^Q$. Hence we conclude that $I$ is covered by finitely many subtrees of $\mathcal{Z}$, a contradiction.\\
Therefore $\mathcal{Z}$ is a transverse covering of $T$ and it now follows from Lemma 1.5 in \cite{guirardel} that the action of $L$ on $T$ splits as a graph of actions $\A$ such that the vertex trees are precisely the trees contained in $\mathcal{Z}$. Associated to this graph of actions there is a simplicial tree $S$ on which $L$ acts (in \cite{guirardel} this tree is called the skeleton of the graph of actions). For convenience we recall the definition of $S$.
The vertex set $V(S)$ is $V_0(S)\cup V_1(S)$, where $V_1(S)$ contains a vertex $x_Y$ for every subtree $Y$ of $\mathcal{Z}$ and $V_0(S)$ is the set of points $x\in T$ lying in the intersection of two distinct subtrees in $\mathcal{Z}$. There is an edge $e=(x,x_Y)$ between a vertex $x\in V_0(S)$ and a vertex $v_Y\in V_1(S)$ if and only if $x\in Y$.\\
The stabilizer of a vertex $v_Y\in V_1(S)$ is the setwise stabilizer of $Y$, while the stabilizer of a vertex $x\in V_0(S)$ is the stabilizer of the corresponding point $x\in T$. It follows from the discussion above that the vertex $v_{T_{\min}^Q}\in V_1(S)$ is stabilized by $Q'$.\\
Let $e=(x,v_{T_{\min}^Q})$ be an edge in $S$ adjacent to $v_{T_{\min}^Q}$. Then 
$$\stab_L(e)=\stab_L(x)\cap \stab_L(v_{T_{\min}^Q})=\stab_L(x)\cap Q'.$$
Since $L$ is one-ended relative $G$, either $\stab_L(e)$ corresponds to a boundary component of $\mO$ or is a finite group $N$ containing a finite edge group adjacent to $Q$ in $\B$. Moreover every edge group of an edge in $\B$ adjacent to $Q$ appears as a finite index subgroup of an edge group of $\A$ adjacent to $Q'$.
\end{proof}

Lemma \ref{orbitest2} completes the proof of Theorem \ref{orbifoldflat}.
\end{proof}

\subsection{Extensions along finite groups}
We start with the following definition.

\begin{definition}
Let $F_m=\langle a_1,\ldots,a_m\rangle$ be a non-abelian free group.
\begin{enumerate}[(a)]
\item Let $x=(x_1,\ldots,x_k)$, $x_1,\ldots,x_k \in F_m$. We call a word $w\in F_m$ a \textnormal{piece} of $x$ if $w$ is either a subword of $x_i^{\pm 1}$ and $x_j^{\pm 1}$ for $i\neq j\in \{1,\ldots,k\}$, or $w$ appears twice in $x_j^{\pm 1}$ for some $j\in\{1,\ldots,k\}$.
\item Let $0< p< 1$. We say that $x=(x_1,\ldots,x_k)$ satisfies the \textnormal{small cancellation property} $C'(p)$ in $F_m$, if for any piece $w$ of $x$, if $w$ is a subword of say $x_i$, then we have $|w|_{F_m}<p\cdot|x_i|_{F_m}$.
\item Let $(x^{(n)})_{n\in\N}$ be a sequence of tuples $x^{(n)}=(x^{(n)}_1,\ldots,x_k^{(n)})\in F_m^k$. We say that $(x^{(n)})$ satisfies a \textnormal{very small cancellation property} if for all $n\geq 1$, $x^{(n)}$ satisfies $C'(1/n)$ in $F_m$.
\end{enumerate}
\end{definition}

\begin{satz}\label{testamalgamated}
Suppose $G$ has the structure of an amalgamated product $H\ast_E K$ where $E$ is some finite group and $K$ is isomorphic to a one-ended subgroup of $\Gamma$.
Assume that there exists a test sequence $(\lambda^H_n)$ for $H$. Assume moreover that for all $n\in\N$ there exists an embedding $\tau_n:K\to \Gamma$ such that $\lambda^H_n(E)=\tau_n(E)$ and $C_{\Gamma}(\tau_n(E))$ contains a non-abelian free subgroup which is quasi-convex in $\Gamma$. Then there exists a test sequence $(\lambda_n)$ for $G$.
\end{satz}

Before we can start the proof, we need the following result about hyperbolic groups, which is well-known to the experts, so we only sketch a proof here.

\begin{lemma}\label{bla}
Let $\Gamma$ be a hyperbolic group and $E\leq E'\leq \Gamma$ two finite subgroups of $\Gamma$. Suppose that the centralizer $C_{\Gamma}(E)$ contains a quasi-convex non-abelian free group. Then one of the following holds:
\begin{enumerate}[(a)]
\item $|C_{\Gamma}(E):C_{\Gamma}(E')|< \infty$.
\item $C_{\Gamma}(E)\setminus \left(C_{\Gamma}(E')\setminus\{1\}\right)$ contains a quasi-convex non-abelian free group.
\end{enumerate}
\end{lemma}

\begin{proof}
Let $C:=C_{\Gamma}(E)$ and $C':=C_{\Gamma}(E')$ and suppose that $|C:C'|= \infty$. Then there exists a hyperbolic element $w\in C$ such that the diameter of the intersection of the axis $\Ax(w)$ with the convex hull $\CH(\Lambda(C'))$ of the limit set of $C'$ is finite. Moreover there exists by assumption a quasi-convex non-abelian free group, say $F$, of infinite index in $C$. Hence again there exists a hyperbolic element $v\in C$ such that the diameter of the intersection of the axis $\Ax(v)$ with the convex hull $\CH(\Lambda(F))$ of the limit set of $F$ is finite.\\
After possibly conjugating $w$ or $v$, we can assume that the limit points of $\Ax(w)$ and $\Ax(v)$ are disjoint. Hence we can find an element $u\in C$ such that the path corresponding to the element $w^kuv^{-n}$ is a quasi-geodesic for large $k,n\in\N$. Therefore for $k,n$ large enough $(w^kuv^{-n})F(w^kuv^{-n})^{-1}$ is  the desired quasi-convex non-abelian free group.
\end{proof}

\begin{proof}[Proof (Theorem 9.16)]
Denote by $|\cdot |_{\Gamma}$ and $|\cdot |_G$ the respective word metric on $\Gamma$ and $G$ with respect to some fixed finite generating sets. Denote by $X$ the Cayley graph of $\Gamma$ with respect to this generating set. For all $n\in\N$ let $B_H(n)=\{h\in H\ |\ |h|_G\leq n\}$ and 
$$L(n)=\max\{|\lambda^H_n(h)|_{\Gamma}\ |\ h\in B_H(n)\}.$$
Fix $n\in\N$ and denote by $C_n:=C_{\Gamma}(\lambda_n^H(E))$ the centralizer of the image of $E$ under the test sequence for $H$. 
We distinguish two cases:\\
First assume that for every finite group $E'\leq H$, containing $E$, $$C_n\setminus \left(C_{\Gamma}(\lambda^H_n(E'))\setminus\{1\}\right)$$ contains a quasi-convex non-abelian free group. 
Let $F(a,b)$ be this quasi-convex non-abelian free subgroup of $C_n$ which exists by assumption. Let $F=F(x,y)\leq F(a,b)$ be a non-abelian free subgroup which is quasi-convex in $\Gamma$ such that any segment lying in the $10\delta$-neighborhood of both axis $\Ax(a)$ and $\Ax(cbc^{-1})$ in $X$ of $a$ and any conjugate of $b$ by an element $c\in\Gamma$ has length at most
$$\frac{1}{1000}\min\{|x|_{\Gamma},|y|_{\Gamma}\}.$$ This can be achieved by taking $x=a^k$ and $y=b^ka^k$ for a large integer $k$, since it follows from Proposition 4.1 in \cite{weidmann} that there exists a bound on the length of how long the axis of $a$ and $b$ can lie in the $10\delta$-neighborhood of each other. For all $n\geq 1$ we define
$$w_n:=xyxy^2xy^3x\cdots xy^{n\cdot L(n)}x$$
and a homomorphism $\lambda_n: G\to \Gamma$ by
\begin{itemize}
\item $\lambda_n(H)=\lambda_n^H(H)$ and
\item $\lambda_n(K)= w_n\tau_n(K)w_n^{-1}$.
\end{itemize}
Since $w_n\in C_n$ this is a well-defined homomorphism. Moreover by the construction of $w_n$ we have that 
$$n\cdot|\lambda_n(h)|_{\Gamma}\leq n\cdot L(n)< |\lambda_n(k)|_{\Gamma}$$
for all $n\in\N$, $h\in B_H(n)$, $k\in K\setminus E$. Hence $(\lambda_n|_K)$ dominates the growth of $(\lambda_n|_H)=(\lambda^H_n)$. Note that $(w_n)$ satisfies the small cancellation property $C'(1/n)$ for all $n\in\N$.\\
For the second case assume that there exists some maximal finite subgroup $E'\leq H$, containing $E$ such that $$C_n\setminus \left(C_{\Gamma}(\lambda^H_n(E'))\setminus\{1\}\right)$$ does not contain a quasi-convex non-abelian free group.  By Lemma \ref{bla} this implies that $|C_n:C_{\Gamma}(\lambda^H_n(E'))|<\infty$. Let now $(\vp_n)\subset\Hom(E'\ast_EK,\Gamma)$ be a stably injective sequence. Then for $n$ large enough $$\vp_n|_{E'}=c_{g_n}\circ\iota_n\circ\alpha_n$$ for one of finitely many embeddings $\iota_n:E'\to\Gamma$, an automorphism $\alpha_n\in\Aut(E')$ and some element $g_n\in C_n$ (see \cite{rewe}). Since $\Aut(E')$ is finite and $|C_n:C_{\Gamma}(\lambda^H_n(E'))|<\infty$ we can assume that (after passing to a subsequence) $\vp_n|_{E'}=\lambda_n^H$. Hence we can extend the homomorphisms $\lambda_n^H:H\to\Gamma$ and $\vp_n:E'\ast_EK\to\Gamma$ to a homomorphism $$\mu_n:H\ast_{E'}(E'\ast_EK)\to\Gamma.$$
Let $F(a,b)$ be the quasi-convex non-abelian free subgroup of $C_n':=C_{\Gamma}(\lambda_n^H(E'))$ which exists by assumption. Let again $F=F(x,y)\leq F(a,b)$ be a non-abelian free subgroup which is quasi-convex in $\Gamma$ such that any segment lying in the $10\delta$-neighborhood of both axis $\Ax(a)$ and $\Ax(cbc^{-1})$ in $X$ of $a$ and any conjugate of $b$ by an element $c\in\Gamma$ has length at most
$$\frac{1}{1000}\min\{|x|_{\Gamma},|y|_{\Gamma}\}.$$
Let $$L'(n)=\max\{|\mu_n(g)|_{\Gamma}\ |\ g\in H\cup (E'\ast_E K) \text{ with }|g|_G\leq n\}.$$
For all $n\geq 1$ we define
$$w_n:=xyxy^2xy^3x\cdots xy^{n\cdot L'(n)}x$$
and a homomorphism $\lambda_n:G\to\Gamma$ by 
\begin{itemize}
\item $\lambda_n|_H=\mu_n|_H$ and
\item $\lambda_n|_{E'\ast_EK}=c_{w_n}\circ\mu_n|_{E'\ast_EK}$.
\end{itemize}
Since $w_n\in C_n'$ this is again a well-defined homomorphism. Moreover by the construction of $w_n$ we have that 
$$n\cdot|\lambda_n(h)|_{\Gamma}\leq n\cdot L'(n)< |\lambda_n(k)|_{\Gamma}$$
for all $n\in\N$, $h\in B_H(n)$, $k\in (E'\ast_EK)\setminus E'$. Hence $(\lambda_n|_{E'\ast_EK})$ dominates the growth of $(\lambda_n|_H)=(\lambda^H_n)$. Note that $(w_n)$ satisfies the small cancellation property $C'(1/n)$ for all $n\in\N$. For simplicity we assume from now on that we are in the first case from above. The second one is similar.\\

Let $\tilde{G}$ be a group containing $G$ and in abuse of notation let $\lambda_n: \tilde{G}\to\Gamma$ be an extension of $\lambda_n$ which is short relative to $G$ for all $n\in\N$. Let $T$ be the real tree into which (a convergent subsequence of) $(\lambda_n)$ converges and let $L:=\tilde{G}/\underrightarrow{\ker} \lambda_n$. Assume that $L$ is one-ended relative $G$ and moreover that $G$ does not act elliptically on $T$.
Since $G$ does not fix a point when acting on $T$, there exists a non-trivial minimal $G$-invariant subtree $T_{\min}$ of $T$ on which $G$ acts. 

\begin{lemma}\label{intersection}
The following hold:
\begin{enumerate}[(1)]
\item $T_{\min}$ is equivariantly isomorphic as a $G$-tree to the Bass-Serre tree of the splitting $G=H\ast_EK$.
\item For every $g\in L$ either $T_{\min}\cap gT_{\min}$ is at most one point or $g$ is contained in the setwise stabilizer of $T_{\min}$.
\item $T_{\min}$ intersects any indecomposable subtree of $T$ in at most one point.
\item The setwise stabilizer of $T_{\min}$ in $L$ is $\langle H,E_L\rangle\ast_{E_L}\langle K,E_L\rangle$ for some finite group $E_L$, containing $E$.
\item Edges in $T_{\min}$ are stabilized by conjugates of $E_L$ and vertices by conjugates of either $\langle H,E_L\rangle$ or $\langle K,E_L\rangle$.
\end{enumerate}
\end{lemma}

\begin{proof}
By Lemma 1.10(3) in \cite{rewe} for all $k\in K$, $(1)_{n\in\N}$ and $(\lambda_n(k))=(w_n\tau_n(k)w_n^{-1})$ are approximating sequences for $t_0$ and $kt_0$ respectively. Hence there exists a point $w$ on the arc $[t_0,kt_0]$ such that $(w_n)$ is an approximating sequence for $w$. It follows that for all $k\in K$: 
$$d_T(kw,w)=\lim_{n\to\infty} d_n(\lambda_n(k)w_n,w_n)
=\lim_{n\to\infty} d_n(w_n\tau_n(k),w_n)=\lim_{n\to\infty}d_n(\tau_n(k),1).$$
Here $d_n$ denotes the scaled metric on the Cayleygraph of $\Gamma$.
Since $w_n$ dominates the growth of $\tau_n$, $K$ fixes the point $w$ in the limit tree $T$ and therefore acts elliptically on $T$.
Since $\Gamma t_0=t_0$ and $Kw=w$, it follows from Lemma 1.14 in \cite{guirardel} that $T_{\min}$ is covered by translates of the convex hull of the points $\{kt_0\ |\ k\in K\}$, i.e. translates of the arcs $[t_0,kt_0], k\in K$.
The first claim is now obvious. For $(2)-(5)$ we need the following lemma.

\begin{lemma}\label{claim1}
Let $I=[a,b]\subset [t_0,w]$ be an arc of $T_{\min}\subset T$. There exists a finite group $E_L\leq L$ such that for any $g\in L\setminus E_L$, $gI$ intersects $[t_0,w]$ in at most one point. Moreover $E_L$ stabilizes the segment $[1,w]\subset T_{\min}$, where $w$ is the point in $T$ approximated by $(w_n)$.
\end{lemma}

\begin{proof}
We suppose that there exists some $g\in L$ such that $gI\cap [t_0,w]$ is non-trivial and not a point.
 Let $(a_n)_{n\in\N}$ and $(b_n)_{n\in\N}$ be approximating sequences for $a$ and $b$ respectively, and $f\in \tilde{G}$, such that $\eta(f)=g$, where $\eta:\tilde{G}\to L$ is the limit quotient map. Let $u_n\in F(x,y)$ be the unique reduced word such that $b_n=a_nu_n$, i.e. the segment $[a_n,b_n]\subset (X,d_n)$ is labeled by the word $u_n$. Since $I\subset [t_0,w]$, $u_n$ is a subword of $w_n$ (since $(w_n)$ is an approximating sequence of $w$). Clearly the segments $J_n:=[\lambda_n(f)a_n,\lambda_n(f)b_n]$ approximate the segment $gI$ and are labeled by $u_n$.\\
Let $t_{a_n}$ and $t_{b_n}$ be points on the segment $[1,w_n]\leq X$, such that $(t_{a_n})$ and $(t_{b_n})$ are approximating sequences for $ga$ and $gb$ respectively. The subsegment $$I_n:=[t_{a_n},t_{b_n}]\subset [1,w_n]$$ is labeled by some word $z_n\in F(x,y)$ and we want to show that $u_n=z_n$ as words in $F(x,y)$.\\
\begin{figure}[htbp]
\centering
\includegraphics{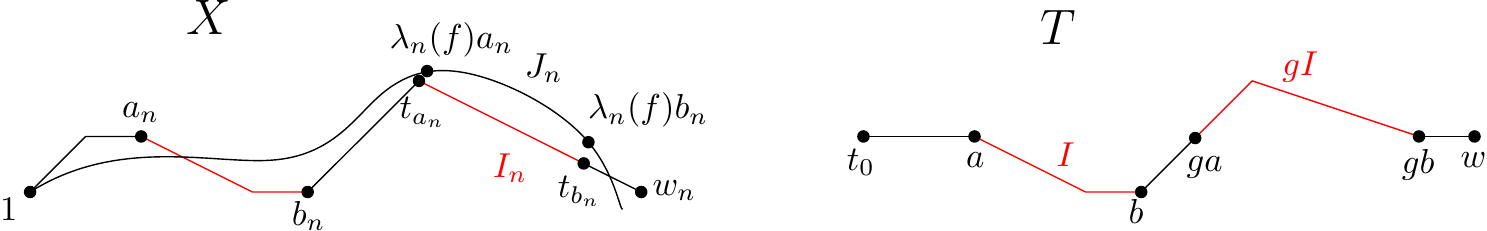}
\caption{Approximating the arc $[1,w]$}
\end{figure}
Since $(\lambda_n(f)a_n)$ and $(t_{a_n})$ are both approximating sequences for $ga$ it follows from Lemma 1.10 in \cite{rewe} that $$\lim_{n\to\infty}d_n(\lambda_n(f)a_n,t_{a_n})=0.$$ From the same argument also follows that 
$$\lim_{n\to\infty}d_n(\lambda_n(f)b_n,t_{b_n})=0.$$
Therefore we can assume that for large $n$ the segment $J_n$ is in the $2\delta$-neighborhood of $I_n$, where $\delta$ is the hyperbolicity constant of $X$. Suppose now that for large $n$ the first letter of $u_n$ is $x$ and the first letter of $z_n$ is $y$. Since $x=a^k$ and $y=b^ka^k$ there exists $c_n\in\Gamma$ such that 
$$\Ax(t_{a_n}at_{a_n}^{-1})\subset N_{10\delta}(\Ax(t_{a_n}c_nb(t_{a_n}c_n)^{-1}))$$
for length at least $\frac{1}{2}\min\{|x|_{\Gamma},|y|_{\Gamma}\}$.
\begin{figure}[htbp]
\centering
\includegraphics{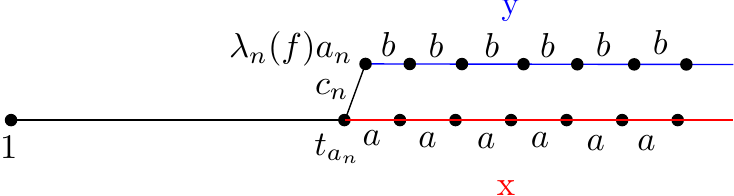}
\caption{$\Ax(t_{a_n}at_{a_n}^{-1})\subset N_{10\delta}(\Ax(t_{a_n}c_nb(t_{a_n}c_n)^{-1}))$}
\end{figure}
But this implies that 
$$\Ax(a)\cap N_{10\delta}(\Ax(c_nbc_n^{-1}))$$
has diameter at least $\frac{1}{2}\min\{|x|_{\Gamma},|y|_{\Gamma}\}$, a contradiction to the assumption that the axis $\Ax(a)$ and $\Ax(cbc^{-1})$ fellow travel for length at most 
$$\frac{1}{1000}\min\{|x|_{\Gamma},|y|_{\Gamma}\}$$
for any $c\in\Gamma$. Essentially the same argument guarantees that the first letter of $z_n$ can also not be $y^{-1}$ or $x^{-1}$. Therefore both $u_n$ and $z_n$ start with the letter $x$ and by induction we conclude that $z_n=u_n$ as words in $F(x,y)$.\\
Note that in the case that $\Gamma$ is a free group and therefore $X$ is a tree, $I_n=J_n$ and the above argument becomes trivial.\\
So either $gI=I$ or $u_n$ appears twice in the label of the segment $[1,w_n]$. In the second case $u_n$ is a piece of $(w_n)$. Since $(w_n)$ satisfies the small cancellation property $C'(1/n)$ it follows that
$$\lim_{n\to\infty} d_n(1,u_n)<\lim_{n\to\infty} \frac{1}{n} d_n(1,w_n)=\lim_{n\to\infty} \frac{1}{n} d_T(t_0,w)=0,$$
a contradiction to the assumption that $gI\cap [t_0,w]$ is not a point.\\
So suppose we are in the case that $gI=I$. Hence $g\in\stab(I)$. If $g\notin \stab([1,w])$ then $[1,w]$ is an unstable arc and therefore $\stab([1,w])$ is finite by Theorem \ref{eigenschaften}. Otherwise $g\in\stab([1,w])$ and $[1,\lambda_n(f)w_n]\subset N_{2\delta}([1,w_n])$ for large $n$. This implies that $\lambda_n(f)$ lies in a ball $B_r(1)$ of radius $r$ around the identity for large $n$ and the diameter of the ball does not depend on $n$. Moreover $d_X(w_n,\lambda_n(f)w_n)$ is also globally bounded for all $n$, implying that $w_n^{-1}\lambda_n(f)w_n\in B_r(1)$ (after possibly enlarging $r$). Hence after passing to a subsequence we can assume that $w_n^{-1}\lambda_n(f)w_n=\lambda_n(f)$, which implies that $w_n\in C_{\Gamma}(\lambda_n(f))$ for $n$ sufficiently large. But by the same arguments every subword of $w_n$ lies in the centralizer of $\lambda_n(f)$. Therefore $\langle \lambda_n(f)\rangle$ is finite. Hence $\langle g\rangle$ is a finite subgroup of $L$. The claim follows.
\end{proof}

First of all it follows from Lemma \ref{claim1} that all arcs in $T_{\min}$ have finite stabilizer. Let now $l\in L\setminus \{1\}$ and $g_1,g_2\in G$ such that 
$$l[g_1t_0,g_1w]\cap[g_2t_0,g_2w]$$
is non-trivial and not a point. Lemma \ref{claim1} yields that $g_2^{-1}lg_1$ stabilizes the segment $[t_0,w]$ and in particular $g_2^{-1}lg_1\in E_L$. Hence $$l\in \langle E_L, G\rangle=\langle H,E_L\rangle\ast_{E_L}\langle K,E_L\rangle.$$ This finally proves $(2)-(5)$.
\end{proof}

We define a family $\mathcal{Z}$ of subtrees of $T$ consisting of:
\begin{itemize}
\item $L$-translates of $T_{\min}$ and 
\item closures of connected components of $T\setminus L\cdot T_{\min}$.
\end{itemize}
By $(2)$ of Lemma \ref{intersection} the family $\mathcal{Z}$ gives rise to an equivariant covering of $T$ such that any two subtrees intersect in at most one point. If we can show that every arc of $T$ is covered by finitely many subtrees of $\mathcal{Z}$ then $\mathcal{Z}$ is precisely what V. Guirardel in \cite{guirardel} calls a transverse covering of the tree $T$. So suppose for the the sake of contradiction that there exists some arc $I\subset T$ which is not covered by finitely many subtrees of $\mathcal{Z}$.\\
Let $\mathcal{H}=\{H,K\}$. 
Then by Theorem \ref{relativegeometric} the pair action $(L,\mathcal{H})\curvearrowright T$ is the strong limit of a direct system of geometric pair actions 

\begin{center}
\begin{tikzpicture}[descr/.style={fill=white}]
\matrix(m)[matrix of math nodes,
row sep=1em, column sep=2.8em,
text height=1.5ex, text depth=0.25ex]
{L_k&L_{k+1}&\cdots&L\\
T_k&T_{k+1}&\cdots&T\\};
\path[->>,font=\scriptsize]
(m-1-1) edge node[above] {$\vp_k$}(m-1-2)
(m-1-1) edge [bend left=60] node[descr] {$\Phi_{k}$} (m-1-4)

(m-1-2) edge [bend left=40] node[descr] {$\Phi_{k+1}$} (m-1-4)

(m-2-1) edge node[above] {$f_k$} (m-2-2)
(m-2-1) edge [bend right=60] node[descr] {$F_{k}$} (m-2-4)
(m-2-2) edge [bend right=40] node[descr] {$F_{k+1}$} (m-2-4);
\path[->,font=\scriptsize]
(m-1-1.south) edge [bend left=60](m-2-1.north)
(m-1-2.south) edge [bend left=60](m-2-2.north)
(m-1-4.south) edge [bend left=60](m-2-4.north);
\end{tikzpicture}
\end{center}

such that
\begin{itemize}
\item $\vp_k$ and $\Phi_k$ are one-to-one in restriction to arc stabilizers of $T_k$.
\item $\vp_k$ (resp. $\Phi_k$) restricts to an isomorphism between $\mathcal{H}_k$ and $\mathcal{H}_{k+1}$ (resp. $\mathcal{H}$).
\item The action of $L_k$ on $T_k$ splits as a graph of actions where each non-degenerate vertex action is either axial, thin or of orbifold type and therefore indecomposable, or is an arc containing no branch point except at its endpoints.
\end{itemize}

Since $G$ is finitely presented relative $\mathcal{H}$ there exists $k_0\in\N$ such that for all $k\geq k_0$, $G\leq L_k$. Denote by $T_{\min}$ the minimal $G$-invariant subtree of $T$ and by $T_{\min}^k$ the minimal $G$-invariant subtree of $T_k$ for all $k\geq k_0$. Clearly the action of $G$ on $T_{\min}$ and on $T_{\min}^k$ is minimal. Hence the direct system of pair actions $(G,\mathcal{H})\curvearrowright T_{\min}^k$ converges strongly to $(G,\mathcal{H})\curvearrowright T_{\min}$.\\ 
Clearly the action of $G$ on $T_{\min}$ is geometric and since by Theorem 2.5 in \cite{levittpaulin} a geometric action of a finitely generated group cannot be a non-trivial strong limit of actions, this implies that there exists $k_1\geq k_0$ such that $F_k$ is an isometry when restricted to $T_{\min}^k$ for all $k\geq k_1$.\\
By Lemma \ref{intersection}(3) $T_{\min}$ intersects any indecomposable subtree of $T$ in at most one point. It follows that $T_{\min}^k$ intersects each indecomposable subtree of $T_k$ in at most one point for $k\geq k_1$. Since the action of $L_k$ on $T_k$ splits as a graph of actions $\A_k$ where each non-degenerate vertex action is either axial, thin or of orbifold type and therefore indecomposable, or is an arc containing no branch point except at its endpoints, it follows that $T_{\min}^k$ is contained in the simplicial part of $T_k$.\\
Now there exists $k_2\geq k_1$ and an arc $J\subset F_k^{-1}(I)$ such that $F_k$ maps $J$ isometrically onto $I$ for all $k\geq k_2$. Let now $k\geq k_2$. Since the action of $L_k$ on $T_k$ splits as a graph of actions, every finite subtree of $T_k$ is covered by finitely many translates of the vertex trees of $\A_k$. In particular there exists a decomposition of $J$ into non-trivial subsegments
$$J=J_1\cup\ldots \cup J_l$$
such that each $J_i$ is contained in a unique translate of a vertex tree of $\A_k$. But this implies that $J$ intersects only finitely many translates of $T_{\min}^k$ in a segment (i.e. not just in a point). But since $F_k$ is an isometry when restricted to $T_{\min}^k$, it follows that $I$ also intersects only finitely many translates of $T_{\min}\subset T$ in more than a point. So assume that $I_0\subset I$ is a non-trivial subsegment of $I$ that intersects translates of $T_{\min}$ in at most one point. Then $I_0$ is contained in the closure of a connected component of $T\setminus L\cdot T_{\min}$. Hence we conclude that $I$ is covered by finitely many subtrees of $\mathcal{Z}$, a contradiction.\\
Therefore $\mathcal{Z}$ is a transverse covering of $T$ and it now follows from Lemma 1.5 in \cite{guirardel} that the action of $L$ on $T$ splits as a graph of actions such that the vertex trees are precisely the trees contained in $\mathcal{Z}$. Associated to this graph of actions there is a simplicial tree $S$ on which $L$ acts (in \cite{guirardel} this tree is called the skeleton of the graph of actions). For convenience we recall the definition of $S$.\\
The vertex set $V(S)$ is $V_0(S)\cup V_1(S)$, where $V_1(S)$ contains a vertex $x_Y$ for every subtree $Y$ of $\mathcal{Z}$ and $V_0(S)$ is the set of points $x\in T$ lying in the intersection of two distinct subtrees in $\mathcal{Z}$. There is an edge $e=(x,x_Y)$ between a vertex $x\in V_0(S)$ and a vertex $v_Y\in V_1(S)$ if and only if $x\in Y$.\\
The stabilizer of a vertex $v_Y\in V_1(S)$ is the setwise stabilizer of $Y$, while the stabilizer of a vertex $x\in V_0(S)$ is the stabilizer of the corresponding point $x\in T$. By Lemma \ref{intersection} $(4)$ the vertex $v_{T_{\min}}\in V_1(S)$ is stabilized by 
$$\langle H,E_L\rangle\ast_{E_L}\langle K,E_L\rangle.$$
Let $e=(x,v_{T_{\min}})$ be an edge in $S$ adjacent to $v_{T_{\min}}$. Then 
$$\stab_L(e)=\stab_L(x)\cap \stab_L(v_{T_{\min}})=\stab_L(x)\cap (\langle H,E_L\rangle\ast_{E_L}\langle K,E_L\rangle).$$
By Lemma \ref{intersection}$(5)$ either $\stab_L(x)=\langle H,E_L\rangle=:H'$ or $\stab(x)=\langle K,E_L\rangle=:K'$. Therefore we can refine $S$ by replacing each vertex $v_{gT_{\min}}$ corresponding to a translate of $T_{\min}$ by $gT_{\min}$. By Lemma \ref{intersection}$(1)$ this yields again a simplicial tree, say $S'$. Hence from the action of $L$ on $S'$, $L$ admits a splitting as a graph of groups $\A'$, such that there exists an edge $e\in EA'$ with edge group $A'_e=E_L$ and $H'\in A'_{\alpha(e)}$, $K'\in A'_{\omega(e)}$. Let $\A$ be the graph of groups which we get by collapsing all edges in $\A'$ except for $e$.\\
Suppose that $\A$ has only one vertex group, i.e. $L$ splits as the HNN-extension $L=M\ast_{E_L}$.
Since $L$ is one-ended relative $G$ and $H', K'\in M$, we conclude that $M$ is one-ended relative $\{H',K'\}$.\\
Let $\{m_1,\ldots,m_l\}$ be a generating set of $M$. We can extend this to a generating set $\{m_1,\ldots,m_l,t\}$ of $L$, where $t$ is the stable letter of the HNN-extension $L=M\ast_{E_L}$.
For every homomorphism $\vp:\tilde{G}\to \Gamma$ that factors through the limit map $\eta:\tilde{G}\to L$, we denote by $\bar{\vp}: L\to\Gamma$ the unique homomorphism such that $\vp=\bar{\vp}\circ\eta$.
Let $A_n$ be the following set:
\begin{align*}
A_n=\{ &\vp_n\in \Hom(L,\Gamma) \ |\ \vp_n|_{\langle t\rangle}=\bar{\lambda}_n|_{\langle t\rangle},\ \vp_n|_{G}=\bar{\lambda}_n|_{G},\\
&\vp_n(v_j)\neq 1 \text{ for all } j\in\{1,\ldots,r\}\}.
\end{align*}
Since $\bar{\lambda}_n\in A_n$ the set $A_n$ is non-empty. Let $U$ be the set of all converging sequences $(\vp_n)$, such that for all $n\in\N$, $\vp_n\in A_n$ and  
$$\sum_{i=1}^l|\vp_n(m_i)|_{\Gamma}$$
is minimal (with respect to the word metric on $\Gamma$). By the same arguments as in the proof of Lemma \ref{maximal1} and Lemma \ref{maximal} the collection of all sequences from $U$ factor through a finite collection of maximal $\Gamma$-limit groups. By abuse of notation let $(\vp_n)\in U$ be a sequence that converges into such a maximal limit group, say $N$.\\ 
In particular $(\vp_n|_M)$ converges into the action of a $\Gamma$-limit group $M_1$ on some real tree $T$ without a global fixpoint. Suppose that $M_1$ is not one-ended relative $H,K$. In this case we can pass to finitely many $\Gamma$-limit quotients $L_1,\ldots,L_k$ of $N$, each admitting a decomposition along finite subgroups with one vertex group $Y$ which has a decomposition either as $Y=N_1\ast_{E_L}$ or as $Y=N_1\ast_{E_L}N_2$, and all other vertex groups are isomorphic to subgroups of $\Gamma$. Moreover either $H,K\leq N_1$ and $N_1$ is one-ended relative $\{H,K\}$ (in the first case) or $H\leq N_1$, $K\leq N_2$ and $N_1$ is one-ended relative $H$ and $N_2$ is one-ended relative $K$ (in the second case).\\
This procedure of passing to finitely many $\Gamma$-limit quotients with the desired properties is straightforward but rather lengthy and we explain it in great detail in the proof of the main theorem (Theorem \ref{maintheorem}). So we omit the proof here and just note that, to simplify notation, we can assume without loss of generality that $N$ already is of this type and moreover that in the decomposition of $N$ along finite subgroups there are no vertices with vertex group isomorphic to a subgroup of $\Gamma$. That is we can assume that $N$ admits a decomposition either as $N=M_1\ast_{E_L}$, where $H,K\leq M_1$ and $M_1$ is one-ended relative $\{H,K\}$, or as $N=M_1\ast_{E_L}M_2$, where $H\leq M_1$, $K\leq M_2$ and $M_1$ is one-ended relative $H$ and $M_2$ is one-ended relative $K$.\\
Suppose we are in the first case., i.e. $M_1$ is one-ended relative $\{H,K\}$. Since $\vp_n$ is an extension of $\lambda_n$, the subgroups $H$ and $K$ act elliptically on $T$. Hence we can apply the relative version of the Rips machine (Theorem \ref{relativerips}) to analyze the action of $M_1$ on $T$. By Theorem \ref{relativerips} the action of $M_1$ on $T$ splits as a graph of actions with corresponding graph of groups $\B$.
Suppose that there exist a vertex in $\B$ corresponding to an axial or orbifold component. By the shortening argument (Corollary \ref{shortening}) we can shorten the action of $M_1$ on $T$ while keeping $\vp_n|_G$ fixed, a contradiction to the assumption that $\vp_n$ was chosen short relative to $G$. Hence there are no axial or orbifold components in $\B$ and therefore the action of $M_1$ on $T$ is simplicial.\\
Suppose that an edge $e$ in $T$ has infinite stabilizer. Again in this case we can shorten the length of $\vp_n$ while keeping $\vp_n|_G$ fixed, a contradiction. Hence all edges in $T$ have finite stabilizer. Therefore an edge $e\in \B$ induces a splitting of $M_1$ along a finite subgroup relative $H,K$ a contradiction to the assumption that $M_1$ is one-ended relative $H,K$.\\
We conclude that there exists a decomposition of $N$ as an amalgamated product $$N=V\ast_{E_L}P$$ along some finite subgroup $E_L$ containing $E$ and $H\leq V$, $K\leq P$. Moreover $V$ and $P$ are one-ended relative $H$, $K$ respectively.\\
To simplify notation we assume that $L$ is already of this type, i.e. $L=V\ast_{E_L}P$ with the properties above.
Let $\{u_1,\ldots,u_s\}$ be a generating set of $V$. We can extend this to a generating set $\{u_1,\ldots,u_s,p_1,\ldots,p_l\}$ of $L$.
Recall that for every homomorphism $\vp:\tilde{G}\to \Gamma$ that factors through the limit map $\eta:\tilde{G}\to L$, we denote by $\bar{\vp}: L\to\Gamma$ the unique homomorphism such that $\vp=\bar{\vp}\circ\eta$.
Let $B_n$ be the following set:
\begin{align*}
B_n=\{ &\vp_n\in \Hom(L,\Gamma) \ |\ \vp_n|_{V}=\bar{\lambda}_n|_{V},\ \vp_n|_{G}=\bar{\lambda}_n|_{G},\\
&\vp_n(v_j)\neq 1 \text{ for all } j\in\{1,\ldots,r\}\}.
\end{align*}
Since $\bar{\lambda}_n\in B_n$ the set $B_n$ is non-empty. Let $U$ be the set of all converging sequences $(\vp_n)$, such that for all $n\in\N$, $\vp_n\in B_n$ and  
$$\sum_{i=1}^l|\vp_n(p_i)|_{\Gamma}$$
is minimal (with respect to the word metric on $\Gamma$). By the same arguments as in the proof of Lemma \ref{maximal1} and Lemma \ref{maximal} the collection of all sequences from $U$ factor through a finite collection of maximal $\Gamma$-limit quotients. By abuse of notation let $(\vp_n)\in U$ be a sequence that converges into such a maximal limit group, say $N$.\\ 
In particular $(\vp_n|_P)$ converges into the action of a $\Gamma$-limit group $P_1$ on some real tree $T$ without a global fixpoint. As noted before we can assume without loss of generality that $P_1$ is one-ended relative $K$.\\
Since $\vp_n$ is an extension of $\lambda_n$, the subgroup $K$ acts elliptically on $T$. Hence we can once again apply the relative version of the Rips machine (Theorem \ref{relativerips}) to analyze the action of $P_1$ on $T$. By Theorem \ref{relativerips} the action of $P_1$ on $T$ splits as a graph of actions with corresponding graph of groups $\mathbb{P}$.
Suppose that there exist an vertex in $\mathbb{P}$ corresponding to an axial or orbifold component. By the shortening argument (Corollary \ref{shortening}) we can shorten the action of $P_1$ on $T$ while keeping $\vp|_K$ fixed, a contradiction to the assumption that $\vp_n|_P$ was chosen to be short. Hence there are no axial or orbifold components in $\mathbb{P}$ and therefore the action of $P_1$ on $T$ is simplicial.\\
Suppose that an edge $e$ in $T$ has infinite stabilizer. Again in this case we can shorten the length of $\vp_n|_P$ while keeping $\vp_n|_K$ fixed, a contradiction. Hence all edges in $T$ have finite stabilizer. Therefore an edge $e\in \mathbb{P}$ induces a splitting of $P_1$ along a finite subgroup relative $K$ a contradiction to the assumption that $P_1$ is one-ended relative $K$. But this contradicts our assumption that $P_1$ acts without a global fixed point on $T$. Hence there exists no sequence of homomorphisms $(\vp_n)\in U$ such that the $\vp_n|_P$ are pairwise distinct.\\
Therefore the sequence $(\vp_n|_P)$ has a constant subsequence and hence after passing to this subsequence, say $(\vp)$, we conclude that $P_1=P/\underrightarrow{\ker} (\vp_n)=P/\ker (\vp)$ is isomorphic to a subgroup of $\Gamma$.
Hence $N=V\ast_{E_L}P_1$, where $E_L$ is some finite group containing $E$ and $P_1$ is isomorphic to a subgroup of $\Gamma$ and contains $K$.
\end{proof}

%

\begin{satz}\label{testHNN}
Suppose $G$ has the structure of an HNN-extension 
$$\langle H,s\ |\ s\alpha(e)s^{-1}=\beta(e)\ \forall e\in E\rangle$$
 over $H$ along some finite group $E$. Assume that there exists a test sequence $(\lambda^H_n)$ for $H$ and for all $n\in\N$ a homomorphism $\Psi_n:G\to\Gamma$ such that $\Psi_n(\alpha(E))=\lambda_n^H(\alpha(E))$ and $\Psi_n(\omega(E))=\lambda_n^H(\omega(E))$. Assume moreover that for all $n\in\N$, $C_{\Gamma}(\lambda_n^H(\alpha(E)))$ contains a non-abelian free subgroup which is quasi-convex in $\Gamma$. Then there exists a test sequence $(\lambda_n)$ for $G$.
\end{satz}

\begin{proof}
Denote by $|\cdot |_{\Gamma}$ and $|\cdot |_G$ the word metric on $\Gamma$ and $G$ with respect to some fixed generating set respectively. Let $X$ be the Cayleygraph of $\Gamma$ with respect to this generating set. For all $n\in\N$ let $B_H(n)=\{h\in H\ |\ |h|_G\leq n\}$ and 
$$L(n)=\max\{|\lambda^H_n(h)|_{\Gamma}\ |\ h\in B_H(n)\}.$$ 
Let $n\in\N$. Denote by $C_n:=C_{\Gamma}(\lambda^H_n(\alpha(E)))$ the centralizer of $\lambda^H_n(E)$ in $\Gamma$. Let $F(a,b)$ be the quasi-convex non-abelian free subgroup of $C_n$ which exists by assumption. As in the proof of Theorem \ref{testamalgamated} let $F(x,y)\leq F(a,b)$ be a non-abelian free subgroup of $F(a,b)$ such that any segment lying in the $10\delta$-neighborhood of both axis $\Ax(a)$ and $\Ax(cbc^{-1})$ in $X$ of $a$ and any conjugate of $b$ by an element $c\in\Gamma$ has length at most $$\frac{1}{1000}\min\{|x|_{\Gamma},|y|_{\Gamma}\}.$$ We define
$$w_n:=xyxy^2xy^3x\cdots xy^{n\cdot L(n)}x$$
and a homomorphism $\lambda_n: G\to \Gamma$ by
\begin{itemize}
\item $\lambda_n(H)=\lambda_n^H(H)$ and
\item $\lambda_n(s)=\Psi_n(s)w_n$.
\end{itemize}

Then since $w_n\in C_{\Gamma}(\lambda_n^H(\alpha(E)))$:
\begin{align*}
\lambda_n(s\alpha(E)s^{-1}\omega(E)^{-1})&=\Psi_n(s)w_n\lambda_n^H(\alpha(E))w_n^{-1}\Psi_n(s)^{-1}\lambda_n^H(\omega(E))^{-1}\\
&=\Psi_n(s)\lambda_n^H(\alpha(E))\Psi_n(s)^{-1}\lambda_n^H(\omega(E))^{-1}\\
&=\Psi_n(s\alpha(E)s^{-1}\omega(E)^{-1})=1
\end{align*}
and hence $\lambda_n$ is a well-defined homomorphism.
Moreover by the construction of $w_n$ we have that 
$$n\cdot|\lambda_n(h)|_{\Gamma}\leq n\cdot L(n)< |\lambda_n(s)|_{\Gamma}$$
for all $n\in\N$, $h\in H$. Hence $\lambda_n(s)$ dominates the growth of $\lambda_n|_H=\lambda^H_n$.\\

Let $\tilde{G}$ be a group containing $G$ and in abuse of notation $\lambda_n: \tilde{G}\to\Gamma$ an extension of $\lambda_n$ which is short relative $G$ for all $n\in\N$. Let $(T,t_0)$ be the real tree into which (a convergent subsequence of) $(\lambda_n)$ converges and let $L:=\tilde{G}/\underrightarrow{\ker} \lambda_n$. Assume that $L$ is one-ended relative $G$ and moreover that $G$ does not act elliptically on $T$.
Since $G$ does not fix a point when acting on $T$, there exists a non-trivial minimal $G$-invariant subtree $T_{\min}$ of $T$ on which $G$ acts.

\begin{lemma}\label{intersection2}
The following holds:
\begin{enumerate}[(1)]
\item $T_{\min}$ is equivariantly isomorphic as a $G$-tree to the Bass-Serre tree of the splitting $G=H\ast_E$.
\item For every $g\in L$ either $T_{\min}\cap gT_{\min}$ is at most one point or $g$ is contained in the setwise stabilizer of $T_{\min}$.
\item $T_{\min}$ intersects any indecomposable subtree of $T$ in at most one point.
\item The setwise stabilizer of $T_{\min}$ in $L$ is $\langle H,E_L\rangle\ast_{E_L}$ for some finite group $E_L$, containing $E$.
\item Edges in $T_{\min}$ are stabilized by conjugates of $E_L$ and vertices by conjugates of $\langle H,E_L\rangle$.
\end{enumerate}
\end{lemma}

\begin{proof}
Identical to the proof of Lemma \ref{intersection}.
\end{proof}

We again define a family $\mathcal{Z}$ of subtrees of $T$ consisting of:
\begin{itemize}
\item $L$-translates of $T_{\min}$ and 
\item closures of connected components of $T\setminus L\cdot T_{\min}$.
\end{itemize}
By $(2)$ of Lemma \ref{intersection2} the family $\mathcal{Z}$ gives rise to an equivariant covering of $T$ such that any two subtrees intersect in at most one point. It follows from Lemma 1.5 in \cite{guirardel} that the action of $L$ on $T$ splits as a graph of actions such that the vertex trees are precisely the trees contained in $\mathcal{Z}$. Associated to this graph of actions there is a simplicial tree $S$ on which $L$ acts (in \cite{guirardel} this tree is called the skeleton of the graph of actions).  Refining this tree precisely as we did in the proof of Theorem \ref{testamalgamated}, finally yields a decomposition of $L$ as $L=V\ast_{E_L}$, where $H\leq V$ and $E_L$ is a finite group containing $E$.
\end{proof}

\subsection{Virtually abelian flats}
We distinguish several cases how a virtually abelian flat can appear in the construction of the completion. First consider the case that the virtually abelian flat is coming from an edge connecting two rigid vertices, i.e. is of the form $G=H\ast_CC\oplus \Z^n$, with $C$ infinite finite-by-abelian. This case is also covered by the more general Proposition \ref{testabelian2} below, but we still prefer to present a proof here, because the underlying ideas needed to prove the existence of test sequences for virtually abelian flats are easier to grasp in this simpler case.

\begin{prop}\label{testabelian}
Suppose $G=H\ast_CC\oplus A$ has the structure of a virtually abelian flat over some subgroup $H$, such that $A\cong \Z^k$, $C$ is an infinite finite-by-abelian subgroup of $H$, and there exists a test sequence $(\lambda_n^H)$ for $H$. Then there exists a test sequence $(\lambda_n)$ for $G$.
\end{prop}

\begin{proof}
Since $\lambda_n^H$ is a test sequence, $\lambda_n^H(C)$ is infinite finite-by-abelian for $n$ sufficiently large. By Proposition \ref{abelian} the center $Z(\lambda_n^H(C)$) is infinite and hence contains an element, say $w$, of infinite order which generates a quasi-convex subgroup of $\Gamma$. We define the test sequence $(\lambda_n)\subset\Hom(G,\Gamma)$ as follows. 
For all $n\in\N$
\begin{itemize}
\item we set $\lambda_n|_H=\lambda_n^H$.
\item Let $\{a_1,\ldots,a_k\}$ be a basis of $A$ and $B(n)$ the intersection of $H$ with the ball of radius $n$ around the identity in the Cayleygraph of $G$ (i.e. all elements in $H$ of word length at most $n$ in the word metric of $G$ with respect to some fixed finite generating set). Further let
$$\dist(n)=\max\{|\lambda_n^H(h)|_{\Gamma}\ |\ h\in B(n)\}.$$
We then choose $k_n^1\in\N$ such that 
$$|w^{k_n^1}|_{\Gamma}> n\cdot \dist(n)$$
and define 
$$\lambda_n(a_1)=w^{k_n^1}$$
Moreover for all $i\in\{2,\ldots,n\}$ we choose inductively $k_n^i\in\N$ such that 
$$|w^{k_n^i}|_{\Gamma}> n\cdot |w^{k_n^{i-1}}|_{\Gamma}$$
and define $$\lambda_n(a_i)=w^{k_n^i}.$$
\end{itemize}

Clearly $\lambda_n$ is a homomorphism and $\lambda_n|_A$ dominates the growth of $\lambda_n^H$. It remains to show that $(\lambda_n)$ is indeed a test sequence.\\

So let $\tilde{G}$ be a group containing $G$ and in abuse of notation let $\lambda_n: \tilde{G}\to\Gamma$ be an extension of $\lambda_n$ which is short relative $G$ for all $n\in\N$. Let $(T,t_0)$ be the real tree into which (a convergent subsequence of) $(\lambda_n)$ converges and let $L:=\tilde{G}/\underrightarrow{\ker} \lambda_n$. Assume that $L$ is one-ended relative $G$ and moreover that $G$ does not act elliptically on $T$. Let $\{a_1,\ldots,a_k\}$ be the basis of $A$ chosen above and $c$ a generating set of $C$. Then $C\oplus A=\langle c,a_1,\ldots,a_k\rangle$.\\
We claim that $L$ admits a splitting as an amalgamated product $L=V\ast_PM$, where $H\leq V$, the subgroups $P$ and $M$ are virtually abelian and $\langle c,a_1,\ldots,a_{k-1}\rangle\leq P$, while $a_k\in M\setminus P$.\\ 
To see this we now take a closer look at the action of $C\oplus A$ on $T$.
Since $\lambda_n|_A$ dominates the growth of $\lambda_n^H$, $H$ fixes the base point of $T$.

\begin{lemma}\label{lemma0}
\begin{enumerate}[(1)]
\item The finite-by-abelian subgroup $B:=\langle a_1,\ldots,a_{k-1},c\rangle$ fixes the base point $t_0$ of $T$.
\item $a_k$ acts hyperbolically on $T$.
\item $B$ fixes the segment $[t_0,a_kt_0]\subset T$.
\end{enumerate}
\end{lemma}
  
\begin{proof}
Let $(X,d_X)$ be the Cayley graph of $\Gamma$ with respect to some fixed finite generating set. Let $Y$ be a generating set of $L$ containing $\{c,a_1,\ldots,a_k\}$ and set
$$\mu_n:=\max_{y\in Y} d_X(1,\lambda_n(y))=|\lambda_n(y)|_{\Gamma}.$$
In particular $\mu_n\geq d_X(1,\lambda_n(a_k))$.
Denote by $d_T$ the metric on the real tree $T$ and by $d_n$ the scaled metric on $X$, i.e. $(X_n,d_n)=(X,\frac{d_X}{\mu_n})$. We assume without loss of generality that $(1)$ is an approximating sequence for $t_0$. By definition of the test sequence $(\lambda_n)$ it holds for all $i\in\{1,\ldots,k-1\}$ that 
$$|\lambda_n(a_k)|_{\Gamma}>n^{k-i}\cdot|\lambda_n(a_i)|_{\Gamma}$$
and therefore:
\begin{align*}d_T(t_0,a_it_0)&=\lim_{n\to\infty}d_n(1,\lambda_n(a_i).1)\\
&=\lim_{n\to\infty}\frac{1}{\mu_n}d_X(1,\lambda_n(a_i))\\
&\leq\lim_{n\to\infty}\frac{|\lambda_n(a_i)|_{\Gamma}}{|\lambda_n(a_k)|_{\Gamma}}\\
&<\lim_{n\to\infty}\frac{1}{n^{k-i}}=0.
\end{align*}
We have shown $(1)$ since $C$ is contained in $H$ and therefore fixes $t_0$.\\
Since $G$ does not act with a global fixed point on $T$, we have that $a_k$ does not fix $t_0$ and it is easy to see that $a_k$ acts by translation on the infinite axis $\bigcup_{i\in\Z}[a_k^it_0,a_k^{i+1}t_0]$ and hence $(2)$ holds.\\
By $(1)$ the group $B=\langle c,a_1,\ldots, a_{k-1}\rangle$ fixes $t_0$, hence it remains to show that $B$ also fixes $a_kt_0$. Since $(1)$ is an approximating sequence for $t_0$, $(\lambda_n(a_k))$ is an approximating sequence for $a_kt_0$ by Lemma 1.10 in \cite{rewe}. For all $i\in\{1,\ldots,k-1\}$ we have that $[\lambda_n(a_i),\lambda_n(a_k)]=1$ and therefore
\begin{align*}
d_T(a_i.a_kt_0,a_kt_0)&=\lim_{n\to\infty}\frac{d_X(\lambda_n(a_i)\lambda_n(a_k),\lambda_n(a_k))}{\mu_n}\\
&=\lim_{n\to\infty}\frac{|\lambda_n(a_k)^{-1}\lambda_n(a_i)\lambda_n(a_k)|_{\Gamma}}{\mu_n}\\
&\leq\lim_{n\to\infty}\frac{|\lambda_n(a_i)|_{\Gamma}}{|\lambda_n(a_k)|_{\Gamma}}\\
&<\lim_{n\to\infty}\frac{1}{n^{k-i}}=0.
\end{align*}
The same argument yields that $C$ fixes $a_kt_0$ and hence the claim follows.
\end{proof}

By assumption $L$ is one-ended relative $G$. In particular $L$ is one-ended relative $\langle B, H\rangle$. It follows from Lemma \ref{lemma0} that $\langle H,B\rangle$ acts elliptically on the real tree $T$ and hence we can apply the relative version of the Rips machine (Theorem \ref{relativerips}). This yields a graph of actions with corresponding graph of groups, say $\A$. By Lemma \ref{lemma0}
$$B\leq \stab_L([t_0,a_kt_0])$$
and hence the stabilizer of the segment $[t_0,a_kt_0]\subset T$ contains a finite-by-abelian group. Therefore this segment has trivial intersection with orbifold components, since segments in orbifold components have finite stabilizer. So suppose there exists a proper subsegment $J$ of $[t_0,a_kt_0]$ which is contained in an axial component, while $[t_0,a_kt_0]\setminus J$ is contained in the discrete part of $T$. Denote by $Z$ the stabilizer of the whole axial component. Since the subgroup $B$ stabilizes in particular $J$, it also stabilizes the whole axial component and therefore $B\leq Z$. Since $L$ is a $\Gamma$-limit group and therefore $N(\Gamma)$-CSA and $Z$ is a maximal virtually abelian subgroup of $L$ (since it is the stabilizer of an axial component) this implies that $a_k\in Z$. But the axis of $a_k$ contains $[t_0,a_kt_0]$, hence is not contained in the axial component stabilized by $Z$, a contradiction.\\
Therefore $[t_0,a_kt_0]$ is either contained in an axial component or contained in the discrete part of $T$.

\begin{lemma}\label{lemmaaxial}
If $[t_0,a_kt_0]$ is contained in an axial component of $T$, then $$L=V\ast_PM,$$ where $M,P$ are finitely generated virtually abelian groups, $a_k\in  M\setminus P$ and $$B=\langle c, a_1,\ldots, a_{k-1}\rangle\leq P.$$
\end{lemma}

\begin{proof}
Let $v$ be the vertex in the graph of groups $\A$, which corresponds to the axial component of $T$ containing $[t_0,a_kt_0]$. Let $M:=A_v$ be the associated vertex group. In particular $M$ is the stabilizer of the axial component and therefore virtually abelian. Let $P$ be the kernel of the action of $M$ on the axial component. Let $e_1,\ldots, e_k$ be the edges in $\A$  connected to $v$, with corresponding edge groups $A_{e_1},\ldots, A_{e_k}$. Hence $A_{e_1},\ldots, A_{e_k}$ are subgroups of $P$.
\begin{figure}[htbp]
\centering
\includegraphics[scale=0.9]{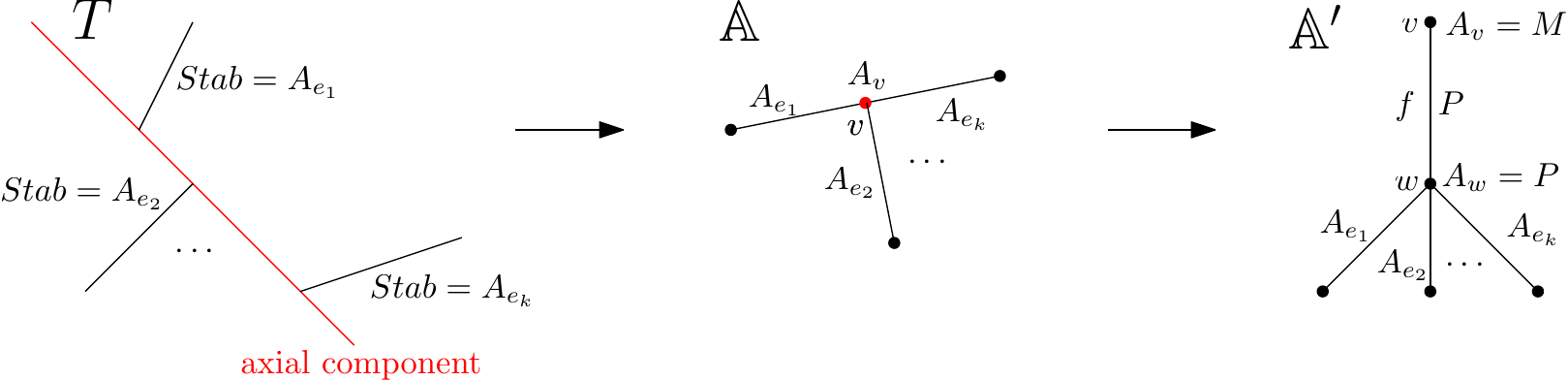}
\caption{The axial case}
\end{figure}
Therefore we can refine $\A$ by adding a new vertex $w$ with vertex group $A_w:=P$ and a new edge $f$ with edge group $A_f:=P$ connecting $v$ and $w$. Moreover the edges $e_1,\ldots,e_k$ are now connected to $w$ instead of $v$. Since $A_{e_1},\ldots, A_{e_k}\leq P=A_w$ this yields a graph of groups $\A'$, with $\pi_1(\A')=\pi_1(\A)$. In particular 
$$L=\pi_1(\A')=V\ast_{P}M.$$
By Lemma \ref{lemma0} the remaining properties of the Lemma follow. 
\end{proof}

\begin{lemma}\label{lemmadiscreteabelian}
If $[t_0,a_kt_0]$ is contained in the discrete part of $T$, then $$L=V\ast_PM,$$ where $M,P$ are finitely generated virtually abelian groups, $a_k\in  M\setminus P$ and $$B=\langle c, a_1,\ldots, a_{k-1}\rangle\leq P.$$
\end{lemma}

\begin{proof}
Let $e_1,\ldots,e_m$ be the edges contained in $e:=[t_0,a_kt_0]$ and let $i\in \{1,\ldots,m\}$. Clearly $B\leq \stab_L(e_i)=:A_{e_i}$ by Lemma \ref{lemma0}. Let $\Ax(a_k)$ be the axis of $a_k$ in $T$.\\
We claim that $A_{e_i}$ stabilizes pointwise the entire axis $\Ax(a_k)$.
\begin{figure}[htbp]
\centering
\includegraphics{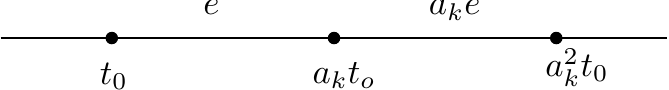}
\caption{Axis of $a_k$ in $T$}
\end{figure}
First we note that since $$B=\langle c,a_1,\ldots,a_{k-1}\rangle\leq \stab(e)$$ and $[B,a_k]=1$ it follows that
$$B=a_kB a_k^{-1}\leq a_k\stab(e)a_k^{-1}=\stab(a_ke)$$
and therefore $B\leq \stab(e)\cap\stab(a_ke)$. Hence $B$ stabilizes pointwise the entire axis of $a_k$ in $T$. Now suppose that there exists a segment $I\subset \Ax(a_k)$ such that $A_{e_i}\nleq \stab(I)$. But this implies that $\Ax(a_k)$ is an unstable subtree of $T$. It follows from Theorem \ref{eigenschaften}(4) that $\stab(\Ax(a_k))$ is finite, a contradiction to the fact that $B\leq \stab(\Ax(a_k))$. Therefore $A_{e_i}$ stabilizes the entire axis of $a_k$ for all $i\in\{1,\ldots,m\}$.\\
Denote by $M$ the setwise stabilizer of $\Ax(a_k)$. By Theorem \ref{eigenschaften} (4) $M$ is finitely generated virtually abelian and contains a finite-by-abelian subgroup of index at most $2$ which leaves $\Ax(a_k)$ invariant and fixes its ends. As shown above there exists a finite-by-abelian group $P'$ such that $\stab(f)=P'$ for all edges $f$ contained in $\Ax(a_k)$. Hence $L$ admits a decomposition as $L\ast_PM$, where $P$ is finitely generated virtually abelian and contains $P'$ as a subgroup of index at most $2$. Moreover it follows from Lemma \ref{lemma0} that $B\leq P$ and $a_k\in M\setminus P$.
\end{proof}

Therefore in both cases (axial and discrete), $L=V\ast_{P}M$, where $B \leq P$, $a_k\in M\setminus P$ and $P,M$ are finitely generated virtually abelian. The remaining part of the proof is identical to the equivalent part of the proof of Theorem \ref{testabelian2}, and we prefer to do this proof there in the more general case.
\end{proof}
 
Now suppose that the virtually abelian flat comes from a virtually abelian vertex group $A'$, i.e. is of the form $G=H\ast_{C}A$, where $C$ is the maximal virtually abelian subgroup of $H$ containing a peripheral subgroup $P(A')$ of $A'$ and $A$ is a virtually abelian group containing $A'$.

\begin{satz}\label{testabelian2}
Suppose $G=H\ast_CA$ has the structure of a virtually abelian flat over some subgroup $H$ coming from a virtually abelian vertex group $A'$, i.e.  $A$ is virtually abelian, $C$ is the maximal virtually abelian subgroup of $H$ containing a peripheral subgroup $P(A')$ of $A'$, and there exists a test sequence $(\lambda_n^H)$ for $H$. Then there exists a test sequence $(\lambda_n)$ for $G$.
\end{satz}

\begin{proof}
For every infinite virtually abelian subgroup $M$ of a $\Gamma$-limit group  there exists by Proposition \ref{abelian} a unique finite-by-abelian subgroup of index at most $2$ which we denote by $M^+$.\\
Let $E$ be the torsion subgroup of $A^+$. It follows from Lemma \ref{eig} that $E$ is a finite normal subgroup of $A$. Denote by $\pi:A\to A/E$ the canonical projection. Recall that $P(A')^+$ consists of all elements of $A'^+$ which vanish under every homomorphism from $A'^+$ to $\Z$. It follows from the construction of the completion that $E\leq P(A')^+\leq C^+$. Let $W:=\pi(A)=A/E$ and set $W^+:=\pi(A^+)$. Since $A^+$ is finitely generated finite-by-abelian it follows immediately that $W^+$ is finitely generated free abelian.\\
Moreover since $E\leq C^+$ it follows from the construction of the completion that $$A^+/C^+\cong \Z^n$$ for some $n\in\N$. From the second isomorphism theorem then follows
$$W^+/_{\displaystyle (C^+/E)}=(A^+/E)/_{\displaystyle (C^+/E)}\cong A^+/C^+\cong\Z^n.$$
Hence there exists a short exact sequence 
$$1\to C^+/E\to A^+/E \xrightarrow{\tau}\Z^n\to 1$$
and we claim that this sequence splits. To see this choose a basis $\{z_1,\ldots,z_n\}$ of $\Z^n$ and define a map $s_0:\{z_1,\ldots,z_n\} \to A^+/E$ such that $\tau\circ s_0=\id_{\{z_1,\ldots,z_n\}}$. Since $A^+/E$ is abelian there exists a unique extension of $s_0$ to a homomorphism $s:\Z^n\to A^+/E$ by the universal property of the free abelian group $\Z^n$. By construction clearly $\tau\circ s=\id_{\Z^n}$ and therefore the sequence splits. It follows that $$W^+=A^+/E\cong \Z^n\oplus C^+/E,$$
i.e. $K:=\pi(C^+)$ is a direct summand of the free abelian group $W^+=\pi(A^+)$. Hence there exists a subgroup $B\cong\Z^n$ of $W^+$ such that $W^+=K\oplus B$.\\
By construction of the completion, $G$ comes equipped with a retract $\eta:G\to H$. Since by assumption $C$ is the maximal virtually abelian subgroup of $W$ containing $P(A')$ and $\pi(A)$ is virtually abelian, we conclude that $\eta(A)=C$.\\
Let $\{k_1,\ldots,k_m\}$ be a basis of $K$ and $\{b_1,\ldots,b_k\}$ a basis of $B$. We define the following sequence of automorphisms of $W^+$:\\
For $n\in\N$ let $\alpha_n:W^+\to W^+$ be defined by:
\begin{align*}
\alpha_n(k_j)&=k_j\quad \text{ for all } j\in \{1,\ldots,m\},\\
\alpha_n(b_j)&=\sum_{i=1}^jn^{j-i}b_i+n^j\sum_{l=1}^mk_l\quad \text{ for all } j\in \{1,\ldots,k\}. 
\end{align*}
Set $\tilde{K}=\pi^{-1}(K)$, $\tilde{B}=\pi^{-1}(B)$. Then $A=\langle \tilde{K},\tilde{B},s\rangle$ for some element $s$, where either $\pi(s)$ has order $2$ (see \cite{rewe} Lemma 4.9) or $s=1$. Denote by $\Aut_s(A)$ the subgroup of $\Aut(A)$ consisting of those automorphisms which restrict to the identity on $\langle \tilde{K},s\rangle$ and preserve $\tilde{B}$.\\
It follows from Corollary 4.13 in \cite{rewe} that after passing to a subsequence we can extend $\alpha_n$ to an automorphism (still denoted) $\alpha_n\in \Aut_s(A)$ for all $n\in\N$.\\
We then define our test sequence $\lambda_n: G\to \Gamma$ on $A$ as $\lambda_n|_A=\lambda^H_n\circ\eta\circ\alpha_n$ and on $H$ as $\lambda_n|_H=\lambda_n^H$ for all $n\in\N$.\\
Clearly $\lambda_n|_A$ dominates the growth of $\lambda_n|_H=\lambda_n^H$. Moreover for $n$ sufficiently large it holds that $$|\lambda_n(\tilde{b}_i)|_{\Gamma}> n\cdot|\lambda_n(\tilde{b}_{i-1})|_{\Gamma}$$
and $$|\lambda_n(\tilde{b}_1)|_{\Gamma}> n\cdot|\lambda_n(\tilde{k})|_{\Gamma}$$
for all $i\in\{2,\ldots,k\}$, $\tilde{b}_i\in\pi^{-1}(b_i)$ and $\tilde{k}\in \pi^{-1}(K)$.
It remains to show that $\lambda_n$ is indeed a test sequence.\\

So let $\tilde{G}$ be a group containing $G$ and in abuse of notation let $\lambda_n: \tilde{G}\to\Gamma$ be an extension of $\lambda_n$ which is short relative $G$ for all $n\in\N$. Let $(T,t_0)$ be the real tree into which (a convergent subsequence of) $(\lambda_n)$ converges and let $L:=\tilde{G}/\underrightarrow{\ker} \lambda_n$. Assume that $L$ is one-ended relative $G$ and moreover that $G$ does not act elliptically on $T$.\\
We claim that $L$ admits a splitting as an amalgamated product $L=V\ast_PM$, where $H\leq V$, the subgroups $P,M$ are virtually abelian and $$N:=\langle k_1\ldots,k_m,b_1,\ldots,b_{k-1}\rangle\leq P^+/E_P,$$ while $b_k\in (M^+/E_M)\setminus (P^+/E_M)$, here $E_P$, $E_M$ denote the torsion subgroups of $P$ and $M$ respectively.\\ 
To see this we now take a closer look at the action of $A$ on $T$.
Since $\lambda_n|_A$ dominates the growth of $\lambda_n^H$, $H$ fixes the base point of $T$. Recall that $\pi:A \to A/E$ denotes the canonical projection and that $A=\langle \tilde{K},\tilde{B}, \tilde s\rangle$, where $\tilde{K}=\pi^{-1}(K)$, $\tilde{B}=\pi^{-1}(B)$ and either $\pi(\tilde{s})=s$ has order $2$ or $\tilde{s}=1$.

\begin{lemma}\label{lemma01}
\begin{enumerate}[(1)]
\item The subgroup $\tilde{N}:=\pi^{-1}(N)$ fixes the base point $t_0$ of $T$.
\item Elements from $\pi^{-1}(b_k)$ act hyperbolically on $T$.
\item If $\tilde{b}_k\in\pi^{-1}(b_k)$ then $N$ fixes the segment $[t_0,\tilde{b}_kt_0]\subset T$.
\item If $\tilde{s}\neq 1$ then $\tilde{s}$ exchanges the ends of the axis of $\tilde{b}_k$ in $T$.
\end{enumerate}
\end{lemma}

\begin{proof}
Let $(X,d_X)$ be the Cayley graph of $\Gamma$ with respect to some fixed generating set. Let $Y$ be a generating set of $L$ containing the generators of $\pi^{-1}(K\oplus B)$ and $\tilde{s}$ and set
$$\mu_n:=\max_{y\in Y} d_X(1,\lambda_n(y))=|\lambda_n(y)|_{\Gamma}.$$
Let $\tilde{b}_k\in \pi^{-1}(b_k)$. Then in particular $\mu_n\geq d_X(1,\lambda_n(\tilde{b}_k))$.
Denote by $d_T$ the metric on the real tree $T$ and by $d_n$ the scaled metric on $X$, i.e. $(X_n,d_n)=(X,\frac{d_X}{\mu_n})$. We assume without loss of generality that $(1)$ is an approximating sequence for $t_0$. 
Since $\tilde{K}\leq C\leq H$, $\tilde{K}$ fixes $t_0$. In particular $E$ fixes $t_0$. 
By definition of the test sequence $(\lambda_n)$ it holds for all $i\in\{1,\ldots,k-1\}$ and $\tilde{b_i}\in\pi^{-1}(b_i)$ that 
$$|\lambda_n(\tilde{b}_k)|_{\Gamma}>n^{k-i}\cdot|\lambda_n(\tilde{b}_i)|_{\Gamma}$$
and therefore:
\begin{align*}
d_T(t_0,\tilde{b}_it_0)&=\lim_{n\to\infty}d_n(1,\lambda_n(\tilde{b}_i).1)\\
&=\lim_{n\to\infty}\frac{1}{\mu_n}d_X(1,\lambda_n(\tilde{b}_i))\\
&\leq\lim_{n\to\infty}\frac{|\lambda_n(\tilde{b}_i)|_{\Gamma}}{|\lambda_n(\tilde{b}_k)|_{\Gamma}}\\
&<\lim_{n\to\infty}\frac{1}{n^{k-i}}=0.
\end{align*}
We have shown $(1)$.\\
Since $A$ does not act with a global fixed point on $T$, $\tilde{b}_k$ does not fix $t_0$ and it is easy to see that $\tilde{b}_k$ acts by translation on the infinite axis 
$$Y:= \bigcup_{i\in\Z}[\tilde{b}_k^it_0,\tilde{b}_k^{i+1}t_0].$$
By $(1)$ the group $\tilde{N}=\pi^{-1}(N)$ fixes $t_0$, hence it remains to show that $\tilde{N}$ also fixes $\tilde{b}_kt_0$. Since $(1)$ is an approximating sequence for $t_0$, $(\lambda_n(\tilde{b}_k))$ is an approximating sequence for $\tilde{b}_kt_0$ by Lemma 1.10 in\cite{rewe}. For all $i\in\{1,\ldots,k-1\}$, $\tilde{k}\in\tilde{K}$, there exists $e_{i},e_i^{(\tilde{k})}\in E$ such that $[\lambda_n(\tilde{b}_i),\lambda_n(\tilde{b}_k)]=\lambda_n(e_{i})$ and $[\lambda_n(\tilde{b}_i),\lambda_n(\tilde{k})]=\lambda_n(e_{i}^{(\tilde{k})})$ and therefore
\begin{align*}
d_T(\tilde{b}_i.\tilde{b}_kt_0,\tilde{b}_kt_0)&=\lim_{n\to\infty}\frac{d(\lambda_n(\tilde{b}_i)\lambda_n(\tilde{b}_k),\lambda_n(\tilde{b}_k))}{\mu_n}\\
&=\lim_{n\to\infty}\frac{|\lambda_n(\tilde{b}_k)^{-1}\lambda_n(\tilde{b}_i)\lambda_n(\tilde{b}_k)|_{\Gamma}}{\mu_n}\\
&\leq\lim_{n\to\infty}\frac{|\lambda_n(\tilde{b}_i)\lambda_n(e_{i})|_{\Gamma}}{|\lambda_n(\tilde{b}_k)|_{\Gamma}}\\
&<\lim_{n\to\infty}\frac{1}{n^{k-i}}=0.
\end{align*}
The same argument yields that $\tilde{K}$ fixes $\tilde{b}_kt_0$ and hence $(3)$ follows.\\
Assume that $\tilde{s}\neq 1$. It remains to show that $\tilde{s}$ exchanges the ends of $Y$. This follows from Lemma 4.9 in \cite{rewe}. There it is shown that $W=\pi(A)=W^+\rtimes \Z_2$, $\Z_2=\langle s\rangle$ and the action of $\Z_2$ on $W^+$ is given by $sws^{-1}=w^{-1}$ for all $w\in W^+$. In particular $\tilde{b}_k\tilde{s}\tilde{b}_k\tilde{s}^{-1}\in E$. Moreover by construction of the test sequence $\tilde{s}$ fixes $t_0$. The claim now follows from 
$$\tilde{s}\tilde{b}_kt_0=\tilde{b}^{-1}_k\tilde{s}t_0=\tilde{b}_k^{-1}t_0.$$
\end{proof}

Lemma \ref{lemma01} shows in particular that $A$ acts invariantly on $Y$ and $E$ lies in the kernel of this action. Hence the action factors through $\pi$ and induces an action of $$W=\pi(A)=(K\oplus B)\rtimes\Z_2$$ on $Y$.\\
By assumption $L$ is one-ended relative $G$. In particular $L$ is one-ended relative $\langle H,\tilde{N}\rangle$. It follows from Lemma \ref{lemma01} that $\langle\tilde{N},H\rangle$ acts elliptically on the real tree $T$ and hence we can apply the relative version of the Rips machine (Theorem \ref{relativerips}). This yields a graph of actions with corresponding graph of groups, say $\A$. As in the proof of Proposition \ref{testabelian} we conclude that $Y$ is either contained in an axial component or in the discrete part of $T$.

\begin{lemma}\label{lemmaaxial1}
If $Y$ is contained in an axial component of $T$, then $L=V\ast_PM$, where $M,P$ are finitely generated virtually abelian groups, where $H\leq V$, $\pi^{-1}(N)\subset P$ and $$b_k\in (M/E_M)\setminus (P/E_M)$$ (here $E_P,E_M$ denote the torsion subgroups of $P$ and $M$ respectively).
\end{lemma}

\begin{proof}
The proof is identical to the one of Lemma \ref{lemmaaxial}.
\end{proof}

\begin{lemma}\label{lemmadiscrete}
If $Y$ is contained in the discrete part of $T$, then $L=V\ast_PM$, where $M,P$ are finitely generated virtually abelian groups, where $H\leq V$, $\pi^{-1}(N)\subset P$ and $b_k\in (M/E_M)\setminus (P/E_M)$.
\end{lemma}

\begin{proof}
Since $E$ lies in the kernel of the action of $A$ on $Y$ we get an induced action of $W=(K\oplus B)\rtimes \Z_2$ on $Y$.
Let $e_1,\ldots,e_p$ be the edges contained in $e:=[t_0,\tilde{b}_kt_0]$. Let $i\in \{1,\ldots,m\}$. Clearly $\pi^{-1}(N)\leq \stab_L(e_i)=:A_{e_i}$ by Lemma \ref{lemma01}.
We claim that $A_{e_i}$ stabilizes pointwise the entire axis $Y=\Ax(\tilde{b}_k)$ of $\tilde{b}_k$ in $T$.\\
First note that since $[N,b_k]=1_W$ it follows that
$$N=b_kN b_k^{-1}\leq b_k\stab_W(e)b_k^{-1}=\stab_W(b_ke)$$
and therefore $\pi^{-1}(N)\leq \stab_L(e)\cap\stab_L(\tilde{b}_ke)$. Hence $\tilde{N}$ stabilizes pointwise the entire axis of $\tilde{b}_k$ in $T$. Now suppose that there exists a segment $I\subset Y$ such that $A_{e_i}\nleq \stab(I)$. But this implies that $Y$ is an unstable subtree of $T$. It follows from Theorem \ref{eigenschaften}(4) that $\stab(Y)$ is finite, a contradiction to the fact that $\pi^{-1}(N)\leq \stab(Y)$. Therefore $A_{e_i}$ stabilizes the entire axis of $\tilde{b}_k$ for all $i\in\{1,\ldots,p\}$.\\
Denote by $M$ the setwise stabilizer of $Y$. By Theorem \ref{eigenschaften}(4) $M$ is finitely generated virtually abelian and contains a finite-by-abelian subgroup of index at most $2$ which leaves $Y$ invariant and fixes its ends. As shown above there exists a finite-by-abelian group $P'$, containing $\tilde{N}$, such that $\stab(f)=P'$ for all edges $f$ contained in $Y$. Hence $L$ admits a decomposition as $L\ast_PM$, where $P$ is finitely generated virtually abelian and contains $P'$ as a subgroup of index at most $2$. In particular $P=\langle P',\tilde{s}\rangle$  by Lemma \ref{lemma01}(4). The remaining part of the claim follows from Lemma \ref{lemma01}.
\end{proof}

In the following we denote the maximal free abelian subgroup of a virtually abelian group $P$ by $Z_P$.
Let $\{m_1,\ldots,m_l\}$ be a generating set of $M$ and $\{v_1,\ldots,v_q\}$ a generating set for $V$. Then clearly $\{m_1,\ldots,m_l,v_1,\ldots,v_q\}$ is a generating set of $L$.
For every homomorphism $\vp:\tilde{G}\to \Gamma$ that factors through the $\Gamma$-limit map $\eta:\tilde{G}\to L$, we denote by $\bar{\vp}: L\to\Gamma$ the unique homomorphism such that $\vp=\bar{\vp}\circ\eta$.
Let $A_n$ be the following set:
\begin{align*}
A_n=\{ &\vp_n\in \Hom(L,\Gamma) \ |\  \vp_n|_{M}=\bar{\lambda}_n|_{M},\ \vp_n|_{G}=\bar{\lambda}_n|_{G},\\ &\vp_n(v_j)\neq 1 \text{ for all } j\in\{1,\ldots,r\}\}.
\end{align*}
Since $\bar{\lambda}_n\in A_n$ the set $C_n$ is non-empty. Let $U$ be the set of all converging sequences $(\vp_n)$, such that for all $n\in\N$, $\vp_n\in C_n$ and  
$$\sum_{i=1}^q|\vp_n(v_i)|_{\Gamma}$$
is minimal (with respect to the word metric on $\Gamma$). By the same arguments as in the proof of Lemma \ref{maximal1} and Lemma \ref{maximal} the collection of all sequences from $U$ factor through a finite collection of maximal limit groups. Let $(\vp_n)\in U$ be a sequence that converges into such a maximal limit group, say $N$.\\ 
In particular $(\vp_n|_V)$ converges into the action of a $\Gamma$-limit group $V_1$ on some real tree $T$ without a global fixpoint. 
Moreover $(\vp_n|_V)$ is the extension of $\lambda_n|_Y$ where $Y=H\ast_{C}A_1$ and $A_1\leq A$ is the subgroup of the original virtually abelian group $A$ such that $A_1=\langle \tilde{N},\tilde{s}\rangle$, i.e. $A_1^+=\tilde{N}$. In particular $\rk(Z_{A_1})=\rk(Z_A)-1$.\\
As in the proof of Theorem \ref{testamalgamated} (where we referred to the proof of Theorem \ref{maintheorem}) we can assume without loss of generality that $N$ is one-ended relative $G$ and moreover that $Y$ does not act elliptically on the limit tree $T$.
Hence we can apply the same procedure to analyze the action of $Y$ on $T$ as we did before for the action of $A$ on the corresponding real tree.\\
Therefore (in general) we end up with finitely many $\Gamma$-limit quotients of $L$ which admit a decomposition $\D$ relative $G$ along finite subgroups with one vertex with vertex group $U_1$, which contains $G$, while all other vertices are isomorphic to subgroups of $\Gamma$. Moreover $U_1=V_1\ast_{P_1}M_1$, where $H\leq V_1$, $P_1$ is virtually abelian and contains $C$ and $M_1$ is virtually abelian and contains $A$. In addition $b_{k-1},b_k\in (M_1/E_{M_1})\setminus (P_1/ E_{M_1})$ and $\tilde{K},\tilde{b}_1,\ldots,\tilde{b}_{k-2}\subset P_1$.\\

Hence after repeating these steps finitely many times we end up with finitely many $\Gamma$-limit quotients of $\tilde{G}$ which admit a decomposition $\D$ relative $G$ along finite subgroups with one vertex with vertex group $\hat{U}$, which contains $G$, while all other vertices are isomorphic to subgroups of $\Gamma$. Moreover $\hat{U}=\hat{V}\ast_{\hat{P}}\hat{M}$, where $H\leq \hat{V}$, $\hat{P}$ is virtually abelian and contains $C$ and $\hat{M}$ is virtually abelian and contains $A$. In addition $$B=\langle b_1,\ldots,b_k\rangle \subset (\hat{M}/E_{\hat{M}})\setminus (\hat{P}/E_{\hat{M}})$$ and $\tilde{K}\subset \hat{P}$.\\
To simplify notation we assume that $L$ is already of this form. We moreover assume without loss of generality that $L=\hat{U}=\hat{V}\ast_{\hat{P}}\hat{M}$, i.e. that the decomposition $\D$ of $L$ along finite subgroups is trivial.
Let $(\vp_n)\subset\Hom(L,\Gamma)$ be a stably injective sequence. After passing to a subsequence we can assume that $\vp_n|_{\hat{M}}\subset \Hom(\hat{M},Z)$ for some virtually cyclic group $Z\leq \Gamma$ and that moreover $\vp_n$ is injective on a virtually cyclic subgroup of $\hat{M}$, containing $\hat{E}$, and $\vp_n(v_j)\neq 1$ for all $j\in\{1,\ldots,r\}$ and $n\in\N$. In particular $\hat{M}$ is a $Z$-limit group.\\
Let us briefly recall the structure of $A$ and $\hat{M}$. $E\leq A$ is the torsion subgroup of $A$ and $\pi:A\to A/E$ denotes the canonical projection. Then $\pi(A)=(K\oplus B)\rtimes \Z_2$, where $\Z_2=\langle s\rangle$, $K=\pi(C^+)$ and $K,B$ are finitely generated free abelian. Let $\tilde{X}=\pi^{-1}(X)$ for all $X\leq \pi(A)$. Then $A=\langle \tilde{K},\tilde{B}, \tilde{s}\rangle$. As shown before $\hat{M}$ is finitely generated virtually abelian, $$A\leq \hat{M}, B\leq (\hat{M}/E_{\hat{M}})\setminus (\hat{P}/E_{\hat{M}})\text{ and }C\leq \hat{P},$$ where $E_{\hat{M}}$ is the torsion subgroup of $\hat{M}$. Denote by $\hat{\pi}:\hat{M}\to\hat{M}/\hat{E}$ the canonical projection.\\
Then there exists a decomposition of $\hat{\pi}(\hat{M}^+)$ as $\hat{\pi}(\hat{M}^+)=R\oplus S$ such that $\hat{\pi}(A)$ is a subgroup of finite index of $R$. In particular $\hat{M}=\langle \hat{\pi}^{-1}(R),\hat{\pi}^{-1}(S), \hat{s}\rangle$ for some element $\hat{s}$ such that either $\hat{\pi}(\hat{s})$ has order $2$ or $\hat{s}=1$.\\
Let $\Aut_{\hat{s}}(\hat{M})\leq \Aut(\hat{M})$ be the subgroup consisting of those automorphisms that restrict to the identity on $\langle \hat{\pi}^{-1}(S), \hat{s}\rangle$ and preserve $\hat{\pi}^{-1}(R)$. Any $\alpha\in \Aut_{\hat{s}}(\hat{M})$ restricts to an automorphism of $\hat{\pi}^{-1}(R)$ and therefore induces an automorphism of $R$. Denote the subgroup of $\Aut(R)$ induced in this fashion by $K_{\hat{s}}$. By Lemma 4.12 in \cite{rewe} $K_{\hat{s}}$ has finite index in $\Aut(R)$. Therefore
for every sequence of automorphisms $(\alpha_n)\subset\Aut(R)$ we can assume, after passing to a subsequence, that $(\alpha_n)\subset K_{\hat{s}}$.\\ 
Let $(\alpha_n)\subset \Aut(R)$. Then (after passing to a subsequence) there exists a lift $\hat{\alpha}_n\in \Aut_{\hat{s}}(\hat{M})$ of $\alpha_n$ for all $n\in\N$. Define $\Psi_n:=\vp_1\circ\hat{\alpha}_n$ for all $n\in\N$ and let $M':=\hat{M}/\underrightarrow{\ker} \Psi_n$ be the corresponding $\Gamma$-limit group. Since $M'$ is also a $Z$-limit group, it is in particular virtually abelian. We now choose the sequence $(\alpha_n)\subset \Aut(R)$ such that $(\Psi_n|_{\langle\pi^{-1}(R),\hat{s}\rangle})$ is stably injective. In particular it follows that $A\leq M'$ is a subgroup of finite index. Let $L':=\hat{V}\ast_{\hat{P}}M'$ and denote by $\eta:L\to L'$ the corresponding $\Gamma$-limit map. The claim follows.
\end{proof}

Now suppose that the virtually abelian flat comes from an isolated virtually abelian vertex group $A$, i.e. is of the form $G=H\ast_{E}A$ for some finite group $E$. 

\begin{satz}\label{testabelian3}
Suppose $G=H\ast_EA$ where $A$ is virtually abelian, $E$ is finite and there exists a test sequence $(\lambda_n^H)$ for $H$. Let $F$ be the torsion subgroup of $A^+$. Denote by $\pi:A\to A/F$ the canonical projection. Let $s\in A$ be an element such that $A=\langle A^+,s\rangle$ and let $\{z_1,\ldots,z_k\}$ be a basis of $\pi(A^+)\cong\Z^k$.
Assume that there exists a sequence of homomorphisms $(\lambda_n^A)\subset\Hom(A,\Gamma)$ such that $\lambda_n^H(E)=\lambda_n^A(E)$ for all $n\in\N$. Assume moreover that $\lambda_n^A$ is injective when restricted to the virtually cyclic subgroup $Z_1:=\langle s, \pi^{-1}(\langle z_1\rangle)\rangle$, $\lambda_n^A(A)=\lambda_n^A(Z_1)$ and that $C_{\Gamma}(\lambda_n^A(E))$ contains a non-abelian free subgroup which is quasi-convex in $\Gamma$. Then there exists a test sequence $(\lambda_n)$ for $G$.
\end{satz}

\begin{proof}
Denote by $|\cdot |_{\Gamma}$ and $|\cdot |_G$ the word metric on $\Gamma$ and $G$ respectively. For all $n\in\N$ let $B_H(n)=\{h\in H\ |\ |h|_G\leq n\}$ and 
$$L(n)=\max\{|\lambda^H_n(h)|_{\Gamma}\ |\ h\in B_H(n)\}.$$
Denote by $C_n:=C_{\Gamma}(\lambda_n^A(E))$ the centralizer of $\lambda_n^A(E)$ in $\Gamma$. Let $F(a,b)$ be the quasi-convex non-abelian free subgroup of $C_n$ which exists by assumption. As in the proof of Theorem \ref{testamalgamated} let $F(x,y)\leq F(a,b)$ be a non-abelian free subgroup of $F(a,b)$ such that any segment lying in the $10\delta$-neighborhood of both axis $\Ax(a)$ and $\Ax(cbc^{-1})$ in $X$ of $a$ and any conjugate of $b$ by an element $c\in\Gamma$ has length at most $$\frac{1}{1000}\min\{|x|_{\Gamma},|y|_{\Gamma}\}.$$ For all $n\geq 1$ we define
$$w_n:=xyxy^2xy^3x\cdots xy^{n\cdot L(n)}x.$$
Let $W:=\pi(A)=A/E$ and set $W^+:=\pi(A^+)$.
We define the following sequence of automorphisms of $W^+$:\\
For $n\in\N$ let $\alpha_n:W^+\to W^+$ be defined by:
$$\alpha_n(z_j)=\sum_{i=1}^jn^{j-i}z_i\quad \text{ for all } j\in \{1,\ldots,k\}.$$ 
It follows from Corollary 4.13 in \cite{rewe} that after passing to a subsequence we can lift $\alpha_n$ to an automorphism (still denoted) $\alpha_n\in \Aut(A)$ for all $n\in\N$.\\
We then define our test sequence $\lambda_n: G\to \Gamma$ on $A$ as $\lambda_n|_A=c_{w_n}\circ\lambda^A_n\circ\alpha_n$ and on $H$ as $\lambda_n|_H=\lambda_n^H$ for all $n\in\N$.\\
By further modifying the automorphisms $(\alpha_n)$ (if necessary) we can assume that $\lambda_n|_A$ dominates the growth of $\lambda_n^H$.
It remains to show that $\lambda_n$ is indeed a test sequence.\\

So let $\tilde{G}$ be a group containing $G$ and in abuse of notation let $\lambda_n: \tilde{G}\to\Gamma$ be an extension of $\lambda_n$ which is short relative $G$ for all $n\in\N$. Let $(T,t_0)$ be the real tree into which (a convergent subsequence of) $(\lambda_n)$ converges and let $L:=\tilde{G}/\underrightarrow{\ker} \lambda_n$. Assume that $L$ is one-ended relative $G$ and moreover that $G$ does not act elliptically on $T$.\\
After applying the same same procedure as we did in the proof of Theorem \ref{testabelian2} we end up with finitely many $\Gamma$-limit quotients of $\tilde{G}$ which admit a decomposition $\D$ relative $G$ along finite subgroups with one vertex with vertex group $Q$, which contains $G$, while all other vertices are isomorphic to subgroups of $\Gamma$. Moreover $Q=V\ast_PM$, where $H\leq V$, $P$ is virtually abelian and contains $E$ and $M$ is virtually abelian and contains $A$. In addition $z_1\leq P^+/E_P$ and $\{z_2,\ldots,z_k\}\leq (M^+/E_M)\setminus (P^+/E_P)$, where $E_M$, $E_P$ denote the respective torsion subgroups of $M$ and $P$.\\
To simplify notation we assume without loss of generality that $L=V\ast_PM$ is already of this type and moreover that $L$ is one-ended relative $G$, i.e. $\D$ contains only a single vertex. Let $\{u_1,\ldots,u_l\}$ be a a generating set for $V$.
For every homomorphism $\vp:\tilde{G}\to \Gamma$ that factors through the $\Gamma$-limit map $\eta:\tilde{G}\to L$, we denote by $\bar{\vp}: L\to\Gamma$ the unique homomorphism such that $\vp=\bar{\vp}\circ\eta$.
Let $A_n$ be the following set:
\begin{align*}
A_n=\{&\vp_n\in \Hom(L,\Gamma) \ |\  \vp_n|_{M}=\bar{\lambda}_n|_{M},\ \vp_n|_{G}=\bar{\lambda}_n|_{G},\\
&\vp_n(v_j)\neq 1 \text{ for all } j\in\{1,\ldots,r\}\}.
\end{align*}
Since $\bar{\lambda}_n\in A_n$ the set $A_n$ is non-empty. Let $U$ be the set of all converging sequences $(\vp_n)$, such that for all $n\in\N$, $\vp_n\in A_n$ and  
$$\sum_{i=1}^l|\vp_n(u_i)|_{\Gamma}$$
is minimal (with respect to the word metric on $\Gamma$). By the same arguments as in the proof of Lemma \ref{maximal1} and Lemma \ref{maximal} the collection of all sequences from $U$ factor through a finite collection of maximal limit groups. Let $(\vp_n)\in U$ be a sequence that converges into such a maximal $\Gamma$-limit group, say $N$.\\ 
In particular $(\vp_n|_V)$ converges into the action of a $\Gamma$-limit group $V_1$ on some real tree $T$ without a global fixpoint. Recall that $Z_1=\langle s, \pi^{-1}(\langle z_1\rangle)\rangle$. As in the proof of Theorem \ref{testamalgamated} (where we referred to the proof of Theorem \ref{maintheorem}) we can assume without loss of generality that $N$ is one-ended relative $G$ and moreover that $X:=H\ast_EZ_1$ does not act elliptically on the limit tree $T$.\\
But now we are in the case of Theorem \ref{testamalgamated}. $\lambda_n|_{Z_1}$ is by assumption an embedding postcomposed by conjugation with an element satisfying a $C'(1/n)$ small-cancellation property, while $H$ acts elliptically on $T$. Hence by the same arguments as in the proof of Theorem \ref{testamalgamated} it follows that (finitely many quotients of) $V_1$ (still denoted $V_1$) have the structure $V_1=R\ast_{E'}C_1$ where $E'$ is a finite group containing $E$, and $C_1$ is virtually cyclic and contains $Z_1$ as a subgroup of finite index.\\
Applying the same procedure as in the end of the proof of Theorem \ref{testabelian2} yields (finitely many) $\Gamma$-limit quotients $L_1,\ldots,L_m$ of $\tilde{G}$ such that every $L_i$ has a decomposition $\D$ along finite subgroups with one distinguished vertex group $M$ and all other vertex group are isomorphic to subgroups of $\Gamma$. Moreover $M=V_1\ast_{E'} A_1$, where $H\leq V_1$, $E'$ is a finite group containing $E$, and $A_1$ is virtually abelian and contains $A$ as a subgroup of finite index. This completes the proof of the theorem.
\end{proof}

\subsection{Test sequences for Completions of Resolutions}
\begin{satz}\label{testsequence}
Let $\Comp(L)$ be the completion of a good well-structured resolution with completed resolution $$\Comp(L)=Comp(L)_0\xrightarrow{\eta_1}\Comp(L)_1\xrightarrow{\eta_2}\ldots\xrightarrow{\eta_l}\Comp(L)_l.$$ Then there exists a test sequence for $\Comp(L)$.
\end{satz}

\begin{proof}
We construct a test sequence for $\Comp(L)$ iteratively from bottom to top along the completed resolution by essentially combining Theorem \ref{orbifoldflat}, Theorem \ref{testamalgamated}, Theorem \ref{testHNN}, Theorem \ref{testabelian2} and Theorem \ref{testabelian3}.\\
Denote by $\B_i$ the completed decomposition of $\Comp(L)_i$ for all $i\in\{1,\ldots,l\}$. Then $\B_l$ is a Dunwoody decomposition of $\Comp(L)_l$ with one vertex group, say $H$, isomorphic to $\Gamma$ and all other vertex groups isomorphic to subgroups of $\Gamma$. By Theorem \ref{MRdiagram} there exists a restricted locally injective homomorphism $\vp:\Comp(L)_l\to\Gamma$ such that the centralizer in $\Gamma$ of the image of each edge group of $\B_l$ contains a quasi-convex non-abelian free group. We then construct the test sequence $\lambda_n$ of $\Comp(L)_l$ by setting 
$\lambda_n|_H=\id$ and applying Theorem \ref{testamalgamated} and Theorem \ref{testHNN} finitely many times.\\
So now let $i<l$ and assume that we already have constructed a test sequence $$(\lambda_n)\subset\Hom(\Comp(L)_{i+1},\Gamma).$$ Let us first recall how we have constructed $\Comp(L)_i$ and its completed decomposition $\B_i$ out of $\Comp(L)_{i+1}$ and the corresponding completed decomposition $\B_{i+1}$.\\
First we have replaced subgraphs of groups whose vertex groups were isomorphic to subgroups of $\Gamma$ by either isolated QH groups or by isolated virtually abelian groups. Afterwards we have added several virtually abelian flats to this graph of groups. Then we have added several (non-isolated) QH vertex groups, all amalgamated along their extended boundary subgroups and finally we have removed some subgraphs which were contained in the image of these QH subgroups under $\eta_{i+1}$. The resulting graph of groups $\B_i$ is the completed decomposition of $\Comp(L)_i$.\\
So let $\A_1$ be the (possibly disconnected) subgraph of $\B_{i+1}$ which "survives" in $\B_i$, that is all vertices and edges which neither get replaced by isolated QH or virtually abelian groups nor which get removed after the addition of QH vertex groups. We then define our test sequence $\vp_n$ on $\A_1$ to be equal to $\lambda_n$. Let $\A_2$ be the (still possibly disconnected) subgraph of groups of $\B_i$ which we get out of $\A_1$ after having added all the isolated QH and virtually abelian vertex groups. We define the test sequence $\vp_n$ on $\pi_1(\A_2)$ by applying finitely many times Theorem \ref{testabelian3} and Theorem \ref{orbifoldflat}.\\
Now let $\A_3$ be the graph of groups which we get after having added all the virtually abelian flats to the graph of groups $\A_3$. Note that $\A_3$ can still be disconnected. Applying Theorem  \ref{testabelian2} finitely many times yields a test sequence (still denoted) $\vp_n$ for $\pi_1(\A_3)$.  The only thing left is to extend the test sequence onto the (non-isolated) QH vertex groups added in the construction of the test sequence. This is done by applying Theorem \ref{orbifoldflat}. Note that $\eta_{i+1}\circ\lambda_n$ plays the role of the homomorphism $\Theta_n$ in the assumption of Theorem \ref{orbifoldflat}. Hence we end up with a test sequence $\vp_n: \Comp(L)_{i}\to\Gamma$. 
\end{proof}

The following Lemma follows from the construction of the test sequences.

\begin{lemma}\label{opq}
Let $\Comp(\Res(L))$ be a completed resolution of a completed $\Gamma$-limit group $\Comp(L)$ and $s,t\in\Z$. Let moreover $A$ be a virtually abelian vertex group in one of the completed decompositions appearing along $\Comp(\Res(L))$. Denote by $Z_A$ the maximal free abelian subgroup of $A$. Let $a$ be a basis element of $Z_A$. Then there exists a test sequence $(\lambda_n)$ for $\Comp(L)$, integers $m_n\in\Z$ and an element $g\in\Gamma$ which has infinite order and no non-trivial roots, such that $\lambda_n(a)=g^{m_ns+t}$.
\end{lemma}

\section{The Generalized Merzlyakov's Theorem over hyperbolic groups}\label{9}
Having constructed test sequences for completions of well-structured resolutions we are finally ready to prove a generalization of Theorem 1.18 in \cite{sela}, respectively Theorem 2.3 in \cite{selahyp}, that is a Generalized Merzlyakov's Theorem over hyperbolic groups with torsion.

\begin{satz}[Generalized Merzlyakov's Theorem]\label{maintheorem}
Let $\Gamma=\langle a\rangle$ be a hyperbolic group, and let 
$R:=\langle y,a\ |\ V(y,a)\rangle$ be a restricted $\Gamma$-limit group. Denote by $V_y$ the basic definable set corresponding to $R$. Let $\Res(R)$ be a good well-structured resolution of $R$, and let 
$\Comp(\Res(R))$ be the completion of the resolution $\Res(R)$ with corresponding completed $\Gamma$-limit group $\Comp(R)$.
Let $$\Sigma(x,y,a):\ w_1(x,y,a)=1,\ldots, w_s(x,y,a)=1$$ be a system of equations over $\Gamma$ and let $v_1(x,y,a),\ldots,v_r(x,y,a)$ be a collection of words in the alphabet $\{x,y,a\}$. Let $$G_{\Sigma}:=\langle x,y,a\ |\ V(y,a), w_1(x,y,a),\ldots,w_s(x,y,a)\rangle.$$
Suppose that 
$$\Gamma \models \forall y\in V_y\ \exists x:\ w_1(x,y,a)=1,\ldots, w_s(x,y,a)\wedge v_1(x,y,a)\neq 1,\ldots, v_r(x,y,a)\neq 1.$$
Then there exist closures $$\Cl(\Res(R))_1,\ldots,\Cl(\Res(R))_q$$
of $\Res(R)$ and for each index $1\leq i\leq q$, there exists a generalized closure $\GCl(\Res(R))_i$ and a retraction 
$$\pi_i: G_{\Sigma}\ast_{R}\GCl(\Comp(R))_i\to \GCl(\Comp(R))_i.$$
In addition for each $1\leq i\leq q$ there exists a homomorphism $$\Psi_i: \GCl(\Comp(R))_i\to \Gamma$$ that factors through the resolution $\GCl(\Res(R))_i$ such that $\Psi_i\circ\pi_i(v_j)\neq 1$ for all $j\in\{1,\ldots,r\}$. $\pi_i(x)$ is called a formal solution.\\
If $\Gamma$ is torsion-free then in addition $\Cl(\Res(R))_1,\ldots,\Cl(\Res(R))_q$ form a covering closure.
\end{satz}

\begin{bem}
\begin{enumerate}[(1)]
\item We are still not sure if one can possibly always replace the generalized closure in the above theorem by an ordinary closure. In the torsion-free case this is always true. On the other hand if $\Gamma$ is torsion-free hyperbolic and the words $$w_1(x,y),\ldots,w_s(x,y),v_1(x,y),\ldots,v_r(x,y)$$ are coefficient free, then there exist coefficient free formal solutions  but only in a generalized closure and not in an ordinary closure of the corresponding well-structured resolution.
\item We expect the last claim of the theorem to also be true in the case with torsion, but the proof from the torsion-free case does no longer work in this setting.
\item We expect that the assumption that $\Res(R)$ is a good resolution can be dropped by either refining the construction of the MR diagram in \cite{rewe} to guarantee that every strict resolution in this diagram is already good, or by defining a new type of test sequences in the case that $\Res(R)$ is not good.
\end{enumerate}
\end{bem}

\begin{bem} An alternative way to formulate Theorem \ref{maintheorem} is as follows:\\
Let $M_i:=\GCl(\Comp(R))_i$, $\iota_i: R\to M_i$ the canonical embeddings into the closures of the completion of $R$ and $\nu:R\to G_{\Sigma}$ the canonical (injective) map.\\
For all restricted homomorphisms $\vp:R\to\Gamma$ there exists an extension $\tilde{\vp}:M_i\to\Gamma$ of $\vp$ for some $i$ and under the assumption of the theorem there exists an extension $\bar{\vp}:G_{\Sigma}\to\Gamma$ making the following diagram commute.  
$$\begin{xy}
  \xymatrix{
      & G_{\Sigma}\ar[rd]^{\exists\bar{\vp}}   &       \\
      R\ar[ur]_{\nu}\ar[r]_{\iota_i} & M_i \ar[r]_{\tilde{\vp}} & \Gamma
  }
\end{xy}$$
Now the theorem states that there exists a map (not necessarily a retraction) $\pi:G_{\Sigma}\to M_i$ such that both sides of the following diagram commute:
$$\begin{xy}
  \xymatrix{
      & G_{\Sigma}\ar[rd]^{\exists\bar{\vp}}\ar[d]^{\pi}   &       \\
      R\ar[ur]_{\nu}\ar[r]_{\iota_i} & M_i \ar[r]_{\tilde{\vp}} & \Gamma
  }
\end{xy}$$
\end{bem}

To prove Theorem \ref{maintheorem} we have to plug together all results from the previous chapters.

\begin{proof}[Proof of Theorem \ref{maintheorem}]
By Theorem \ref{completion} $\Comp(R)$ is a $\Gamma$-limit tower. We proceed by induction on the height of the tower $\Comp(R)$ over $\Gamma$. If the height is zero then necessarily $\Comp(R)$ is of the form $\Comp(R)=\Gamma$ and the claim is trivial. Hence we may assume that the height of $\Comp(R)$ is at least one.\\
By Theorem \ref{testsequence} there exists a test sequence 
$(\lambda_n)\subset\Hom(\Comp(R),\Gamma)$ for $\Comp(R)$.
For all $n\geq 1$ by the assumption of the theorem we can extend $\lambda_n|_R: R\to \Gamma$ to a homomorphism $\bar{\lambda}_n:G_{\Sigma}\to \Gamma$, where 
$$G_{\Sigma}=\langle x,y,a\ |\ V(y,a),w_1(x,y,a),\ldots,w_s(x,y,a)\rangle,$$
such that $\bar{\lambda}_n(v_j)\neq 1$ for all $j\in\{1,\ldots,r\}$. Note that we can view $R$ as a subgroup of $G_{\Sigma}$. In particular $\bar{\lambda}_n|_R=\lambda_n$.
Moreover for all $n\geq 1$ we can choose $\bar{\lambda}_n$ in such a way that $\bar{\lambda}_n(x)$ is shortest possible with respect to the word metric on $\Gamma$ for the generating set $\{a\}$.
Clearly $\bar{\lambda}_n$ can be extended to a homomorphism $$\vp_n: G:=G_{\Sigma}\ast_R\Comp(R)\to \Gamma,$$ such that $\vp_n|_{\Comp(R)}=\lambda_n$ for all $n\in\N$. 

\begin{definition}
If such a sequence of homomorphisms $(\vp_n)\subset\Hom(G,\Gamma)$ corresponding to a test sequence (and some shortest possible specializations $\{\vp_n(x)\}$) converges, we call the obtained $\Gamma$-limit group $T$ a \textnormal{test limit group} with corresponding \textnormal{test limit map} $\eta:G\to T$.\\ On the set of test limit groups, we define a partial order $\geq$, by setting $T_1\geq T_2$ if there exists an epimorphism $\pi: T_1\to T_2$, such that $\eta^2=\pi\circ\eta^1$, where $\eta^i:G\to T_i$ is the corresponding test limit map for $i\in\{1,2\}$. We further say that $T_1>T_2$ if $T_1\geq T_2$ and $T_2\ngeq T_1$.
\end{definition}

\begin{lemma}\label{maximal1}
There exist maximal elements in the set of all test limit groups, with the partial order defined above.
\end{lemma}

\begin{proof}
Assume that there exists an infinite ascending sequence of test limit groups $$T_1<T_2<T_3<T_4<\ldots$$
For all $i\geq 1$, we denote the test limit map corresponding to $T_i$ by $\eta^i$ and the sequence of homomorphisms from $G$ to $\Gamma$ which converges into $T_i$ by $\vp_n^{(i)}$. Let $B_i$ be the ball of radius $i$ in $\Cay(G)$ around the identity with respect to some fixed finite generating set.\\
Since for all $i\geq 1$, $(\vp_n^{(i)})$ is a stable sequence, there exists $n^{(i)}\in\N$, such that $$\ker (\vp^{(i)}_{n^{(i)}})\cap B_i=\ker(\eta^i)\cap B_i.$$
After possibly increasing some of the $n^{(i)}$'s, the sequence $(\vp^{(i)}_{n^{(i)}})_{i\in\N}$ is an extension of the test sequence $(\lambda_n)$. We denote the corresponding test limit group into which (a subsequence of) $(\vp^{(i)}_{n^{(i)}})$ converges by $T$ and the test limit map by $\eta: G\to T$. It remains to show that $T_i<T$ for all $i\in\N$. Assume that there exists $i_0\in\N$ such that $T_{i_0}\nless T$, i.e. there exists no proper epimorphism $\pi:T\to T_{i_0}$ such that $\eta^{i_0}=\pi\circ\eta$. Hence there exists some element $g\in\ker\eta$, such that $g\notin\ker \eta^{i_0}$. Let $m=\max\{i_0,|g|\}$, where $|\cdot|$ denotes the metric on the Cayley graph of $G$ with respect to the fixed generating set chosen before. Since $T_1<T_2<T_3<T_4<\ldots$ is an infinite ascending sequence of test limit groups, for all $i\geq m$, $g\notin\ker \vp^{(i)}_{n^{(i)}}$ and hence $g\notin \ker \eta$, a contradiction.
\end{proof}

\begin{lemma}\label{maximal}
There exist only finitely many equivalence classes (with respect to the equivalence relation $\leq$) of maximal test limit groups.
\end{lemma}

\begin{proof}
Assume that there exist infinitely many equivalence classes of maximal test limit maps $T_1,T_2,T_3,\ldots$. For all $i\geq 1$, we again denote the test limit map corresponding to $T_i$ by $\eta^i$ and the sequence of homomorphisms from $G$ to $\Gamma$ which converges into $T_i$ by $\vp_n^{(i)}$. Let $B_i$ be the ball of radius $i$ in $\Cay(G)$ around the identity with respect to some fixed finite generating set.\\
Since for all $i\geq 1$, $B_i$ contains only finitely many elements, there exists a sequence of pairwise distinct maximal test limit maps (which we denote for simplicity again by) $(\eta^i)_{i\in\N}$ such that for each $m,n\geq i$ the following hold:
$$\ker \eta^m\cap B_i=\ker \eta^n\cap B_i.$$
As in the proof of Lemma \ref{maximal1} for all $i\geq 1$ there exists $n^{(i)}\in\N$, such that $$\ker (\vp^{(i)}_{n^{(i)}})\cap B_i=\ker(\eta^i)\cap B_i$$
and the sequence $(\vp^{(i)}_{n^{(i)}})_{i\in\N}$ is an extension of the test sequence $(\lambda_n)$ after possibly enlarging some of the $n^{(i)}$'s. Let $T$ be the associated test limit group and $\eta:G\to T$ the corresponding test limit map. Since the sequence $(\eta^i)_{i\in\N}$ consists of pairwise distinct maximal test limit maps, after possibly removing $\eta$ from the sequence, we have that $T\ngeq T_i$ for all $i\geq 1$.\\
Hence for all $i\geq 1$ there exists some element $g_i\in G$, such that $\eta^i(g_i)\neq 1$ and $\eta(g_i)=1$. Since for all $i\geq 1$, $(\vp_n^{(i)})$ is a stable sequence, by possibly enlarging the ${n^{(i)}}$'s we can assume that 
$$\ker (\vp^{(i)}_{n^{(i)}})\cap B_i=\ker(\eta^i)\cap B_i$$
and $\vp^{(i)}_{n^{(i)}}(g_i)\neq 1$. In particular none of the $\vp^{(i)}_{n^{(i)}}$'s factors through $\eta$, since $\eta(g_i)=1$ for all $i\geq 1$. By Theorem 6.5 in \cite{rewe} a subsequence of $(\vp^{(i)}_{n^{(i)}})$ factors through $\eta$ (recall that $\eta$ is the limit map corresponding to the sequence $(\vp^{(i)}_{n^{(i)}})$), a contradiction.
\end{proof}

By Lemma \ref{maximal} there exist only finitely many maximal test limit groups, which we denote by $M_1,\ldots,M_m$.
For all $i\in\{1,\ldots,m\}$ we denote the test limit map corresponding to $M_i$ by $\eta^i$. Moreover for all $i\in\{1,\ldots,m\}$ let $(\vp_n^{(i)})_{n\in\N}\subset\Hom(G,\Gamma)$ be the extension of a test sequence, such that $$M_i=G/\underrightarrow{\ker}(\vp_n^{(i)}).$$
Clearly the images of the words $v_j$ under the test limit maps $\eta^i$ are non-trivial in the maximal test limit groups $M_1,\ldots,M_m$.\\
Let $M\in\{M_1,\ldots,M_m\}$ be one of these maximal test limit groups. For simplicity we denote the extension of a test sequence which converges into the action of $M$ on the associated real tree $T$ by $(\vp_n)$ and the corresponding test limit map by $\eta:G\to M$. We identify $\Comp(R)$ and the words $v_j$ with their images in $M$ under $\eta$. Then there exists a sequence of homomorphisms $\bar{\vp}_n: M\to \Gamma$ such that $\vp_n=\bar{\vp}_n\circ\eta$ for all $n\geq 1$. Let $\D$ be a Dunwoody decomposition of $M$ relative $\Comp(R)$, i.e. $\pi_1(\D)=M$ and  $\Comp(R)$ is contained in a vertex group.\\
Suppose that $\Comp(R)\leq M$ fixes a point in $T$. Denote by $H$ the vertex group of $\D$ containing $\Comp(R)$ and by $D_1,\ldots,D_t$ the vertex group of $\D$ which do not contain $\Comp(R)$ and are not isomorphic to subgroups of $\Gamma$. Suppose that $t\geq 1$, i.e. there exists a vertex group in $\D$ not isomorphic to a subgroup of $\Gamma$, which does not contain $\Comp(R)$. Let $D_{t+1},\ldots,D_s$ be the vertex groups of $\D$ isomorphic to subgroups of $\Gamma$.\\
Let $\Gamma_1,\ldots,\Gamma_s$ be isomorphic copies of $\Gamma$. Let moreover $\D'$ be the graph of groups  which we get by replacing the vertex groups $D_1,\ldots,D_s$ in $\D$ by $\Gamma_1,\ldots,\Gamma_s$ respectively and the vertex group $H$ by $\Gamma$. For all $i\in\{t+1,\ldots,s\}$ denote by $\iota_i:D_i\to\Gamma_i$ the canonical embedding. By the Bestvina-Feighn combination theorem $\pi_1(\D')$ is hyperbolic.\\
Clearly $\bar{\vp}_n\in\Hom(M,\Gamma)$ induces a homomorphism (still denoted by) $$\bar{\vp}_n\in \Hom(M,\pi_1(\D')).$$
For all $n\in\N$ we define the following set:
\begin{align*}
A_n=\{&\Psi_n\in \Hom(M,\pi_1(\D')) \ |\ \Psi_n|_H=\bar{\vp}_n|_H,\  \Psi_n|_{\pi_1(D)}=\id,\\
 &\Psi_n|_{D_i}=\iota_i \text{ for } i\in\{t+1,\ldots,s\}\text{ i.e. if } D_i \text{ is isomorphic to a subgroup of } \Gamma,\\
 &\Psi_n(v_j)\neq 1 \text{ for all } j\in\{1,\ldots,r\}\}.
\end{align*}
Since $\bar{\vp}_n\in A_n$, the set $A_n$ is non-empty. We fix a finite generating set $B$ of $D_t$. Let $U$ be the set of all converging sequences $(\Psi_n)$, such that for all $n\in\N$, $\Psi_n\in A_n$ and  
$$\sum_{b\in B}|\Psi_n(b)|_{\Gamma_t}$$
is minimal (with respect to the word metric on $\Gamma_t$). By the same arguments as in the proof of Lemma \ref{maximal1} and Lemma \ref{maximal} the collection of all sequences from $U$ factor through a finite collection of maximal test limit groups. Let $(\Psi_n)\in U$ be a sequence that converges into such a maximal test limit group, say $N$.\\ 
In particular $(\Psi_n|_{D_t})$ converges either into a finite group or into the action of an infinite $\Gamma$-limit group $\tilde{D}_t$ on some real tree $T$ without a global fixpoint. Assume that $\tilde{D}_t$ is one-ended. By the shortening argument $T$ contains no axial or orbifold components and all stabilizers of edges in the discrete part of the output of the Rips machine are finite. Since $\tilde{D}_t$ does not act with a global fixpoint on $T$, $\tilde{D}_t$ admits a decomposition along a finite subgroup, a contradiction to the one-endedness assumption. Hence $\tilde{D}_t$ is either finite or admits a decomposition along a finite subgroup and therefore $\tilde{D}_t$ is a proper quotient of $D_t$. It follows that $N$ is a proper quotient of the original maximal test limit group $M$.\\
It follows from the descending chain condition for $\Gamma$-limit groups (Proposition \ref{descendingchain}) that after repeating this procedure finitely many times, we end up with finitely many maximal test limit groups $N_1,\ldots, N_k$ such that for all $N\in\{N_1,\ldots,N_k\}$ the Dunwoody decomposition $\D$ relative $\Comp(R)$ of $N$ consists of one vertex with vertex group $H$ which contains $\Comp(R)$ and all other vertex groups are isomorphic to subgroups of $\Gamma$. To simplify notation we assume that our original test limit groups $M_1,\ldots,M_k$ are already of this type.\\
We continue with all the maximal test limit groups $M_1,\ldots,M_k$ in parallel. So let $M\in\{M_1,\ldots,M_k\}$ be such a maximal test limit group with Dunwoody decomposition $\D$ relative $\Comp(R)$. Denote by $v_0$ the unique vertex of $\D$ whose vertex group $H$ contains $\Comp(R)$. Recall that $(\vp_n)\subset \Hom(G,\Gamma)$ is the extension of a test sequence which converges into the action of $M$ on the associated real tree $T$ and $\eta:G\to M$ is the corresponding test limit map. Then there exists a sequence of homomorphisms $\bar{\vp}_n:M\to \Gamma$ such that $\vp_n=\bar{\vp}_n\circ\eta$ for all $n\geq 1$. (In general, i.e. without the simplifying assumption made above, we would have to precompose by a modular automorphism $\alpha_n\in\Mod(M)$, i.e. $\vp_n=\bar{\vp}_n\circ\pi\circ\alpha_n\circ\eta$, where $\pi:M\to N$ is the limit map from $M$ to $N$).\\
Suppose that $\Comp(R)$ is still elliptic when acting on the limit tree $T$. Let $\D'$ be the graph of groups which we get by replacing the vertex group $H$ in $\D$ by $\Gamma$ and all other vertex groups by isomorphic copies of $\Gamma$. By the Bestvina-Feighn combination theorem $\pi_1(\D')$ is hyperbolic. Again $\bar{\vp}_n\in\Hom(M,\Gamma)$ induces a homomorphism (still denoted by) $$\bar{\vp}_n\in \Hom(M,\pi_1(\D')).$$
For all $n\in\N$ we define the following set:
\begin{align*}
B_n=\{ &\Psi_n\in \Hom(M,\pi_1(\D')) \ |\ \Psi_n|_{D_v}=\id \text{ for all } v\in VD\setminus \{v_0\},\\
&\Psi_n|_{\pi_1(D)}=\id,
 \Psi_n(v_j)\neq 1 \text{ for all } j\in\{1,\ldots,r\},\ \Psi_n|_{\Comp(R)}=\bar{\vp}_n|_{\Comp(R)}\}.
\end{align*}
Since $\bar{\vp}_n\in B_n$, the set $B_n$ is non-empty. We fix a generating set $\{h_1,\ldots,h_m\}$ of $H$. Let $U$ be the set of all converging sequences $(\Psi_n)$, such that for all $n\in\N$, $\Psi_n\in B_n$ and  
$$\sum_{i=1}^m|\Psi_n(h_i)|_{\Gamma}$$
is minimal (with respect to the word metric on $\Gamma$). By the same arguments as in the proof of Lemma \ref{maximal1} and Lemma \ref{maximal} the collection of all sequences from $U$ factor through a finite collection of maximal test limit groups. Let $(\Psi_n)\in U$ be a sequence that converges into such a maximal test limit group, say $N$.\\ 
In particular $(\Psi_n|_H)$ converges into the action of a $\Gamma$-limit group $\tilde{H}$ on some real tree $T$ without a global fixpoint. Assume that $\tilde{H}$ is one-ended relative $\Comp(R)$. Suppose that $\Comp(R)$ fixes a point in $T$. Then it follows from the shortening argument (relative $\Comp(R)$) that $T$ contains no axial or orbifold components and all stabilizers of edges in the discrete part in the output of the Rips machine are finite. Since $\tilde{H}$ does not act with a global fixpoint on $T$, $\tilde{H}$ admits a decomposition along a finite subgroup relative $\Comp(R)$, a contradiction to the one-endedness assumption. Hence $\tilde{H}$ is a proper quotient of $H$. It follows that $N$ is a proper quotient of the original maximal test limit group $M$.\\
It follows from the descending chain condition for $\Gamma$-limit groups (Proposition \ref{descendingchain}) that after repeating this procedure finitely many times we end up with finitely many maximal test limit groups $N_1,\ldots, N_k$ such that for all $N\in\{N_1,\ldots,N_k\}$ the Dunwoody decomposition $\D$ relative $\Comp(R)$ of $N$ consists of one vertex with vertex group $H$ which contains $\Comp(R)$, all other vertex groups are isomorphic to subgroups of $\Gamma$ and $\Comp(R)$ does not fix a point when acting on the corresponding limit tree. To simplify notation we assume that our original test limit groups $M_1,\ldots,M_k$ are already of this type.\\
We continue with all the maximal test limit groups $M_1,\ldots,M_k$ in parallel. So let $M\in\{M_1,\ldots,M_k\}$ be such a maximal test limit group with Dunwoody decomposition $\D$ relative $\Comp(R)$. Recall that $(\vp_n)\subset \Hom(G,\Gamma)$ is the extension of a test sequence which converges into the action of $M$ on the associated real tree $T$ and $\eta:G\to M$ is the corresponding test limit map. Then there exists a sequence of homomorphisms $\bar{\vp}_n:M\to \Gamma$ such that $\vp_n=\bar{\vp}_n\circ\eta$ for all $n\geq 1$.\\
The graph of groups $\D$ consists of a unique vertex $v_0$ whose vertex group $H$ contains $\Comp(R)$ and all other vertex groups are isomorphic to subgroups of $\Gamma$. Note that $\bar{\vp}_n$ is still short with respect to the fixed generating set of $H$. Moreover $\Comp(R)$ does not fix a point when acting on $T$.\\
We can describe the action of $M$ on $T$ more precisely. To do this we exploit the fact that $\bar{\vp}_n|_{\Comp(R)}$ is a test sequence.  Recall that $\Comp(R)$ has the structure of a $\Gamma$-limit tower $$\Gamma=T_0\leq T_1\leq\ldots\leq T_n=\Comp(R).$$ Denote by $\B$ the graph of groups decomposition of $T_n$ corresponding to the top level tower structure, i.e. the graph of groups corresponding to the orbifold/virtually abelian flat or extension along a finite group which witnesses the tower structure of $T_n$ over $T_{n-1}$.

\begin{lemma}\label{operation}
\begin{enumerate}[(1)]
\item Assume that $T_n$ has the structure of an orbifold flat over $T_{n-1}$ and let $$1\to E\to Q\to \pi_1(\mO)\to 1$$ be the short exact sequence corresponding to the QH vertex group $Q$ of $\B$. Then $H$ inherits from its action on $T$ an orbifold flat decomposition $\A$ over a group $V$, which has one vertex $v$ with vertex group $Q'$ which contains the subgroup $Q\leq \Comp(R)$ as a subgroup of finite index, and several vertices with vertex groups $V_1,\ldots,V_r$ corresponding to the various orbits of point stabilizers in the action of $H$ on $T$. In addition $Q'$ fits into the short exact sequence $$1\to E'\to Q'\to \pi_1(\mO)\to 1.$$ Moreover the underlying graphs of $\A$ and $\B$ are isomorphic, edge groups in $\A$ are isomorphic to finite index supergroups of the corresponding edge groups in $\B$ and $T_{n-1}\leq V$.
\item Assume that $T_n$ has the structure of a virtually abelian flat $T_n=T_{n-1}\ast_CK$ over $T_{n-1}$. Then there exist finitely many $\Gamma$-limit quotients $N_1,\ldots,N_p$ of the test limit group $M$ and each $N_i$ admits a decomposition $\D$ along finite groups with one distinguished vertex group $U$ and all other vertex groups of $\D$ are isomorphic to subgroups of $\Gamma$. Moreover $U$ admits a decomposition as an amalgamated product of the form $U=V\ast_{A_1}A$, where the subgroup $A$ is virtually abelian and contains the virtually abelian subgroup $K$ as a subgroup of finite index, $A_1$ contains  $C$ as a subgroup of finite index and $T_{n-1}\leq V$.
\item If $T_n=T_{n-1}\ast_{E'}K$ is an extension of $T_{n-1}$ along a finite subgroup and $K$ is isomorphic to a one-ended subgroup of $\Gamma$, then there exist finitely many $\Gamma$-limit quotients $N_1,\ldots,N_p$ of the test limit group $M$ and each $N_i$ admits a decomposition $\D$ along finite groups with one distinguished vertex group $U$ and all other vertex groups of $\D$ are isomorphic to subgroups of $\Gamma$. Moreover $U$ admits a splitting of the form $U=V\ast_E P$, where $E$ is a finite subgroup containing $E'$, $T_{n-1}\leq V$ and $P$ is isomorphic to a subgroup of $\Gamma$ containing $K$.
\item If $T_n=T_{n-1}\ast_{E'}$ is an extension of $T_{n-1}$ along a finite subgroup $E'$ then $H$ admits a splitting $\A$ as an HNN-extension $M=V\ast_E$, where $E$ is a finite subgroup containing $E'$ and $T_{n-1}\leq V$. 
\end{enumerate} 
\end{lemma}

\begin{proof}
This follows immediately from the definition of test sequences (Definition \ref{testseq}).
\end{proof}

To improve readability we split the remaining part of the proof of Theorem \ref{maintheorem} into several lemmas.

\begin{lemma}\label{induction1}
If the action of the maximal test limit group $M$ satisfies the properties of case (3) or (4) in Lemma \ref{operation}, i.e. $Comp(R)=T_n$ is an extension of $T_{n-1}$ along a finite subgroup, then there exists finitely many closures $$\Cl(\Res(R))_1,\ldots,\Cl(\Res(R))_q$$
of $\Res(R)$ and for each index $1\leq i\leq q$, there exists a generalized closure $\GCl(\Res(R))_i$ and a retraction 
$$\pi_i: G_{\Sigma}\ast_{R}\GCl(\Comp(R))_i\to \GCl(\Comp(R))_i.$$
\end{lemma}

\begin{proof}
By assumption either $T_n=T_{n-1}\ast_{E'}K$ or $T_n=T_{n-1}\ast_{E'}$, where $E'$ is some finite group and $K$ is a one-ended subgroup of $\Gamma$. By Lemma \ref{operation} there exist finitely many $\Gamma$-limit quotients $N_1,\ldots,N_k$ of the test limit group $M$ and each $N_i$ admits a decomposition $\D$ along finite groups with one distinguished vertex group $H$ and all other vertex groups of $\D$ are isomorphic to subgroups of $\Gamma$.
Moreover $H$ admits a splitting as $H=V\ast_EP$ in the first case and as $H=V\ast_E$ in the second case, where $T_{n-1}\leq V$, $K\leq P$, $P$ is isomorphic to a subgroup of $\Gamma$ and $E'\leq E$, where $E$ is some finite group. 
To simplify notation we assume that $\{M_1,\ldots,M_k\}$ is already of this type. So let $M\in\{M_1,\ldots,M_k\}$.\\
We only consider the amalgamated product case, the HNN case is similar.
Let $C_n$ be the following set:
\begin{align*}
C_n=\{ &\Psi_n\in \Hom(M,\pi_1(\D')) \ |\ \Psi_n|_{D_v}=\id \text{ for all } D_v \text{ isomorphic to a subgroup of } \Gamma,\\
&\Psi_n|_{\pi_1(D)}=\id,\
 \Psi_n(v_j)\neq 1 \text{ for all } j\in\{1,\ldots,r\},\ \Psi_n|_P=\bar{\vp}_n|_P,\\
 &\Psi_n|_{\Comp(R)}=\bar{\vp}_n|_{\Comp(R)}\}.
\end{align*}
Since $\bar{\vp}_n\in C_n$ the set $C_n$ is non-empty. We fix a generating set $\{u_1,\ldots,u_n\}$ of $V$. Let $U$ be the set of all converging sequences $(\Psi_n)$, such that for all $n\in\N$, $\Psi_n\in C_n$ and  
$$\sum_{i=1}^n|\Psi_n(u_i)|_{\Gamma}$$
is minimal (with respect to the word metric on $\Gamma$). By the same arguments as in the proof of Lemma \ref{maximal1} and Lemma \ref{maximal} the collection of all sequences from $U$ factor through a finite collection of maximal test limit groups. Let $(\Psi_n)\in U$ be a sequence that converges into such a maximal test limit group, say $N$.\\ 
In particular $(\Psi_n|_V)$ converges into the action of a $\Gamma$-limit group $\tilde{V}$ on some real tree $T$ without a global fixpoint. 
Moreover $(\Psi_n|_V)$ is the extension of a test sequence of $T_{n-1}$.  Since $T_{n-1}$ is a tower of height $n-1$, we can apply the induction hypothesis. Denote by $\tilde{H}$ the limit group into which $\Psi_n|_H$ converges. It follows that $\tilde{H}$ admits a decomposition along finite edge groups with one distinguished vertex $h_0$ whose vertex group $H_0$ is a closure of the tower $T_n=\Comp(R)$ and all other vertices have vertex groups isomorphic to subgroups of $\Gamma$. Therefore $N$ admits a decomposition along finite edge groups with one distinguished vertex whose vertex group is a closure of the tower $T_n=\Comp(R)$ and all other vertices have vertex groups isomorphic to subgroups of $\Gamma$. Hence we can assume that every extension of a test sequence to a convergent sequence of homomorphisms from $G$ to $\Gamma$ factors through one of finitely many maximal test limit groups (still denoted) $M_1,\ldots,M_m$ with corresponding quotient maps $\eta_1,\ldots,\eta_m$. Each $M_i$ admits a decomposition $\A$ along finite edge groups with one vertex $v_0$ with vertex group $\Cl(\Comp(R))_i$ which is a closure of $\Comp(R)$ and all other vertex groups are isomorphic to subgroups of $\Gamma$. By construction the images of the words $v_j$ under the maps $\eta_1,\ldots,\eta_m$ are non-trivial in these test limit groups.\\The map $\eta_i: G\to \GCl(\Comp(R))_i$ can be extended in the obvious way to a retraction $$G_{\Sigma}\ast_{R}\GCl(\Comp(R))_i\to \GCl(\Comp(R))_i$$
which maps the words $v_j$ to non-trivial elements.
\end{proof}

%

\begin{lemma}
If the action of the maximal test limit group $M$ satisfies the properties of case (1) in Lemma \ref{operation}, i.e. $Comp(R)=T_n$ is an orbifold flat over $T_{n-1}$, then there exist finitely many closures $$\Cl(\Res(R))_1,\ldots,\Cl(\Res(R))_q$$
of $\Res(R)$ and for each index $1\leq i\leq q$, there exists a generalized closure $\GCl(\Res(R))_i$ and a retraction 
$$\pi_i: G_{\Sigma}\ast_{R}\GCl(\Comp(R))_i\to \GCl(\Comp(R))_i.$$
\end{lemma}

\begin{proof}
By assumption $T_n$ is an orbifold flat over $T_{n-1}$, i.e. there exists some finite group $N$, a cone-type orbifold $\mO$ and a short exact sequence 
$$1\to N\hookrightarrow Q\twoheadrightarrow \pi_1(\mO)\to 1,$$
such that $T_n$ admits a graph of groups decomposition $\B$ with one vertex $q$ with vertex group $Q$ and several vertices $v_1,\ldots, v_r$ with vertex groups $H_1,\ldots, H_r$. Moreover $T_{n-1}$ admits a decomposition as a graph of groups with finite edge groups, whose vertex groups are precisely $H_1,\ldots,H_r$. Edges adjacent to $q$ in $\B$ are either finite or have virtually cyclic edge groups which correspond to boundary components of $\mO$ and every boundary component of $\mO$ corresponds to precisely one adjacent edge.\\
Since we are in case (1) of Lemma \ref{operation} $H$ inherits from its action on $T$ an orbifold flat decomposition $\A$ over $V$ with the same underlying graph as $\B$. One vertex $q'$ in $\A$ has vertex group $Q'$, which fits into the short exact sequence
$$1\to N'\to Q'\to\pi_1(\mO)\to 1,$$
where $N'$ is a finite group, $N\leq N'$ and the other vertices of $\A$ have vertex groups $V_1,\ldots,V_r$ and $H_i\leq V_i$ for all $i\in\{1,\ldots, r\}$. Edges adjacent to $q'$ are isomorphic to finite index supergroups of the corresponding ones in $\B$. In particular $T_{n-1}\leq V$.\\
Let $C_n$ be the following set:
\begin{align*}
C_n=\{ &\Psi_n\in \Hom(M,\pi_1(\D')) \ |\ \Psi_n|_{D_v}=\id \text{ for all } v\in VD\setminus \{v_0\},\\
&\Psi_n|_{\pi_1(D)}=\id,\
 \Psi_n(v_j)\neq 1 \text{ for all } j\in\{1,\ldots,r\},\ \Psi_n|_{Q'}=\bar{\vp}_n|_{Q'},\\
&\Psi_n|_{\Comp(R)}=\bar{\vp}_n|_{\Comp(R)}\}.
\end{align*}
Since $\bar{\vp}_n\in C_n$ the set $C_n$ is non-empty. We fix a generating set $\{u_1,\ldots,u_n\}$ of $V$. Let $U$ be the set of all converging sequences $(\Psi_n)$, such that for all $n\in\N$, $\Psi_n\in C_n$ and  
$$\sum_{i=1}^n|\Psi_n(u_i)|_{\Gamma}$$
is minimal (with respect to the word metric on $\Gamma$). By the same arguments as in the proof of Lemma \ref{maximal1} and Lemma \ref{maximal} the collection of all sequences from $U$ factor through a finite collection of maximal test limit groups. Let $(\Psi_n)\in U$ be a sequence that converges into such a maximal test limit group, say $N$.\\ 
In particular $(\Psi_n|_V)$ converges into the action of a $\Gamma$-limit group $\tilde{V}$ on some real tree $T$ without a global fixpoint. 
Moreover $(\Psi_n|_V)$ is the extension of a test sequence of $T_{n-1}$. Since $T_{n-1}$ is a tower of height $n-1$, we can apply the induction hypothesis and the claim follows in the same way as in the proof of Lemma \ref{induction1}. 
\end{proof}

\begin{lemma}
If the action of the maximal test limit group $M$ satisfies the properties of case (2) in Proposition \ref{operation}, i.e. $Comp(R)=T_n$ is a virtually abelian flat over $T_{n-1}$, then there exist finitely many closures $$\Cl(\Res(R))_1,\ldots,\Cl(\Res(R))_q$$
of $\Res(R)$ and for each index $1\leq i\leq q$, there exists a generalized closure $\GCl(\Res(R))_i$ and a retraction 
$$\pi_i: G_{\Sigma}\ast_{R}\GCl(\Comp(R))_i\to \GCl(\Comp(R))_i.$$
\end{lemma}

\begin{proof}
By assumption $T_n$ has the structure of a virtually abelian flat $T_n=T_{n-1}\ast_CK$ over $T_{n-1}$. It follows from Lemma \ref{operation} that there exist finitely many $\Gamma$-limit quotients $N_1,\ldots,N_k$ of the test limit group $M$ and each $N_i$ admits a decomposition $\D$ along finite groups with one distinguished vertex group $H_0$ and all other vertex groups of $\D$ are isomorphic to subgroups of $\Gamma$. Moreover $H_0$ admits a decomposition as an amalgamated product of the form $H_0=V\ast_{A_1}A$, where the subgroup $A$ is virtually abelian and contains the virtually abelian subgroup $K$ as a subgroup of finite index, $A_1$ contains  $C$ as a subgroup of finite index and $T_{n-1}\leq V$.
To simplify notation we assume that $\{M_1,\ldots,M_k\}$ is already of this type. So let $M\in\{M_1,\ldots,M_k\}$.\\
Let $C_n$ be the following set:
\begin{align*}
C_n=\{ &\Psi_n\in \Hom(M,\pi_1(\D')) \ |\ \Psi_n|_{D_v}=\id \text{ for all } D_v \text{ isomorphic to a subgroup of } \Gamma,\\
&\Psi_n|_{\pi_1(D)}=\id,\
 \Psi_n(v_j)\neq 1 \text{ for all } j\in\{1,\ldots,r\},\ \Psi_n|_A=\bar{\vp}_n|_A,\\
 &\Psi_n|_{\Comp(R)}=\bar{\vp}_n|_{\Comp(R)}\}
\end{align*}
Since $\bar{\vp}_n\in C_n$ the set $C_n$ is non-empty. We fix a generating set $\{u_1,\ldots,u_n\}$ of $V$. Let $U$ be the set of all converging sequences $(\Psi_n)$, such that for all $n\in\N$, $\Psi_n\in C_n$ and  
$$\sum_{i=1}^n|\Psi_n(u_i)|_{\Gamma}$$
is minimal (with respect to the word metric on $\Gamma$). By the same arguments as in the proof of Lemma \ref{maximal1} and Lemma \ref{maximal} the collection of all sequences from $U$ factor through a finite collection of maximal test limit groups. Let $(\Psi_n)\in U$ be a sequence that converges into such a maximal test limit group, say $N$.\\ 
In particular $(\Psi_n|_V)$ converges into the action of a $\Gamma$-limit group $\tilde{V}$ on some real tree $T$ without a global fixpoint. 
Moreover $(\Psi_n|_V)$ is the extension of a test sequence of $T_{n-1}$.  Since $T_{n-1}$ is a tower of height $n-1$, we can apply the induction hypothesis. Denote by $\tilde{H}_0$ the limit group into which $\Psi_n|_{H_0}$ converges. It follows that $\tilde{H}_0$ admits a decomposition along finite edge groups with one distinguished vertex $h_1$ whose vertex group $H_1$ is a closure of the tower $T_n=\Comp(R)$ and all other vertices have vertex groups isomorphic to subgroups of $\Gamma$. Therefore $N$ admits a decomposition along finite edge groups with one distinguished vertex whose vertex group is a closure of the tower $T_n=\Comp(R)$ and all other vertices have vertex groups isomorphic to subgroups of $\Gamma$. Hence we can assume that every extension of a test sequence to a convergent sequence of homomorphisms from $G$ to $\Gamma$ factors through one of finitely many maximal test limit groups (still denoted) $M_1,\ldots,M_q$, with corresponding quotient maps $\eta_1,\ldots,\eta_q$. Each $M_i$ admits a decomposition $\A_i$ along finite edge groups with one vertex $v_0$ with vertex group $\Cl(\Comp(R))_i$ which is a closure of $\Comp(R)$ and all other vertex groups are isomorphic to subgroups of $\Gamma$. By construction the images of the words $v_j$ under the maps $\eta_1,\ldots,\eta_q$ are non-trivial in these test limit groups. The claim follows.
\end{proof}

Hence in all cases there exists for every $i\in\{1,\ldots,m\}$ some homomorphism $$\Psi_i: \GCl(\Comp(R))_i\to \Gamma,$$ such that $\Psi_i\circ\eta_i(v_j)\neq 1$ in $\Gamma$. It remains to show that if $\Gamma$ is torsion-free the finitely many closures $\Cl(\Res(R))_1,\ldots, \Cl(\Res(R))_q$ form a covering closure.\\
So let from now on $\Gamma$ be a torsion-free hyperbolic group and suppose that $$\Cl(\Res(R))_1,\ldots, \Cl(\Res(R))_q$$ do not form a covering closure. Note that in particular $R$ and $\Comp(R)$ are torsion-free.
Then there exists a homomorphism $\vp: R\to\Gamma$ that factors through $\Res(R)$ which cannot be extended to a homomorphism of any of the closures $\Cl(\Res(R))_i$. By Lemma \ref{completion property} $\vp$ can be extended to a homomorphism (still denoted $\vp$) from $\Comp(R)$ to $\Gamma$ which factors through the completed resolution $\Comp(\Res(R))$. $\Comp(R)$ has the structure of a $\Gamma$-limit tower  
$$\Comp(R)=T_l\geq\ldots\geq T_1\geq T_0=\Gamma$$
and hence there are two possibilities:
\begin{enumerate}[(a)]
\item There exists $i\in\{1,\ldots,l\}$ such that $T_i=T_{i-1}\ast_CA$ is an abelian flat, i.e. either $C$ is trivial or infinite free abelian and $A=\langle a_1,\ldots,a_m\rangle$ is a free abelian subgroup of rank $m\geq 2$, and $\vp|_A$ cannot be extended to the closure of $A$ corresponding to any of the closures $\Cl(\Comp(R))_1,\ldots,\Cl(\Comp(R))_q$. 
\item  There exists $i\in\{1,\ldots,l\}$ such that $T_i=T_{i-1}\ast H$, where $H$ is isomorphic to a subgroup of $\Gamma$ and $\vp|_H$ cannot be extended to the closure of $H$ corresponding to any of the closures $\Cl(\Comp(R))_1,\ldots,\Cl(\Comp(R))_q$. 
\end{enumerate}
Assume we are in the first case. Denote by $f_1,\ldots,f_q$ the respective embeddings of the free abelian group $A$ into $\Cl(\Comp(R))_1,\ldots,\Cl(\Comp(R))_q$. For all $i\in\{1,\ldots,q\}$ there exists a maximal free abelian subgroup $\bar{A}_i\leq \Cl(\Comp(R))_i$ that contains $f_i(A)$ as a subgroup of finite index. Following Remark \ref{abeliantorsionfree} we can either associate a subgroup $U_{f_i}\leq \Z^m$ to every embedding $f_i$ (if $C=1$) or we can associate a coset $K_{f_i}+U_{f_i}\leq \Z^m$ to every embedding $f_i$ (if $C\neq 1$). We assume without loss of generality that we can associate a subgroup $U_{f_i}\leq \Z^m$ to every embedding $f_i$ (the other case is identical). Since  $\Cl(\Res(R))_1,\ldots,\Cl(\Res(R))_q$ do not form a covering closure, the union of the subgroups $\bigcup_{i=1}^q U_{f_i}$ does not cover $\Z^m$ by Remark \ref{abeliantorsionfree}. Hence there exists an element
$$p=(p_1,\ldots,p_m)\in\Z^m\setminus \bigcup_{i=1}^q U_{f_i}.$$
Let $d_i:=|\Z^m:U_{f_i}|$ for all $i\in\{1,\ldots,q\}$. Then clearly for every element $v\in\Z^m$
$$d_1\cdots d_qv+p\notin U_{f_1},\ldots,U_{f_q}.$$
By Lemma \ref{opq} for all $n\in\N$ there exist elements $v^{(n)}\in \Z^m$ and a test sequence $(\lambda_n)\subset\Hom(\Comp(R),\Gamma)$, such that $$\lambda_n(a_i)=u^{d_1\cdots d_qv^{(n)}_i+p_i}$$ for all $i\in\{1,\ldots,m\}$ and some element $u\in\Gamma$ with no root. By the assumption of the theorem $\lambda_n$ can be extended to a homomorphism (still denoted) $\lambda_n:G_{\Sigma}\to\Gamma$ such that $\vp_n(v_j)\neq 1$ for all $j\in\{1,\ldots,r\}$. By the construction of $\Cl(\Res(R))_1,\ldots,\Cl(\Res(R))_q$ in the proof, there exist $i\in\{1,\ldots,q\}$, quotient maps $\eta_1,\ldots,\eta_p$, modular automorphism $\alpha_1,\ldots,\alpha_p$ and a homomorphism $\bar{\lambda}_n:\Cl(\Comp(R))_i\to\Gamma$ such that 
$$\lambda_n=\bar{\lambda}_n\circ\eta_p\circ\alpha_p\circ\ldots\circ\eta_1\circ\alpha_1.$$
Moreover $\eta_p\circ\alpha_p\circ\ldots\circ\eta_1\circ\alpha_1|_A=f_i$. Hence $\lambda_n=\bar{\lambda}_n\circ f_i$ and therefore by Remark \ref{abeliantorsionfree} $d_1\cdots d_qv^{(n)}+p\in U_{f_i}$ a contradiction to the choice of $p$. Hence $$\Cl(\Res(R))_1,\ldots, \Cl(\Res(R))_q$$  is a covering closure.\\
So now assume that we are in the second case. Again denote by $f_1,\ldots,f_q$ the respective embeddings of the subgroup $H$ into $\Cl(\Comp(R))_1,\ldots,\Cl(\Comp(R))_q$. For all $i\in\{1,\ldots,q\}$ denote by $V_i$ the maximal subgroup of $\Cl(\Comp(R))_i$ which is isomorphic to a subgroup of $\Gamma$ and contains $f_i(H)$. By assumption $\vp|_H$ cannot be extended to $V_1,\ldots,V_q$. Let $(\lambda_n)\subset\Hom(\Comp(R),\Gamma)$ be a test sequence. Then $\lambda_n|_H$ is an embedding of $H$ into $\Gamma$. Changing $\lambda_n|_H$ to $\vp|_H$ yields a new test sequence, say $(\Psi_n)$.\\
Again by the construction of $\Cl(\Res(R))_1,\ldots,\Cl(\Res(R))_q$ in the proof, there exist $i\in\{1,\ldots,q\}$, quotient maps $\eta_1,\ldots,\eta_p$, modular automorphism $\alpha_1,\ldots,\alpha_p$ and a homomorphism $\bar{\Psi}_n:\Cl(\Comp(R))_i\to\Gamma$ such that 
$$\Psi_n=\bar{\Psi}_n\circ\eta_p\circ\alpha_p\circ\ldots\circ\eta_1\circ\alpha_1.$$
Moreover $\eta_p\circ\alpha_p\circ\ldots\circ\eta_1\circ\alpha_1|_H=f_i$. Hence $\Psi_n|_H=\bar{\Psi}_n\circ f_i$ and therefore 
$$\vp|_H=\Psi_n|_H=\bar{\Psi}_n|_{V_i}\circ f_i,$$ a contradiction to the assumption that $\vp|_H$ cannot be extended to $V_i$. We have shown that $$\Cl(\Res(R))_1,\ldots, \Cl(\Res(R))_q$$ is a covering closure.
\end{proof}

\end{document}